\documentclass[reqno,10pt]{amsart}

\usepackage{amsmath}
\usepackage{amsfonts,amsthm,amssymb}
\usepackage{mathabx}
\usepackage{indentfirst}
\usepackage{graphicx}
\usepackage{graphics}
\usepackage{pict2e}
\usepackage{epic}
\numberwithin{equation}{section}
\usepackage[margin=1in]{geometry}
\usepackage{epstopdf} 
\usepackage[colorlinks,linkcolor=blue]{hyperref}

\usepackage[capitalise,noabbrev]{cleveref}

\usepackage{cleveref}  

\crefformat{equation}{(#2#1#3)}
\crefrangeformat{equation}{(#3#1#4)--(#5#2#6)}
\newsavebox{\dummybox}

\newcommand{\fullrangefour}[4]{%
	(\ref{#1})--(\ref{#4})
	\sbox{\dummybox}{\ref{#2}\ref{#3}}
}

\usepackage{todonotes}
\makeatletter
\def\thanks#1{\protected@xdef\@thanks{\@thanks
		\protect\footnotetext{#1}}}
\makeatother

\usepackage[normalem]{ulem}

\allowdisplaybreaks

\theoremstyle{plain}
\newtheorem{Thm}{Theorem}[section]
\newtheorem*{Thm*}{Theorem}
\newtheorem{Lem}[Thm]{Lemma}
\newtheorem{Cor}[Thm]{Corollary}
\newtheorem{Prop}[Thm]{Proposition}

\newtheorem{Def}[Thm]{Definition}

\newtheorem{Rem}[Thm]{Remark}
\newtheorem{?}[Thm]{Problem}

\newcommand{\aaa}{\mathbf{a}}
\newcommand{\bb}{\mathbf{b}}
\newcommand{\cc}{\mathbf{c}}

\newcommand{\Ff}{\mathbf{F}}
\newcommand{\Ee}{\mathbf{E}}
\newcommand{\Pp}{\mathbf{P}}

\newcommand{\p}{\partial}
\newcommand{\pt}{\partial_t}

\newcommand{\R}{\mathbb{R}}
\newcommand{\T}{\mathbf{T}}
\newcommand{\E}{\mathbb{E}}

\newcommand{\Z}{\mathbf{Z}}
\newcommand{\D}{\mathbf{D}}
\newcommand{\Do}{\mathbf{D}_0}
\newcommand{\Dn}{\mathbf{D}_{\neq}}
\newcommand{\Et}{\tilde{\mathbb{E}}}

\newcommand{\Eb}{\bar{\mathbb{E}}}
\newcommand{\Ec}{\check{\mathbb{E}}}
\newcommand{\Em}{\mathring{\mathbb{E}}}

\newcommand{\Tb}{\bar{\theta}}
\newcommand{\Tc}{\check{\theta}}

\newcommand{\varphim}{\mathring{\varphi}}

\newcommand{\Wbj}{\bar{\mathbf{W}}^{\#}}
\newcommand{\W}{\mathbf{W}}
\newcommand{\Abj}{\bar{\mathcal{A}}^{\#}}
\newcommand{\Bbj}{\bar{\mathcal{B}}^{\#}}
\newcommand{\Rm}{\mathcal{R}}
\newcommand{\vm}{\mathbf{v}}
\newcommand{\Vm}{\mathbf{V}}

\newcommand{\Vmc}{\check{\mathbf{V}}}

\newcommand{\Psic}{\check{\Psi}}
\newcommand{\Phic}{\check{\Phi}}

\newcommand{\Hc}{\check{H}}
\newcommand{\Dt}{\tilde{D}}
\newcommand{\Mt}{\tilde{M}}

\newcommand{\Mb}{\bar{M}}
\newcommand{\Pb}{\bar{P}}
\newcommand{\Wb}{\bar{W}}

\newcommand{\U}{\mathbf{U}}

\newcommand{\uf}{u}

\newcommand{\f}{\mathbf{f}}

\newcommand{\Torus}{\mathbb{T}}

\newcommand{\ra}{\rangle}
\newcommand{\la}{\langle}

\newcommand{\dv}{\text{div}}

\newcommand{\B}{\mathbf{B}}

\newcommand{\rhob}{\bar{\rho}}
\newcommand{\rhoc}{\check{\rho}}
\newcommand{\rhot}{\tilde{\rho}}

\newcommand{\rhom}{\mathring{\rho}}

\newcommand{\ub}{\bar{u}}
\newcommand{\uc}{\check{u}}
\newcommand{\ut}{\tilde{u}}

\newcommand{\um}{\mathring{u}}

\newcommand{\mb}{\bar{m}}
\newcommand{\mc}{\check{m}}
\newcommand{\mt}{\tilde{m}}

\newcommand{\mm}{\mathring{m}}

\newcommand{\thetat}{\tilde{\theta}}
\newcommand{\thetab}{\bar{\theta}}
\newcommand{\Thetab}{\bar{\Theta}}

\newcommand{\thetam}{\mathring{\theta}}

\newcommand{\deltab}{\bar{\delta}}
\newcommand{\deltac}{\check{\delta}}

\newcommand{\phim}{\mathring{\phi}}
\newcommand{\psim}{\mathring{\psi}}
\newcommand{\zetam}{\mathring{\zeta}}

\newcommand{\abs}[1]{\left\lvert#1\right\rvert}

\newcommand{\norm}[1]{\left\lVert#1\right\rVert}

\begin{document}
	
	\title[Stability and decay rate of planar entropy wave for Landau equation]{Nonlinear stability and optimal decay rate of the planar entropy wave for Landau equation}
	
	\author[R.-J. Duan]{Renjun Duan}
	\address[R.-J. Duan]{Department of Mathematics, The Chinese University of Hong Kong, Shatin, Hong Kong, China}
	\email{rjduan@math.cuhk.edu.hk}

	\author[F.~Huang]{Feimin Huang}
	\address[F.~Huang]{Academy of Mathematics and Systems Science, Chinese Academy of Sciences, Beijing 100190, China, and School of Mathematical Sciences, University of Chinese Academy of Sciences, Beijing 100049, China.}
	\email{fhuang@amt.ac.cn}

	\author[R.~Li]{Rui Li}
	\address[R.~Li]{Department of Mathematics, The Chinese University of Hong Kong, Shatin, Hong Kong, China.}
	\email{ruili001@cuhk.edu.hk}

	\author[L.~Xu]{Lingda Xu}
	\address[L.~Xu]{Department of Applied Mathematics, The Hong Kong Polytechnic University, Hong Kong, China.}
	\email{lingda.xu@polyu.edu.hk}
	
	\begin{abstract}
		
		This paper investigates the nonlinear asymptotic stability and optimal decay rates of entropy waves for the Landau equation with physically realistic Coulomb interactions under general perturbations. We consider the infinite channel domain $\mathbb{R} \times \mathbb{T}^2$ in three dimensions, which possesses both one-dimensional and high-dimensional   characteristics, thereby posing  two primary analytical challenges: (i) for the one-dimensional Landau equation with Coulomb potentials, the absence of a spectral gap in the linearized operator has obstructed the derivation of wave pattern stability results with explicit time decay rates; (ii) in the study of contact discontinuities, the multidimensional case fundamentally differs from the one-dimensional setting due to lack of a key structural condition. We develop effective analytical approaches to treat those difficulties. To overcome the weak dissipation caused by the spectral gap deficiency, we implement a time-velocity interpolation  technique to enhance dissipation and simultaneously construct coupled diffusion waves to compensate for the loss of time decay. To address the missing structural condition in higher dimensions, a novel transformation is introduced to recover the two-sided structural condition within the perturbation system. By developing a derivative-level transformation and a refined energy framework, we restore the necessary structural condition for derivatives, establish the optimal decay of the solution, and prove the stretched exponential decay of its non-zero modes. 
		
	\end{abstract}
	
	\date{\today}
	
	\subjclass[2020]{35Q84, 35Q20, 35L65, 35B40, 35B20}
	
	
	\keywords{Landau equation, planar entropy wave,  nonlinear stability, optimal decay rates, stretched exponential decay}
	
	\maketitle
	
	\setcounter{tocdepth}{2}
	\tableofcontents
	\thispagestyle{empty}
	
	\section{Introduction}
\subsection{Formulation of the problem}
The Landau equation in three spatial dimensions reads
\begin{align}\label{equ-landau}
	&\p_t{f} + {\xi }\cdot\nabla_x{f} =  Q\left( {f,f}\right) ,
\end{align}
where \(f\left( {t,x,\xi }\right)  \geq  0\) is the density distribution function for the particles with velocity \(\xi  = \left( {{\xi }_{1},{\xi }_{2},{\xi }_{3}}\right)  \in  {\R}^{3}\) at time \(t \geq  0\) and position \(x=(x_1,x_2,x_3) \in \Omega:= \R \times\Torus^2\), where $\Torus:=\R/\mathbb{Z}$. Through the paper, we only consider the Landau collision operator for physically most relevant Coulomb potentials:
\begin{align}\label{landau-operator}
	Q\left( {f,g}\right) \left( \xi \right)  = \mathop{\sum }\limits_{{i,j = 1}}^{3}{\partial }_{\xi_i}{\int }_{{\R}^{3}}\tilde{\phi}^{ij}\left( {\xi  - {\xi }^{\prime }}\right) \left\{  {f\left( {\xi }^{\prime }\right) {\partial }_{\xi_j}g\left( \xi \right)  - {\partial }_{\xi_j}f\left( {\xi }^{\prime }\right) g\left( \xi \right) }\right\}  d{\xi }^{\prime },
\end{align}
where the collision kernel is given by the nonnegative definite matrix-valued function \(\tilde\phi(z)=(\tilde\phi(z))_{1\leq i,j\leq 3}\) with $z=\xi  - {\xi }^{\prime }$ taking the form
\begin{align}\notag
	\tilde{\phi}^{ij}\left( z\right)  = \left\{  {{\delta }^{ij} - \frac{{z }_{i}{z }_{j}}{{\left| z \right| }^{2}}}\right\}  {\left| z \right| }^{\gamma  + 2},\;\gamma  =  - 3.
\end{align}
We also call very soft potentials in case $-3\leq \gamma<-2$ including the Coulomb potentials given above. Since their derivation, the Boltzmann and Landau equations have been closely linked to fluid dynamics, especially the Euler and Navier-Stokes equations. Early attempts to derive fluid equations from kinetic theory trace back to Maxwell and Boltzmann, who used intuitive, somewhat ad hoc arguments. To formalize this, Hilbert introduced a systematic expansion method in 1912, later refined independently by Enskog (1916) and Chapman (1917); see the recent breakthrough \cite{DHM} by Deng-Hani-Ma.

The Riemann problem, first studied in the 1860s for one-dimensional isentropic flow, is foundational to hyperbolic conservation laws. Its solution captures the local and global nonlinear structure of such systems. The Euler equations feature three basic waves: compressive shocks, expansive rarefactions, and contact discontinuities. The last one splits into entropy waves (discontinuous entropy) and vortex sheets (discontinuous shear velocity), both with a diffusive structure. Consequently, analyzing the hydrodynamic limit of the Boltzmann or Landau equation for a full Riemann solution --- the  superposition of all three wave patterns --- remains  a major mathematical challenge in kinetic theory; see \cite{HWWY} and references therein.

In this paper, we are interested in the Cauchy problem related to the basic wave patterns for \eqref{equ-landau} supplemented with the following initial condition
\begin{equation}\label{ini}
	f(0,x,\xi) = f_0(x,\xi).
\end{equation}
We assume that $f_0(x,\xi)$ connects two different global Maxwellians at the far fields of the first spatial direction \( x_1 = \pm\infty \):
\begin{equation}\notag
	f_0(x,\xi) \to \frac{\rho_\pm}{(2\pi R \theta_\pm)^{3/2}} \exp\left( -\frac{|\xi - u_\pm|^2}{2 R \theta_\pm} \right)
	\quad \text{as}\quad x_1 \to \pm\infty,
\end{equation}
where \( \rho_\pm > 0 \), \( \theta_\pm > 0 \), and  \( u_\pm = (u_{1\pm}, 0, 0)  \) with \( u_{1\pm} \) are constants independent of $(x_2,x_3)\in \Torus^2$, and \( R>0 \) is the gas constant. To the end, we study only the entropy wave for which the far-field data above satisfy 
\begin{align}\label{eqs11.se1.1}
	\rho_-\neq\rho_+\qquad\qquad u_{1-}=u_{1+},\qquad\quad R\rho_+\theta_+=p_+=p_-=R\rho_-\theta_-.
\end{align}
Next, we shall start from the micro-macro decomposition of \eqref{equ-landau}.

\subsection{The micro-macro decomposition}

It is well known  that the collision invariants \({\tilde{\psi}}_{\alpha }\left( \xi \right)\) for the Landau collision operator \eqref{landau-operator} are given by
\begin{align}\notag
	\tilde{\psi }_{0}\left( \xi \right)  =  1, \quad\quad\tilde{\psi }_{i}\left( \xi \right)  =  {\xi }_{i},\;\text{ for }i = 1,2,3, \quad\quad  \tilde{\psi }_{4}\left( \xi \right)  =  \frac{1}{2}{\left| \xi \right| }^{2}, 
\end{align}
satisfying
\begin{align}\notag
	{\int }_{{\R}^{3}}\tilde{\psi }_{i}\left( \xi \right) Q\left( {f,f}\right) {d\xi } = 0,\text{ for }i = 0,1,2,3,4.
\end{align}
Motivated by \cite{Liu-Yang-Yu}, we decompose the solution of Landau equation \eqref{equ-landau} as
\begin{align}\label{assumption-1}
	f\left( {t,x,\xi }\right)  = M_{[\rho,u, \theta]}\left( {t,x,\xi }\right)  +  G\left( {t,x,\xi }\right) ,
\end{align}
where the local Maxwellian \(M_{[\rho,u,\theta]}\) and \(G\) represent the macroscopic and microscopic component in the solution, respectively. Precisely,  the local Maxwellian \(M_{[\rho,u,\theta]}\) is defined as 
\begin{align}\label{2025.6.14-1}
	M =  {M}_{\left[ \rho , u, \theta \right] }\left( t,x,\xi \right):=  \frac{\rho \left( t,x \right) }{\sqrt{\left( 2\pi R \theta \left( t,x\right) \right) ^{3}}} \exp{ \left( -\frac{\left| \xi  - u \left( t,x \right) \right| ^{2}}{2R \theta \left( t,x\right)}\right)},  
\end{align}
in terms of five conserved quantities mass density \(\rho \left( {t,x}\right)\), momentum \(m\left( {t,x}\right)  :=(\rho u) \left( {t,x}\right)\), and energy density \(\mathbb{E}\left( {t,x}\right)  :=  \rho\left( {t,x}\right) \left( {e\left( {t,x}\right)  + \frac{1}{2}{\left| u\left( t,x\right) \right| }^{2}}\right)\) given by
\begin{align*}
	\left\{ \begin{aligned}
		&\rho \left( {t,x}\right)  :=  {\int }_{{\R}^{3}}f\left( {t,x,\xi }\right) {d\xi }, \\ 
		&{m}_{i}\left( {t,x}\right)  :=  {\int }_{{\R}^{3}}{\tilde{\psi} }_{i}\left( \xi \right) f\left( {t,x,\xi }\right) {d\xi },\text{ for }i = 1,2,3, \\  
		&\mathbb{E}\left( {t,x}\right)  = \left[  \rho \left( e \left( {t,x}\right)  + \frac{1}{2}{\left| u \left( t,x\right) \right| }^{2} \right) \right]  :=  {\int }_{{\R}^{3}}{\tilde{\psi} }_{4}\left( \xi \right) f\left( {t,x,\xi }\right) {d\xi }.  
	\end{aligned} \right.
\end{align*}
Here \(\theta \left( {t,x}\right)\) is the temperature which is related to the internal energy \(e\) by \(e = \frac{3}{2}{R\theta }\), and \(u \left( {t,x}\right)\) is the fluid velocity. 
In the sequel, the inner product of two functions \(h\) and \(g\) in \({L}^{2}( {\R}^{3}_\xi)\) or \({L}^{2}( \Omega\times \R^{3}_{\xi })\)  with respect to an arbitrary Maxwellian \(\widetilde{M}\) is defined by
\begin{align}\notag
	\langle h,g \rangle_{\widetilde{M}} =  {\int }_{{\R}^{3}}\frac{1}{\widetilde{M}}h\left( \xi \right) g\left( \xi \right) {d\xi },\quad {\left( h,g\right) _{\widetilde{M}} =  \int_{\Omega}\int_{\R^{3}}\frac{1}{\widetilde{M}}h\left( \xi \right) g\left( \xi \right) {d\xi dx}.}
\end{align}
If \(\widetilde{M}\) is taken as the local Maxwellian \(M\) associated with $f(t,x,\xi)$ as in \eqref{2025.6.14-1}, then the macroscopic $L^2_\xi$ function space with respect to  the corresponding inner product is spanned by the following  five pairwise orthonormal functions
\begin{align}\label{chi-1-1-1}
	\left\{  \begin{array}{l} {\chi }_{0}\left( \xi \right)  =  \frac{1}{\sqrt{\rho }}M, \qquad\qquad\qquad\qquad
		{\chi }_{i}\left( \xi \right)  =  \frac{{\xi }_{i} - {u}_{i}}{\sqrt{R\theta \rho }}M,\text{ for }i = 1,2,3, \\ 
		{\chi }_{4}\left( \xi \right)  =  \frac{1}{\sqrt{6\rho} }\left( {\frac{{\left| \xi  -u\right| }^{2}}{R\theta } - 3}\right) M, \qquad
		\left\langle  {{\chi }_{i},{\chi }_{j}}\right\rangle   = {\delta }_{ij},\;i,j = 0,1,2,3,4. \end{array}\right.
\end{align}
Using these five base functions, we define the macroscopic projection \({P}_{0}\) and  microscopic projection \({P}_{1}\) as follows:
\begin{align}\label{2025-2025}
	{P}_{0}h =  \mathop{\sum }\limits_{{j = 0}}^{4}\left\langle  {h,{\chi }_{j}}\right\rangle  {\chi }_{j},\qquad{P}_{1}h =  h - {P}_{0}h.
\end{align}
The projections \({P}_{0}\) and \({P}_{1}\) are orthogonal and satisfy
\begin{align}\notag
	{P}_{0}{P}_{0} = {P}_{0},\;{P}_{1}{P}_{1} = {P}_{1},\;{P}_{1}{P}_{0} = {P}_{0}{P}_{1} = 0.
\end{align}
A function \(h\left( \xi \right)\) is called microscopic if
\begin{align}\notag
	{\int }_{{\R}^{3}}h\left( \xi \right) \tilde{\psi }_{i}\left( \xi \right) {d\xi } = 0,\text{ for }i = 0,1,2,3,4.
\end{align}
Under the decomposition \eqref{2025-2025}, the solution \(f\left( {t,x,\xi }\right)\) of the Landau equation \eqref{equ-landau} satisfies
\begin{align}\notag
	{P}_{0}f = M,\qquad{P}_{1}f = G.
\end{align}
Moreover, the Landau equation \eqref{equ-landau} can be re-written as
\begin{align}\notag
	\p_t{\left(M +  G\right) } + \xi \cdot \nabla_x{\left( M + G\right) } = Q\left( {M,G}\right)  + Q\left( {G,M}\right)  + Q\left( {G,G}\right) ,
\end{align}
which is equivalent to the following fluid-type system for the macroscopic quantities of $f$:
\begin{align}\label{landau-1}
	\left\{\begin{array}{l}
		\p_t \rho +\operatorname{div}_x(\rho u)=0, \\[2mm]
		\p_t (\rho u)+\operatorname{div}_{x}(\rho u \otimes u)+\nabla_x p=-\int_{\R^3} \xi \otimes \xi \cdot \nabla_x G\, d \xi, \\[2mm]
		{\p_t \big[\rho\big(e+\frac{|u|^{2}}{2}\big)\big]+\operatorname{div}_x\big[\rho u\big(e+\frac{|u|^{2}}{2}\big)+p u\big]=-\int_{\R^3} \frac{1}{2}|\xi|^{2} \xi \cdot \nabla_x G\, d \xi},
	\end{array}\right.
\end{align}    
coupled with the equation for the microscopic component \(G\):
\begin{align}\label{equ-G}
	\p_t{ G} + {P}_{1}\left( \xi \cdot\nabla_x G \right)  + {P}_{1}\left(\xi \cdot\nabla_x M \right)  = {L}_{M}G + Q\left( {G,G}\right) ,
\end{align}
where \(p = {R\rho \theta }\), 
\begin{align}\label{G}
	G = {L}_{M}^{-1}\left( {{P}_{1}\left( \xi \cdot \nabla_x M\right) }\right)  + {L}_{M}^{-1}\Pi,
\end{align}
and
\begin{align}\label{Theta}
	\Pi  = \p_t {G} + {P}_{1}\left( \xi \cdot \nabla_x G\right)  - Q\left( {G,G}\right) .
\end{align}
Here \({L}_{M}\) is the linearized collision operator with respect to the local  Maxwellian \(M\):
\begin{align}\label{2025.6.14.-2}
	{L}_{M}h = Q\left( {M,h}\right)  + Q\left( {h,M}\right),
\end{align}
and the null space \(\mathcal{N}\) of \({L}_{M}\) is spanned by the macroscopic variables: \({\chi }_{j},\ j =\)  \(0,1,2,3,4\) \eqref{chi-1-1-1}.
Plugging \eqref{G} into \eqref{landau-1}, we have
\begin{equation}\label{landau-3}
	\begin{cases}
		\p_t \rho +\operatorname{div}_{x}(\rho u)=0, \\[1mm]
		\p_t (\rho u)+\operatorname{div}_{x}(\rho u \otimes u)+\nabla_{x} p \\
		\qquad\qquad\qquad =-\int_{\R^3} \xi \otimes \xi \cdot \nabla_{x}\left(L_{M}^{-1}\left[P_{1}\left(\xi \cdot \nabla_{x} \mathbf{M}\right)\right]\right) \,d \xi-\int_{\R^3} \xi \otimes \xi \cdot \nabla_{x}\left( L_M^{-1}\Pi\right)\, d \xi, \\[1mm]
		{\p_t \big[\rho\big(e+\frac{|u|^{2}}{2}\big)\big]+\operatorname{div}_{x}\big[\rho u\big(e+\frac{|u|^{2}}{2}\big)+p u\big]} \\
		\qquad\qquad\qquad =-\int_{\R^3} \frac{1}{2}|\xi|^{2} \xi \cdot \nabla_{x}\left(L_{M}^{-1}\left[P_{1}\left(\xi \cdot \nabla_{x} \mathbf{M}\right)\right]\right)\,\, d \xi-\int_{\R^3} \frac{1}{2}|\xi|^{2} \xi \cdot\nabla_{x}\left( L_M^{-1}\Pi\right)\, d \xi.
	\end{cases}
\end{equation}
Furthermore, a direct calculation yields 
\begin{align}
	&-\int_{\R^3} \xi_{i} \xi \cdot \nabla_{x}\left(L_{M}^{-1}\left[P_{1}\left(\xi \cdot \nabla_{x} \mathbf{M}\right)\right]\right) d \xi=\sum_{j=1}^{3}\p_{x_j}\left[\mu(\theta)\left(\p_{x_j}u_{i }+\p_{x_i}u_{j}-\frac{2}{3} \delta_{i j} \operatorname{div}_{x} u\right)\right]=:\sum_{j=1}^3\p_{x_j}[\mu(\theta)S_{ij}],\label{diff}\\
	&-\int_{\R^3} \frac{1}{2}|\xi|^{2} \xi \cdot \nabla_{x}\left(L_{M}^{-1}\left[P_{1}\left(\xi \cdot \nabla_{x} \mathbf{M}\right)\right]\right) d \xi 
	=\dv_x\left(\kappa(\theta) \nabla_x \theta\right)+\dv_x\left\{\mu(\theta) \uf\cdot \mathbf{S}\right\},\label{diff.2}
\end{align} 
where $\mathbf{S}:=(S_{ij})_{1\leq i,j\leq 3}\in\R^{3\times3}$ in \eqref{diff.2} is defined in \eqref{diff}. Here,  $\mu(\theta)$ and $\kappa(\theta)$ are to be introduced in subsection \ref{sec7.1} using the Burnett functions in the Appendix. Therefore, \eqref{landau-3} can also be written as 
\begin{align}\label{landau-4}
	\left\{
	\begin{aligned}
		&\p_t \rho +u \cdot \nabla_x \rho+\rho \operatorname{div}_x u=0, \\[-1mm]
		&\rho \p_tu+\rho u \cdot \nabla_x u+\frac{2 }{3 }\theta \nabla_x \rho+\frac{2}{3}\rho \nabla_x \theta=\mu(\theta)\big(\Delta_x u+\frac{1}{3} \nabla_x \operatorname{div}_x u\big) \\[-1mm]
		&\qquad\;+\mu^{\prime}(\theta) \nabla_x \theta \cdot\big(\nabla_x u+(\nabla_x u)^{t}-\frac{2}{3}\mathbb{I}_1 \operatorname{div}_x u\big)-\int_{\R^3} \xi \otimes \xi \cdot \nabla_{x} (L_M^{-1}\Pi) d \xi, \\[-1mm]
		&\rho\p_t\theta+\rho u \cdot \nabla_x \theta+\frac{2}{3} \rho\theta \operatorname{div}_x u=\kappa(\theta) \Delta_x \theta+{\mu(\theta)}\Big[\frac{\left(\nabla_x u+(\nabla_x u)^{t}\right)^{2}}{2}-\frac{2}{3}(\operatorname{div}_x u)^{2}\Big]  \\[-1mm]
		&\qquad\;+{\kappa^{\prime}(\theta)}|\nabla_x \theta|^{2}-\int_{\R^3} \frac{1}{2}|\xi|^{2} \xi \cdot \nabla_{x} (L_M^{-1}\Pi) d \xi+u \cdot \int_{\R^3} \xi \otimes \xi \cdot \nabla_{x} (L_M^{-1}\Pi) d \xi .
	\end{aligned}\right.
\end{align}
To the end we have normalized the gas constant $R$ to be $\frac{2}{3}$ for convenience, so $e=\theta$ and $p=\frac{2}{3}\rho\theta.$  

\subsection{The entropy wave}
The macroscopic system \eqref{landau-3} can be regarded as the corresponding compressible Navier-Stokes equations with some source terms coming from non-fluid components. We construct the ansatz for the entropy wave profile of the Landau equation as follows. It is worth giving an explicit expression of the entropy wave for the hyperbolic conservation laws. Corresponding to \eqref{landau-1} for $x=(x_1,x_2,x_3) \in \Omega= \R \times\Torus^2$, the 1-D Euler equations depending only one the first spatial direction $x_1\in \R$ read as
\begin{align}\label{eqs36}
	\left\{\begin{aligned}
		&\p_t\rho+\p_{x_1}(\rho u_1)=0,\\
		&\p_t(\rho u_1)+\p_{x_1}(\rho u_1^2+p)=0,\ \ \ x_1\in\mathbb{R},\ t>0,\\
		&
		\p_t\mathbb{E}+\p_{x_1}\left( \mathbb{E}u_1+pu_1
		\right)=0,
	\end{aligned}\right.
\end{align}
with the initial data:
\begin{align*}
	(\rho, u_1,\theta)(0,x_1)=
	\begin{cases}
		(\rho_-,0,\theta_-),\ \ x_1<0,\\
		(\rho_+,0,\theta_+),\ \ x_1>0.
	\end{cases}
\end{align*}
Using the conserved quantities, \eqref{eqs36} can be re-written as
\begin{align}\label{eq-rmns}
	&\left\{\begin{array}{l}
		\pt\rho+\p_{x_1} m_1=0,\\
		\pt{m_1}+\frac23\p_{x_1}\left({\E}+\frac{{m}_{1}^2}{\rho}\right)=0, \\
		\pt{\E}+\p_{x_1}\left(\frac{5{m}_1 {\E}}{3\rho}-\frac{{m}_1^{3}}{3 \rho^{2}}\right)=0.
	\end{array}\right.
\end{align}
Studies for entropy waves are somehow very subtle because of their linear degeneracy. Accordingly, there are some structural conditions that need to be considered when studying entropy waves. The Jacobi matrix of the flux for \eqref{eq-rmns} is
\begin{align*}
	A\left(\rho, m_1, \E\right)=\left(\begin{array}{ccc}
		0 & 1 & 0 \\
		-\frac23\frac{m_1^2}{\rho^2} & \frac43\frac{m_1}{ \rho} & \frac23 \\
		-\frac{5 m_1 \E}{ 3\rho^2}+\frac{2m_1^3}{ 3\rho^3} & \frac{5 \E}{3 \rho}-\frac{m_1^2}{ \rho^2} & \frac{5 m_1}{ 3\rho}
	\end{array}\right),
\end{align*}
which has the second eigenvalue $\lambda_2=\frac{m_1}{\rho}$ with the corresponding left and right eigenvectors given by
\begin{align}\label{NSlr}
	r_2=\left(1,\frac{m_1}{\rho},\frac{m^2_1}{2\rho^2}\right)^t,\qquad\qquad l_2=\left(\frac{5\E}{3\rho}-\frac{4m_1^2}{3\rho^2},\frac{2m_1}{3\rho},-1\right).
\end{align}
From direct computations, \eqref{NSlr} further gives that
\begin{align}\label{NSlr.veri}
	\nabla l_2\cdot r_2\neq0,\qquad\nabla r_2\cdot r_2=0.
\end{align}
However, the classical left and right structural conditions as in \cite{Liu-Xin} are given as 
\begin{align}\label{SC}
	\nabla l_2\cdot r_2=0, \quad \text{and} \quad \nabla r_2\cdot r_2=0,\quad\text{ respectively. } 
\end{align}
Therefore, from \eqref{NSlr.veri} the left structural condition is violated for the Euler equations \eqref{eq-rmns} in Eulerian
coordinates. The presence of two-sides structural conditions has enhanced the decay of lower-order terms \cite{Huang-Matsumura-Xin,Huang-Xin-Yang,Liu-Xin,DYY-CMP}. For more details, see difficulties and strategies to be specified in subsection \ref{sss} later on.

As in \eqref{eqs11.se1.1}, the entropy wave for \eqref{eq-rmns}  satisfies
\begin{align}\label{eqs33}
	\frac{2}{3}\rho_+\theta_+=p_+=p_-=\frac{2}{3}\rho_-\theta_-,\qquad u_{1+}=u_{1-}.
\end{align}
We assume that the strength of wave $\delta:=|\theta_+-\theta_-|$ is small. As in \cite{Huang-Xin-Yang, Huang-Yang}, let $\rho^{cd}(\frac{x_1}{\sqrt{1+t}})$ be the self-similar solution of the following nonlinear diffusion equation
\begin{align}\label{eqs14.cdw}
	\left\{\begin{aligned}
		&\p_t\rho^{cd}=\p_{x_1}\left(\frac{3\tilde{\kappa}(\rho^{cd})}{5}\frac{\p_{x_1}\rho^{cd}}{\rho^{cd}}\right), \ \ \tilde{\kappa}(\rho^{cd})=\kappa(\frac{3p_{+}}{2 \rho^{cd}})
		,\\
		&\rho^{cd}(t,+\infty)=\rho_{+},\qquad\qquad \rho^{cd}(t,-\infty)=\rho_{-}.
	\end{aligned}\right.
\end{align}
It then holds
\begin{align}\label{eqs27}
	\left\{\begin{aligned}
		&\left|(1+t)^{\frac{1}{2}}\p_{x_1}\rho^{cd}\right|=O(\delta)e^{-\frac{d x_1^2}{1+t}},\ x_1\in \R, \\
		&|\rho^{cd}-\rho_{+}|=O(\delta) e^{-\frac{d x_1^2}{1+t}},\ x_1\geq 0; 
		\quad |\rho^{cd}-\rho_{-}|=O(\delta) e^{-\frac{d x_1^2}{1+t}},\ x_1\leq 0, 
	\end{aligned}\right.
\end{align}
where $d>0$ is a constant.

In terms of the self-similar diffusion profile $\rho^{cd}(\cdot)$ \eqref{eqs14.cdw}, we can further construct the viscous entropy wave profiles $\rhoc,\check{u},\check{\theta}$ corresponding to the compressible Euler fluid part of \eqref{landau-3} for the Landau equation: 
\begin{equation}\label{efvew}
	\begin{aligned}
		\rhoc=\rho^{cd},\quad \check{u}=\left(-\frac{3\tilde{\kappa}(\rhoc)\p_{x_1}\rhoc}{5\rhoc^2},0,0\right)^t, \quad
		\check{\theta}=\frac{3p_+}{2\rhoc}-
		\frac{\uc_1^2}{2},\quad\frac23\rhoc\check\theta=\check p=p_+-\frac{1}{3}
		\rhoc \uc_1^2.
	\end{aligned}
\end{equation}
Moreover, the associated 1-D local Maxwellian corresponding to \eqref{efvew} above is defined as
\begin{align}\label{2026-5-9-1}
	{M}_{[ \rhoc , \uc, \check\theta ] }\left( t,x_1,\xi \right)\equiv  \frac{\rhoc \left( t,x_1 \right) }{\sqrt{\left( 2\pi R \check\theta \left( t,x_1\right) \right) ^{3}}} \exp{ \left( -\frac{\left| \xi  - \uc \left( t,x_1 \right) \right| ^{2}}{2R \check\theta \left( t,x_1\right)}\right)} .
\end{align}
\begin{Rem}
	For convenience of the proof regarding the dynamical stability of the planar entropy wave \eqref{2026-5-9-1}, throughout the paper we fix a normalized global Maxwellian with the
	fluid constant state $(1,0,\frac32)$:
	\begin{align}\label{2026-5-9-2}
		\mu=M_{[1,0,\frac32]}=(2\pi)^{-\frac32}e^{-\frac{\abs{\xi}^2}{2}}
	\end{align}
	as a reference equilibrium state, and choose both the far-field data \eqref{eqs33} to be close enough to the constant state $(1,0,\frac32)$ such that the viscous contact
	wave in \eqref{efvew} further satisfies that
	\begin{align*}
		\left\{
		\begin{aligned}
			&\sup_{t\ge0,x_1\in\R}\{ \abs{\rhoc(t,x_1)-1}+\abs{\uc(t,x_1)}+|\check \theta(t,x_1)-\frac32| \} \le3\delta,\\
			&\frac12 \sup_{t\ge0,x_1\in\R}\check\theta(t,x_1)\le \frac32 \le \inf_{t\ge0,x_1\in\R} \check\theta(t,x_1),
		\end{aligned}\right.
	\end{align*}
	where $\delta>0$ is the wave strength small enough.
\end{Rem}

Setting $\mc_i=\rhoc\uc_i$ and $\check{\mathbb{E}}=\rhoc\left(\check{\theta}+\frac{\abs{\uc}^2}{2}\right)$, by direct calculations, it follows from \eqref{efvew} that
\begin{align}\label{contact-wave-check}
	\left\{  \begin{aligned} 
		&\p_t{\rhoc} + \p_{x_1}{\mc}_{1 } = 0 \\ 
		&\p_t{\mc}_{1} + \p_{x_1}{\left( \frac{{\mc}_{1}^{2}}{\rhoc } + \check{p}\right) } = \frac{4}{3} \p_{x_1}{\left( \mu \left( \check{\theta} \right) \p_{x_1} \uc_{1}\right) } +\p_{x_1}\Rm_{1 }, \\ 
		&\p_t{\mc}_{i} + \p_{x_1}{\left( \frac{{\mc}_{1}{\mc}_{i}}{\rhoc }\right) } = \p_{x_1}{\left( \mu \left( \check{\theta} \right) {\p_{x_1}\check{u}}_{i}\right) } ,\;i = 2,3, \\ 
		&\p_t{\check{\mathbb{E}}} + \p_{x_1}{\left( \frac{{\mc}_{1}\check{\mathbb{E}}}{\rhoc } + \frac{{\mc}_{1}}{\rhoc }\check{p}\right) } = \p_{x_1}{\left( \kappa \left( \check{\theta} \right) {\p_{x_1}\check{\theta} }\right) } + \frac{4}{3}\p_{x_1}{\left(  \mu \left( \check{\theta} \right) {\uc}_{1}\p_{x_1} \uc_{1}\right) }+ \mathop{\sum }\limits_{{i = 2}}^{3} \p_{x_1}{ \left(  \mu \left( \check{\theta} \right) {\uc}_{i}\p_{x_1} \uc_{i}\right) }+\p_{x_1}\Rm_{4} ,  \end{aligned}\right. 
\end{align}
where
\begin{align*}
	&\Rm_1=-\frac{3\tilde{\kappa}(\rhoc)\p_t\rhoc}{5\rhoc}+\frac{2}{3}\rhoc \uc_1^2-\frac{4}{3}\mu(\check \theta) \p_{x_1}\uc_{1}=O(\delta)(1+t)^{-1}e^{-\frac{c x_1^2}{1+t}},\\
	&\Rm_4=\left[\kappa(\check{\theta})-\kappa(\frac{3p_{+}}{2 \rhoc}) \right]\frac{3p_{+}\p_{x_1}\rhoc}{2 \rhoc^2}-\frac{1}{3}\rhoc \uc_1^3-\left(\frac{4}{3}\mu(\check{\theta})+\kappa(\check\theta)\right)\uc_1\p_{x_1}\uc_{1}=O(\delta)(1+t)^{-\frac{3}{2}}e^{-\frac{c x_1^2}{1+t}}, 
\end{align*}
with some constant $c>0$. For convenience, we denote the conserved quantities by
\begin{align}\label{conserved-law}
	\U=(\rho,m_1,m_2,m_3,\mathbb{E})^{t},\;\;\;\check{\U}=(\rhoc,\mc_1, \mc_2,\mc_3,\check{\mathbb{E}})^{t},\;\;\; \U^{\mathcal{\#}}=(\rho,m_1,\E)^{t}, \;\;\;\check{\U}^{\#}=(\rhoc,\mc_1, \check{\mathbb{E}})^{t}.  
\end{align}

\subsection{Literature review}

The Landau equation with Coulomb interactions is central to plasma physics (cf.~\cite{Hilton}) and has been widely studied. Among most fundamental results in the mathematical literature, we mention: global weak solutions (with defect measure) by Lions \cite{Lions} and Villani \cite{Vi98, Vi96}; grazing collision limit from the non-cutoff Boltzmann equation by Desvillettes \cite{Des} and Alexandre–Villani \cite{AV04}; spectral analysis of the linearized equation by Degond–Lemou \cite{DL}; derivation from particle systems by Bobylev–Pulvirenti–Saffirio \cite{BPS}; recent breakthrough on no blow-up via the Fisher information by Guillen-Silvestre \cite{GuSi} for the spatially homogeneous case; and all references therein. 

Directly related work concerns solutions near global Maxwellians: Guo \cite{Guo-2002} first established global existence on the torus, with decay rates given by Strain–Guo \cite{Strain-Guo-2006, Strain-Guo-2008}. For the Vlasov–Poisson–Landau system, global solutions near Maxwellians were obtained by Guo \cite{Guo-JAMS} on the torus and Strain–Zhu \cite{Strain-Zhu} on $\R^3$
(see also \cite{Duan, Wang, Wang-1}). Other contributions include \cite{CaMi, CTW, DV2000, GIMV, GHJO, HeSn, KGH, Luk}.

A fundamental open problem is the global existence and behavior of solutions to the Cauchy problem \eqref{equ-landau} with initial data of small total variation connecting two different Maxwellians at the far fields (cf. \cite{SR, Ukai-Yang}). Based on fluid dynamic limits, the long-time profile is expected to be a nonlinear wave pattern—shock, rarefaction, or contact wave—or their superposition (cf. \cite{Smoller}), an expectation motivated by pointwise estimates via the Green's function method \cite{LY-G04}.

For the cutoff Boltzmann equation, wave patterns are well-studied: shock profiles \cite{Caflisch-Nicolaenko, Liu-Yu, Yu}; contact wave stability and hydrodynamics \cite{Huang-Xin-Yang, Huang-Yang, Huang-Wang-Yang}; rarefaction wave stability \cite{Liu-Yang-Yu-Zhao, Xin-Yang-Yu, Yang-Zhao}; composite wave stability \cite{Wang-Yu}; and patterns with self-consistent forces \cite{Duan-Liu-2015, Li-Wang-Yang-Zhong}. For viscous conservation laws (e.g., compressible Navier-Stokes), contact wave stability was established earlier \cite{Duan-Huang-Li-Xu,Huang-Li-Matsumura, HMS, Huang-Matsumura-Xin, Kawashima-M,LWX, Liu-Xin, Xin} (see survey \cite{Ma}).

In contrast, the known results for the non-cutoff Boltzmann and Landau equations are quite few, with all previous ones concentrated on only the one-dimensional case over $\R$. Recent progress includes the stability of rarefaction waves for the Landau equation \cite{Duan-Yu1} and the analysis of its small-Knudsen limit \cite{DYY}. The stability of viscous contact waves was established in \cite{DYY-CMP}, with its corresponding small-Knudsen limit further investigated in \cite{YDC}. Meanwhile, the existence of shock profiles has been proven in \cite{Bed}, although their nonlinear stability and behavior in the small-Knudsen limit remain open questions.

Several important issues merit further investigation. The result in \cite{DYY-CMP}, while establishing stability, was only obtained in 1-d Lagrangian coordinates and does not provide an explicit decay rate—a property that is well-understood for the corresponding Navier-Stokes equations. This problem is particularly challenging for the Landau equation with very soft potentials ($\gamma<-2$), where the linearized operator lacks a spectral gap, distinguishing it fundamentally from the fluid dynamic case. Furthermore, it should be noted that the known decay-rate results for the Landau equation (e.g., \cite{Wang,Wang-1}) are obtained in a multi-dimensional whole-space setting. In such multi-dimensional frameworks, the analytical difficulty associated with the spectral gap is generally less severe.

The current work aims at studying the nonlinear stability of viscous entropy waves for the  Landau equation with Coulomb interactions in the three-dimensional infinite channel domain $\Omega=\R\times\Torus^2$. We are able to obtain (i) the existence of a unique global solution near a local Maxwellian whose fluid components are viscous entropy wave profiles, and (ii) the time-asymptotic stability and optimal decay rate of this solution. This is the first result of such kind for the Landau equation. The rough version of the main result will be given in Theorem \ref{thm.rvmt}; we also refer to Remarks \ref{thm.rk1}, \ref{thm.rk2} and \ref{thm.rk3} for more discussions on the result. The precise statement will be given in Theorem \ref{mt} together with the more detailed version Theorem \ref{mt-1} in the next section.

\subsection{Norms and rough version of the main theorem}
\subsubsection{Norms}
Corresponding to the global Maxwellian $\mu$ in \eqref{2026-5-9-2}, the Landau collision frequency is
\begin{equation}
	\label{1.34}
	\sigma^{ij}(\xi):=\tilde{\phi}^{ij}\ast \mu(\xi)=\int_{{\mathbb R}^3}\tilde{\phi}^{ij}(\xi-\xi')\mu(\xi')\,d\xi'.
\end{equation}
We remark that $\sigma^{ij}(\xi)$ is a positive definite symmetric matrix. 
We first denote the weight function
\begin{align}\label{1.32}
	\omega=\omega(\beta)(\xi):=\langle\xi\rangle^{(l-|\beta|)}e^{q\langle\xi\rangle^2},\quad l\geq|\beta|,\quad\langle\xi\rangle=\sqrt{1+|\xi|^2},\quad 0<q<1.
\end{align} We denote the
weighted $L^2$ norms as
$$
|\omega g|^2_{2}:= \int_{{\mathbb R}^3}\omega^2 g^2\,d\xi,\ \ \ \|\omega g\|_2^2:={ \int_{\mathbb R\times\Torus^2}\int_{{\mathbb
			R}^3}}\omega^2 g^2\,d\xi dx.
$$
In terms of linearization of the nonlinear Landau operator around $\mu$ \eqref{2026-5-9-2} (cf.~\cite{Guo-2002}), with \eqref{1.34} we define the weighted dissipative norms:
$$| g|^2_{\sigma,\omega}:=\sum_{i,j=1}^3\int_{{\mathbb
		R}^3}\omega^{2}[\sigma^{ij}\partial_{\xi_i} g\partial_{\xi_j} g+\sigma^{ij}\frac{\xi_i}{2}\frac{\xi_j}{2} g^2]\,d\xi,\quad \mbox{and}\quad \ \| g\|_{\sigma,\omega}:= \|| g|_{\sigma,\omega} \|_{L_x^2}.$$
And let $| g|_{\sigma}=| g|_{\sigma,1}$ and $\| g\|_{\sigma}=\| g\|_{\sigma,1}$.
From \cite[Corollary 1, p.399]{Guo-2002} and \cite[Lemma 5, p.315]{Strain-Guo-2008}, one has
\begin{equation}
	\label{1.35}
	| g|_\sigma\approx |\langle \xi\rangle^{-\frac{1}{2}} g|_2+\Big|\langle \xi\rangle^{-\frac{3}{2}}\frac{\xi}{|\xi|}\cdot\nabla_\xi  g\Big|_2+\Big|\langle \xi\rangle^{-\frac{1}{2}}\frac{\xi}{|\xi|}\times \nabla_\xi  g\Big|_2.
\end{equation}
We also denote
$$\|\partial^\alpha_\beta
g\|^2_{2,\omega}:=\int_{\mathbb R\times\Torus^2}\int_{{\mathbb
		R}^3}\omega^{2}(\beta)|\partial^\alpha_\beta  g(x,\xi)|^2\, d\xi dx,$$
and
$$
\|\partial^\alpha_\beta g\|^2_{\sigma,\omega}:=\sum_{i,j=1}^3\int_{\mathbb R\times\Torus^2}\int_{{\mathbb
		R}^3}\omega^{2}(\beta)[\sigma^{ij}\partial_{\xi_i}\partial^\alpha_\beta g(x,\xi)\partial_{\xi_j}\partial^\alpha_\beta  g(x,\xi)+\sigma^{ij}\frac{\xi_i}{2}\frac{\xi_j}{2}|\partial^\alpha_\beta  g(x,\xi)|^2]\, d\xi dx.
$$

\subsubsection{Rough version of the main result}
\begin{Thm}[Rough version]\label{thm.rvmt}
	Let $\Omega=\R\times\Torus^2$. The Cauchy problem \eqref{equ-landau} and \eqref{ini} admits a unique global-in-time solution $f=f(t,x,\xi)\geq 0$ satisfying 
	\begin{align}
		\norm{\frac{f-{M}_{[\rhoc,\uc,\Tc]}}{\sqrt{\mu}}}_{L^{\infty}_{x}(\Omega)L^2_\xi(\R^3)}\leq C\tilde{\delta} (1+t)^{-\frac{1}{2}}, \label{decay} \qquad \norm{\frac{(f-{M}_{[\rhoc,\uc,\Tc]})_{\neq}}{\sqrt{\mu}}}_{L_x^\infty(\Omega) L_\xi^2(\R^3)}^2 \leq C \tilde\delta e^{-c t^{\frac23}}. 
	\end{align}
	where the non-zero mode $f_{\neq}$ is defined by $f_{\neq}:=f-\int_{\Torus^2}fdx_2d_3$,  ${M}_{[\rhoc,\uc,\Tc]}$ and $\mu$ are defined in \eqref{2026-5-9-1} and \eqref{2026-5-9-2}, respectively, and $\tilde{\delta}$ is a small constant  associated with the wave strength and the appropriate initial data $f_0(x,\xi)$ near $\mu$.
\end{Thm}

\begin{Rem}\label{thm.rk1}
	To the best of our knowledge, there have been no existing results addressing time-decay rate of solutions around entropy waves for the Landau equation in the spatial domain either the infinite 3-D channel $\Omega=\mathbb{R}\times\mathbb{T}^2$ or 1-D whole line $\Omega=\mathbb{R}$ with slab symmetry. Indeed, even the global-in-time existence result around planar waves for the Landau equation is not yet known. Therefore, Theorem \ref{thm.rvmt} above provides the first result of such kind and also resolves the question on rate of convergence left open in \cite{DYY-CMP}. Moreover, our proof of Theorem \ref{thm.rvmt} furnishes a robust framework for analyzing the time-decay behavior of the Landau equation near a local Maxwellian.
\end{Rem}

\begin{Rem}\label{thm.rk2}
	In this paper, we consider the spatial domain $\mathbb{R}\times\mathbb{T}^2$. The analysis in this setting presents two main challenges: 
	(i) It is highly nontrivial to establish stability results with an explicit time-decay rate for the Landau equation with very soft potentials $\gamma<-2$. Indeed, the methods used in existing works, such as \cite{Wang, Wang-1}, may not be directly adapted to the one-dimensional problem in an effective manner. 
	(ii) Studying the planar entropy wave in $\mathbb{R}\times\mathbb{T}^2$ introduces fundamental analytical difficulties that are distinct from those arising in the purely one-dimensional setting, cf.~\cite{DYY-CMP}.
\end{Rem}

\begin{Rem}\label{thm.rk3}
	Regarding the rate of convergence of solutions toward entropy waves in \eqref{decay}, the polynomial decay rate $(1+t)^{-\frac{1}{2}}$ is optimal under generic perturbations, due to the presence of diffusion waves propagating along the one-dimensional unbounded direction $x_1\in\mathbb{R}$. Moreover, for non-zero modes we obtain the stretched exponential decay $\exp(-c t^{\frac{2}{3}})$ corresponding to Coulomb interaction potentials $\gamma=-3$. As far as we know, this is the first result of exponential decay of perturbations in non-zero modes near a local Maxwellian.
\end{Rem}

\subsection{Difficulties and Strategies}\label{sss}
The main difficulties of this work lie in the following aspects. First, we consider the feature of the entropy waves.  Previous studies often relied on the fact that entropy waves satisfy certain structural conditions \eqref{SC} at the macroscopic level, cf. \cite{DYY-CMP,Huang-Matsumura-Xin,Huang-Xin-Yang,Liu-Xin}. The structural conditions lead to cancellation in lower order terms, resulting in the following form
\begin{align}\notag
	\int_{\Omega}(1+t)^{-1}e^{-\frac{cx^2_1}{1+t}}|(\rho-\rhoc,u-\uc,\theta-\Tc)|^2dx.
\end{align}
However, those structural conditions are derived from one-dimensional models formulated in Lagrangian coordinates, and they generally do not hold in higher dimensions. Actually, only the following form with the slower time decay coefficient can be expected
\begin{align}\notag
	\int_{\Omega}(1+t)^{-\frac{1}{2}}e^{-\frac{cx^2_1}{1+t}}|(\rho-\rhoc,u-\uc,\theta-\Tc)|^2dx.
\end{align}
In this work, we introduce a novel transformation under which the perturbed equations in multiple dimensions can satisfy the structural condition in a suitable sense.

Second, for the Landau operator with Coulomb interactions, the lack of spectral gap leads to very weak dissipation, {\it i.e.}, $\abs{g}_{\sigma}$ \eqref{1.35}. Specifically, the dissipation induced by the linearized collision operator is not strong enough to control the $L^2$-norm $\abs{g}_{L^2}$. Moreover, it is important to note that the local Maxwellian $M_{[\rhoc,\uc,\check \theta]}$ \eqref{2026-5-9-1} associated with the entropy wave is not an exact solution to the underlying equation, and the corresponding error $\bar{\mathcal{R}}$ \eqref{tildeE} is not well-controlled.
The combination of weak dissipation and slowly decaying errors makes it extremely challenging to establish asymptotic stability of planar entropy wave, and it has remained an outstanding open problem for obtaining optimal decay rates, cf. \cite{DYY,DYY-CMP,HY}.

To overcome these challenges, we develop the following strategies. 

\begin{enumerate}
	\item {\it Macro-micro decomposition for Landau equation around the entropy wave.}
	
	By decomposing the solution of the Landau equation \eqref{equ-landau} into its macroscopic and microscopic components, we are able to take advantage of the desirable properties of the entropy wave in the macroscopic equations.
	
	\item {\it The time-velocity interpolation  technique  sacrifices time decay in exchange for stronger dissipation.}\label{2026-5-14}
	
	In this paper, due to lack of spectral gap, we use a time-velocity interpolation  technique as in Lemma \ref{key-trans} introduced by Strain-Guo \cite{Strain-Guo-2006,Strain-Guo-2008} to obtain stronger dissipation, {\it i.e.}
	\begin{align*}
		\abs{\p^{\alpha} f}_2^2 \leq \upsilon(1+t)^{\epsilon} \abs{\p^{\alpha} f}_{\sigma}^2+e^{-\frac{q}{4}\upsilon^2( 1+t )^{{2\epsilon}}} \abs{e^{\frac{q}{8}\la \xi\ra^2} \p^{\alpha} f}_2^2.
	\end{align*}
	but at the cost of a $(1+t)^{\epsilon}$  loss in the decay rate. Unfortunately, previous results \cite{Huang-Xin-Yang} imply that the growth of energy 
	for anti-derivatives is $(1+t)^{1/2}$ due to the error terms $\bar{\mathcal{R}}$ \eqref{tildeE}, and thus 
	the decay of the original energy should be at least $(1+t)^{-1/2}$ to close the 
	{\it a priori} assumptions. However, heat dissipation indicates that the original 
	energy decays faster than its anti-derivative counterpart by at most a factor of 
	$(1+t)^{-1}$.  Therefore, to afford the additional $(1+t)^{\epsilon}$ decay loss introduced by the time-velocity interpolation  technique,  we need to remove these error terms $\bar{\mathcal{R}}$ \eqref{tildeE}.
	
	\item {\it Construct diffusion waves and coupled diffusion waves.}
	
	Under generic perturbations, solutions to the macroscopic equations \eqref{landau-3} generate diffusion waves $\Theta_i$ \eqref{eqs19}, which correspond to the first-order expansion of the macroscopic fluid system; see \cite{Kai,LiuZ}. In general, the  background solution $\bar{\U}:=(\rhob,\mb,\Eb)=(\rhoc,\mc,\Ec)+\Theta_i$, coupled to the entropy wave  $(\rhoc,\mc,\Ec)$ \eqref{contact-wave-check} and  diffusion wave $\Theta_i$ \eqref{eqs19}, satisfies  the system (see \eqref{landau-2} and \eqref{U-bar} for details)
	\begin{align*}
		\p_t\bar\U + \p_{x_1}F(\bar \U)=\p_{x_1}\big(\mathcal{B}(\bar{\U})\p_{x_1}\bar{\U}\big)+\p_{x_1}\bar{\Rm}.
	\end{align*}
	These diffusion waves $\Theta_i$ exhibit insufficient time decay of the spatial $L^2$-norm, making it difficult to control the induced error terms $\bar{\Rm}$ \eqref{tildeE} (see the reason for strategy (\ref{2026-5-14})). 
	In this paper, we construct a family of coupled diffusion waves $\Xi_i$ by exploiting specific features of the entropy wave to cancel the slowly decaying error terms $\bar{\Rm}$ (see \eqref{coupled-diffusion-wave-1} for details)
	\begin{align}\notag
		\quad\qquad\left\{ \begin{aligned}
			&\p_t\Xi+\p_{x_1}(F'(\bar{\U})\Xi)=\p_{x_1}\big(\mathcal{B}(\bar{\U})\p_{x_1}\Xi\big)-\frac{1}{2}\p_{x_1}(\Xi^{t}F''(\bar{\U})\Xi)+\p_{x_1}\big(\Xi^{t}\mathcal{B}{'}(\bar{\U})\p_{x_1}\bar{\U}\big)  -\p_{x_1}\bar{\Rm}-\p_{x_1}\mathcal{R}_{G}, \\
			&\Xi|_{t=0}=0,
		\end{aligned}\right.
	\end{align}
	where $\mathcal{R}_G$ is the slowly decaying term induced by the microscopic parts.
	To obtain optimal decay rates for derivatives of all orders of the coupled diffusion waves $\Xi_i$, we diagonalize the system \eqref{W-t-1}, which is derived from \eqref{coupled-diffusion-wave-1}.  In the diagonalized system \eqref{eq-diaB}, we observe that different components propagate at distinct speeds. Based on this feature, we find out a new dissipation mechanism \eqref{2025.6.07-1}. This enables us to effectively control the slowly decaying terms in \eqref{eq-diaB} and achieve the optimal decay rates.

	\item {\it The transformations to ensure structural conditions hold.}
	
	For the new ansatz $(\rhot,\mt,\Et)$ (see \eqref{New-Ansatz}) formed by coupling $\bar \U$ and the coupled diffusion wave $\Xi_i$, we study the perturbation in an integrated system with
	\begin{align*}
		(\Phi,\Psi,H)=\int_{-\infty}^{x_1}\int_{\Torus^2}(\rho-\rhot,m-\mt,\E-\tilde{\E})(t,y)dy'dy_1,
	\end{align*}
	which presents a possibility to obtain the decay rate. Next, we apply the following transformation\begin{align}\notag 
		&\check{\Phi}=\thetat\Phi,\quad \Psic=\Psi,\quad \Hc=H-\Phic,
	\end{align}
	where $\thetat$ is a non-conserved quantity of temperature for the new ansatz defined in \eqref{non-conserved quantities}. Through this transformation, the new system is simplified a lot and satisfies both the two structural conditions, see \eqref{T-A-D-1}, {\it i.e.},  neither 
	$$
	\int_{\R} (1+t)^{-\frac12}e^{-\frac{cx_1^2}{1+t}}(\Phic^2+|\Psic|^2+\Hc^2)dx_1
	$$
	nor 
	$$\int_{\R} (1+t)^{-\frac12}e^{-\frac{cx_1^2}{1+t}}[(\p_{x_1}\Phic)^2+|\p_{x_1}\Psic|^2+(\p_{x_1}\Hc)^2]dx_1
	$$
	appears in the $L^2$-estimates of the anti-derivative itself and its first derivative, respectively. Unfortunately, the structural conditions \eqref{SC} are once again not met in the derivative system $$(\p_{x_1}^2\Phic,\p_{x_1}^2\Psic,\p_{x_1}^2\Hc).$$ 
	This is also why it is so difficult to obtain the optimal decay rate for contact discontinuities. In this paper, we find a new transformation in the derivative level, {\it i.e.}, we study the following quantities 
	\begin{align}\notag
		\left(\frac{2}{3}\thetat\p_{x_1}^2\Phic,\p_{x_1}\left(\thetat\p_{x_1}\Psic_{1}\right), \p_{x_1}^2\Psic_{2},\p_{x_1}^2\Psic_{3}, \; \thetat \p_{x_1}^2\Hc \right).
	\end{align}
	Several cancellations are achieved under this transformation and thus the optimal decay rates are available, see Remark \ref{2026-5-14-1}.

	\item  {\it The stretched exponential decay for the non-zero mode.}
	
	For non-zero modes, we observe that the Poincar\'{e}'s inequality is applicable. By utilizing this inequality and combining the small strength of the entropy wave, the low-order terms arising in the flux can be effectively controlled through  dissipation. Furthermore, for the non-zero modes, we observe that 
	$$\frac{(f-M_{[\rhoc,\uc,\check{\theta}]})_{\neq}}{\sqrt{\mu}}=\frac{(f-\mu)_{\neq}}{\sqrt{\mu}}.$$
	To obtain the stretched exponential decay for these modes, we perturb the Landau equation around the global Maxwellian $\mu$, thereby deriving dissipation estimates for the non-zero modes of the macroscopic part. Then, combining these estimates with the $\xi$-derivative and the exponentially weighted estimates for the microscopic part, the Poincar\'{e}'s inequality, and the time-velocity interpolation  technique Lemma \ref{key-trans} establishes the stretched exponential decay for the non-zero modes, see Theorem \ref{lem999}.

	\item {\it The delicate energy structure to obtain the optimal polynomial decay rate.}
	
	Since the slowly decaying error terms are eliminated by constructing coupled diffusion waves, and the structural conditions are restored, we obtain the optimal decay rate by deriving the following differential inequalities 
	\begin{align*}
		&\frac{d}{dt}\mathcal{E}_1 + \mathcal{D}_1\leq C\deltac(1+t)^{-1}\mathcal{E}_1 + C\deltab (1+t)^{-\frac{3}{2}},  \\
		&\frac{d}{dt}\mathcal{E}_2 + \mathcal{D}_2\leq C\deltac \left[(1+t)^{-1}\mathcal{D}_1+(1+t)^{-2}\mathcal{E}_1 \right]+ C\deltab (1+t)^{-\frac{3}{2}}, \\
		&\frac{d}{dt}\tilde{\mathcal{E}}_2 +\tilde{\mathcal{D}}_2 \leq C\deltac\left[(1+t)^{-1}\mathcal{D}_1+(1+t)^{-2}\mathcal{E}_1\right]+C\deltab(1+t)^{-\frac52},\\
		&\frac{d}{dt}\mathcal{E}_3 + \mathcal{D}_3\leq C\deltac \left[(1+t)^{-1}\mathcal{D}_2+(1+t)^{-2}\mathcal{D}_1+(1+t)^{-3}\mathcal{E}_1 \right]+ C\deltab (1+t)^{-\frac{5}{2}},
	\end{align*}
	where $\mathcal{E}_{(1,2),\omega}$, $\mathcal{E}_{1,2,3}$ and $\tilde{\mathcal{E}}_{2}$ are transient energies  (see \eqref{2.10}–\eqref{2.10-5}), and  $\mathcal{D}_{(1,2,3),\omega}$, $\mathcal{D}_{1,2,3}$ and $\tilde{\mathcal{D}}_{2}$ are dissipative energies (see \eqref{2.11}–\eqref{2.11-5}), and $\deltab,\deltac$ are small constants associated with wave strength and initial perturbation. It is noteworthy that the introduction of $\tilde{\mathcal{E}}_{2}$ \eqref{2.10-4} and $\tilde{\mathcal{D}}_{2}$ \eqref{2.11-4} effectively circumvents the impact of the time loss induced by the time-velocity interpolation technique  on the optimal decay rate; specifically, we need to use the relation
	\begin{align*}
		&\int_0^t(1+\tau) \tilde{\mathcal{D}}_2 d\tau\le C \deltab (1+t)^{\frac1{5}},\\
		&\int_0^t (1+\tau)\mathcal{E}_3 d \tau \le C(1+t)^{\frac{1}{10}} \int_0^t(1+\tau) \tilde{\mathcal{D}}_2 d\tau +C\deltab.
	\end{align*}
	For the case of $\gamma\ge-2$, it is unnecessary to introduce $\tilde{\mathcal{E}}_2$ and $\tilde{\mathcal{D}}_2$.  See Section \ref{sec5} for all the details.
\end{enumerate}

\subsection{Notations}

Throughout the paper we shall use $\langle \cdot , \cdot \rangle$  to denote the standard $L^{2}$ inner product in $\mathbb{R}_{\xi}^{3}$ with its corresponding $L^{2}$ norm $|\cdot|_2$. For $\Omega= \R \times \Torus^2$, we also use $( \cdot , \cdot )$ to denote the $L^{2}$ inner product in $\Omega\times \mathbb{R}_{\xi}^{3}$  with its corresponding $L^{2}$ norm $\|\cdot\|_2$. Let nonnegative integer $\alpha$ and $\beta$ be multi indices $\alpha=[\alpha_{0},\alpha_{1},\alpha_{2},\alpha_{3}]$ and $\beta=[\beta_{1},\beta_{2},\beta_{3}]$, respectively. Denote $\partial_{\beta}^{\alpha}=\partial_{t}^{\alpha_{0}}\partial_{x_1}^{\alpha_{1}}\partial_{x_2}^{\alpha_{2}}\partial_{x_3}^{\alpha_{3}}\partial_{\xi_{1}}^{\beta_{1}}\partial_{\xi_{2}}^{\beta_{2}}\partial_{\xi_{3}}^{\beta_{3}}$. If each component of $\beta$ is not greater than the corresponding one of $\overline{\beta}$, we use the standard notation $\beta\leq\overline{\beta}$. And $\beta<\overline{\beta}$ means that $\beta\leq\overline{\beta}$ and $|\beta|<|\overline{\beta}|$. $C^{\bar\beta}_{\beta}$ is the usual  binomial coefficient. Throughout the paper, generic positive constants are denoted  by either $c$ or $C$, and $c_{1}$, $c_{2}$ or $C_{1}$, $C_{2}$ etc. are some given constants. The notation $A\lesssim B$ is used to denote that there exists a constant $c_{0}>1$ such that $ A\leq c_{0}B$ and $A\approx B$ is used to denote  $c_{0}^{-1}B\leq A\leq c_{0}B$. For an integrable function $f$ in $\Omega= \R_{x_1} \times \Torus^2_{x_2,x_3}$, we define the zero and non-zero modes, respectively, as
\begin{align}\label{2025-11-10-4}
	\Do f:=\mathring{f}:=\int_{\Torus^2}fdx_2dx_3,\qquad \Dn f:={f}_{\neq}:=f-\Do f.
\end{align}

The rest of the present paper is organized as follows. In Section \ref{sec2}, we construct the diffusion waves and the coupled diffusion waves, and then introduce the new ansatz. With these, we are able to precisely state the main result Theorem \ref{mt}. Then, in Section \ref{sec3}, we give the estimate of the coupled diffusion wave. In Section \ref{sec4}, we study the stability under the new ansatz. The optimal decay rate and the stretched exponential decay are obtained in Sections \ref{sec5} and \ref{sec6}, respectively. Finally, in Section \ref{sec7}, we present some necessary technical lemmas that have been used in the previous sections.

\section{Construction of the ansatz and the main theorem}\label{sec2}
\subsection{Diffusion waves generated by non-zero initial mass}
Recall the definition of $\U,\check\U,\U^{\#},\check\U^{\#}$ in \eqref{conserved-law}.
In this paper, we are concerned with the general non-zero initial perturbation in $\Omega=\mathbb R\times\mathbb T^2$ with $\mathbb T={\mathbb R}/{\mathbb Z}$, that is, the integral $\int_{\Omega}(\U-\check \U)(0,x)dx\neq0$. 
Note that the extra initial mass is distributed along the $x_1$-direction. We consider the far-field Jacobi matrices for the flux of 1-d Navier-Stokes equations in Eulerian coordinates after taking $\int_{\Torus^2}$\eqref{landau-3}$dx_2dx_3$  is 
\begin{align*}
	A_\pm=\left(\begin{array}{ccc}
		0&1&0\\
		0&0&\frac{2}{3}\\
		0&\frac{5}{3}\frac{\E_\pm}{\rho_\pm}&0
	\end{array}\right),
\end{align*}
corresponding to $\U^{\#}_{+}$ and $\U^{\#}_{-}$, respectively. 
It is easy to see that $\lambda_1^-=-\sqrt{\frac{10\E_-}{9\rho_-}}$ is the first eigenvalue of $A_-$ corresponding to $r_1^-=\left(1,\lambda_1^-,\frac{3(\lambda_1^-)^2}{2}\right)^t$ and $\lambda_3^+=\sqrt{\frac{10\E_+}{9\rho_+}}$ is the third eigenvalue of $A_+$ corresponding to $r_3^+=\left(1,\lambda_3^+,\frac{3(\lambda_3^+)^2}{2}\right)^t$. Since the three vectors $r_1^-,\U^{\#}_{-}-\U^{\#}_{+}$ and $r_3^+$ are linearly independent by strict hyperbolicity, we have the following identity in case of non-zero initial mass:
\begin{align}\notag
	\int_{\Omega}
	(\U^{\#}-\check \U^{\#})(0,x)dx=\bar\Theta_1r_1^-+\bar\Theta_2(\U^{\#}_{-}-\U^{\#}_{+})+\bar\Theta_3r_3^+,
\end{align}
with unique constants $\bar\Theta_i,i=1,2,3.$ Without loss of generality, we can assume that $\bar\Theta_2=0$ as in \cite{Huang-Xin-Yang}. Then, 
we construct two diffusion waves to carry the extra initial mass as follows
\begin{align}\label{eqs19}
	\Theta_1(t,x_1)=\frac{1}{\sqrt{4 \pi(1+t)}} e^{-\frac{\left(x_1-\lambda_1^{-}(1+t)\right)^2}{4(1+t)}}, \quad \Theta_3(t,x_1)=\frac{1}{\sqrt{4 \pi(1+t)}} e^{-\frac{\left(x_1-\lambda_3^{+}(1+t)\right)^2}{4(1+t)}},
\end{align}
satisfying
\begin{align}\notag
	\p_t\Theta_{1 }+\lambda_1^{-} \p_{x_1}\Theta_{1 }=\p_{x_1}^2\Theta_{1}, \quad \p_t \Theta_{3 }+\lambda_3^{+} \p_{x_1}\Theta_{3 }=\p_{x_1}^2\Theta_{3 },
\end{align}
and $\int_{-\infty}^{+\infty}\Theta_i(t,x_1)dx_1=1$ for $i=1,3,$ and  $t\ge0.$ Let
\begin{align}\notag
	\bar\Theta_{i+2}:=\int_{\Omega}
	m_i(0,x)-\bar m_i(0,x_1)dx,\ \ \ i=2,3.
\end{align}
We define the new ansatz $(\rhob,\mb,\Eb)$ as
\begin{align}\label{new-ansatz-new}
	\begin{aligned}
		&\bar{\rho}(t,x_1)={\rhoc}\left(t,x_1\right)+\bar{\Theta}_{1} \Theta_{1}+\bar{\Theta}_{3} \Theta_{3}, \\
		&\bar{m}_{1}(t,x_1)={\mc}_{1}\left(t,x_1\right)+\lambda_{1-} \bar{\Theta}_{1} \Theta_{1}+\lambda_{3+} \bar{\Theta}_{3} \Theta_{3}, \\
		&\bar{m}_{i}(t,x_1)=\frac{\bar{\Theta}_{i+2}}{\sqrt{4 \pi(1+t)}} e^{-\frac{x_1^{2}}{4(1+t)}}, \quad i=2,3, \\
		&\bar{\E}(t,x_1)=\check{\mathbb{E}}\left(t,x_1\right)+\left(\frac{3(\lambda_{1-})^2}{2}\bar{\Theta}_{1} \Theta_{1}+\frac{3(\lambda_{3+})^2}{2}\bar{\Theta}_{3} \Theta_{3}\right).
	\end{aligned}
\end{align}
Then the initial extra mass under the new ansatz above becomes 
\begin{align}\label{eq-initial mass}
	&\int_{\Omega}\left(\U-\bar{\U}\right)(0,x)dx=\int_{\Omega}\left(\U^{\#}-\check{\U}^{\#}\right)(0,x)dx+\sum_{i=2}^3\bar{\Theta}_{i+2}+\int_{\Omega}\left(\check{\U}-\bar{\U}\right)(0,x)dx=0.
\end{align}
For $i=1,2,3$, we also define non-conserved quantities: velocity  $\ub_i$, temperature $\thetab$ and pressure $\bar{p}$ as follows
\begin{align}\notag
	\ub_i:=\frac{\mb_i}{\rhob},\qquad \qquad\thetab:=\frac{\Eb}{\rhob}-\frac{\abs{\mb}^2}{2\rhob^2},\qquad\qquad \bar{p}:=\frac{2}{3}\Eb-\frac{\abs{\mb}^2}{3\rhob}.
\end{align}
By \eqref{contact-wave-check} and \eqref{new-ansatz-new}, we find that the  ansatz \eqref{new-ansatz-new} satisfies the following approximate Navier-Stokes equations
\begin{align}\label{landau-2}
	\left\{  \begin{aligned} 
		&\p_t{\rhob} + \p_{x_1}{\mb}_{1} = \p_{x_1}\bar{\Rm}_{0}, \\ 
		&\p_t{\mb}_{1} + \p_{x_1}{\left( \frac{{\mb}_{1}^{2}}{\rhob } + \bar{p}\right) } = \frac{4}{3}\p_{x_1} {\left( \mu \left( \thetab \right)\p_{x_1} {\ub}_{1 }\right) } +\p_{x_1}\bar{\Rm}_{1 }, \\ 
		&\p_t{\mb}_{i} + \p_{x_1}{\left( \frac{{\mb}_{1}{\mb}_{i}}{\rhob }\right) } = \p_{x_1}{\left( \mu \left( \thetab \right) \p_{x_1}{\ub}_{i }\right) } + \p_{x_1}\bar{\Rm}_{i },\;i = 2,3, \\ 
		&\p_t{\Eb} +\p_{x_1} {\left( \frac{{\mb}_{1}\Eb}{\rhob } + \frac{{\mb}_{1}}{\rhob }\bar{p}\right) } = \p_{x_1}{\left( \kappa \left( \thetab \right) \p_{x_1}{\thetab }\right) } \\
		&\qquad\qquad\qquad+ \frac{4}{3}\p_{x_1}{\left(  \mu \left( \thetab \right) {\ub}_{1}\p_{x_1}{\ub}_{1 }\right) }+ \mathop{\sum }\limits_{{i = 2}}^{3}\p_{x_1} { \left(  \mu \left( \thetab \right) {\ub}_{i}\p_{x_1}{\ub}_{i }\right) }+\p_{x_1}\bar{\Rm}_{4} ,  \end{aligned}\right. 
\end{align}
where
\begin{align*}
	& \bar{\Rm}_0:=\Thetab_1\p_{x_1}\Theta_{1}+\Thetab_3\p_{x_1}\Theta_{3},\\
	&\bar{\Rm}_1:=\lambda_{1-}\Thetab_1\p_{x_1}\Theta_{1}+\lambda_{3+}\Thetab_3\p_{x_1}\Theta_{3}-\frac{4}{3}\left( \mu(\bar{\theta})\p_{x_1} \ub_{1}-\mu(\check{\theta})\p_{x_1} \uc_{1} \right)+\Rm_{1 } \nonumber\\
	&\qquad\;\;+\p_{x_1}\left(\frac{\mb_1^2}{\rhob}-\frac{ \abs{\mb}^2}{3\rhob} -\frac{\mc^2_1}{\rhoc}+\frac{ \abs{\mc}^2}{3\rhoc}\right),\nonumber\\
	&\bar{\Rm}_i:=\Thetab_{i+2}\p_{x_1}\Theta_{i+2}-\left( \mu(\bar{\theta})\p_{x_1} \ub_{i}-\mu(\check{\theta})\p_{x_1} \uc_{i} \right)\;+\p_{x_1}\left(\frac{\mb_1\mb_i}{\rhob} -\frac{\mc_1\mc_i}{\rhoc}\right),\quad i=2,3,\nonumber\\
	&\bar{\Rm}_4:=\frac{3(\lambda_{1-})^2}{2}\Thetab_1\p_{x_1}\Theta_{1}+\frac{3(\lambda_{3+})^2}{2}\Thetab_3\p_{x_1}\Theta_{3} - (\kappa(\thetab)\p_{x_1}\thetab - \kappa(\check{\theta})\p_{x_1}\check{\theta}) + \Rm_{4}   \nonumber\\
	&\qquad\;\;+\frac{4}{3}\p_{x_1}\left(  \mu \left( \thetab \right) {\ub}_{1}\p_{x_1} \ub_{1}-\mu \left( \check{\theta} \right) {\uc}_{1}\p_{x_1} \uc_{1} \right)  + \mathop{\sum }\limits_{{i = 2}}^{3} \p_{x_1}{ \left(  \mu \left( \thetab \right) {\ub}_{i}\p_{x_1} \ub_{i}-\left(  \mu \left( \check{\theta} \right) {\uc}_{i}\p_{x_1} \uc_{i}\right)\right) } \nonumber\\
	&\qquad\;\;+\frac53 \left( \frac{\mb_1 \bar{\E}}{\rhob}-\frac{\mc_1 \check{\E}}{\rhoc} -\lambda_1^{-} \frac{\E_{-}}{\rho_{-}}\bar{\Theta}_1\Theta_1 -\lambda_3^{+}\frac{\E_{+}}{\rho_{+}}\bar{\Theta}_3\Theta_3    \right)+\frac23 \left( \frac{\abs{\mb}^2\mb_1}{2\rhob^2} - \frac{\abs{\mc}^2\mc_1}{2\rhoc^2} \right) .\nonumber
\end{align*}
We set $\bar{\Rm}=(\bar{\Rm}_0,\cdots,\bar{\Rm}_4)^{t}$ and $\bar{\Rm}^{*}=(\bar{\Rm}_0,\bar{\Rm}_1,\bar{\Rm}_4)^{t}$. For $i=0,\cdots,4$, simple calculations yield
\begin{align}\label{tildeE}
	\big|\bar{\Rm}_i\big|\leq O\left(\delta+\sum_{j=1}^{5}\left|\bar{\Theta}_{j}\right|\right) \frac{1}{1+t}\left(e^{-\frac{c x_1^{2}}{1+t}}+e^{-\frac{c\left(x_1-\lambda_{1-}(1+t)\right)^{2}}{1+t}}+e^{-\frac{c\left(x_1-\lambda_{3+}(1+t)\right)^{2}}{1+t}}\right), 
\end{align}
where $c>0$ is a constant independent of any small parameters throughout the paper. 
We use the following notation for convenience
\begin{align}\label{errors}
	D_{-\alpha}= \frac{1}{(1+t)^{\alpha}}\left(e^{-\frac{c x_1^{2}}{1+t}}+e^{-\frac{c\left(x_1-\lambda_{1-}(1+t)\right)^{2}}{1+t}}+e^{-\frac{c\left(x_1-\lambda_{3+}(1+t)\right)^{2}}{1+t}}\right), \qquad
	\Upsilon_{-\alpha}={(1+t)^{-\alpha}}e^{-\frac{\tilde{d} x_1^{2}}{1+t}}.
\end{align}
Compared with \eqref{eqs27}, the constant $\tilde{d}>0$ in $\Upsilon_{-
	\alpha}$ above is less than $d$.
We also always use $\tilde{D}_{-\alpha}$ to denote quantities that have the same decay rate as $D_{-\alpha}$ in the sense of the $L^p$-norm, namely, for $1\leq p \le +\infty$,
\begin{align}\label{tilde-d}
	\norm{\Dt_{-\alpha}}_{L^p}\thickapprox \norm{D_{-\alpha}}_{L^p}.
\end{align}

\subsection{Coupled diffusion wave to refine the errors}
In order to write system \eqref{landau-2} in the form of conservation laws with respect to $\rhob,\mb_i,\Eb$, we introduce the following notation:
\begin{align}
	&F(\U):=\left(m_1,\frac{2m_1^2}{3\rho}+\frac{2}{3}\E-\frac{m_2^2+m_3^2}{3\rho},\frac{m_1m_2}{\rho},\frac{m_1m_3}{\rho},\frac{5m_1\E}{3\rho}-\frac{m_1 \abs{m}^2}{3\rho^2}\right)^t,\label{F}\\
	&\mathcal{B}(\U):=\mathcal{B}_1(\U)+\mathcal{B}_1^{*}(\U),\label{mathcal-B}
\end{align}
where
\begin{align*}
	&\mathcal{B}_1(\U)=
	\left(\begin{array}{ccccc}
		0&0&0&0&0\\
		- \frac{4\mu(\theta)m_1}{3\rho^2}& \frac{4\mu(\theta)}{3\rho}&0&0&0\\
		0&0&\frac{\mu(\theta)}{\rho}&0&0\\
		0&0&0&\frac{\mu(\theta)}{\rho}&0\\
		-\frac{\theta\kappa(\theta)}{\rho}&0&0&0&\frac{\kappa(\theta)}{\rho}
	\end{array}\right),\quad
	\mathcal{B}_{1}^{*}(\U)=
	\left(\begin{array}{ccccc}
		0&0&0&0&0\\
		0&0&0&0&0\\
		\mathcal{B}_{31}&0&0&0&0\\
		\mathcal{B}_{41}&0&0&0&0\\
		\mathcal{B}_{51}&\mathcal{B}_{52}&\mathcal{B}_{53}&\mathcal{B}_{54}&0
	\end{array}\right),\\
	&\mathcal{B}_{31}=-\frac{\mu(\theta)m_2}{\rho^2},\quad \mathcal{B}_{41}= -\frac{\mu(\theta)m_3}{\rho^2}, \quad \mathcal{B}_{51}=-\frac{\mu(\theta)}{\rho^3}\left( \abs{m}^2+\frac{m_1^2}{3} \right)+\frac{\abs{m}^2}{2\rho^3}\kappa(\theta),\\
	&\mathcal{B}_{52}=\frac{4 \mu(\theta)m_1-3\kappa(\theta)m_1}{3\rho^2},\quad \mathcal{B}_{53}=\frac{(\mu(\theta)-\kappa(\theta))m_2}{\rho^2},\quad \mathcal{B}_{54}=\frac{(\mu(\theta)-\kappa(\theta))m_3}{\rho^2}.
\end{align*}
Since $\theta=\frac{\E}{\rho}-\frac{|m|^2}{2\rho^2},\;\;p=\frac{2}{3}\rho\theta=\frac{2}{3}\E-\frac{|m|^2}{3\rho}$, combining \eqref{F} and \eqref{mathcal-B},  system \eqref{landau-2} can be written as
\begin{align}\label{U-bar}
	\p_t\bar\U + \p_{x_1}F(\bar \U)=\p_{x_1}\big(\mathcal{B}(\bar{\U})\p_{x_1}\bar{\U}\big)+\p_{x_1}\bar{\Rm}.
\end{align}
For the errors generated by macroscopic quantities, we expect to construct the coupled diffusion wave $\Xi=(\Xi_1,\Xi_2,\Xi_3,\Xi_4,\Xi_5)^t$ such that the error terms  of  the equation for $\tilde{\U}:=\bar{\U}+\Xi$ are improved from $\bar{\Rm}$ to $\tilde{\Rm}$, that is, $\tilde{\U}$ satisfies
\begin{align*}
	\p_t\tilde{\U}+\p_{x_1}F(\tilde{\mathcal\U})=\p_{x_1}\big(\mathcal{B}(\tilde{\U})\p_{x_1}\tilde{\U}\big)+\p_{x_1}\tilde{\Rm},
\end{align*}
with the improved time-decay $\tilde{\Rm}\approx \tilde\delta\Dt_{-\frac32}$, where $\tilde\delta$ is a small constant associated with the wave strength and the initial perturbation, and $\tilde{D}_{-\alpha}$ is defined in \eqref{tilde-d}.  To achieve this goal, we need to calculate the following quantities:
\begin{align}
	&F'(\U)=\mathcal{A}_1(\U)+\mathcal{A}_1^{*}(\U)  ,\label{F-dev-1}\\
	&\Xi^{t}F''(\U)=   \left(\begin{array}{ccccc}  0&0&0&0&0\\
		\mathcal{F}^{\Xi}_{21}&\mathcal{F}^{\Xi}_{22}&\mathcal{F}^{\Xi}_{23}&\mathcal{F}^{\Xi}_{24}&0\\
		\mathcal{F}^{\Xi}_{31}&\mathcal{F}^{\Xi}_{32}&\mathcal{F}^{\Xi}_{33}&0&0\\
		\mathcal{F}^{\Xi}_{41}&\mathcal{F}^{\Xi}_{42}&0&\mathcal{F}^{\Xi}_{44}&0\\
		\mathcal{F}^{\Xi}_{51}&\mathcal{F}^{\Xi}_{52}&\mathcal{F}^{\Xi}_{53}&\mathcal{F}^{\Xi}_{54}&\mathcal{F}^{\Xi}_{55}
	\end{array}\right),\quad \Xi^{t}  \mathcal{B}'(\U)= \left(\begin{array}{ccccc}
		0&0&0&0&0\\
		\mathcal{B}^{\Xi}_{21}&\mathcal{B}^{\Xi}_{22}&0&0&0\\
		\mathcal{B}^{\Xi}_{31}&0&\mathcal{B}^{\Xi}_{33}&0&0\\
		\mathcal{B}^{\Xi}_{41}&0&0&\mathcal{B}^{\Xi}_{44}&0\\
		\mathcal{B}^{\Xi}_{51}&\mathcal{B}^{\Xi}_{52}&\mathcal{B}^{\Xi}_{53}&\mathcal{B}^{\Xi}_{54}&\mathcal{B}^{\Xi}_{55}
	\end{array}\right), \label{higher-F-B-taylor}
\end{align}
where
\begin{align*}
	&\mathcal{A}_1(\U)=\left(\begin{array}{ccccc}
		0&1&0&0&0\\
		0&0&0&0&\frac23\\
		0&0&0&0&0\\
		0&0&0&0&0\\
		0&\frac{5\theta}{3}&0&0&0
	\end{array}\right),\quad \mathcal{A}_1^{*}(\U)=\left(\begin{array}{ccccc}
		0&0&0&0&0\\
		\mathcal{A}_{21}&\mathcal{A}_{22}&\mathcal{A}_{23}&\mathcal{A}_{24}&0\\
		\mathcal{A}_{31}&\mathcal{A}_{32}&\mathcal{A}_{33}&0&0\\
		\mathcal{A}_{41}&\mathcal{A}_{42}&0&\mathcal{A}_{44}&0\\
		\mathcal{A}_{51}&\mathcal{A}_{52}&\mathcal{A}_{53}&\mathcal{A}_{54}&\mathcal{A}_{55}
	\end{array}\right), \\
	&\mathcal{A}_{21}=-\frac{2m_1^2}{3\rho^2}+\frac{m_2^2}{3\rho^2}+\frac{m_3^2}{3\rho^2},\quad \mathcal{A}_{22}=\frac{4m_1}{3\rho}, \quad \mathcal{A}_{23}=-\frac{2m_2}{3\rho}, \quad\mathcal{A}_{24}=-\frac{2m_3}{3\rho},\quad \mathcal{A}_{31}=-\frac{m_1 m_2}{\rho^2}, \\
	& \mathcal{A}_{32}=\frac{m_2}{\rho},\quad \mathcal{A}_{33}=\frac{m_1}{\rho},\quad \mathcal{A}_{41}=-\frac{m_1 m_3}{\rho^2},\quad \mathcal{A}_{42}=\frac{m_3}{\rho},\quad \mathcal{A}_{44}=\frac{m_1}{\rho},\quad \mathcal{A}_{51}=\frac{2m_1 \abs{m}^2}{3\rho^3}-\frac{5m_1 \E}{3\rho^2},\\
	&\mathcal{A}_{52}=\frac{ \abs{m}^2}{2\rho^2}-\frac{2m_1^2 }{3\rho^2},\quad \mathcal{A}_{53}=-\frac{2m_1 m_2 }{3\rho^2},\quad \mathcal{A}_{54}=-\frac{2m_1 m_3 }{3\rho^2},\quad \mathcal{A}_{55}=\frac{5m_1 }{3\rho},
\end{align*}
and
\begin{align*}
	&\Xi^{t} \cdot \nabla_{*} f:=\Xi_1 \p_{\rho} f +  \Xi_2 \p_{m_1} f+  \Xi_3 \p_{m_2} f+  \Xi_4 \p_{m_3} f + \Xi_5 \p_{\E} f,\\[1.5mm]
	&\mathcal{F}^{\Xi}_{21}=\Xi^{t} \cdot \nabla_{*} \mathcal{A}_{21},\quad\mathcal{F}^{\Xi}_{22}=\Xi^{t} \cdot \nabla_{*} \mathcal{A}_{22}, \quad \mathcal{F}^{\Xi}_{23}=\Xi^{t} \cdot \nabla_{*} \mathcal{A}_{23}, \quad\mathcal{F}^{\Xi}_{24}=\Xi^{t} \cdot \nabla_{*} \mathcal{A}_{24}, \quad\mathcal{F}^{\Xi}_{31}=\Xi^{t} \cdot \nabla_{*} \mathcal{A}_{31}, \\[1.5mm]
	& \mathcal{F}^{\Xi}_{32}=\Xi^{t} \cdot \nabla_{*} \mathcal{A}_{32},\quad \mathcal{F}^{\Xi}_{33}=\Xi^{t} \cdot \nabla_{*} \mathcal{A}_{33}, \quad\mathcal{F}^{\Xi}_{41}=\Xi^{t} \cdot \nabla_{*} \mathcal{A}_{41}, \quad\mathcal{F}^{\Xi}_{42}=\Xi^{t} \cdot \nabla_{*} \mathcal{A}_{42}, \quad\mathcal{F}^{\Xi}_{44}=\Xi^{t} \cdot \nabla_{*} \mathcal{A}_{44}, \\[1.5mm]
	& \mathcal{F}^{\Xi}_{51}=\Xi^{t} \cdot \nabla_{*} \mathcal{A}_{51},\quad\mathcal{F}^{\Xi}_{52}=\Xi^{t} \cdot \nabla_{*} \left(\mathcal{A}_{52} +\frac{5\theta}{3}\right), \quad \mathcal{F}^{\Xi}_{53}=\Xi^{t} \cdot \nabla_{*} \mathcal{A}_{53}, \quad \mathcal{F}^{\Xi}_{54}=\Xi^{t} \cdot \nabla_{*} \mathcal{A}_{54}, \\
	& \mathcal{F}^{\Xi}_{55}=\Xi^{t} \cdot \nabla_{*} \mathcal{A}_{55},\quad \mathcal{B}^{\Xi}_{21}=-\Xi^{t} \cdot \nabla_{*} \left(\frac{4 \mu(\theta) m_1}{3 \rho^2}\right),\quad  \mathcal{B}^{\Xi}_{22}=\Xi^{t} \cdot \nabla_{*} \left(\frac{4 \mu(\theta) }{3 \rho}\right), \quad \mathcal{B}^{\Xi}_{31}=\Xi^{t} \cdot \nabla_{*} \mathcal{B}_{31},\\
	& \mathcal{B}^{\Xi}_{33}=\Xi^{t} \cdot \nabla_{*} \left(\frac{ \mu(\theta) }{ \rho}\right),\;\;\mathcal{B}^{\Xi}_{41}=\Xi^{t} \cdot \nabla_{*} \mathcal{B}_{41}, \;\; \mathcal{B}^{\Xi}_{44}=\Xi^{t} \cdot \nabla_{*} \left(\frac{ \mu(\theta) }{ \rho}\right), \;\; \mathcal{B}^{\Xi}_{51}=\Xi^{t} \cdot \nabla_{*}  \left(\mathcal{B}_{51}-\frac{\theta \mu(\theta)}{\rho} \right),  \\
	& \mathcal{B}^{\Xi}_{52}=\Xi^{t} \cdot \nabla_{*} \mathcal{B}_{52},\quad\mathcal{B}^{\Xi}_{53}=\Xi^{t} \cdot \nabla_{*} \mathcal{B}_{53},\quad \mathcal{B}^{\Xi}_{54}=\Xi^{t} \cdot \nabla_{*} \mathcal{B}_{54},\quad \mathcal{B}^{\Xi}_{55}=\Xi^{t} \cdot \nabla_{*}  \left(\frac{ \kappa(\theta) }{ \rho}\right).
\end{align*}

With the preparation of notations above, we now introduce the coupled diffusion wave which plays a vital role in obtaining the optimal time-decay for convergence of the kinetic Landau solution toward the planar entropy wave.

\begin{Def}
	Corresponding to the diffusion wave $\bar{\U}$ which satisfies \eqref{U-bar}, the coupled diffusion wave $\Xi$ is defined to satisfy the following equation with zero initial data:
	\begin{align}\label{coupled-diffusion-wave-1}
		\left\{ \begin{aligned}
			&\p_t\Xi+\p_{x_1}(F'(\bar{\U})\Xi)=\p_{x_1}\big(\mathcal{B}(\bar{\U})\p_{x_1}\Xi\big)-\frac{1}{2}\p_{x_1}(\Xi^{t}F''(\bar{\U})\Xi)+\p_{x_1}\big(\Xi^{t}\mathcal{B}{'}(\bar{\U})\p_{x_1}\bar{\U}\big)  -\p_{x_1}\bar{\Rm}-\p_{x_1}\mathcal{R}_{G}, \\
			&\Xi|_{t=0}=0.
		\end{aligned}\right.
	\end{align}
	Here, $\mathcal{R}_{G}$ in the source term is to be specified in \eqref{2025-11-10-1} later on. The purpose of introducing the coupled diffusion wave  is to cancel out the slowly decaying terms in the macroscopic part corresponding to \eqref{landau-3} or \eqref{landau-4}.
	
\end{Def}

In order to construct $\mathcal{R}_{G}$ in the desired way,  we define the linearized operator around the local Maxwellian 
$
\Mb:=M_{[\rhob,\ub,\thetab]}
$ 
as
\begin{align}\notag
	L_{\Mb}h=Q(\bar{M},h)+Q(h,\bar{M}),
\end{align}
and the base functions for the kernel space of $L_{\Mb}$ are given as
\begin{align}\label{chi-1-1}
	\left\{  \begin{array}{ll} \bar{\chi}_{0}\left( \xi \right)  \equiv  \frac{1}{\sqrt{\rhob }}\Mb, \qquad\quad
		&\bar{\chi}_{i}\left( \xi \right)  \equiv  \frac{{\xi }_{i} - \bar{u}_{i}}{\sqrt{R\bar{\theta} \bar{\rho} }}\Mb,\text{ for }i = 1,2,3, \\ [2mm]
		\bar{\chi}_{4}\left( \xi \right)  \equiv  \frac{1}{\sqrt{6\bar{\rho}} }\left( {\frac{{\left| \xi  -\bar{u}\right| }^{2}}{R\bar{\theta} } - 3}\right) \Mb, \qquad\quad
		&\left\langle  {\bar{\chi}_{i},\bar{\chi }_{j}}\right\rangle   = {\delta}_{ij},\;i,j = 0,1,2,3,4. \end{array}\right.
\end{align}
Using these five base functions in \eqref{chi-1-1} above, we define the macroscopic projection \(\bar{P}_{0}\) and the microscopic projection \(\bar{P}_{1}\) as follows:
\begin{align}\notag
	\bar{P}_{0}h \equiv  \mathop{\sum }\limits_{{j = 0}}^{4}\left\langle  {h,\bar{\chi}_{j}}\right\rangle  \bar{\chi}_{j},\qquad\bar{P}_{1}h \equiv  h - \bar{P}_{0}h.
\end{align}
According to $G$ in \eqref{G}, we then introduce two correction functions by 
\begin{align}
	\bar{G}_0&=\frac{3}{2\theta}{\bf L}_{M}^{-1}\left\{P_1 \left[\xi_1\left(\frac{|\xi-u|^2}{2\theta}\p_{x_1} \bar{\theta}+(\xi-u)\cdot\p_{x_1}\bar u\right){M}\right]\right\}\notag\\
	&=\frac{\sqrt{R} \p_{x_1} \bar{\theta}   }{\sqrt{\theta} } {A}_1\left( \frac{\xi-u}{\sqrt{R \theta}}  \right) + \sum_{j=1}^3 \p_{x_1}\bar u_{j}{B}_{1 j}\left( \frac{\xi-u}{\sqrt{R \theta}}  \right),\label{correction-G-2} \\
	\bar{G}&=\frac{3}{2\thetab}{\bf L}_{\bar{M}}^{-1}\left\{{\bar{P}}_1\left[\xi_1\left(\frac{|\xi-\bar{u}|^2}{2\thetab}\p_{x_1} \bar{\theta}+(\xi-\bar{u})\cdot\p_{x_1}\bar{u}\right){\bar{M}}\right]\right\}\notag\\
	&=\frac{\sqrt{R} \p_{x_1} \bar{\theta}   }{\sqrt{\bar{\theta}} } \bar{A}_1\left( \frac{\xi-\bar{u}}{\sqrt{R \bar{\theta}}}  \right) + \sum_{j=1}^3 \p_{x_1}\bar{u}_{j}\bar{B}_{1 j}\left( \frac{\xi-\bar{u}}{\sqrt{R \bar{\theta}}}  \right).\label{correction-G-1}  
\end{align}
By \eqref{Theta} and \eqref{correction-G-1}, we define $
\bar{\Pi}_1:=\bar{P}_1(\xi_1\p_{x_1} \bar{G})-Q(\bar{G},\bar{G}).$
Thus, the errors of the microscopic components can be written as
\begin{align}\label{2025-11-10-1}
	\mathcal{R}_{G}:=\left(0,\;\int_{\R^3}\xi_1^2L^{-1}_{\bar{M}}\bar{\Pi}_1d\xi,\;\int_{\R^3}\xi_1\xi_2L^{-1}_{\bar{M}}\bar{\Pi}_1d\xi,\;\int_{\R^3}\xi_1\xi_3L^{-1}_{\bar{M}}\bar{\Pi}_1d\xi,\;\frac{1}{2}\int_{\R^3}\xi_1|\xi|^2L^{-1}_{\bar{M}}\bar{\Pi}_1d\xi\right)^t. 
\end{align}
For $ k\geq1$ with $k\in \mathbb{N}$, using \eqref{5.1}, \eqref{5.2} and the property of $L^{-1}_{\Mb}$, we have
\begin{align*}
	&\p_{x_1}^k\int_{\mathbb{R}^{3}}\xi_{1}\xi_{i}L^{-1}_{\bar{M}}\bar{\Pi}_{1} \,d\xi=
	\p_{x_1}^k\int_{\mathbb{R}^{3}} L^{-1}_{\bar{M}}\{\Pb_{1}( \xi_{1}\xi_{i}\bar{M})\}\frac{\bar{\Pi}_{1}}{\bar{M}} \,d\xi
	\nonumber\\
	=&\p_{x_1}^k\int_{\mathbb{R}^{3}} L^{-1}_{\bar{M}}\{R\bar{\theta}\hat{B}_{1i}(\frac{\xi-\ub}{\sqrt{R\bar{\theta}}})\bar{M}\}\frac{\bar{\Pi}_{1}}{\bar{M}}\, d\xi
	=R\p_{x_1}^k\left[\bar{\theta}\int_{\mathbb{R}^{3}}\bar{B}_{1i}(\frac{\xi-\bar{u}}{\sqrt{R\bar{\theta}}})\frac{\bar{\Pi}_{1}}{\bar{M}} \,d\xi\right]\le C\deltab\p_{x_1}^k D_{-1},
\end{align*}
and
\begin{align*}
	&\p_{x_1}^k \int_{\mathbb{R}^{3}} (\frac{1}{2}\xi_{1}|\xi|^{2}-\xi_{1}\xi\cdot \ub)L^{-1}_{\bar{M}}\bar{\Pi}_{1} \,d\xi= \p_{x_1}^k
	\int_{\mathbb{R}^{3}} L^{-1}_{\bar{M}}\{\bar{P}_{1}(\frac{1}{2}\xi_{1}|\xi|^{2}-\xi_{1}\xi\cdot \bar{u})\bar{M}\}\frac{\bar{\Pi}_{1}}{\bar{M}}\, d\xi
	\nonumber\\
	=&\p_{x_1}^k \int_{\mathbb{R}^{3}} L^{-1}_{\bar{M}}\{(R\thetab)^{\frac{3}{2}}\hat{A}_{1}(\frac{\xi-\bar{u}}{\sqrt{R\thetab}})\bar{M}\}\frac{\bar{\Pi}_{1}}{\bar{M}}\, d\xi
	=\p_{x_1}^k \left[(R\thetab)^{\frac{3}{2}}\int_{\mathbb{R}^{3}}\bar{A}_{1}(\frac{\xi-\bar{u}}{\sqrt{R\thetat}})\frac{\bar{\Pi}_{1}}{\bar{M}}\, d\xi\right] \le C \deltab\p_{x_1}^k D_{-1}.
\end{align*}
Combining \eqref{U-bar} and \eqref{coupled-diffusion-wave-1}, the function $\tilde{\U}=(\rhot,\mt_1,\mt_2,\mt_3,\Et)^{t}:=\bar{\U}+\Xi$ defined as the superposition of the background solution $\bar{\U}$ and the coupled diffusion wave  $\Xi$ satisfies
\begin{align}\label{New-Ansatz}
	\p_t\tilde{\U}+\p_{x_1}F(\tilde{\U})&=\p_{x_1}\Big(\mathcal{B}(\tilde{\U})\p_{x_1}\tilde{\U}\Big) +\p_{x_1}\tilde{\Rm} -\p_{x_1} \mathcal{R}_{G},
\end{align}
where 
\begin{align}\label{new-error}
	\tilde{\Rm}=\Big(F(\tilde{\U})-F(\bar{\U})-F'(\bar{\U})\Xi-\frac{1}{2}\Xi^{t}F''(\bar{\U})\Xi\Big)+\Big[\left( \mathcal{B}(\bar{\U})+\Xi^{t}\mathcal{B}'(\bar{\U})-\mathcal{B}(\tilde{\U}) \right) \p_{x_1}\tilde{\U}
	- \Xi^{t}\mathcal{B}'(\bar{\U})\p_{x_1}\Xi  \Big].
\end{align}
It can be seen from \eqref{eq-initial mass} and \eqref{coupled-diffusion-wave-1} that $\U-\tilde{\U}$ still satisfies the zero mass condition, {\it i.e.}
\begin{align}\label{initial-zero-mass-2}
	\int_{\Omega} \U(0,x)-\tilde{\U}(0,x)dx=\int_{\Omega} \U(0,x)-\bar{\U}(0,x)dx-\int_{\Omega} \Xi (0,x)dx=0.
\end{align}
Moreover, as can be seen from Corollary \ref{R-m-1-1-1}, $\tilde{\mathcal{R}}\approx \deltab\tilde{D}_{-\frac32}$ turns out to be satisfied.
For $i=1,2,3$, we also define the background non-conserved quantities: velocity  $\ut_i$, temperature $\thetat$ and pressure $\tilde{p}$ as follows
\begin{align}\label{non-conserved quantities}
	\ut_i:=\frac{\mt_i}{\rhot},\qquad \qquad\thetat:=\frac{\Et}{\rhot}-\frac{\abs{\mt}^2}{2\rhot^2},\qquad\qquad\tilde{p}=\frac{2}{3}\Et-\frac{\abs{\mt}^2}{3\rhot}.
\end{align}
The $L^p$ norm decay rate of the coupled diffusion wave $\Xi$ \eqref{coupled-diffusion-wave-1} aligns with that of the diffusion wave $\Theta$  \eqref{eqs19}. For details, we refer to Theorem \ref{res-Wt} and Corollary  \ref{R-m-1-1-1} later on.

\subsection{Main result}

We define the perturbation for the new ansatz \eqref{correction-G-2},  \eqref{New-Ansatz} and  \eqref{non-conserved quantities} by
\begin{align}\label{ori-perb-1}
	(\phi,\varphi,h,\psi,\zeta,\sqrt{\mu}g):=(\rho-\rhot,m-\mt,\E-\Et,u-\ut,\theta-\thetat,G-\bar{G}_0).
\end{align}
For convenience, we use the following notations:
\begin{align}\label{2025-11-5-1}
	{\Vm}:=(\Phi,\Psi,H)^{t},\qquad{\Vmc}:=(\Phic,\Psic,\Hc)^{t},\qquad \Vmc^{\ast}:=(0,\Psic,\Hc),\qquad \vm:=(\phi,\psi,\zeta)^{t},\qquad {\vm}^{\ast}:=(0,\psi,\zeta)^{t},
\end{align}
where $(\Phi,\Psi,H)$ and $(\Phic,\Psic,\Hc)$ are defined in \eqref{anti-11-5} and \eqref{anti-trans-qwe} respectively.   Next, for the solution  $f$ of landau equation \eqref{equ-landau} and the definition of $\mu$ \eqref{2026-5-9-2}, we introduce the instant energy
functionals  $\mathcal{E}_{1,\omega} (t),\mathcal{E}_{2,\omega}(t)$ with weight $\omega$ \eqref{1.32} and the unweighted instant energy functionals $\mathcal{E}_{1} (t),\mathcal{E}_{2} (t),\mathcal{E}_{3} (t)$, $\tilde{\mathcal{E}}_{2} (t)$ by
\begin{align}
	\mathcal{E}_{1,\omega}(t)=&\norm{\Vmc}_{H^1}^2+\norm{\p_{x_1}^2\left(\Phic,\Hc,\sum_{i=2}^3 \Psic_i\right)}_{L^2}^2 + \norm{\p_{x_1}\left( \thetat \p_{x_1} \Psic_1\right)}_{L^2}^2  +\norm{\vm_{\neq}}_{H^1}^2+\norm{\nabla_x^2 \vm}_{L^2}^2+\tilde{C}\sum_{\abs{\alpha}=3}\norm{\frac{\p^{\alpha}f}{\sqrt{\mu}}}_{2,\omega}^2 \notag\\
	&+\sum_{l=1}^4\mathcal{X}^l+\tilde{c}\left(\sum_{k=0}^1\int_{\R} \p_{x_1}^{k+1}\Phic\p_{x_1}^k\Psic_1 dx_1 +\sum_{k=1}^2\int_{\Omega} \nabla_x^k \psi \cdot\nabla_x^{k+1} \phi dx\right)+\tilde{C}\sum_{0\le \abs{\alpha}\le 2\atop 0\leq\abs{\alpha}+\abs{\beta}\leq3}\|\partial^\alpha_\beta g\|^2_{2,\omega},\label{2.10}\\
	\mathcal{E}_{2,\omega}(t)=&\norm{\p_{x_1}\Vmc}_{L^2}^2+\norm{\p_{x_1}^2\left(\Phic,\Hc,\sum_{i=2}^3 \Psic_i\right)}_{L^2}^2 + \norm{\p_{x_1}\left( \thetat \p_{x_1} \Psic_1\right)}_{L^2}^2 +\norm{\vm_{\neq}}_{H^1}^2+\norm{\nabla_x^2 \vm}_{L^2}^2 \notag\\
	&+c\left(\int_{\R} \p_{x_1}^{2}\Phic\p_{x_1}\Psic_1 dx_1 +\sum_{k=1}^2\int_{\Omega} \nabla_x^k \psi \cdot\nabla_x^{k+1} \phi dx\right)+\tilde{C}\sum_{0\le \abs{\alpha}\le 2\atop 0\leq\abs{\alpha}+\abs{\beta}\leq3}\|\partial^\alpha_\beta g\|^2_{2,\omega}+\tilde{C}\sum_{\abs{\alpha}=3}\norm{\frac{\p^{\alpha}f}{\sqrt{\mu}}}_{2,\omega}^2,\label{2.10-1}\\
	\mathcal{E}_{1}(t)=&\norm{\Vmc}_{H^1}^2+\norm{\p_{x_1}^2\left(\Phic,\Hc,\sum_{i=2}^3 \Psic_i\right)}_{L^2}^2 + \norm{\p_{x_1}\left( \thetat \p_{x_1} \Psic_1\right)}_{L^2}^2 +\norm{\vm_{\neq}}_{H^1}^2+\norm{\nabla_x^2 \vm}_{L^2}^2 +\tilde{C}\sum_{\abs{\alpha}=3}\norm{\frac{\p^{\alpha}f}{\sqrt{\mu}}}^2_{2}\notag\\
	&+\sum_{l=1}^4\mathcal{X}^l+\tilde{c}\left(\sum_{k=0}^1\int_{\R} \p_{x_1}^{k+1}\Phic\p_{x_1}^k\Psic_1 dx_1 +\sum_{k=1}^2\int_{\Omega} \nabla_x^k \psi \cdot\nabla_x^{k+1} \phi dx\right)+\tilde{C}\sum_{0\le \abs{\alpha}\leq 2}\|\partial^\alpha g\|^2_{2},\label{2.10-2}\\
	\mathcal{E}_{2}(t)=&\norm{\p_{x_1}\Vmc}_{L^2}^2+\norm{\p_{x_1}^2\left(\Phic,\Hc,\sum_{i=2}^3 \Psic_i\right)}_{L^2}^2 + \norm{\p_{x_1}\left( \thetat \p_{x_1} \Psic_1\right)}_{L^2}^2+\norm{\vm_{\neq}}_{H^1}^2+\norm{\nabla_x^2 \vm}_{L^2}^2 \notag\\
	&+\tilde{c}\left(\int_{\R} \p_{x_1}^{2}\Phic\p_{x_1}\Psic_1 dx_1 +\sum_{k=1}^2\int_{\Omega} \nabla_x^k \psi \cdot\nabla_x^{k+1} \phi dx\right)+\tilde{C}\sum_{0\le \abs{\alpha}\leq 2}\|\partial^\alpha g\|^2_{2}+\tilde{C}\sum_{\abs{\alpha}=3}\norm{\frac{\p^{\alpha}f}{\sqrt{\mu}}}^2_{2},\label{2.10-3}\\
	\tilde{\mathcal{E}}_{2}(t)=&\norm{\p_{x_1}\Vmc}_{L^2}^2+\norm{\p_{x_1}^2\left(\Phic,\Hc,\sum_{i=2}^3 \Psic_i\right)}_{L^2}^2 + \norm{\p_{x_1}\left( \thetat \p_{x_1} \Psic_1\right)}_{L^2}^2+\norm{\vm_{\neq}}_{H^1}^2+\norm{\nabla_x^2 \vm}_{L^2}^2 \notag\\
	&+\tilde{c}\left(\int_{\R} \p_{x_1}^{2}\Phic\p_{x_1}\Psic_1 dx_1 +\sum_{k=1}^2\int_{\Omega} \nabla_x^k \psi \cdot\nabla_x^{k+1} \phi dx\right)+\tilde{C}\sum_{1\le \abs{\alpha}\leq 2}\|\partial^\alpha g\|^2_{2}+\tilde{C}\sum_{\abs{\alpha}=3}\norm{\frac{\p^{\alpha}f}{\sqrt{\mu}}}^2_{2},\label{2.10-4}\\
	\mathcal{E}_{3}(t)=&\norm{\p_{x_1}^2\left(\Phic,\Hc,\sum_{i=2}^3 \Psic_i\right)}_{L^2}^2 + \norm{\p_{x_1}\left( \thetat \p_{x_1} \Psic_1\right)}_{L^2}^2 +\norm{\nabla_x\vm_{\neq}}_{L^2}^2+\norm{\nabla_x^2 \vm}_{L^2}^2 \notag\\
	&+\tilde{c}\sum_{k=1}^2\int_{\Omega} \nabla_x^k \psi \cdot\nabla_x^{k+1} \phi dx+\tilde{C}\sum_{1\le \abs{\alpha}\leq 2}\|\partial^\alpha g\|^2_{2}+\tilde{C}\sum_{\abs{\alpha}=3}\norm{\frac{\p^{\alpha}f}{\sqrt{\mu}}}^2_{2},\label{2.10-5}
\end{align}
where $\mathcal{X}^l$ is defined in \eqref{definition-of-Xm} and $\tilde{c}$ and $\tilde{C}$  are sufficiently small and sufficiently large constants, respectively, used to ensure that all the transient energy functionals defined above are positive. 

\begin{Rem}
	For $N=1,2$, a direct calculation shows that the transient energy functionals defined above have the following equivalent relationship 
	\begin{align*}
		& \mathcal{E}_{N,\omega}\approx\abs{N-2}\norm{\Vmc}_{L^2}^2+\norm{\vm}_{H^2}^2+\sum_{0\le \abs{\alpha}\le 2\atop 0\leq\abs{\alpha}+\abs{\beta}\leq3}\|\partial^\alpha_\beta g\|^2_{2,\omega}+\sum_{\abs{\alpha}=3}\norm{\frac{\p^{\alpha}f}{\sqrt{\mu}}}_{2,\omega}^2,\\
		&\mathcal{E}_{N}\approx\abs{N-2}\norm{\Vmc}_{L^2}^2+\norm{\vm}_{H^2}^2+\sum_{0\le \abs{\alpha}\le 2}\|\partial^\alpha g\|^2_{2}+\sum_{\abs{\alpha}=3}\norm{\frac{\p^{\alpha}f}{\sqrt{\mu}}}_{2}^2,\\
		&\tilde{\mathcal{E}}_{2}\approx\norm{\vm}_{H^2}^2+\sum_{1\le \abs{\alpha}\le 2}\|\partial^\alpha g\|^2_{2}+\sum_{\abs{\alpha}=3}\norm{\frac{\p^{\alpha}f}{\sqrt{\mu}}}_{2}^2,\qquad\mathcal{E}_{3}\approx\norm{\nabla_x\vm}_{H^1}^2+\sum_{1\le \abs{\alpha}\le 2}\|\partial^\alpha g\|^2_{2}+\sum_{\abs{\alpha}=3}\norm{\frac{\p^{\alpha}f}{\sqrt{\mu}}}_{2}^2.
	\end{align*}
\end{Rem}
And the corresponding dissipation energy  functionals $\mathcal{D}_{1,\omega} (t),\mathcal{D}_{2,\omega} (t),\mathcal{D}_{3,\omega} (t)$ with weight $\omega$ and the unweighted dissipation  energy functionals $\mathcal{D}_{1} (t),\mathcal{D}_{2} (t),\mathcal{D}_{3} (t),\tilde{\mathcal{D}}_{2} (t)$ by
\begin{align}
	&\mathcal{D}_{1,\omega}(t)=\norm{\p_{x_1}\Vmc}_{H^2}^2+\norm{\vm_{\neq}}_{H^2}^2+\norm{\nabla_x^3 \vm}_{L^2}^2+\sum_{0\leq|\alpha|+|\beta|\leq 3}\|\partial^\alpha_\beta g\|^2_{\sigma,\omega}, \label{2.11}\\
	&\mathcal{D}_{2,\omega}(t)=\norm{\p_{x_1}^2\Vmc}_{H^1}^2+\norm{\nabla_x \vm_{\neq}}_{H^1}^2+\norm{\nabla_x^3 \vm}_{L^2}^2+\sum_{0\leq|\alpha|+|\beta|\leq 3}\|\partial^\alpha_\beta g\|^2_{\sigma,\omega},\label{2.11-1}\\
	&\mathcal{D}_{3,\omega}(t)=\norm{\p_{x_1}^3\Vmc}_{L^2}^2+\norm{\nabla_x^2 \vm_{\neq}}_{L^2}^2+\norm{\nabla_x^3 \vm}_{L^2}^2+\sum_{1\leq|\alpha|+|\beta|\leq 3}\|\partial^\alpha_\beta g\|^2_{\sigma,\omega},\label{2.11-1-t}\\
	&\mathcal{D}_{1}(t)=\norm{\p_{x_1}\Vmc}_{H^2}^2+\norm{\vm_{\neq}}_{H^2}^2+\norm{\nabla_x^3 \vm}_{L^2}^2+\sum_{0\leq|\alpha|\leq 3}\|\partial^\alpha g\|^2_{\sigma}, \label{2.11-2}\\
	&\mathcal{D}_{2}(t)=\norm{\p_{x_1}^2\Vmc}_{H^1}^2+\norm{\nabla_x \vm_{\neq}}_{H^1}^2+\norm{\nabla_x^3 \vm}_{L^2}^2+\sum_{0\leq|\alpha|\leq 3}\|\partial^\alpha g\|^2_{\sigma},\label{2.11-3}\\
	&\tilde{\mathcal{D}}_{2}(t)=\norm{\p_{x_1}^2\Vmc}_{H^1}^2+\norm{\nabla_x \vm_{\neq}}_{H^1}^2+\norm{\nabla_x^3 \vm}_{L^2}^2+\sum_{1\leq|\alpha|\leq 3}\|\partial^\alpha g\|^2_{\sigma},\label{2.11-4}\\
	&\mathcal{D}_{3}(t)=\norm{\p_{x_1}^3\Vmc}_{L^2}^2+\norm{\nabla_x^2 \vm_{\neq}}_{L^2}^2+\norm{\nabla_x^3 \vm}_{L^2}^2+\sum_{1\leq|\alpha|\leq 3}\|\partial^\alpha g\|^2_{\sigma}.\label{2.11-5}
\end{align}

Now we can precisely state the main result of this paper corresponding to the rough version Theorem \ref{thm.rvmt} given before.

\begin{Thm}\label{mt}
	Let $M_{[\rhoc,\uc,\check{\theta}](t,x_1)}(\xi)$ be the viscous entropy wave defined in \eqref{2026-5-9-1}  with the
	small wave strength
	$\delta=|\theta_{+}-\theta_{-}|>0$.
	Then, there is a sufficiently small constant $\varepsilon_{0}>0$  such that if the initial data $f_{0}(x,\xi)\geq 0$ satisfies
	\begin{align}\notag
		\norm{(\rho-\rhoc,m-\mc,\E-\Ec)|_{t=0}}_{L^1}^2+\mathcal{E}_{1,\omega}(0)\leq C\varepsilon^{2}_{0},
	\end{align}
	where $\mathcal{E}_{1,\omega}$ is defined in \eqref{2.10} with  $l\geq 2$ being arbitrarily given in \eqref{1.32}, then the Cauchy problem  \eqref{equ-landau} and \eqref{ini} on the Landau equation with Coulomb interaction  admits a unique  global solution $f(t,x,\xi)\geq 0$ satisfying
	\begin{align}\notag
		\Big\|\frac{f-M_{[\rhoc,\uc,\check{\theta}]}}{\sqrt{\mu}}\Big\|_{L_{x}^{\infty}L_{\xi}^{2}}\leq C(\varepsilon_0^{\frac12}+{\delta}^{\frac12})(1+t)^{-\frac{1}{2}}, \quad \norm{\Dn\left(\frac{f-M_{[\rhoc,\uc,\check{\theta}](t,x_1)}}{\sqrt{\mu}}\right)}_{L_x^\infty L_{\xi}^2}^2\leq C (\delta+\varepsilon_0) e^{-ct^{\frac23}}.
	\end{align}
\end{Thm}

In order to prove Theorem \ref{mt}, we first show the following result.

\begin{Thm}\label{mt-1}
	Under  the same assumptions of Theorem \ref{mt}, it holds that
	\begin{align}\notag
		\mathcal{E}_{2,\omega}\le C(\varepsilon_0+\delta),\qquad\quad\mathcal{E}_{2}(t)\leq C(\varepsilon_0+\delta)(1+t)^{-\frac{1}{2}},\qquad\quad\mathcal{E}_{3}(t)\leq C(\varepsilon_0+\delta)(1+t)^{-\frac{3}{2}},
	\end{align}
	where $\mathcal{E}_{2,\omega},\mathcal{E}_{2}(t),\mathcal{E}_{3}(t)$ are defined in \eqref{2.10-1}, \eqref{2.10-3} and \eqref{2.10-5}. Moreover, for the non-zero modes, it holds that
	\begin{align*}
		\norm{\Dn\left(\frac{f-M_{[\rhoc,\uc,\check{\theta}](t,x_1)}}{\sqrt{\mu}}\right)}_{L_x^\infty L_{\xi}^2}^2\leq C (\delta+\varepsilon_0) e^{-ct^{\frac23}},
	\end{align*}
	where  $C$ is a positive constant independent of $\varepsilon_0$ and $\delta$, and $\Dn$ is defined in \eqref{2025-11-10-4}. 
\end{Thm}

Once we have Theorem \ref{mt-1}, we are ready to prove Theorem \ref{mt}.

\medskip
\noindent{\bf{Proof of Theorem \ref{mt}}.} In terms of \eqref{assumption-1}, \eqref{eqs19}, \eqref{new-ansatz-new}, \eqref{New-Ansatz}, \eqref{2.10-3} and \eqref{2.10-5}, one has
\begin{align*}
	&\norm{\frac{f(t,x,\xi)-M_{[\rhoc,\ut,\check{\theta}]}}{\sqrt{\mu}}}_{L_{x}^{\infty}L_{\xi}^{2}}\le \norm{\frac{M_{[\rho,u,\theta]}-M_{[\rhot,\ut,\thetat]}+M_{[\rhot,\ut,\thetat]}-M_{[\rhoc,\uc,\check{\theta}]}+\bar{G}_0+\sqrt{\mu}g}{\sqrt{\mu}}}_{L_{x}^{\infty}L_{\xi}^{2}}\\
	&\qquad\qquad\quad\le C\left(\norm{\vm}_{L_x^{\infty}}+\sum_{i=\{1,3\}}\Thetab_i\norm{\Theta_i}_{L^\infty} +\sum_{i=2}^3\norm{\bar{m}_i}_{L^\infty}+\norm{\Xi}_{L^\infty}+\norm{\abs{\bar{G}_0}_2}_{L_x^\infty}+\norm{\abs{g}_2}_{L_x^\infty}\right)\\
	&\qquad\qquad\quad\le C\left( \mathcal{E}_2^{\frac14}\mathcal{E}_3^{\frac14}+\mathcal{E}_3^{\frac12}+\sum_{i=\{1,3\}}\Thetab_i\norm{\Theta_i}_{L^\infty} +\sum_{i=2}^3\norm{\bar{m}_i}_{L^\infty}+\norm{\Xi}_{L^\infty}+\norm{\abs{\bar{G}_0}_2}_{L_x^\infty}  \right).
\end{align*}
According to \eqref{eqs27}, \eqref{eqs19}, \eqref{new-ansatz-new}, \eqref{correction-G-2}, Theorem \ref{mt-1} and the estimate of coupled diffusion wave $\Xi_i$ (see Corollary \ref{R-m-1-1-1} below), we have completed the proof of Theorem \ref{mt}. \qed

\subsection{The {\it a priori} estimates}
The local existence of the solutions to the Landau system \eqref{equ-landau} near a global Maxwellian was proved in \cite{Guo-2002}.  By a straightforward modification of the argument there, we can obtain the local existence of the solutions to the Landau system \eqref{equ-landau} and \eqref{ini} with $f(t,x,\xi)\ge0$ under the assumptions in Theorem \ref{mt}. As for the local existence of the anti-derivative variables $(\Phi,\Psi,W)$, since 
$f$ already exists, the corresponding macroscopic system has a similar structure to the Navier-Stokes equations, and thus its local existence can be established analogously. For brevity, we omit all details of the proof.

Thus, by the continuity method, it suffices to provide a uniform {\it a priori} estimate. The {\it a priori} assumptions are given by
\begin{align}\label{apa}
	\sup_{0\le t \le T} \bigg\{ &\norm{\Vm}_{L_x^\infty}^2+\sum_{i=0}^1(1+t)^{\frac12+i}\norm{\nabla_x^i \vm}_{L^2}^2 + \sum_{k=2}^3(1+t)^{\frac32}\norm{\nabla_x^k \vm}_{L^2}^2+(1+t)^{\frac12}\norm{g}_2^2+\sum_{\abs{\alpha}=1}^2(1+t)^{\frac32}\norm{\p^{\alpha}g}_2^2 \notag\\
	&\qquad+\sum_{\abs{\alpha}=3}(1+t)^{\frac32}\norm{\frac{\p^{\alpha}f}{\sqrt{\mu}}}_2^2+\sum_{\abs{\alpha}+\abs{\beta}=0}^2 \norm{\p_{\beta}^{\alpha}g}_{2,\omega}^2+\sum_{\abs{\beta}=3} \norm{\p_{\beta}g}_{2,\omega}^2+\sum_{\abs{\alpha}=3}\norm{\frac{\p^{\alpha}f}{\sqrt{\mu}}}_{2,\omega}^2\bigg\} \leq \chi^2.
\end{align}
Then we need to close the {\it a priori} assumptions. Indeed, we are able to prove the following {\it a priori} estimates.
\begin{Prop}[{\it a priori} estimates]\label{Thm-ape}
	Assume that $(\Vm,\vm,g,f)$ \eqref{2025-11-5-1} is the unique solution given in the local existence and satisfies the {\it a priori} assumptions \eqref{apa}, then the following estimates hold
	\begin{align*}
		&\norm{\Vm}_{L_x^\infty}^2+\sum_{i=0}^1(1+t)^{\frac12+i}\norm{\nabla_x^i \vm}_{L^2}^2 + \sum_{k=2}^3(1+t)^{\frac32}\norm{\nabla_x^k \vm}_{L^2}^2+(1+t)^{\frac12}\norm{g}_2^2+\sum_{\abs{\alpha}=1}^2(1+t)^{\frac32}\norm{\p^{\alpha}g}_2^2 \notag\\
		&\qquad\qquad+\sum_{\abs{\alpha}=3}(1+t)^{\frac32}\norm{\frac{\p^{\alpha}f}{\sqrt{\mu}}}_2^2+\sum_{\abs{\alpha}+\abs{\beta}=0}^2 \norm{\p_{\beta}^{\alpha}g}_{2,\omega}^2+\sum_{\abs{\beta}=3} \norm{\p_{\beta}g}_{2,\omega}^2+\sum_{\abs{\alpha}=3}\norm{\frac{\p^{\alpha}f}{\sqrt{\mu}}}_{2,\omega}^2 \le C(\delta+\varepsilon_0),
	\end{align*}
	where $C>0$ is a universal constant independent of any small parameters in this paper.
\end{Prop}

By the local existence combined with Proposition \ref{Thm-ape} and  Theorem \ref{lem999} (to be stated later), using a continuity argument, we are able to obtain Theorem \ref{mt-1}. In the following sections, we are devoted to the proof of Proposition \ref{Thm-ape} and  Theorem \ref{lem999}. And, the estimate of the coupled diffusion wave $\Xi_i$ will be given in Corollary \ref{R-m-1-1-1}. 

\section{Decay estimate for coupled diffusion waves}\label{sec3}
Combining \eqref{F}, \eqref{mathcal-B}, \eqref{U-bar}, \eqref{F-dev-1} and \eqref{higher-F-B-taylor}, the coupled diffusion wave system \eqref{coupled-diffusion-wave-1} can be written as 
\begin{align}\label{omega}
	\left\{\begin{aligned}
		&\p_t\Xi+\p_{x_1}\big(\mathcal{A}_1(\bar{\U})\Xi\big)=\p_{x_1}\big(\mathcal{B}_1(\bar{\U})\p_{x_1}\Xi\big)-\p_{x_1}\big( \mathcal{A}_1^{*}(\bar{\U})\Xi \big)+\p_{x_1}\big(\mathcal{B}_1^{*}(\bar{\U}){\p_{x_1}\Xi}\big) \\
		&\qquad\qquad\qquad\qquad\qquad-\frac12\p_{x_1}(\Xi^{t}F''(\bar{\U})\Xi)+\p_{x_1}\big(\Xi^{t}\mathcal{B}{'}(\bar{\U}){\p_{x_1}\bar{\U}}\big)  -\p_{x_1}\bar{\Rm}-\p_{x_1}\mathcal{R}_{G}, \\
		&\Xi(0,x_1)=0.
	\end{aligned}\right.
\end{align}
Since $\Xi|_{t=0} = 0$, we can define the anti-derivative of $\Xi_i$ as $  W_i=\int_{-\infty}^{x_1}\Xi_i(\cdot,t)dy$, and thus system \eqref{omega} becomes
\begin{align}\label{W}
	\p_t\W+\mathcal{A}_1(\bar{\U})\p_{x_1}\W=&\mathcal{B}_1(\bar{\U})\p_{x_1}^2\W- \mathcal{A}_1^{*}(\bar{\U})\p_{x_1}\W+\mathcal{B}_{1}^{*}(\bar{\U})\p_{x_1}^2\W \notag\\
	&-\frac{1}{2}\W_{x_1}^{t}F''(\bar{\U})\p_{x_1}\W+\p_{x_1}\W^{t}\mathcal{B}{'}(\bar{\U}){\p_{x_1}\bar{\U}}  -\bar{\Rm}-\mathcal{R}_{G},
\end{align}
for $\W=(W_1,\cdots,W_5)^{t}$. 
To ensure that \eqref{W} satisfies the left-right structural conditions,   we make the following transformations
\begin{align}\label{trans-W-aa}
	\bar{W}_1=\thetab W_1,\quad \bar{W}_2=W_2,\quad \bar{W}_3=W_3,\quad \bar{W}_4=W_4,\quad \bar{W}_5=W_5-\thetab W_1,
\end{align}
then \eqref{W} becomes
\begin{align}\label{Wt}
	\p_t\bar{\W}+\bar{\mathcal{A}}_1(\bar{\U})\p_{x_1}\bar{\W}=\bar{\mathcal{B}}_1(\bar{\U})\p_{x_1}^2\bar{\W}+\mathcal{S}_{\bar \W},
\end{align}
where $\bar{\W}=(\Wb_1,\cdots,\Wb_5)^{t}$,
\begin{align*}
	\bar{\mathcal{A}}_1(\bar{\U})=\left(\begin{array}{ccccc}
		0 &\thetab &0 &0&0 \\
		\frac{2}{3}& 0&0&0  &\frac{2}{3}\\
		0 &0&0&0&0\\
		0&0&0&0&0\\
		0&\frac{2}{3}\thetab&0&0&0
	\end{array}\right),\quad \bar{\mathcal{B}}_{1}(\bar{\U})=\left(\begin{array}{ccccc}
		0 &0 &0&0&0  \\
		-\frac{4}{3}\frac{\mu(\thetab)\mb_1}{\rhob^2\thetab} & \frac{4}{3}\frac{\mu(\thetab)}{\rhob}  &0&0&0\\
		0&0&\frac{\mu(\thetab)}{\rhob}&0&0\\
		0&0&0&\frac{\mu(\thetab)}{\rhob}&0\\
		0&0&0&0&\frac{\kappa(\thetab)}{\rhob}
	\end{array}\right),
\end{align*}
and 
\begin{align}
	&\mathcal{S}_{\bar \W}=\left( (\mathcal{S}_{\bar \W})_1, (\mathcal{S}_{\bar \W})_2, (\mathcal{S}_{\bar \W})_3,(\mathcal{S}_{\bar \W})_4,(\mathcal{S}_{\bar \W})_5 \right)^{t}, \notag \\
	&(\mathcal{S}_{\bar \W})_1= \thetab (\mathcal{S}_{\W})_1-\p_t\thetab W_1,\quad (\mathcal{S}_{\bar \W})_2=(\mathcal{S}_{\W})_2-\frac{4 \mu(\thetab) \mb_1}{3 \rhob^2} \left[ \p_{x_1}^2 \left( \frac{\Wb_1}{\thetab} \right)-\frac{\p_{x_1}^2\Wb_1}{\thetab} \right] ,\quad (\mathcal{S}_{\bar \W})_3=(\mathcal{S}_{\W})_3, \notag\\
	&(\mathcal{S}_{\bar \W})_4=(\mathcal{S}_{\W})_4,\quad (\mathcal{S}_{\bar \W})_5= \frac{\kappa(\thetab)}{\rhob}\p_{x_1}^2 \Wb_1-\frac{\thetab \kappa(\thetab)}{\rhob}\p_{x_1}^2 \left( \frac{\Wb_1}{\thetab} \right)+(\mathcal{S}_{\W})_5-(\mathcal{S}_{\bar{\W}})_1, \notag\\
	&\left((\mathcal{S}_{\W})_1,\cdots,(\mathcal{S}_{\W})_5 \right)^{t}=- \mathcal{A}_1^{*}(\bar{\U})\p_{x_1}\W+\mathcal{B}_{1}^{*}(\bar{\U})\p_{x_1}^2\W -\frac{1}{2}\W_{x_1}^{t}F''(\bar{\U})\p_{x_1}\W+\p_{x_1}\W^{t}\mathcal{B}{'}(\bar{\U}){\p_{x_1}\bar{\U}}  -\bar{\Rm}-\mathcal{R}_{G}. \notag
\end{align}
By the definition of $\mathcal{A}_1^{*},\mathcal{B}_{1}^{*},\Xi^{t}F''(\bar{\U}),\Xi^{t}\mathcal{B}{'}(\bar{\U})$ \eqref{F}-\eqref{higher-F-B-taylor} and $\mathcal{R}_{G}$ \eqref{2025-11-10-1}, one has
\begin{align}
	& \abs{(\mathcal{S}_{\bar \W})_1}\lesssim \deltab D_{-1} + \deltab D_{-1} \abs{\Wb_1},\quad \abs{(\mathcal{S}_{\bar \W})_2}\lesssim \deltab D_{-1}+\abs{\p_{x_1}\bar{\W}}^2+\deltab D_{-\frac12}\abs{\p_{x_1}\bar{\W}}+\deltab D_{-1}\abs{\bar{\W}},\label{Error-1-1}\\
	&\abs{(\mathcal{S}_{\bar \W})_3}\thickapprox\abs{(\mathcal{S}_{\bar \W})_4}\lesssim \deltab D_{-1}+\abs{\p_{x_1}\bar{\W}}^2+\deltab D_{-\frac12}\abs{\p_{x_1}^2\bar{\W}}+\deltab D_{-\frac12}\abs{\p_{x_1}\bar{\W}}+\deltab D_{-1}\abs{\bar{\W}},\label{Error-1-2}\\
	&\abs{(\mathcal{S}_{\bar \W})_5}\lesssim \deltab D_{-1}+\abs{\p_{x_1}\bar{\W}}^2+\deltab D_{-\frac12}\sum_{i=1}^4\abs{\p_{x_1}^2\Wb_i}+\deltab D_{-\frac12}\abs{\p_{x_1}\bar{\W}}+\deltab D_{-1}\abs{\bar{\W}}\label{Error-1-3},
\end{align}
where $\deltab:=\delta+\varepsilon_0$. Now, let us state the results on the coupled diffusion wave $\bar{\W}$ \eqref{Wt}.

\begin{Thm}\label{res-Wt}
	Under the same assumptions of Theorem \ref{mt}, one has
	\begin{align*}
		\norm{\bar{\W}}_{L^\infty}\le C(\varepsilon_0^{\frac12}+\delta^{\frac12}),\qquad \norm{\p_{x_1}^k \bar{\W}}_{L^2}^2 \leq C(\varepsilon_0 +\delta)(1+t)^{\frac{1}{2}-k}, \quad k \geq 1 \;\text{and}\; k\in \mathbb{N}.
	\end{align*}
\end{Thm}

The proof of Theorem \ref{res-Wt} for small $t_0$ is standard and the details are omitted for simplicity of presentation.  Next, the  {\it a priori} estimates will be carried out under the following {\it a  priori} assumptions
\begin{align}\label{aps-Wt}
	\sup_{T_0 \leq t \leq T_0+ t_0}\left\{\norm{\bar{\W}}_{L^\infty}^2+(1+t)^{k-\frac{1}{2}}\norm{\p_{x_1}^k \bar{\W}}_{L^2}^2\right\}\leq \chi^2.
\end{align}
Under the {\it a priori} assumptions \eqref{aps-Wt}, we obtain the following {\it a priori} estimate.

\begin{Prop}[{\it a priori} estimates for $\bar{\W}$]\label{ape-Wt}
	Under the same assumptions of Theorem \ref{mt}, assuming further that $\bar{\W}$ is the unique solution for \eqref{Wt}  in the interval $[T_0,T_0+t_0]$, satisfying the  {\it a priori} assumptions \eqref{aps-Wt}, it holds that
	\begin{align*}
		&\norm{\bar{\W}}_{L^\infty}\leq C(\varepsilon_0^{\frac12}+\delta^{\frac12}),\qquad \norm{\p_{x_1}^k \bar{\W}}_{L^2}^2 \leq C(\varepsilon_0+\delta) (1+t)^{\frac{1}{2}-k}.
	\end{align*}
\end{Prop}

In the rest of this section, we are devoted to the proof of Proposition \ref{ape-Wt} above. It is noted that in the linear part of \eqref{Wt}, $\Wbj:=(\Wb_1,\Wb_2,\Wb_5)$ and $(\Wb_3,\Wb_4)$
are decoupled. Therefore, we first estimate $\Wbj$. Indeed, $\Wbj$  satisfies the following equation
\begin{align}\label{W-t-1}
	\p_t\Wbj+\Abj_1(\bar{\U})\p_{x_1}\Wbj=\Bbj_1(\bar{\U})\p_{x_1}^2\Wbj+(\mathcal{S}_{\bar \W})^{\#},
\end{align}
where
\begin{align*}
	&\Abj_1(\bar{\U})=\left(\begin{array}{ccc}
		0 &\thetab &0  \\
		\frac{2}{3}& 0  &\frac{2}{3}\\
		0&\frac{2}{3}\thetab&0
	\end{array}\right),\qquad \Bbj_{1}(\bar{\U})=\left(\begin{array}{ccc}
		0 &0 &0  \\
		0 & \frac{4}{3}\frac{\mu(\thetab)}{\rhob}  &0\\
		0&0&\frac{\kappa(\thetab)}{\rhob}
	\end{array}\right), \\
	&(\mathcal{S}_{\bar \W})^{\#}=((\mathcal{S}_{\bar \W})_1,(\mathcal{S}_{\bar \W})_2-\frac{4\mu(\thetab)\mb_1}{3\rhob^2 \thetab}\p_{x_1}^2\Wb_1,(\mathcal{S}_{\bar \W})_5)^{t}. 
\end{align*}
Direct computation yields the eigenvalues of $\Abj_1(\bar{\U})$ as
\begin{align}\notag
	\bar{\lambda}_1=-\sqrt{\frac{10}{9}\Tb},\quad \lambda_2=0,\quad\bar{\lambda}_3=\sqrt{\frac{10}{9}\Tb},\qquad \Lambda:=\text{diag}\{\bar{\lambda}_1,0,\bar{\lambda}_3\},
\end{align}
and the corresponding eigenvectors as
\begin{align*}
	&\bar{L}:=\left(\begin{array}{ccc}
		\sqrt{\frac{3}{10}} & \sqrt{\frac{3}{10}}\frac{3\bar{\lambda}_1}{2} &\sqrt{\frac{3}{10}} \\
		\sqrt{\frac{2}{5}} & 0 & -\sqrt{\frac{9}{10}}\\
		\sqrt{\frac{3}{10}} & -\sqrt{\frac{3}{10}}\frac{3\bar{\lambda}_1}{2} &\sqrt{\frac{3}{10}}
	\end{array}\right),\quad
	&\bar{R}:=\left(\begin{array}{ccc}
		\sqrt{\frac{3}{10}} & \sqrt{\frac{2}{5}} &\sqrt{\frac{3}{10}} \\
		\sqrt{\frac{3}{10}}\frac{\bar{\lambda}_1}{\bar{\theta}} & 0 & -\sqrt{\frac{3}{10}}\frac{\bar{\lambda}_1}{\bar{\theta}}\\
		\frac{2}{3}\sqrt{\frac{3}{10}} & -\sqrt{\frac{2}{5}} &\frac{2}{3}\sqrt{\frac{3}{10}}
	\end{array}\right).
\end{align*}

\begin{Rem}
	Direct calculations yield that the two sides of structural conditions \eqref{SC} are satisfied for the system of $\bar{\W}^{\#}$ \eqref{W-t-1}.
\end{Rem}

Set $\mathbf{B}=\bar{L}\bar{\W}^{\#}=(b_1,b_2,b_3)^t$, then $\Wbj=\bar{R}\mathbf{B}$ and \eqref{W-t-1} can be written in a diagonalized form
\begin{align}\label{eq-diaB}
	\pt \B+\Lambda \p_{x_1} \B=\bar{L}\Bbj_1\bar{R}\p_{x_1} ^2\B+2\bar{L}\Bbj_1\p_{x_1} \bar{R}\p_{x_1} \B+\left[\left(\pt\bar{L}+\Lambda \p_{x_1} \bar{L}\right)\bar{R}+\bar{L}\Bbj_1 \p_{x_1} ^2\bar{R}\right]\B+\bar{L} (\mathcal{S}_{\bar \W})^{\#}.
\end{align}
Set $v_{1}=\frac{\rhoc}{\rho_{+}}$ , one has $\left|v_{1}-1\right| \leq C \delta .$
For the sake of convenience, we denote $\bar{\mathcal{A}}_4=\bar{L}\Bbj_1\bar{R}$, 
\begin{align}
	\bar{\mathcal{A}}_4=\left(\begin{array}{ccc}
		\frac{2\bar{\mu}}{3}+\frac{1}{5}\bar{\kappa} & -\frac{\sqrt{3}}{5}\bar{\kappa}& -\frac{2\bar{\mu}}{3}+\frac{1}{5}\bar{\kappa}\\
		-\frac{\sqrt{3}}{5}\bar{\kappa}& \frac{3}{5}\bar{\kappa} &-\frac{\sqrt{3}}{5}\bar{\kappa}\\
		-\frac{2\bar{\mu}}{3}+\frac{1}{5}\bar{\kappa}& -\frac{\sqrt{3}}{5}\bar{\kappa} &\frac{2\bar{\mu}}{3}+\frac{1}{5}\bar{\kappa}
	\end{array}\right),\label{A_3}
	\text{\quad where\quad}
	\bar{\mu}=\frac{\mu(\bar{\theta})}{\bar{\rho}} ,\quad\bar{\kappa}=\frac{\kappa(\thetab)}{\bar{\rho}} .
\end{align}
And we also denote
\begin{align}\notag
	E^{\#}_k:=&\int_{\R}\frac{v_1^N}{2} \abs{\p_{x_1}^kb_1}^2+\frac{1}{2} \abs{\p_{x_1}^kb_2}^2+\frac{v_1^{-N}}{2} \abs{\p_{x_1}^kb_3}^2dx_1, \qquad
	K^{\#}_k:=\int_{\R}\p_{x_1}^{k+1}B^t\bar{\mathcal{A}}_4\p_{x_1}^{k+1}Bdx_1,
\end{align}
where $N=4[\delta^{-\frac{1}{2}}]+1$ is a large positive integer. For $k=0$, $E^{\#}_0=C\norm{\bar{\W}^{\#}}_{L^2}^2$, and for $k\ge 1$,
\begin{align}\label{sec-n-1}
	\norm{\p_{x_1}^{k}\Wbj}_{L^2}^2-\bar{\delta}\sum_{j=0}^{k-1}(1+t)^{-(k-j)}\norm{\p_{x_1}^{j}\Wbj}_{L^2}^2\lesssim E^{\#}_{k}\lesssim  \norm{\p_{x_1}^{k}\Wbj}_{L^2}^2+\bar{\delta}\sum_{j=0}^{k-1}(1+t)^{-(k-j)}\norm{\p_{x_1}^{j}\Wbj}_{L^2}^2.
\end{align}
We further denote
\begin{equation}\label{sec-n-2}
	\left\{\begin{aligned}
		K_i:=&\int_{\R} \abs{\p_{x_1}^{i+1}\bar{\W}}^2dx_1,\\
		{E}_i:=&E^{\#}_i+\int_{\R}\abs{\p_{x_1}^{i}\Wb_3}^2+\abs{\p_{x_1}^{i}\Wb_4}^2 dx_1 +\bar{c} \int_{\R} \p_{x_1}^i\Wb_2\p_{x_1}^{i+1}\Wb_1 +\frac{2\mu(\thetab)}{3\thetab\rhob} \abs{\p_{x_1}^{i+1} \Wb_1}^2dx_1, \end{aligned} \right.
\end{equation}
where $\bar{c}$ is a sufficiently small constant  chosen to ensure $E_i$ is positive.
Then we have the following result.
\begin{Lem}\label{Lem-Wb-1}
	Under the same assumptions of Proposition \ref{ape-Wt}, it holds that
	\begin{align}
		&\frac{d}{dt}\big(\sum_{i=0}^n E_i\big)+\sum_{i=0}^n(K_i+G_{i})\leq C\deltac(1+t)^{-1}\big(\sum_{i=0}^n {E}_i\big)+C\bar{\delta}(1+t)^{-\frac{1}{2}},\label{11}\\
		&\frac{d}{dt}\big(\sum_{i=1}^n {E}_i\big)+\sum_{i=1}^n (K_i+G_{i})\leq C\deltac \left[(1+t)^{-1}\big(\sum_{i=1}^n{E}_i+G_0\big)+(1+t)^{-2}{E}_0\right]+C\deltab(1+t)^{-\frac{3}{2}},\label{22}\\
		&\qquad\qquad\qquad\qquad\qquad\qquad\qquad\qquad\qquad\dots\dots, \notag\\
		&\frac{d}{dt}\big(\sum_{i=n-1}^n {E}_i\big)+\sum_{i=n-1}^n (K_i+G_{i})\leq C\deltac\bigg[(1+t)^{-1}\sum_{i=n-1}^n E_i +\sum_{i=0}^{n-2}(1+t)^{i-n}E_i\notag\\
		&\qquad\qquad\qquad\qquad\qquad\qquad\qquad\qquad\quad\;+\sum_{i=0}^{n-2}(1+t)^{i+1-n}G_i\bigg]+C\deltab(1+t)^{\frac{1}{2}-n},\notag\\
		&\frac{d}{dt}{E}_n+{K}_n+G_{n}\leq C\deltac \left[\sum_{i=0}^{n}(1+t)^{i-n-1}E_i+\sum_{i=0}^{n-1}(1+t)^{i-n}G_i \right]+C\deltab(1+t)^{-\frac{1}{2}-n},\label{kk}
	\end{align}
	where $G_i$, $i=1,\dots,n$ is defined in \eqref{2025.6.07-1}, $\deltab:=\delta+\varepsilon_0$, and $\deltac:=\chi+\deltab^{\frac12}$.
\end{Lem}

\begin{proof}
	\begin{flushleft}
		\textbf{Step 1. The diagonalized system.}
	\end{flushleft}
	Applying $\p_{x_1}^k$ with $k=0,1,2,\dots,n$ to \eqref{eq-diaB}, one has
	\begin{align}\label{equ-Bk}
		\begin{aligned}
			\p_t\p_{x_1}^k\B+\Lambda \p_{x_1}^{k+1}\B=\bar{\mathcal{A}}_4 \p_{x_1}^{k+2}\B+\mathcal{S}_{\B k},
		\end{aligned}
	\end{align}
	where
	\begin{align}
		\mathcal{S}_{\B k}:=&\sum_{j=1}^{k}\bigg[-\p_{x_1}^j\Lambda\p_{x_1}^{k-j+1}\B+\p_{x_1}^j\bar{\mathcal{A}}_4\p_{x_1}^{k-j+2}\B\bigg]+2\sum_{i=0}^{k}\p_{x_1}^i\big(\bar{L} \Bbj_1 \p_{x_1}\bar{R}\big) \p_{x_1}^{k-i+1}\B\notag\\
		&+\sum_{i=0}^{k}\bigg\{\p_{x_1}^i\left[\left(\p_t\bar{L}+\Lambda \p_{x_1}\bar{L}\right) \bar{R}+\bar{L} \Bbj_1 \p_{x_1}^2 \bar{R}\right] \p_{x_1}^{k-i}\B \bigg\}+\p_{x_1}^{k}\left(\bar{L} (\mathcal{S}_{\bar \W})^{\#}\right) \notag\\
		\leq&C\sum_{j=1}^{k+1}\left(\abs{\p_{x_1}^j \rhoc} \abs{\p_{x_1}^{k-j+1}b_1},0,\abs{\p_{x_1}^j\rhoc} \abs{\p_{x_1}^{k-j+1}b_3}\right)^{t}
		+C\bar{\delta} \sum_{j=1}^{k+2}\abs{D_{-\frac{j}{2}}\p_{x_1}^{k-j+2}\B}+\p_{x_1}^{k}\left(\bar{L} (\mathcal{S}_{\bar \W})^{\#}\right):=\sum_{i=1}^{3}\mathcal{S}_{\B k}^{(i)}.\notag
	\end{align}
	We shall use a weighted energy method to derive the intrinsic dissipation. Without loss of generality, we assume that $\p_{x_1}{\rhoc}>0 $ since the proof in the case $\p_{x_1}{\rhoc}<0$ is similar.
	
	Applying $\p_{x_1}^k$ to \eqref{eq-diaB} and then multiplying the resulting equations by $\bar{\B}^{(k)}=\left(v_{1}^{N} \p_{x_1}^kb_{1}, \p_{x_1}^kb_{2}, v_{1}^{-N} \p_{x_1}^kb_{3}\right)$, one has
	\begin{align}\label{d-k-b}
		&\int_\R \p_t\left(\frac{v_1^N}{2} \abs{\p_{x_1}^kb_1}^2+\frac{1}{2} \abs{\p_{x_1}^kb_2}^2+\frac{v_1^{-N}}{2} \abs{\p_{x_1}^kb_3}^2\right)+\p_{x_1}\bar{\B}^{(k)} \bar{\mathcal{A}}_4 \p_{x_1}^{k+1}\B dx_1
		+\int_{\R} a_1\abs{\p_{x_1}^kb_1}^2+ a_3\abs{\p_{x_1}^kb_3}^2dx_1\nonumber\\
		=&\int_{\R}\bar{\B}^{(k)}\p_{x_1} \bar{\mathcal{A}}_{4}\p_{x_1}^{k+1} \B+\bigg[\left(\frac{v_1^N}{2}\right)_t \abs{\p_{x_1}^kb_1}^2+\left(\frac{v_1^{-N}}{2}\right)_t \abs{\p_{x_1}^kb_3}^2\bigg]+\bar{\B}^{(k)}\mathcal{S}_{\B k}dx_1:=I_1+I_2+I_3,
	\end{align}
	where 
	\begin{align*}
		&a_1:=-\frac{v_1^{N-1}}{2}\left(N \bar\lambda_1 \p_{x_1} v_{1}+v_1\p_{x_1}  \bar\lambda_{1 }\right)\ge C\delta ^{-\frac{1}{2}}\p_{x_1} \rhoc-C\bar{\delta}D_{-1},\\ &a_3:=\frac{v_1^{-N-1}}{2}\left(N \bar\lambda_3 \p_{x_1}  v_{1 }-v_1 \p_{x_1} \bar\lambda_{3 }\right)\ge C\delta ^{-\frac{1}{2}}\p_{x_1} \rhoc-C\bar{\delta}D_{-1}.
	\end{align*} 
	Thus, we obtain an additional dissipative structure through diagonalization:
	\begin{align}\label{2025.6.07-1}
		G_k:=C\delta^{-\frac12}\int_{\R} \p_{x_1}\rhoc\abs{\p_{x_1}^kb_1}^2+ \p_{x_1}\rhoc\abs{\p_{x_1}^kb_3}^2dx_1.
	\end{align}
	Recalling the definition of $\bar{\kappa}$ and $\bar{\mu}$ \eqref{A_3}, we find that the dissipation matrix $\bar{\mathcal{A}}_4$ \eqref{A_3} is non-negative definite. Furthermore, it holds
	\begin{align}\label{q-z-x}
		\p_{x_1}^{k+1}{\B}^t\bar{\mathcal{A}}_4\p_{x_1}^{k+1}\B=&\frac{\bar{\kappa}}{5}\left[(\p_{x_1}^{k+1} b_{1}+\p_{x_1}^{k+1} b_{3})-\sqrt{3}\p_{x_1}^{k+1} b_{2}\right]^2+\frac{2\bar{\mu}}{3}\left[(\p_{x_1}^{k+1} b_{3}-\p_{x_1}^{k+1} b_{1})\right]^2 \notag\\
		\geq&C\big(|{\p_{x_1}^{k+1}}\bar{W}_2|^2+|{\p_{x_1}^{k+1}}\bar{W}_5|^2\big)-C\deltab\sum_{j=1}^{k+1}D_{-j}\abs{\p_{x_1}^{k+1-j}\Wb_2}^2.
	\end{align}
	And we also have
	\begin{align}\label{2025-11-11-1}
		\int_{\R} \p_{x_1} \left(\bar{\B}^{(k)}-\p_{x_1}^{k}\B^t\right)\bar{\mathcal{A}}_4\p_{x_1}^{k+1}\B dx_1 \leq C(\chi+\bar{\delta}^{\frac12})(1+t)^{-1}E_k+C(\chi+\bar{\delta}^{\frac12})K_k.
	\end{align}
	By \eqref{sec-n-1},  \eqref{sec-n-2} and the {\it a priori} assumptions \eqref{aps-Wt}, we obtain
	\begin{align}\label{I12}
		I_1+I_2\leq C \left(\deltab^{\frac12}+\chi\right)(1+t)^{-1}E_k+\left(\deltab^{\frac12}+\chi\right) K_k.
	\end{align}
	The estimate of $I_3$ is more intricate. We first present the estimate of $\int_{\R}\bar{\B}^{(k)}\mathcal{S}_{\B k}^{(2)}dx_1$ as 
	\begin{align}
		\int_{\R}\bar{\B}^{(k)}\mathcal{S}_{\B k}^{(2)}dx_1\leq C\deltab\sum_{i=0}^{k+1}(1+t)^{-i}\norm{\p_{x_1}^{k-i+1}\B}_{L^2}^2, \quad \text{for}\quad k\ge 0.\label{I14}
	\end{align}
	Then we consider the estimate of $\int_{\R}\bar{\B}^{(k)}\mathcal{S}_{\B k}^{(1)}dx_1$. For the case of $k=0$, it is easy to see
	\begin{align}\label{IL12}	
		\int_{\R}\bar{\B}^{(0)}\mathcal{S}_{\B 0}^{(1)}dx_1\le C\int_{\R} \p_{x_1} \rhoc \abs{b_1}^2+\p_{x_1} \rhoc \abs{b_3}^2 dx_1.
	\end{align}
	For $k \ge 1$, we have
	\begin{align}\label{2025-11-3-1}
		&\int_{\R}\bar{\B}^{(k)}\mathcal{S}_{\B k}^{(1)}dx_1\lesssim \sum_{l=1,3}\sum_{j=0}^k \int_{\R} \abs{\p_{x_1}^{j+1}\rhoc}\abs{\p_{x_1}^{k-j} b_{l}} \abs{\p_{x_1}^k b_l}dx_1.
	\end{align}
	Noticing 
	\begin{align*}
		&\abs{\p_{x_1}^{2m} \rhoc}\lesssim \sum_{i=0}^{m-1}(1+t)^{i-m+1}\abs{\frac{x_1}{1+t}}^{2i+1}\p_{x_1}\rhoc,\quad \abs{\p_{x_1}^{2m+1} \rhoc}\lesssim \sum_{i=0}^{m}(1+t)^{i-m}\abs{\frac{x_1}{1+t}}^{2i}\p_{x_1}\rhoc,  
	\end{align*}
	we have
	\begin{align}
		&\int_{\R} \abs{\p_{x_1}^{j+1}\rhoc}\abs{\p_{x_1}^{k-j} b_{l}} \abs{\p_{x_1}^k b_l}dx_1\notag\\
		&\lesssim \sum_{i=0}^{\frac{j}{2}}(1+t)^{i-\frac{j}{2}}\int_{\R} \abs{\frac{x_1}{1+t}}^{2i}\p_{x_1}\rhoc \abs{\p_{x_1}^{k-j}b_l}\abs{\p_{x_1}^k b_l} dx_1 \notag\\
		&\lesssim
		(1+t)^{-j} \int_{\R} \p_{x_1}\rhoc \abs{\p_{x_1}^{k-j}b_l}^2 dx_1 + \sum_{i=0}^{\frac{j}{2}} \int_{\R} \abs{\frac{x_1^2}{1+t}}^{2i} \p_{x_1}\rhoc\abs{\p_{x_1}^k b_l}^2 dx_1\notag\\
		&\lesssim
		(1+t)^{-j} \int_{\R} \p_{x_1}\rhoc \abs{\p_{x_1}^{k-j}b_l}^2 dx_1 + \deltab \int_{\R} \Upsilon_{-\frac12} \abs{\p_{x_1}^k b_l}^2 dx_1, \quad \text{for} \; j \; \text{is} \; \text{even},
		\label{2025-11-3-2}
	\end{align}
	and
	\begin{align}
		&\int_{\R} \abs{\p_{x_1}^{j+1}\rhoc}\abs{\p_{x_1}^{k-j} b_{l}} \abs{\p_{x_1}^k b_l}dx_1\notag\\
		&\lesssim \sum_{i=0}^{\frac{j}{2}-\frac12}(1+t)^{i-\frac{j}{2}+\frac12}\int_{\R} \abs{\frac{x_1}{1+t}}^{2i+1}\p_{x_1}\rhoc \abs{\p_{x_1}^{k-j}b_l}\abs{\p_{x_1}^k b_l} dx_1 \notag\\
		&\lesssim
		(1+t)^{-j} \int_{\R} \p_{x_1}\rhoc \abs{\p_{x_1}^{k-j}b_l}^2 dx_1 + \sum_{i=0}^{\frac{j}{2}-\frac12} \int_{\R} \abs{\frac{x_1^2}{1+t}}^{2i+1} \p_{x_1}\rhoc\abs{\p_{x_1}^k b_l}^2 dx_1\notag\\
		&\lesssim
		(1+t)^{-j} \int_{\R} \p_{x_1}\rhoc \abs{\p_{x_1}^{k-j}b_l}^2 dx_1 + \deltab \int_{\R} \Upsilon_{-\frac12} \abs{\p_{x_1}^k b_l}^2 dx_1, \quad \text{for} \; j \; \text{is} \; \text{odd},\label{2025-11-3-3}
	\end{align}
	where we have used the definition of $\Upsilon_{-\frac12}$ in
	\eqref{errors}. Combining \eqref{2025-11-3-1}, \eqref{2025-11-3-2} and \eqref{2025-11-3-3},  we then obtain
	\begin{align}\label{2025-11-3-4}
		\int_{\R}\bar{\B}^{(k)}\mathcal{S}_{\B k}^{(1)}dx_1\lesssim \sum_{l=1,3}\sum_{j=0}^k 
		(1+t)^{-j} \int_{\R} \p_{x_1}\rhoc \abs{\p_{x_1}^{k-j}b_l}^2 dx_1 +\deltab \sum_{l=1,3} \int_{\R} \Upsilon_{-\frac12} \abs{\p_{x_1}^k b_l}^2 dx_1.
	\end{align}
	Next, we estimate terms involving $\mathcal{S}_{\B k}^{(3)}$. For $k=0$
	\begin{align}\label{d-k-M-00}
		\norm{\bar{\B}^{(0)}\mathcal{S}_{\B 0}^{(3)}}_{L^1(\R)}\leq&C\int_{\R} \left\{\deltab\left[D_{-1}+D_{-1}\abs{\bar{\W}}+D_{-\frac{1}{2}}\left(\abs{\p_{x_1}\bar{\W}}+\abs{\p_{x_1}^2 \bar{\W}}\right)\right]+\abs{\p_{x_1} \bar{\W}}^2\right\}\abs{\bar{\B}^{(0)}}dx_1
		:=\sum_{j=1}^4 I_{4}^{(j)}.
	\end{align}
	Then we estimate the above $I_4$ term by term in the way that
	\begin{align}\notag
		&I_{4}^{(1)}\leq C \deltab (1+t)^{-\frac{1}{2}}+C\deltac(1+t)^{-1}\norm{\B}^2_{L^2},\notag\\
		&I_{4}^{(2)}+I_{4}^{(3)}+I_{4}^{(4)}\leq C\deltac(1+t)^{-1}\norm{\bar{\W}}^2_{L^2}+C\deltac\norm{\p_{x_1} \bar{\W}}_{H^1}^2,\notag
	\end{align}
	where we have used the {\it a priori} assumptions \eqref{aps-Wt}. For $k \geq 1, \; k \in \mathbb{N}$, we have
	\begin{align}\label{d-k-M-k11}
		&\norm{\bar{\B}^{(k)}\mathcal{S}_{\B k}^{(3)}}_{L^1(\R)}\le\abs{\int_{\R} \p_{x_1}^{k+1}b_2 \p_{x_1}^{k-1}\left( \sqrt{\frac25}(\mathcal{S}_{\bar \W})_1-\sqrt{\frac{9}{10}}(\mathcal{S}_{\bar \W})_5 \right) dx_1}
		\notag\\
		&\quad\;+\abs{\int_{\R} \p_{x_1}\left( v_1^N \p_{x_1}^k b_1 \right)\p_{x_1}^{k-1} \left[\sqrt{\frac{3}{10}} (\mathcal{S}_{\bar \W})_1+ 
			\sqrt{\frac{3}{10}}\frac{3\bar{\lambda}_1}{2}\left( (\mathcal{S}_{\bar \W})_2-\frac{4\mu(\thetab)\mb_1}{3\rhob^2\thetab}\p_{x_1}^2\Wb_1\right)+\sqrt{\frac{3}{10}} (\mathcal{S}_{\bar \W})_5 \right]dx_1    }\notag\\
		&\quad\;+\abs{\int_{\R} \p_{x_1}\left( v_1^{-N} \p_{x_1}^k b_3 \right)\p_{x_1}^{k-1} \left[\sqrt{\frac{3}{10}} (\mathcal{S}_{\bar \W})_1-\sqrt{\frac{3}{10}}\frac{3\bar{\lambda}_1}{2}\left( (\mathcal{S}_{\bar \W})_2-\frac{4\mu(\thetab)\mb_1}{3\rhob^2\thetab}\p_{x_1}^2\Wb_1\right)+\sqrt{\frac{3}{10}} (\mathcal{S}_{\bar \W})_5 \right]dx_1    }.
	\end{align}
	By \eqref{Error-1-1}, \eqref{Error-1-2} and \eqref{Error-1-3}, we  present the calculations of the error term $\p_{x_1}^{k-1}D_{-1}$ and the nonlinear terms $\p_{x_1}^{k-1} |\p_{x_1} \bar{\W}|^2$ among $\mathcal{S}_{\B k}^{(3)}$.  Applying the {\it a priori} assumptions \eqref{aps-Wt}, we have
	\begin{align*}
		&\deltab\int_{\R} \p_{x_1}^{k-1}D_{-1} \left[\p_{x_1}^{k+1}b_2+\p_{x_1} \left( v_1^N \p_{x_1}^kb_1\right)+\p_{x_1} \left( v_1^{-N} \p_{x_1}^kb_3\right) \right] dx_1 \notag\\
		\leq& C\deltab (1+t)^{-\frac{1}{2}-k}+C\deltac (1+t)^{-1}\norm{\p_{x_1}^k \B}_{L^2}^2 +C\deltab\norm{\p_{x_1}^{k+1}\B}_{L^2}^2 ,\\
		& \int_{\R}\p_{x_1}^{k-1}\abs{\p_{x_1}\bar{\W}}^2\left( \abs{\p_{x_1}^{k+1}b_2}+\abs{\p_{x_1}(v_1^N\p_{x_1}^kb_1)}+\abs{\p_{x_1}(v_1^{-N}\p_{x_1}^k b_3)}\right)dx_1 \notag\\
		\le& C\sum_{i=0}^{[\frac{k-1}{2}]} \norm{\p_{x_1}^{i+1} \bar{\W}}_{L^\infty}\norm{\p_{x_1}^{k-i} \bar{\W}}_{L^2}\norm{\p_{x_1}^{k+1} \bar{\W}}_{L^2}+C\sum_{i=[\frac{k-1}{2}]+1}^{k-1}\norm{\p_{x_1}^{k-i} \bar{\W}}_{L^\infty}\norm{\p_{x_1}^{i+1} \bar{\W}}_{L^2}\norm{\p_{x_1}^{k+1} \bar{\W}}_{L^2} \notag\\
		&+C \deltab(1+t)^{-\frac12} \norm{\p_{x_1}^k\bar{\W}}_{L^2}\left(\sum_{i=0}^{[\frac{k-1}{2}]} \norm{\p_{x_1}^{i+1} \bar{\W}}_{L^\infty}\norm{\p_{x_1}^{k-i} \bar{\W}}_{L^2}+ \sum_{i=[\frac{k-1}{2}]+1}^{k-1}\norm{\p_{x_1}^{k-i} \bar{\W}}_{L^\infty}\norm{\p_{x_1}^{i+1} \bar{\W}}_{L^2}\right)\notag\\
		\leq& C\deltac \sum_{i=0}^{k} (1+t)^{i-k-1}\norm{\p_{x_1}^i \bar{\W}}_{L^2}^2 +C\deltac \norm{\p_{x_1}^{k+1}\bar{\W}}_{L^2}^2  .
	\end{align*}
	The calculations of the remaining terms in \eqref{d-k-M-k11} are similar. Then we have
	\begin{align}\label{d-k-M-k2}
		\norm{\bar{\B}^{(k)}\mathcal{S}_{\B k}^{(3)}}_{L^1(\R)}\le C\deltac \sum_{i=0}^{k} (1+t)^{i-k-1}\norm{\p_{x_1}^i \bar{\W}}_{L^2}^2 +C\deltac \norm{\p_{x_1}^{k+1}\bar{\W}}_{L^2}^2 +C\deltab(1+t)^{-\frac12-k} .
	\end{align}
	\begin{flushleft}
		\textbf{Step 2.  Estimates on $\int_{\R} \Upsilon_{-\frac12} \abs{\p_{x_1}^k b_i}^2 dx_1$.}
	\end{flushleft}
	
	Notice that we still need to control  $\int_{\R} \Upsilon_{-\frac12} \abs{\p_{x_1}^k b_i}^2 dx_1$ in \eqref{2025-11-3-4}. From \eqref{eqs27} and \eqref{errors}, we have $$\deltab\Upsilon_{-1/2}\approx\deltab (1+t)^{-\frac12}e^{-\frac{\tilde{d} x_1^2}{1+t}},\qquad  \p_{x_1}\rhoc \approx\deltab (1+t)^{-\frac12}e^{-\frac{d x_1^2}{1+t}},$$ with $\tilde{d}<d$. Therefore $\int_{\R} \Upsilon_{-\frac12} \abs{\p_{x_1}^k b_i}^2 dx_1$ cannot be directly controlled by $G_k$ \eqref{2025.6.07-1}. 
	We need to further develop weighted heat kernel inequalities \eqref{2025-10-4} and \eqref{2025-10-5}  in order to obtain the control of $\int_{\R} \Upsilon_{-\frac12} \abs{\p_{x_1}^k b_i}^2 dx_1$ for $ k \ge 1$. To derive these estimates, let 
	$$
	\tilde{h}(t,x_1)=\int_{-\infty}^{x_1} \Upsilon_{-\frac12}(t,y)dy,
	$$
	then it follows that $\|\tilde{h}\|_{L^\infty}\leq C$ and $4\tilde{d} \p_t \tilde{h}= \p_{x_1}\Upsilon_{-\frac12}$. 
	Multiplying \eqref{equ-Bk}$_1$ by $\tilde{h}\p_{x_1}^k{b}_1$,  we thus have
	\begin{align}\label{2025-10-30-3}
		& \frac{1}{2} \frac{d}{dt}\left[ \tilde{h} (\p_{x_1}^kb_1)^2 \right]-\frac{1}{2}\p_t \tilde{h}\left(\p_{x_1}^k b_1 \right)^2 + \frac{1}{2}\p_{x_1} \left[ \tilde{h} \bar{\lambda}_1 \left(\p_{x_1}^k b_1 \right)^2 \right] - \frac{1}{2} \p_{x_1} \left( \tilde{h} \bar{\lambda}_1\right)\left( \p_{x_1}^k b_1\right)^2\notag\\
		=&\left(\mathcal{S}_{\B k}^{(1)} \right)_1\tilde{h}\p_{x_1}^k b_1+ \left[ \left( \bar{\mathcal{A}}_4  \p_{x_1}^{k+2} \B \right)_1 + \sum_{i=2}^3  \left(\mathcal{S}_{\B k}^{(i)} \right)_1 \right] \tilde{h} \p_{x_1}^k b_1.
	\end{align}
	Using the same method as for obtaining \eqref{2025-11-3-1}, \eqref{2025-11-3-2}, \eqref{2025-11-3-3} and \eqref{2025-11-3-4}, one has
	\begin{align*}
		\int_{\R}\left(\mathcal{S}_{\B k}^{(1)} \right)_1\tilde{h}\p_{x_1}^k b_1 dx_1 \lesssim \deltab \int_{\R} \Upsilon_{-\frac12} \abs{\p_{x_1}^k b_1}^2 dx_1 + \sum_{i=0}^k(1+t)^{-i} \int_{\R} \p_{x_1}\rhoc \abs{\p_{x_1}^{k-i}b_1}^2 dx_1. 
	\end{align*}
	Applying an argument similar to \eqref{d-k-M-k2}, it holds that
	\begin{align*}
		\int_{\R}\Big[ \left(\bar{\mathcal{A}}_4 \p_{x_1}^{k+2} \B \right)_1 + \sum_{i=2}^3  \left(\mathcal{S}_{\B k}^{(i)} \right)_1 \Big] \tilde{h} \p_{x_1}^k b_1 dx_1 \lesssim \sum_{i=0}^{k+1}(1+t)^{j-k-1}\norm{\p_{x_1}^j \bar{\W}}_{L^2}^2+(1+t)^{-\frac12-k}.
	\end{align*}
	Thus, integrating \eqref{2025-10-30-3} with respect to $x_1$, we obtain
	\begin{align}\label{2025-10-4}
		\int_{\R} \Upsilon_{-\frac12} \abs{\p_{x_1}^k b_1}^2 dx_1 +\p_t \int_{\R} \tilde{h} \left(\p_{x_1}^k b_1 \right)^2 dx_1 \lesssim &  \sum_{j=0}^{k+1} (1+t)^{j-k-1}\norm{\p_{x_1}^j \bar{\W}}_{L^2}^2 + (1+t)^{-\frac12-k} \notag \\
		&+ \sum_{j=0}^k (1+t)^{-j} \int_{\R} \p_{x_1}\rhoc \abs{\p_{x_1}^{k-j}b_1}^2 dx_1 .
	\end{align}
	Using the same argument as for obtaining \eqref{2025-10-4}, we have
	\begin{align}\label{2025-10-5}
		\int_{\R} \Upsilon_{-\frac12} \abs{\p_{x_1}^k b_3}^2 dx_1- \p_t \int_{\R} \tilde{h} \left(\p_{x_1}^k b_3 \right)^2 dx_1 \lesssim &\sum_{j=0}^{k+1} (1+t)^{j-k-1}\norm{\p_{x_1}^j \bar{\W}}_{L^2}^2 + (1+t)^{-\frac12-k} \notag \\
		&+ \sum_{j=0}^k (1+t)^{-j} \int_{\R} \p_{x_1}\rhoc \abs{\p_{x_1}^{k-j}b_3}^2 dx_1 .
	\end{align}
	Combining \eqref{d-k-b}-\eqref{IL12}, \eqref{2025-11-3-4}, \eqref{d-k-M-k2}, \eqref{2025-10-4} and \eqref{2025-10-5}, one has
	\begin{align}
		&\frac{d}{dt}\big(\sum_{i=0}^n E_i^{\#}\big)+\sum_{i=0}^n(K^{\#}_i+G_{i})\notag\\
		\leq& C\deltab(1+t)^{-\frac12}+ C\check{\delta}\left[(1+t)^{-1}\big(\sum_{i=0}^n {E}_i\big)+\norm{\p_{x_1}\Wb_1}_{H^n}^2+\norm{\p_{x_1}\Wb_3}_{H^n}^2+\norm{\p_{x_1}\Wb_4}_{H^n}^2\right],\label{d-0-11}\\
		&\frac{d}{dt}\big(\sum_{i=1}^n {E}^{\#}_i\big)+\sum_{i=1}^n (K^{\#}_i+G_{i})\notag 
		\leq C\bar{\delta}(1+t)^{-\frac{3}{2}}+C\check{\delta}\bigg[(1+t)^{-1}\big(\sum_{i=1}^n{E}_i+\sum_{i=0}^{n-1}G_i\big)+(1+t)^{-2}{E}_0\\
		&\qquad\qquad\qquad\qquad\qquad\qquad\qquad\quad\;+\norm{\p_{x_1}^2\Wb_1}_{H^{n-1}}^2+\norm{\p_{x_1}^2\Wb_3}_{H^{n-1}}^2+\norm{\p_{x_1}^2\Wb_4}_{H^{n-1}}^2\bigg],\label{d-0-21}\\
		&\qquad\qquad\qquad\qquad\qquad\qquad\qquad\qquad\qquad\dots\dots, \notag\\
		&\frac{d}{dt}\big(\sum_{i=n-1}^n {E}_i^{\#}\big)+\sum_{i=n-1}^n (K^{\#}_i+G_{i})\leq C\deltab (1+t)^{\frac{1}{2}-n}+C\deltac\bigg[(1+t)^{-1}\sum_{i=n-1}^n E_i +\sum_{i=0}^{n-2}(1+t)^{i-n}E_i\notag\\
		&\qquad\qquad\qquad\qquad\qquad\qquad+\sum_{i=0}^{n-2}(1+t)^{i+1-n}G_i+\norm{\p_{x_1}^{n}\Wb_1}_{H^1}^2+\norm{\p_{x_1}^{n}\Wb_3}_{H^1}^2+\norm{\p_{x_1}^{n}\Wb_4}_{H^1}^2\bigg],\label{d-0-31}\\
		&\frac{d}{dt}{E}^{\#}_n+K^{\#}_n+G_{n}\le C\deltab(1+t)^{-\frac{1}{2}-n}+ C\deltac \bigg[\sum_{i=0}^{n}(1+t)^{i-n-1}E_i\notag\\
		& \qquad\qquad\qquad\qquad\qquad\qquad+\sum_{i=0}^{n-1}(1+t)^{i-n}G_i+\norm{\p_{x_1}^{n+1}\Wb_1}_{L^2}^2+\norm{\p_{x_1}^{n+1}\Wb_3}_{L^2}^2+\norm{\p_{x_1}^{n+1}\Wb_4}_{L^2}^2  \bigg].\label{d-0-41}
	\end{align}
	\begin{flushleft}
		\textbf{Step 3. Estimates on $\norm{\p_{x_1}^{k+1}\Wb_1}_{L^2}$.}
	\end{flushleft}
	We still need to estimate $\norm{\p_{x_1}^{k+1}\bar{W}_1}_{L^2}$. Applying $\p_{x_1}^k$ to \eqref{Wt}$_2$ and then multiplying the resulting equation by $\p_{x_1}^{k+1}\bar{W}_1$, one has
	\begin{align}\label{eq-phi-x}
		&\frac{d}{dt}\int_{\R} \p_{x_1}^k\Wb_2\p_{x_1}^{k+1}\Wb_1 dx_1 + \frac{2}{3}\norm{\p_{x_1}^{k+1}\Wb_1}_{L^2}^2-\frac{4}{3}\int_{\R}\frac{\mu(\thetab)}{\rhob} \p_{x_1}^{k+2}\Wb_2\p_{x_1}^{k+1}\Wb_1 dx_1 + \int_{\R} \p_t \p_{x_1}^k \Wb_1 \p_{x_1}^{k+1}\Wb_2 dx_1\notag\\
		=&\frac{4}{3}\sum_{i=0}^{k-1}\int_{\R}\p_{x_1}^{k-i}\left( \frac{\mu(\thetab)}{\rhob} \right)\p_{x_1}^{i+2}\Wb_2 \p_{x_1}^{k+1}\Wb_1 dx_1 -\frac{2}{3}\int_{\R}\p_{x_1}^{k+1}\Wb_5 \p_{x_1}^{k+1}\Wb_1 dx_1 \notag\\
		&+\int_{\R} \p_{x_1}^k (\mathcal{S}_{\bar \W})_2\p_{x_1}^{k+1}\Wb_1 dx_1-\frac{4}{3}\int_{\R} \p_{x_1}^{k}\left[ \frac{\mu(\thetab)\mb_1}{\rhob^2 \thetab} \p_{x_1}^2\Wb_{1} \right] \p_{x_1}^{k+1} \Wb_1.
	\end{align}
	Then we use \eqref{Wt}$_1$ to deal with $-\frac{4}{3}\int_{\R}\frac{\mu(\thetab)}{\rhob} \p_{x_1}^{k+2}\Wb_2\p_{x_1}^{k+1}\Wb_1 dx_1$ as
	\begin{align}\label{px-k+1-W1}
		&-\frac{4}{3}\int_{\R}\frac{\mu(\thetab)}{\rhob} \p_{x_1}^{k+2}\Wb_2\p_{x_1}^{k+1}\Wb_1 dx_1 \notag\\
		=&\frac{4}{3}\int_{\R}\frac{\mu(\thetab)}{\rhob} \p_{x_1}^{k+1} \left( \frac{\p_t\Wb_1}{\thetab} \right) \p_{x_1}^{k+1}\Wb_1 dx_1 -\frac{4}{3}\int_{\R}\frac{\mu(\thetab)}{\rhob} \p_{x_1}^{k+1} \left( \frac{(\mathcal{S}_{\bar \W})_1}{\thetab} \right) \p_{x_1}^{k+1}\Wb_1 dx_1 \notag\\
		=&\frac{2}{3}\frac{d}{dt}\left(\int_{\R} \frac{\mu(\thetab)}{\thetab\rhob}  \abs{\p_{x_1}^{k+1} \Wb_1} ^2 dx_1\right)-\frac{2}{3}\int_{\R} \p_t\left(\frac{\mu(\thetab)}{\thetab\rhob}  \right)\abs{\p_{x_1}^{k+1} \Wb_1} ^2 dx_1 -\frac{4}{3}\int_{\R}\frac{\mu(\thetab)}{\rhob} \p_{x_1}^{k+1} \left( \frac{(\mathcal{S}_{\bar \W})_1}{\thetab} \right) \p_{x_1}^{k+1}\Wb_1 dx_1\notag\\
		& - \frac{4}{3}\sum_{i=0}^k \int_{\R}\frac{\mu(\thetab)}{\rhob} \p_{x_1}^{k+1-i}\left( \frac{1}{\thetab} \right) \left[\p_{x_1}^i \left( \thetab \Wb_{2x_1}\right) -\p_{x_1}^i  (\mathcal{S}_{\bar \W})_1 \right]\p_{x_1}^{k+1} \Wb_1 dx_1.
	\end{align}
	And using \eqref{Wt}$_1$ again, one has
	\begin{align}\label{px-k+1-W1-1}
		\int_{\R}\p_t \p_{x_1}^k \Wb_1 \p_{x_1}^{k+1}\Wb_2 dx_1 = -\int_{\R} \p_{x_1}^k\left(\thetat\Wb_{2x_1} \right) \p_{x_1}^{k+1}\Wb_2 dx_1 + \int_{\R}\p_{x_1}^k(\mathcal{S}_{\bar \W})_1 \p_{x_1}^{k+1}\Wb_2 dx_1.
	\end{align}
	Combining \eqref{eq-phi-x}, \eqref{px-k+1-W1}, and \eqref{px-k+1-W1-1}, we obtain
	\begin{align}\label{px-k+1-W1-2}
		&\frac{2}{3}\frac{d}{dt}\norm{ \sqrt{\frac{\mu(\thetab)}{\thetab\rhob}}  \p_{x_1}^{k+1} \Wb_1}_{L^2}^2 + \frac{d}{dt}\int_{\R} \p_{x_1}^k\Wb_2\p_{x_1}^{k+1}\Wb_1 dx_1 + \frac{2}{3}\norm{\p_{x_1}^{k+1}\Wb_1}_{L^2}^2=I_5+I_6,
	\end{align}
	where
	\begin{align*}
		I_5=&\frac{4}{3}\sum_{i=0}^{k-1}\int_{\R}\p_{x_1}^{k-i}\left( \frac{\mu(\thetab)}{\rhob} \right)\p_{x_1}^{i+2}\Wb_2 \p_{x_1}^{k+1}\Wb_1 dx_1 -\frac{2}{3}\int_{\R}\p_{x_1}^{k+1}\Wb_5\p_{x_1}^{k+1}\Wb_1 dx_1 +\int_{\R} \p_{x_1}^k\left(\thetab\Wb_{2x_1} \right) \p_{x_1}^{k+1}\Wb_2  dx_1   \\   
		&
		+\frac{2}{3} \int_{\R}\p_t \left( \frac{\mu(\thetab)}{\thetab \rhob} \right) \abs{\p_{x_1}^{k+1}\Wb_1}^2 dx_1 + \frac{4}{3}\sum_{i=0}^k \int_{\R}\frac{\mu(\thetab)}{\rhob} \p_{x_1}^{k+1-i}\left( \frac{1}{\thetab} \right) \p_{x_1}^i \left( \thetab \Wb_{2x_1}\right)  \p_{x_1}^{k+1} \Wb_1 dx_1,\\
		I_6=& \int_{\R} \p_{x_1}^k (\mathcal{S}_{\bar \W})_2\p_{x_1}^{k+1}\Wb_1 dx_1 +\frac{4}{3}\int_{\R}\frac{\mu(\thetab)}{\rhob} \p_{x_1}^{k+1} \left( \frac{(\mathcal{S}_{\bar \W})_1}{\thetab} \right) \p_{x_1}^{k+1}\Wb_1 dx_1 -\frac{4}{3}\int_{\R} \p_{x_1}^{k}\left[ \frac{\mu(\thetab)\mb_1}{\rhob^2 \thetab} \p_{x_1}^2\Wb_{1} \right] \p_{x_1}^{k+1} \Wb_1 dx_1 \notag\\
		&+ \frac{4}{3}\sum_{i=0}^k \int_{\R}\frac{\mu(\thetab)}{\rhob} \p_{x_1}^{k+1-i}\left( \frac{1}{\thetab} \right)\p_{x_1}^i (\mathcal{S}_{\bar \W})_1\p_{x_1}^{k+1} \Wb_1 dx_1 - \int_{\R}\p_{x_1}^k(\mathcal{S}_{\bar \W})_1 \p_{x_1}^{k+1}\Wb_2  dx_1.                
	\end{align*}
	By H{\"o}lder's inequality, we have
	\begin{align}\label{I-7}
		I_5 \le&C\deltac \sum_{i=1}^k (1+t)^{i-k-1}\norm{\p_{x_1}^i \bar{\W}}_{L^2}^2 + \frac{1}{16000} \norm{\p_{x_1}^{k+1}\Wb_1}_{L^2}^2  +C\sum_{i=2}^5 \norm{\p_{x_1}^{k+1} \Wb_i}_{L^2}^2.
	\end{align}
	For $I_6$, we apply integration by parts to  control  $-\frac{4}{3}\int_{\R} \p_{x_1}^{k}\left[ \frac{\mu(\thetab)\mb_1}{\rhob^2 \thetab} \p_{x_1}^2\Wb_{1} \right] \p_{x_1}^{k+1} \Wb_1 dx_1$ as
	\begin{align*}
		& -\frac{4}{3}\int_{\R} \p_{x_1}^{k}\left[ \frac{\mu(\thetab)\mb_1}{\rhob^2 \thetab} \p_{x_1}^2\Wb_{1} \right] \p_{x_1}^{k+1} \Wb_1 dx_1 \\
		=& \frac{2}{3}\int_{\R} \p_{x_1} \left(  \frac{\mu(\thetab)\mb_1}{\rhob^2 \thetab} \right) \abs{\p_{x_1}^{k+1} \Wb_1}^2 d{x_1} -\frac{4}{3} \sum_{i=0}^{k-1} \int_{\R} \p_{x_1}^{k-i} \left(  \frac{\mu(\thetab)\mb_1}{\rhob^2 \thetab} \right) \p_{x_1}^{i+2} \Wb_1 \p_{x_1}^{k+1} \Wb_1 dx_1 \\
		\le& C\deltac\sum_{i=0}^k (1+t)^{i-k-1}\norm{\p_{x_1}^i \Wb_1}_{L^2}^2 + C\deltac \norm{\p_{x_1}^{k+1}\Wb_1}_{L^2}^2.
	\end{align*}
	Subsequently, we address the control of $\deltab D_{-1}$ in terms of $I_6$ that involves $(\mathcal{S}_{\bar \W})_1$ and $(\mathcal{S}_{\bar \W})_2$ \eqref{Error-1-1}. It holds that
	\begin{align*}
		\deltab\int_{\R} \p_{x_1}^k D_{-1} \left( \p_{x_1}^{k+1} \Wb_1 +  \p_{x_1}^{k+1} \Wb_2\right) dx_1 \le C \deltab(1+t)^{-k-\frac{3}{2}}+C\deltab \norm{\p_{x_1}^{k+1}\bar{\W}}_{L^2}^2.
	\end{align*}
	Then, by the {\it a priori} assumptions \eqref{aps-Wt},	the nonlinear terms of $(\mathcal{S}_{\bar \W})_2$  \eqref{Error-1-1} in $I_6$ can be bounded as 
	\begin{align*}
		&\sum_{i=1}^5 \int_{\R} \p_{x_1}^k \left(\p_{x_1}\Wb_i \right)^2 \p_{x_1}^{k+1} \Wb_1 dx_1 \\
		=& \sum_{i=1}^5 \left( \sum_{j=1}^{[\frac{k-1}{2}]} + \sum_{j=[\frac{k-1}{2}]+1}^{k-1} \right) \int_{\R} \p_{x_1}^{j+1}\Wb_i \p_{x_1}^{k+1-j}\Wb_i \p_{x_1}^{k+1}\Wb_1dx_1+2\sum_{i=1}^5 \int_{\R} \p_{x_1} \Wb_i \p_{x_1}^{k+1}\Wb_i \p_{x_1}^{k+1} \Wb_1 dx_1\\
		\le& C\norm{\p_{x_1}^{k+1}\bar{\W}}_{L^2}\left( \sum_{j=1}^{[\frac{k-1}{2}]} \norm{\p_{x_1}^{j+1}\bar{\W}}_{L^\infty} \norm{\p_{x_1}^{k+1-j}\bar{\W}}_{L^2}+\sum_{j=[\frac{k-1}{2}]+1}^{k-1}\norm{\p_{x_1}^{k+1-j}\Wb}_{L^\infty} \norm{\p_{x_1}^{j+1}\bar{\W}}_{L^2}\right)\\
		\le& C\deltac \sum_{i=0}^k (1+t)^{i-k-1}\norm{\p_{x_1}^i \bar{\W}}_{L^2}^2+C\deltac\norm{\p_{x_1}^{k+1}\bar{\W}}_{L^2}^2.
	\end{align*}
	The other terms in $I_6$ can be controlled by the H{\"o}lder's inequality and the {\it a priori} assumptions \eqref{aps-Wt}. Thus we obtain
	\begin{align}\label{I-8}
		I_6 \le& \sum_{i=0}^k C\deltac(1+t)^{i-k-1}\norm{\p_{x_1}^i \bar{\W}}_{L^2}^2+C\deltac\norm{\p_{x_1}^{k+1}\bar{\W}}_{L^2}^2+C\deltab(1+t)^{-k-\frac{3}{2}}.
	\end{align}
	Combining \eqref{px-k+1-W1-2}, \eqref{I-7} and \eqref{I-8}, one has
	\begin{align}\label{px-k+1-W1-estimate}
		& \frac{d}{dt}\norm{ \sqrt{\frac{\mu(\thetab)}{\thetab\rhob}}  \p_{x_1}^{k+1} \Wb_1}_{L^2}^2 + \frac{d}{dt}\int_{\R} \p_{x_1}^k\Wb_2\p_{x_1}^{k+1}\Wb_1 dx_1 + c\norm{\p_{x_1}^{k+1}\Wb_1}_{L^2}^2 \notag\\
		\le  &C \deltac\sum_{i=0}^k(1+t)^{i-k-1}\norm{\p_{x_1}^i \bar{\W}}_{L^2}^2+C\deltab(1+t)^{-k-\frac{3}{2}}  +C\sum_{i=2}^5\norm{\p_{x_1}^{k+1}W_i}_{L^2}^2.
	\end{align}
	\begin{flushleft}
		\textbf{Step 4. Estimates on $\Wb_i$, $i=3,4$.}
	\end{flushleft}
	Applying $\p_{x_1}^k$ to \eqref{Wt}$_{3,4}$, multiplying them by $\p^k\Wb_i$, $i=3,4$ and then integrating the resulting equation over $\R$ with respect to $x_1$, one has
	\begin{align*}
		&\frac{d}{dt}\norm{\p_{x_1}^k\Wb_i}_{L^2}^2+\int_{\R}\frac{\mu(\thetab)}{\rhob}\abs{\p_{x_1}^{k+1}\Wb_i}^2dx_1\\
		\lesssim& \abs{\int_{\R}\p_{x_1}^k\left( \frac{\mu(\thetab)}{\rhob} \p_{x_1}^2 \bar{W}_i\right)\p_{x_1}^k \bar{W}_i+   \frac{\mu(\thetab)}{\rhob}\abs{\p_{x_1}^{k+1}\Wb_i}^2dx_1}+\int_{\R}\abs{\p_{x_1}^{k+1}\Wb_i \p_{x_1}^{k_1}(\mathcal{S}_{\bar \W})_i}dx_1\nonumber:=I_{W}^{k1}+I_{W}^{k2}.
	\end{align*}
	The direct calculation yields
	\begin{align}\notag
		I_{W}^{k1}\leq C\check{\delta}\sum_{j=0}^{k}(1+t)^{-j}\norm{\p_{x_1}^{k+1-j}\bar{\W}}_{L^2}^2,
	\end{align}
	and the estimate of $I_{W}^{k2}$ is similar to \eqref{d-k-M-00}-\eqref{d-k-M-k2}. Finally, we have
	\begin{align}
		& \sum_{i=3}^4\left[\frac{d}{dt}\norm{\p_{x_1}^k\Wb_i}_{L^2}^2+\int_{\R}\frac{\mu(\thetab)}{\rhob}\abs{\p_{x_1}^{k+1}\Wb_i}^2dx_1\right]\label{psi2}\\
		\leq& C\bar{\delta} (1+t)^{-\frac{1}{2}-k}+C\check{\delta}\left[\norm{\p_{x_1}^{k+1}\Wb_1}_{L^2}^2+\norm{\p_{x_1}^{k+1}\Wb_2}_{L^2}^2+\norm{\p_{x_1}^{k+1}\Wb_5}_{L^2}^2+\sum_{j=1}^{k}(1+t)^{-j}\norm{\p_{x_1}^{k+1-j}\bar{\W}}_{L^2}^2\right].\notag
	\end{align}
	Combining \eqref{d-0-11}-\eqref{d-0-41}, \eqref{px-k+1-W1-estimate} and \eqref{psi2}, we complete the proof of Lemma \ref{Lem-Wb-1}.
\end{proof}

We are now ready for the proof of Proposition \ref{ape-Wt}.

\begin{proof}[Proof of Proposition \ref{ape-Wt}]
	In the sequel, we use the following fact which comes from \eqref{sec-n-1} and \eqref{sec-n-2}:
	\begin{align}\notag
		E_n \le C K_{n-1}+C\deltab \sum_{j=0}^{n-2}(1+t)^{j+1-n}K_j+C\deltab(1+t)^{-n}E_0.
	\end{align}
	By Gr\"onwall's inequality and \eqref{11}, one has
	\begin{align}\label{basic-W-E}
		\sum_{i=0}^n E_i\leq C(\varepsilon^2_0+\bar{\delta})(1+t)^{\frac{1}{2}},\qquad\quad\sum_{i=0}^n\int_{\R}(K_i+G_{i})dx\leq C(\varepsilon^2_0+\bar{\delta})(1+t)^{\frac{1}{2}}.
	\end{align}
	Multiplying \eqref{22} by $(1+t)$ and then integrating on $[0,t]$, one has
	\begin{align}\notag
		(1+t)\sum_{i=1}^nE_i+\sum_{i=1}^n\int_{0}^t(1+\tau)(K_i+G_{i})d\tau\leq C(\varepsilon^2_0+\bar{\delta})(1+t)^{\frac{1}{2}}.
	\end{align}
	Then it follows that
	\begin{align}\label{E1}
		\sum_{i=1}^nE_i\leq C(\varepsilon^2_0+\bar{\delta})(1+t)^{-\frac{1}{2}},\qquad\quad\sum_{i=1}^n\int_{0}^t(1+\tau)(K_i+G_{i})d\tau\leq C(\varepsilon^2_0+\bar{\delta})(1+t)^{\frac{1}{2}}.
	\end{align}
	By the same argument, we obtain
	\begin{align}\label{Ei}
		\sum_{i=k}^nE_i\leq C(\varepsilon^2_0+\bar{\delta})(1+t)^{\frac{1}{2}-k},\qquad\sum_{i=k}^n\int_{0}^t(1+\tau)^k(K_i+G_{i})d\tau\leq C(\varepsilon^2_0+\bar{\delta})(1+t)^{\frac{1}{2}}.
	\end{align}
	Finally, multiplying \eqref{kk} by $(1+t)^n$, one has
	\begin{align*}
		&(1+t)^nE_n+\int_{0}^t(1+\tau)^n(K_n+G_{n})d\tau\\
		\leq&\sum_{i=0}^{n-1}\int_0^t (1+\tau)^{i}\left(K_{i}+G_i \right)d\tau + \int_0^t (1+\tau)^{-1} E_0 d\tau+ C\bar{\delta}(1+t)^{\frac{1}{2}},\\
		\leq& C(\bar{\delta}+\varepsilon^2_0)(1+t)^{\frac{1}{2}}.
	\end{align*}
	Then it holds that
	\begin{align}\label{E2}
		E_n\leq C(\varepsilon^2_0+\bar{\delta})(1+t)^{\frac{1}{2}-n}.
	\end{align}
	By \eqref{sec-n-1}, \eqref{basic-W-E} and \eqref{E1}, we have
	\begin{align}\label{W-final-1}
		\norm{\bar{\W}}_{L^\infty} \le C \norm{\bar{\W}}_{L^2}^{\frac{1}{2}}\norm{\p_{x_1} \bar{\W}}_{L^2}^{\frac{1}{2}} \le C E_0^{\frac{1}{2}}E_1^{\frac{1}{2}} \leq C\deltab^{\frac12}.
	\end{align}
	Combining \eqref{sec-n-1} and \eqref{E1}-\eqref{W-final-1}, we complete the proof of Proposition \ref{ape-Wt}.
\end{proof}

Therefore, Theorem \ref{res-Wt} is  proved with the help of Proposition \ref{ape-Wt}. Using Theorem \ref{res-Wt}, we further have the following result on the estimate of the coupled diffusion wave $\Xi$ \eqref{omega} and the error term $\tilde{\Rm}$ \eqref{new-error}.

\begin{Cor}\label{R-m-1-1-1}
	Under the same assumptions of Theorem \ref{mt}, for $k\geq0$, one has
	\begin{align*}
		\norm{\W}_{L^{\infty}}\leq C \deltab^{\frac12}, \qquad\quad \norm{\p_{x_1}^k \Xi}_{L^2}^2\leq C \deltab(1+t)^{-\frac{1+2k}{2}}, \qquad\quad \norm{\p_{x_1}^k\tilde{\Rm}}_{L^2}^2\leq C \deltab^{3} (1+t)^{-\frac{5+2k}{2}}.
	\end{align*}
\end{Cor}
\begin{proof}
	Applying Theorem \ref{res-Wt} and the relationship between $\Xi=\p_{x_1}\W$ and  $\bar{\W}$ \eqref{trans-W-aa}, one has $ \norm{\W}_{L^{\infty}}\leq C \deltab^{\frac12}$ and $\norm{\p_{x_1}^k \Xi}_{L^2}^2\leq C \deltab (1+t)^{-\frac{1+2k}{2}}$. From the expression of $\tilde\Rm$ \eqref{new-error}, for $k=0$, we have
	\begin{align*}
		\norm{\tilde{\Rm}}_{L^2}^2\le C\norm{\abs{\Xi}^3}_{L^2}^2+C\deltab\norm{D_{-\frac12} \abs{\Xi}^2}_{L^2}^2+C\norm{\abs{\Xi}\abs{\Xi_{x_1}}}_{L^2}^2\leq C \deltab^{3}(1+t)^{-\frac52}.
	\end{align*}
	Similarly, for $k\geq1$, we have $\norm{\p_{x_1}^k\tilde{\Rm}}_{L^2}^2\leq C \deltab^{3}(1+t)^{-\frac{5+2k}{2}}$.
	Then we have completed the proof of Corollary \ref{R-m-1-1-1}.
\end{proof}

\section{Stability analysis}\label{sec4}
In this section, we carry out the stability analysis in order to obtain the estimates for the perturbation \eqref{ori-perb-1} near the local Maxwellian $M_{[\rhot,\ut,\thetat]}$ with its macroscopic quantities constructed in \eqref{New-Ansatz} for the Landau equation \eqref{equ-landau}. It is divided into five parts consisting of estimates of zero modes for macroscopic equations in subsection \ref{sec4.1}, non-zero modes for macroscopic equations in subsection \ref{sec4.2}, higher-order derivatives for the macroscopic equations in subsection \ref{sec4.3}, the microscopic equations in subsection \ref{sec4.4}, and the highest-order derivatives in subsection \ref{sec4.5}.    

\subsection{Reformulated system and estimate of zero modes for macroscopic equations}\label{sec4.1}

Recall the definition of the perturbation $(\phi,\varphi,h,\psi,\zeta,\sqrt{\mu}g)$ \eqref{ori-perb-1},
by equation \eqref{landau-4} and profile \eqref{non-conserved quantities}, we have
\begin{align}\label{perturbation-1}
	\begin{cases}
		\p_t\phi+\rhot \operatorname{div}_x \psi=\mathcal{S}_{\phi}^{f}, \\
		\p_t\psi+\frac23\frac{\thetat }{\rhot} \nabla_x \phi+ \frac23 \nabla_x \zeta-\frac{\mu(\thetat)}{\rhot}\left(\Delta_x \psi+\frac{1}{3} \nabla_x \operatorname{div}_x \psi\right)=\mathcal{S}_{\psi}^{f1}+\mathcal{S}_{\psi}^{f2}+\mathcal{S}_{\psi}^{m},\\
		\p_t\zeta+\frac{2}{3} \thetat \operatorname{div}_x \psi-\frac{\kappa(\thetat)}{\rhot}\Delta_x \zeta=\mathcal{S}_{\zeta}^{f1}+\mathcal{S}_{\zeta}^{f2}+\mathcal{S}_{\zeta}^{m},
	\end{cases}
\end{align}
where 
\begin{align}
	&\mathcal{S}_{\phi}^{f}:=-u \cdot \nabla_x \phi-\phi \operatorname{div}_x \psi-\psi \cdot \nabla_x \tilde{\rho}-\phi \operatorname{div}_x \tilde{u}-\p_{x_1}\tilde{\Rm}_{0}, \label{S-m-0} \\	 
	&\mathcal{S}_{\psi}^{f1}:=-u \cdot \nabla_x \psi-\frac{\nabla_x(p-\tilde{p})}{\rhot}+\frac23\frac{\thetat}{\rhot}\nabla_x\phi + \frac23\nabla_x\zeta-
	\psi \cdot \nabla_x \tilde{u}+\frac{\phi}{\tilde{\rho}\rho}\nabla_x\tilde{p}- \left(\frac{1}{\rho}-\frac{1}{\rhot} \right) \nabla_x\left(p-\tilde{p}\right),\label{S-m} \\
	&\mathcal{S}_{\psi}^{f2}:=\frac{\p_{x_1}\tilde{\Rm}_0 \ut-\p_{x_1} (\tilde{\Rm}_1,\tilde{\Rm}_2,\tilde{\Rm}_3)^t}{\rhot}+\frac{\mu(\theta)-\mu(\thetat)}{\rho}\left(\Delta_x u+\frac{1}{3} \nabla_x \operatorname{div}_x u\right)\notag\\
	&\quad+\frac{\mu^{\prime}(\thetat)}{\rho} \nabla_x \thetat \cdot\left(\nabla_x \psi+(\nabla_x \psi)^{t}-\frac{2}{3}\mathbb{I} \operatorname{div}_x \psi\right)-\mu'(\thetat)\frac{\nabla_x\thetat\phi}{\rhot \rho}\cdot\left(\nabla_x \tilde{u}+(\nabla_x \tilde{u})^{t}-\frac{2}{3}\mathbb{I} \operatorname{div}_x \tilde{u}\right)\nonumber\\
	&\quad+\frac{\left({\mu^{\prime}(\theta)} \nabla_x \theta-\mu'(\thetat)\nabla_x\thetat\right)}{\rho} \cdot\left(\nabla_x {u}+(\nabla_x {u})^{t}-\frac{2}{3}\mathbb{I} \operatorname{div}_x {u}\right) +\mu(\thetat)\left(\frac{1}{\rho}-\frac{1}{\rhot}\right)\left(\Delta_x u+\frac{1}{3} \nabla_x \operatorname{div}_x u\right),  \label{Q-m}\\
	&\mathcal{S}_{\psi}^{m}	:=-\frac{1}{\rho} \int_{\R^3} \xi \otimes \xi \cdot \nabla_{x} \left(L_{M}^{-1}\Pi-L_{\Mb}^{-1}\bar{\Pi}_1\right) d \xi-\frac{\phi}{\rhot \rho} \int_{\R^3} \xi \otimes \xi \cdot \nabla_{x} L_{\Mb}^{-1}\bar{\Pi}_1 d \xi,\label{K-m}\\
	&\mathcal{S}_{\zeta}^{f1}:= -u \cdot \nabla_x \zeta - \frac{2}{3} \zeta \operatorname{div}_x \psi - \psi \cdot \nabla_x \tilde{\theta} - \frac{2}{3} \zeta \operatorname{div}_x \tilde{u}, \label{S-m-4} \\
	&\mathcal{S}_{\zeta}^{f2}\nonumber:=-\frac{1}{\rhot}\left[\p_{x_1} \tilde{\Rm}_4-\p_{x_1} \tilde{\Rm}\cdot \tilde{u} -\left(\tilde{\theta}+\frac{\abs{\ut}^2}{2}\right) \p_{x_1} \tilde{\Rm}_0\right]+\left(\kappa(\theta)-\kappa(\thetat)\right) \frac{\Delta_x \theta}{\rho}+\frac{{\kappa^{\prime}(\theta)}|\nabla_x \theta|^{2}-\kappa^{\prime}(\thetat)|\nabla_x \thetat|^{2}}{\rho} \notag \\
	&\quad+\frac{\mu(\theta)}{\rho}\left[\frac{\left(\nabla_x u+(\nabla_x u)^{t}\right)^{2}}{2}-\frac{2}{3}(\operatorname{div}_x u)^{2}\right]-\frac{\mu(\thetat)}{\rho}\left[\frac{\left(\nabla_x \tilde{u}+(\nabla_x \tilde{u})^{t}\right)^{2}}{2}-\frac{2}{3}(\operatorname{div}_x \tilde{u})^{2}\right] \nonumber\\
	&\quad+\mu(\thetat)\left(\frac{1}{\rho}-\frac{1}{\rhot}\right)\left[\frac{\left(\nabla_x \tilde{u}+(\nabla_x \tilde{u})^{t}\right)^{2}}{2}-\frac{2}{3}(\operatorname{div}_x \tilde{u})^{2}\right] + \kappa(\thetat) \left(\frac{1}{\rho}-\frac{1}{\rhot}\right) \Delta_x \theta+\kappa^{\prime}(\thetat)\left(\frac{1}{\rho}-\frac{1}{\rhot}\right)|\nabla_x \thetat|^{2}, \label{Q-m_4}\\
	&\mathcal{S}_{\zeta}^{m}:=\frac{1}{\rho}\left(-\int_{\R^3} \frac{1}{2}|\xi|^{2} \xi \cdot (\nabla_{x} L_{M}^{-1} \Pi-\nabla_{x}  L_{\Mb}^{-1}\bar{\Pi}_1 )d \xi+\tilde{u} \cdot \int_{\R^3} \xi \otimes \xi \cdot(\nabla_{x} L_{M}^{-1}\Pi-\nabla_{x}  L_{\Mb}^{-1}\bar{\Pi}_1)  d \xi \right) \notag\\
	&\quad+\frac{\psi}{\rho}\cdot \int_{\R^3} \xi \otimes \xi \cdot\nabla_{x} L_{M}^{-1}\Pi  d \xi+\frac{\phi}{\rhot\rho}\left(\int_{\R^3} \frac{1}{2}|\xi|^{2} \xi \cdot \nabla_{x}  L_{\Mb}^{-1}\bar{\Pi}_1 d \xi-\tilde{u} \cdot \int_{\R^3} \xi \otimes \xi \cdot \nabla_{x}  L_{\Mb}^{-1}\bar{\Pi}_1d \xi\right).  \label{K-m-4}
\end{align}
Since $G=\bar{G}_0+\sqrt{\mu}g$, by \eqref{equ-G} and \eqref{correction-G-2}, we derive the equation of the microscopic component $g$ as
\begin{align}\label{mic-perturbation}
	&\p_t g + \xi \cdot \nabla_x g -\mathcal{L} g \notag\\[1.4mm]
	=&-\frac{1}{\sqrt{\mu}}P_1 \left[\xi_1 \left(\frac{\abs{\xi-u}^2\p_{x_1} \zeta}{2 R \theta^2} +\frac{\left( \xi -u\right)\cdot \p_{x_1} \psi}{R \theta} \right)M \right] -\sum_{i=2}^3\frac{1}{\sqrt{\mu}}P_1 \left[\xi_i \left( \frac{\abs{\xi-u}^2\p_{x_i} \zeta}{2 R \theta^2} +\frac{\left( \xi -u\right)\cdot \p_{x_i} \psi}{R \theta} \right) M\right]
	\notag\\
	&+ \Gamma\left(g,\frac{M-\mu}{\sqrt{\mu}} \right)+\Gamma\left(\frac{M-\mu}{\sqrt{\mu}},g \right) -\frac{1}{\sqrt{\mu}}P_1 \left[\xi_1 \left(\frac{\abs{\xi-u}^2\p_{x_1} (\thetat-\thetab)}{2 R \theta^2} +\frac{\left( \xi -u\right)\cdot \p_{x_1} (\tilde u-\bar u)}{R \theta} \right)M \right]   \notag\\
	& +\Gamma(\frac{G}{\sqrt{\mu}},\frac{G}{\sqrt{\mu}})  +\frac{P_0 \left( \xi \cdot \sqrt{\mu}\nabla_x g \right)}{\sqrt{\mu}}  -\frac{P_1 \left( \xi_1 \p_{x_1} \bar{G}_0 \right)}{\sqrt{\mu}}-\frac{\p_t \bar{G}_0}{\sqrt{\mu}},
\end{align}
where $\Gamma$ and $\mathcal{L}$ are defined by
\begin{align}
	\label{2.5}
	\Gamma(f,g):=\frac{1}{\sqrt{\mu}}Q(\sqrt{\mu}f,\sqrt{\mu}g),
	\quad \mathcal{L}f:=\Gamma(\sqrt{\mu},f)+\Gamma(f,\sqrt{\mu}).
\end{align}
Here  we have used the fact that
\begin{equation*}
	\frac{1}{\sqrt{\mu}} L_{M}(\sqrt{\mu}f)=
	\frac{1}{\sqrt{\mu}}\{Q(M,\sqrt{\mu}f)+Q(\sqrt{\mu}f,M)\}
	=\mathcal{L}f+\Gamma(f,\frac{M-\mu}{\sqrt{\mu}})+
	\Gamma(\frac{M-\mu}{\sqrt{\mu}},f).
\end{equation*}
Note that  the linearized Landau operator $\mathcal{L}$ is self-adjoint and non-positive definite, and its null space $\ker\mathcal{L}$ is spanned by the five functions $\{\sqrt{\mu},\xi\sqrt{\mu},|\xi|^{2}\sqrt{\mu}\}$, cf. \cite{Guo-2002}.

In order to use the anti-derivative technique, it is convenient to study the perturbation system for $(\phi,\varphi,h)$ \eqref{ori-perb-1} since these quantities are conserved. According to zero mass \eqref{initial-zero-mass-2}, we define the anti-derivative by
\begin{align}\label{anti-11-5}
	(\Phi,\Psi,H)(t,x_1):=\int_{-\infty}^{x_1}\int_{\Torus^2}(\rho-\rhot,m-\mt,\E-\Et)(t,y)dy'dy_1.
\end{align}
From $\int_{\Torus^2}$\eqref{landau-3}$dx_2dx_3$ and \eqref{New-Ansatz},  we then derive the system for $(\Phi,\Psi,H)$ as 
\begin{align}\label{conserv-1-system}
	\left\{ \begin{aligned}
		&\p_t \Phi + \p_{x_1} \Psi_{1}=-\tilde{\Rm}_{0},\\
		&\p_t \Psi_1 + \frac{2}{3} \p_{x_1} H -\frac{4}{3}\mu(\thetat) \p_{x_1}\psim_{1} + \Do\int_{{\R}^{3}} \xi_1^2\left( L_M^{-1} \Pi  -L_{\bar{M}}^{-1}\bar{\Pi}_1\right)d \xi= \mathcal{S}^{f}_{\Psi1},     \\
		&\p_t \Psi_i - \mu(\thetat) \p_{x_1}\psim_{i} + \Do \int_{\R^3} \xi_1 \xi_i\left( L_M^{-1} \Pi -L_{\bar{M}}^{-1}\bar{\Pi}_1\right)  d\xi =  \mathcal{S}^{f}_{\Psi i}, \quad \text{for} \quad i=2,3,\\
		&\p_t H + \frac{5}{3}\thetat \p_{x_1} \Psi_{1} - \kappa(\thetat) \p_{x_1}\zetam + \frac12\Do \int_{\R^3} \xi_1 \abs{\xi}^2\left( L_M^{-1} \Pi-L_{\bar{M}}^{-1}\bar{\Pi}_1\right) d\xi= \mathcal{S}^{f}_{H}   ,
	\end{aligned}\right.
\end{align}
where
\begin{align*}
	&\mathcal{S}^{f}_{\Psi1}= \Do\left[\left( \frac{\mt_1^2}{\rhot}-\frac{1}{3}\frac{\abs{\mt}^2}{\rhot} \right) - \left( \frac{m_1^2}{\rho} - \frac{1}{3}\frac{\abs{m}^2}{\rho}  \right)\right] + \frac{4}{3} \Do\left[\left( \mu(\theta)-\mu(\thetat) \right) \p_{x_1} u_{1} \right]-\tilde{\Rm}_1,\\
	&\mathcal{S}^{f}_{\Psi i}= \Do\left[\frac{\mt_1\mt_i}{\rhot} - \frac{m_1 m_i}{\rho}  \right]+ \Do\left[ \left( \mu(\theta)-\mu(\thetat) \right)\p_{x_1} u_{i } \right] -\tilde{\Rm}_i,\\
	&\mathcal{S}^{f}_{H}= \Do\left[ \frac{5}{3}\thetat \p_{x_1} \Psi_{1} -\left( \frac{m_1 \E}{\rho}+\frac{m_1 p}{\rho}-\frac{\mt_1 \Et}{\rhot} - \frac{\mt_1 \tilde{p}}{\rhot} \right)\right] + \Do\left[\left( \kappa(\theta) - \kappa(\thetat) \right)\p_{x_1}\theta\right]\notag\\
	& \qquad\;\; +\Do\left( \frac{4 }{3}  \mu(\theta) u_1 \p_{x_1} u_{1}   - \frac{4 }{3} \mu(\thetat) \ut_1 \p_{x_1}\ut_{1} + \sum_{i=2}^3 \mu(\theta) u_i \p_{x_1}u_{i}-\sum_{i=2}^3 \mu(\thetat) \ut_i \p_{x_1}\ut_{i} \right) -\tilde{\Rm}_4.
\end{align*}
For the non-fluid part, it holds that
\begin{align}
	& \Do \left(\int_{\R^3} \xi_1 \xi_i L_{M}^{-1} \Pi d\xi- \int_{\R^3} \xi_1 \xi_i L_{\bar{M}}^{-1} \bar{\Pi}_1 d\xi\right)
	:=(\mathcal{S}^{m}_{\Psi})^{i}=\sum_{j=1}^3\sum_{l=a,b}(\mathcal{S}^{m}_{\Psi})_{jl}^i+(\mathcal{S}^{m}_{\Psi})_{4b}^i,\label{non-fluid-1-a}\\
	&\Do \left(\frac{1}{2}\int_{\R^3}\xi_1\abs{\xi}^{2}L_M^{-1} \Pi d\xi-\frac{1}{2} \int_{\R^3}\xi_1\abs{\xi}^{2} L_{\bar{M}}^{-1} \bar{\Pi}_1 d\xi \right)
	:=(\mathcal{S}^{m}_{H})=\sum_{j=1}^3\sum_{l=a,b}(\mathcal{S}^{m}_{H})_{jl}+(\mathcal{S}^{m}_{H})_{4b}+(\mathcal{S}^{m}_{H})_{4}.\label{non-fluid-1-b}
\end{align}
Using \eqref{5.1}, \eqref{5.2} and the self-adjoint property of $L^{-1}_{M}$, the terms in \eqref{non-fluid-1-a} and \eqref{non-fluid-1-b} are respectively given as
\begin{align}
	&(\mathcal{S}^{m}_{\Psi})_{1a}^i= R\Do \int_{\R^3} \theta B_{1i}\left(\frac{\xi-u}{\sqrt{R\theta}} \right)\frac{\sqrt{\mu}\p_t g}{M} d\xi ,\quad(\mathcal{S}^{m}_{\Psi})_{2a}^i=R\Do \int_{\R^3}\theta B_{1i}\left(\frac{\xi-u}{\sqrt{R\theta}} \right)\frac{ P_1(\xi \sqrt{\mu} \cdot \nabla_x g)}{M} d\xi, \label{M-1} \\
	&(\mathcal{S}^{m}_{\Psi})_{3a}^i= R\Do \int_{\R^3}
	\theta B_{1i}\left(\frac{\xi-u}{\sqrt{R\theta}} \right) \frac{\sqrt{\mu}}{M} \left[ \Gamma\left(\frac{\bar{G}_0}{\sqrt{\mu}},g \right)+\Gamma\left(g,\frac{\bar{G}_0}{\sqrt{\mu}} \right)+\Gamma\left(g,g \right)  \right]d\xi ,\label{M-2t}\\
	&(\mathcal{S}^{m}_{\Psi})_{1b}^i= R\Do \int_{\R^3}\theta B_{1i}\left(\frac{\xi-u}{\sqrt{R\theta}} \right)\frac{\p_t \bar{G}_0}{M} d\xi,\;\; (\mathcal{S}^{m}_{\Psi})_{2b}^i= \Do \int_{\R^3} \xi_1 \xi_i \left( L_M^{-1}Q(\bar{G}_0,\bar{G}_0)- L_{\Mb}^{-1}  Q(\bar{G},\bar{G}) \right)d\xi, \label{M-3t}\\
	&  (\mathcal{S}^{m}_{\Psi})_{3b}^i=\Do \int_{\R^3} \xi_1 \xi_iL_M^{-1} (P_1 \xi_1 \p_{x_1} \bar G_0)-\xi_1\xi_iL_{\Mb}^{-1} (\bar{P}_1 \xi_1 \p_{x_1}\bar G)  d\xi, \label{M-4t-1-1-1-1}\\
	& (\mathcal{S}^{m}_{\Psi})_{4b}^i= \sum_{l=2}^3 R\Do \int_{\R^3} \theta B_{1i}\left(\frac{\xi-u}{\sqrt{R\theta}} \right)\frac{P_1 \left( \xi_l  \p_{x_l} \bar G_0 \right)}{M} d\xi,\quad (\mathcal{S}^{m}_{H})_{1a}=\Do \int_{\R^3} \left( R \theta \right)^{\frac{3}{2}}A_1 \left( \frac{\xi-u}{\sqrt{R\theta}} \right)\frac{\sqrt{\mu}\p_t g}{M} d\xi,  \label{M-4t-1-1}\\
	& (\mathcal{S}^{m}_{H})_{2a}=\Do  \int_{\R^3}\left( R \theta \right)^{\frac{3}{2}} A_1 \left( \frac{\xi-u}{\sqrt{R\theta}} \right) \frac{P_1\left( \xi \cdot \sqrt{\mu}\nabla_x g \right)}{M}d\xi,\label{Q-1}\\ 
	&(\mathcal{S}^{m}_{H})_{3a}=\Do \int_{\R^3}\left( R \theta \right)^{\frac{3}{2}}A_1 \left( \frac{\xi-u}{\sqrt{R\theta}} \right)\frac{\sqrt{\mu}}{M} \left[ \Gamma\left(\frac{\bar{G}_0}{\sqrt{\mu}},g \right)+\Gamma\left(g,\frac{\bar{G}_0}{\sqrt{\mu}} \right)+\Gamma\left(g,g \right)  \right]d\xi  ,\label{Q-2t}\\
	& (\mathcal{S}^{m}_{H})_{4}=\Do \left( \int_{\R^3}u \cdot \xi \xi_1 L_M^{-1}\Pi- \bar u \cdot \xi \xi_1 L_{\Mb}^{-1} \bar{\Pi}_1 d\xi \right) ,\;\;(\mathcal{S}^{m}_{H})_{1b}=\Do\left[\left( R \theta \right)^{\frac{3}{2}} \int_{\R^3} A_1 \left( \frac{\xi-u}{\sqrt{R\theta}} \right)\frac{\p_t \bar{G}_0}{M} d\xi \right],\label{Q-4t}\\
	&(\mathcal{S}^{m}_{H})_{2b}=\Do \int_{\R^3}\left( \frac{1}{2}\xi_1\abs{\xi}^{2}-\xi_1\xi\cdot u \right) L_M^{-1} (P_1 \xi_1 \p_{x_1}\bar{G}_0)-\left( \frac{1}{2}\xi_1\abs{\xi}^{2}-\xi_1\xi\cdot \bar u \right)L_{\Mb}^{-1} (\bar{P}_1  \p_{x_1} \bar{G}) d\xi,\label{Q-5t}\\
	&(\mathcal{S}^{m}_{H})_{3b}=\Do \int_{\R^3}\left( \frac{1}{2}\xi_1\abs{\xi}^{2}-\xi_1\xi\cdot u \right) L_M^{-1}  Q\left( \bar{G}_0, \bar{G}_0\right) -\left( \frac{1}{2}\xi_1\abs{\xi}^{2}-\xi_1\xi\cdot \bar u \right)L_{\Mb}^{-1}    Q\left( \bar{G}, \bar{G}\right)d\xi,\label{Q-6t}\\
	& (\mathcal{S}^{m}_{H})_{4b}=\sum_{l=2}^3\Do  \int_{\R^3}\left( R \theta \right)^{\frac{3}{2}} A_1 \left( \frac{\xi-u}{\sqrt{R\theta}} \right) \frac{P_1\left( \xi_l  \p_{x_l} \bar{G}_0 \right)}{M}d\xi.\label{Q-7t}
\end{align}
For the same reason as for introducing the transformation \eqref{trans-W-aa}, we define the following transformation
\begin{align}\label{anti-trans-qwe}
	\Phic=\thetat \Phi, \quad \Psic=\Psi,\quad\Hc=H-\Phic.
\end{align}
By \eqref{conserv-1-system} and \eqref{anti-trans-qwe}, we have
\begin{align}\label{T-A-D-1}
	\left\{ \begin{aligned}
		&\p_t \Phic + \thetat\p_{x_1}\Psic_{1}=\mathcal{S}_{\check{\Phi}}^{f},\\
		&\p_t \Psic_1 + \frac{2}{3} \p_{x_1}\Hc+\frac{2}{3}\p_{x_1}\Phic -\frac{4\mu(\thetat)}{3\rhot} \p_{x_1}^2 \Psic_1 +\sum_{j=1}^3 (\mathcal{S}^{m}_{\Psi})_{ja}^1= \mathcal{S}_{\check{\Psi}1}^{f}+\mathcal{S}^{m}_{\check{\Psi}1},    \\
		&\p_t \Psic_i - \frac{\mu(\thetat)}{\rhot}\p_{x_1}^2 \Psic_i +\sum_{j=1}^3 (\mathcal{S}^{m}_{\Psi})_{ja}^i   =  \mathcal{S}_{\check{\Psi}i}^{f}+\mathcal{S}^{m}_{\check{\Psi}i}, \quad \text{for} \quad i=2,3,\\
		&\p_t \Hc + \frac{2}{3}\thetat \p_{x_1}\Psic_{1} - \frac{\kappa(\thetat)}{\rhot}\p_{x_1}^2 \Hc + \sum_{i=1}^3(\mathcal{S}^{m}_{H})_{ia}= \mathcal{S}_{\check{H}}^{f}+\mathcal{S}^{m}_{\check{H}}-(\mathcal{S}^{m}_{H})_{4},
	\end{aligned}\right.
\end{align}
where
\begin{align*}
	&\mathcal{S}_{\check{\Phi}}^{f}=-\thetat \tilde{\Rm}_0+\frac{\p_t\thetat}{\thetat}\Phic, \quad \mathcal{S}_{\check{\Psi}1}^{f}=\frac{4\mu(\thetat)}{3}\left( \p_{x_1}\psim_{1}-\frac{\p_{x_1}^2\Psic_1}{\rhot} \right)+ \mathcal{S}^{f}_{\Psi1},\quad \mathcal{S}^{m}_{\check{\Psi}1}=-\sum_{j=1}^4(\mathcal{S}^{m}_{\Psi})_{jb}^1, \\
	&\mathcal{S}_{\check{\Psi}i}^{f}=\mu(\thetat)\left( \p_{x_1}\psim_{i}-\frac{\p_{x_1}^2\Psic_i}{\rhot} \right)+ \mathcal{S}^{f}_{\Psi i}, \quad \mathcal{S}^{m}_{\check{\Psi}i}=-\sum_{j=1}^4(\mathcal{S}^{m}_{\Psi})_{jb}^i, \quad \text{for}\quad i=2,3, \\
	&\mathcal{S}_{\check{H}}^{f}=\kappa(\thetat)\left( \p_{x_1}\zetam-\frac{\p_{x_1}^2 \Hc}{\rhot} \right)+ \mathcal{S}^{f}_{H}+\thetat \tilde{\Rm}_0-\frac{\p_t\thetat}{\thetat}\Phic,\qquad \mathcal{S}^{m}_{\check{H}}=-\sum_{j=1}^4 (\mathcal{S}^{m}_{H})_{jb}.
\end{align*}

Recall the definition of $\Vmc$ \eqref{2025-11-5-1}, instant energy functionals $\mathcal{E}_i$ \eqref{2.10}-\eqref{2.10-5} and dissipation energy functionals $\mathcal{D}_i$ \eqref{2.11}-\eqref{2.11-5}. Then, we can present the $H^2$ energy estimate for the zero mode \eqref{T-A-D-1} as follows.

\begin{Thm}\label{H2-enenrgy-anti}
	Under the same assumptions of  Proposition \ref{Thm-ape}, it holds that
	\begin{align*}
		& {\frac{d}{dt}}\left[\norm{\Vmc}_{L^2}^2+\sum_{l=1}^4\mathcal{X}^l+ \tilde{c}\left(\norm{\p_{x_1}\Phic}_{L^2}^2+\int_{\R} \p_{x_1}\Phic\Psic_1 dx_1 \right)\right] +\tilde{c}\norm{\p_{x_1}\Vmc}_{L^2}\\
		&\leq C\bar{\delta}(1+t)^{-\frac32}+C\check{\delta}\left[(1+t)^{-1}\mathcal{E}_1 + \mathcal{D}_1 +\norm{\p_t\nabla_x \vm^{\ast}}_{L^2}^2\right]+C_\eta\sum_{\abs{\alpha}\le 2}\norm{\p^{\alpha} g}_{\sigma}^2,
	\end{align*}
	and
	\begin{align*}
		&{\frac{d}{dt}}\left[\norm{\p_{x_1}\Vmc}_{L^2}^2+ \tilde{c}\left(\norm{\p_{x_1}^2\Phic}_{L^2}^2+\int_{\R} \p_{x_1}^2\Phic\p_{x_1}\Psic_1 dx_1 \right) \right]+\tilde{c}\norm{\p_{x_1}^2\Vmc}_{L^2}\\
		&\leq C\bar{\delta}(1+t)^{-\frac52}+C\check{\delta}\left[(1+t)^{-2}\mathcal{E}_1 +(1+t)^{-1} \mathcal{D}_1+\mathcal{D}_2 +\norm{\p_t\nabla_x \vm^{\ast}}_{L^2}^2\right]+C_\eta\sum_{1\le\abs{\alpha}\le 2}\norm{\p^{\alpha} g}_{\sigma}^2,
	\end{align*}
	and
	\begin{align*}
		&{\frac{d}{dt}}\left(\norm{\p_{x_1}^2\left(\Phic,\sum_{i=2,3}\Psic_i,\Hc\right)}_{L^2}^2  + \norm{\p_{x_1}\left( \thetat \p_{x_1} \Psic_1\right)}_{L^2}^2 \right) + \norm{\p_{x_1}^3\left(\Psic,\Hc\right)}_{L^2}^2\\
		&\leq C\bar{\delta}(1+t)^{-\frac52}+C\check{\delta}\left[(1+t)^{-3}\mathcal{E}_1 +\sum_{j=1}^3(1+t)^{-3+j} \mathcal{D}_j+\sum_{\abs{\alpha}= 1}\norm{\p^{\alpha} g}_{\sigma}^2+\norm{\p_t\nabla_x \vm^{\ast}}_{L^2}^2\right]+C_\eta\sum_{\abs{\alpha}= 2}\norm{\p^{\alpha} g}_{\sigma}^2,
	\end{align*}
	where $ \deltab:=\delta+\varepsilon_0$, $\deltac:=\chi+\deltab^{\frac12}$ and for $i=1,2,3$,
	\begin{align}\label{definition-of-Xm}
		& \mathcal{X}^i= R\int_{\R} \Do \int_{\R^3} \theta B_{1i}\left( \frac{\xi-u}{\sqrt{R\theta}}\right) \frac{\sqrt{\mu}}{M} g d\xi \Psic_i dx_1 , \quad \mathcal{X}^4= R^{\frac{3}{2}}\int_{\R} \Do \int_{\R^3} \theta^{\frac{3}{2}} A_{1}\left( \frac{\xi-u}{\sqrt{R\theta}}\right) \frac{\sqrt{\mu}}{M} g d\xi \frac{\Hc}{\thetat}  dx_1.
	\end{align}
\end{Thm}
The proof of Theorem \ref{H2-enenrgy-anti} will be decomposed into three parts including Lemma \ref{s-t-s-tt-1}, Lemma \ref{s-t-s-tt-2} and Lemma \ref{H-D-Phi-lem}.
Before proving these three lemmas, we should study the source terms in \eqref{T-A-D-1}. Recall the definition of $\tilde{D}_{-\alpha}$ \eqref{tilde-d} and $(\Vm,\Vmc,\vm)$ \eqref{2025-11-5-1}. We use the following notations for the sake of convenience:
\begin{align}\label{useful-notation}
	\begin{aligned}
		&\D^{(k)}:=\deltab^{\frac12} \sum_{j=1}^{k+2} \Dt_{-\frac{j}{2}} \abs{\p_{x_1}^{k-j+2} \Vmc},\qquad 
		\T^{(0)}:=\abs{\p_{x_1} \Vmc}^2+ \deltab^{\frac12}  \Dt_{-\frac{1}{2}}\abs{\p_{x_1} \Vmc},\\
		&\T^{(1)}:=\abs{\p_{x_1}^2 \Vmc}\abs{\p_{x_1} \Vmc}+\deltab^{\frac12} \left(\Dt_{-\frac{1}{2}}\abs{\p_{x_1} \Vmc}^2+\Dt_{-1}\abs{\p_{x_1} \Vmc}+\Dt_{-\frac{1}{2}}\abs{\p_{x_1}^2 \Vmc}\right),\\
		&\T^{(2)}:=\abs{\p_{x_1}^3 \Vmc}\abs{\p_{x_1} \Vmc}+\abs{\p_{x_1}^2 \Vmc}^2+ \deltab^{\frac12} \Dt_{-\frac{3}{2}}\abs{\p_{x_1} \Vmc}\\
		&\qquad\qquad +\deltab^{\frac12} \Dt_{-\frac{1}{2}}\big(\abs{\p_{x_1} \Vmc}\abs{\p_{x_1}^2 \Vmc}+\abs{\p_{x_1}^3 \Vmc}\big)+ \deltab^{\frac12} \Dt_{-1}\big(\abs{\p_{x_1}^2 \Vmc}+\abs{\p_{x_1} \Vmc}^2\big),\\
				&\Z^{(0)}:=\abs{\vm_{\neq}}^2,\qquad\qquad\Z^{(1)}:=\abs{\vm_{\neq}}\abs{\nabla_x \vm_{\neq}},\qquad\qquad\Z^{(2)}:=\abs{\nabla_x \vm_{\neq}}^2+\abs{\nabla_x^2 \vm_{\neq}}\abs{\vm_{\neq}}.
			\end{aligned}
		\end{align}
		\begin{Lem}\label{nonlinear-error-estimate}
			Under the same assumptions of  Proposition \ref{Thm-ape}, for $k=0,1$, one has
			\begin{align}
				&\sum_{i=1}^3\abs{\mathcal{S}_{\check \Psi i}^{m}}+\abs{\mathcal{S}_{\check H}^m}\le C \deltab\tilde{D}_{-\frac{3}{2}}+C \deltab \tilde{D}_{-1}\abs{\vm}+C\deltab \tilde D_{-\frac12}(\abs{\nabla_x\vm^{\ast}}+\abs{\p_t\vm^{\ast}}),\label{lem-2-1-n-3-1-1-1}\\
				&\abs{(\mathcal{S}^{m}_{\Psi})^i}+\abs{(\mathcal{S}^{m}_{H})}\le C\deltab \Dt_{-1}+C\sum_{\abs{\alpha}=1}\abs{\p^{\alpha}g}_{\sigma}+C\left( \abs{g}_{2}+\deltab\Dt_{-\frac12}\right)\abs{g}_{\sigma}+C\deltab \tilde D_{-\frac12}(\abs{\nabla_x\vm^{\ast}}+\abs{\p_t\vm^{\ast}}),\label{lem-2-1-n-3}\\
				&\abs{\p_{x_1}(\mathcal{S}^{m}_{\Psi})^i}+\abs{\p_{x_1}(\mathcal{S}^{m}_{H})}\le C\sum_{\abs{\alpha}=2}\abs{\p^{\alpha}g}_{\sigma}+C\sum_{\abs{\gamma}=1}\bigg[\abs{g}_{\sigma}\abs{\p^{\gamma}g}_{\sigma}+\deltab \tilde D_{-\frac12}(\abs{\nabla_x^2\vm^{\ast}}+\abs{\p_t\nabla_x\vm^{\ast}})\notag\\
				&\qquad\qquad\qquad\qquad\quad+\left(\deltab^{\frac12} \Dt_{-\frac12}+\abs{\p_t \vm^{\ast}}+\abs{\nabla_x \vm}\right)\left( \deltab \Dt_{-1}+\abs{\p^{\gamma}g}_{\sigma}+\abs{g}_2\abs{g}_{\sigma}+\deltab \Dt_{-\frac12}\abs{g}_\sigma  \right)\bigg],\label{lem-2-1-n-4}\\
				&\abs{(\mathcal{S}^{m}_{H})_{4}}\le C (\deltab^{\frac12}+\chi) (1+t)^{-\frac12}\abs{(\mathcal{S}^{m}_{\Psi})^i}+C\deltab \left(\tilde D_{-1} \abs{\vm^{\ast}}+\tilde D_{-\frac32}\right),\label{lem-2-1-n-5-1-1}\\
				&\abs{\p_{x_1} (\mathcal{S}^{m}_{H})_{4}}\lesssim (\deltab^{\frac12}+\chi) \left[(1+t)^{-\frac34}\abs{(\mathcal{S}^{m}_{\Psi})^i} + (1+t)^{-\frac12}\abs{\p_{x_1}(\mathcal{S}^{m}_{\Psi})^i}\right]+\deltab \left(\tilde D_{-\frac32} \abs{\vm^{\ast}}+\tilde{D}_{-1}\abs{\nabla_x\vm^{\ast}}+\tilde D_{-\frac52}\right),\label{lem-2-1-n-5}\\
				&\abs{\p_{x_1}^k\mathcal{S}_{\check{\Phi}}^{f}}\le C\deltab \tilde{D}_{-\frac{3+k}{2}}+\D^{(k)}, \label{2026-6-8}\\
				&\sum_{i=1}^3\abs{\p_{x_1}^k\mathcal{S}_{\check{\Psi}i}^{f}}+\abs{\p_{x_1}^k\mathcal{S}_{\check H}^f}\leq C\deltab \tilde{D}_{-\frac{3+k}{2}} +\D^{(k)}+\T^{(k)}+\T^{(k+1)}+\Z^{(k)}+\Z^{(k+1)}\label{lem-2-1-n-2}.
			\end{align}
		\end{Lem}
		\begin{proof}
			We first present the calculation of the microscopic parts, that is, to prove \eqref{lem-2-1-n-3}-\eqref{lem-2-1-n-5}.
			For $\mathcal{S}_{\Psic_i}^m$ and $\mathcal{S}_{\Hc}^m$ in \eqref{T-A-D-1}, by the definition of $\bar{G}_0$ \eqref{correction-G-2}, estimate of $\bar{G}_0$ Lemma \ref{2026-4-24-1} and the definition of $(\mathcal{S}^{m}_{\Psi})_{1b}^i$ \eqref{M-3t}, $(\mathcal{S}^{m}_{\Psi})_{4b}^i$ \eqref{M-4t-1-1}, $(\mathcal{S}^{m}_{H})_{1b}$ \eqref{Q-4t} and $(\mathcal{S}^{m}_{H})_{4b}$ \eqref{Q-7t}, it holds that
			\begin{align}\notag
				\abs{(\mathcal{S}^{m}_{\Psi})_{1b}^i}+\abs{(\mathcal{S}^{m}_{\Psi})_{4b}^i}+\abs{(\mathcal{S}^{m}_{H})_{1b}}+\abs{(\mathcal{S}^{m}_{H})_{4b}}\le C \deltab\tilde{D}_{-\frac{3}{2}}+C\deltab \tilde D_{-\frac12} \abs{\p_t\vm^{\ast}}+C\deltab \tilde D_{-\frac12} \abs{\nabla_x\vm^{\ast}}.
			\end{align}
			Moreover, $(\mathcal{S}^{m}_{\Psi})_{2b}^{i}$ \eqref{M-3t}, $(\mathcal{S}^{m}_{\Psi})_{3b}^{i}$ \eqref{M-4t-1-1-1-1}, $(\mathcal{S}^{m}_{H})_{2b}$ \eqref{Q-5t} and $(\mathcal{S}^{m}_{H})_{3b}$ \eqref{Q-6t} can be treated in the same way, so we only calculate $(\mathcal{S}^{m}_{\Psi})_{3b}^{i}$ \eqref{M-4t-1-1-1-1}. By the self-adjoint properties of $L_M^{-1}$ and expansion of $\p_{x_1}\bar G_0$ \eqref{5.20}, we have
			\begin{align}\label{2026-4-24-2}
				&\Do \int_{\R^3} \xi_1\xi_i L_M^{-1} (P_1 \xi_1 \p_{x_1}\bar G_0) d\xi \notag\\
				=&\Do \int_{\R^3} R\theta B_{1i}\left(\frac{\xi-u}{\sqrt{R\theta}} \right) \frac{ \xi_1 \p_{x_1}\bar G_0}{M} d\xi-\sum_{j=0}^4\Do \int_{\R^3} R\theta B_{1i}\left(\frac{\xi-u}{\sqrt{R\theta}} \right) \frac{ \la\xi_1 \p_{x_1}\bar G_0,\chi_j\ra\chi_j}{M} d\xi\notag\\
				=&\sum_{i,j=1}^3 \Do\left( \mathcal{A}_1 \p_{x_1}^2\thetab+\mathcal{A}_2\p_{x_1}\theta\p_{x_1}\thetab+\mathcal{A}_3\p_{x_1}u_i\p_{x_1}\thetab+\mathcal{A}_4\p_{x_1}^2\bar u_i+\mathcal{A}_5 \p_{x_1}\bar u_i\p_{x_1} u_j+\mathcal{A}_6\p_{x_1}\bar u_i \p_{x_1}\theta\right).
			\end{align}
			Using the same method, we also have
			\begin{align}\label{2026-4-24-3}
				&\Do \int_{\R^3} \xi_1\xi_i L_{\bar M}^{-1} (\bar P_1 \xi_1 \p_{x_1}\bar G) d\xi \notag\\
				=&\sum_{i,j=1}^3 \Do\left( \bar{\mathcal{A}}_1 \p_{x_1}^2\thetab+\bar{\mathcal{A}}_2\p_{x_1}\thetab\p_{x_1}\thetab+\bar{\mathcal{A}}_3\p_{x_1}\ub_i\p_{x_1}\thetab+\bar{\mathcal{A}}_4\p_{x_1}^2\bar u_i+\bar{\mathcal{A}}_5 \p_{x_1}\bar u_i\p_{x_1} \ub_j+\bar{\mathcal{A}}_6\p_{x_1}\bar u_i \p_{x_1}\thetab\right),
			\end{align}
			where $\mathcal{A}_{1,\dots,6}$ and $\bar{\mathcal{A}}_{1,\dots,6}$ are the smooth function of $(\rho,u,\theta)$ and $(\bar\rho,\ub,\thetab)$, respectively. By \eqref{2026-4-24-2}, \eqref{2026-4-24-3} and $\abs{(\rho,u,\theta)-(\rhob,\ub,\thetab)}\lesssim \deltab\tilde{D}_{-\frac12}+\abs{\vm}$, we have
			\begin{align*}
				\abs{(\mathcal{S}^{m}_{\Psi})_{3b}^{i}}\le C\deltab\tilde{D}_{-\frac32}+C\deltab \tilde D_{-1} \abs{\vm}+C\deltab \tilde D_{-\frac12} \abs{\nabla_x \vm^{\ast}}.
			\end{align*}
			Based on the above estimates, we then obtain
			\begin{align}\notag
				\sum_{i=1}^3\abs{\mathcal{S}_{\check{\Psi}i}^{m}}+\abs{\mathcal{S}_{\check H}^m} \le  C \deltab\tilde{D}_{-\frac{3}{2}}+C \deltab \tilde{D}_{-1}\abs{\vm}+C\deltab \tilde D_{-\frac12}(\abs{\nabla_x\vm^{\ast}}+\abs{\p_t\vm^{\ast}}).
			\end{align}
			Then we have finished the proof of  \eqref{lem-2-1-n-3-1-1-1}.
			
			From the definition of $(\mathcal{S}^{m}_{\Psi})^i,\;(\mathcal{S}^{m}_{H})$ \eqref{non-fluid-1-a} and \eqref{non-fluid-1-b}, and by noting $\Pi=\p_t G + P_1(\xi \cdot \nabla_x G)-Q(G,G)$, $G=\bar{G}_0+\sqrt{\mu}g$, and $(\rho,u,\theta)=(\rhot,\ut,\thetat)+(\phi,\psi,\zeta)$, it can be directly calculated that
			\begin{align}\notag
				&\abs{(\mathcal{S}^{m}_{\Psi})^i}+\abs{(\mathcal{S}^{m}_{H})}\le C\deltab \Dt_{-1}+C\sum_{\abs{\alpha}=1}\abs{\p^{\alpha}g}_{\sigma}+C\left( \abs{g}_{2}+\deltab\Dt_{-\frac12}\right)\abs{g}_{\sigma}+C\deltab \tilde D_{-\frac12}(\abs{\nabla_x\vm^{\ast}}+\abs{\p_t\vm^{\ast}}).
			\end{align}
			To calculate $\p_{x_1}(\mathcal{S}^{m}_{\Psi})^i$ and $\p_{x_1}(\mathcal{S}^{m}_{H})$, we note that the following fact holds:
			\begin{align}
				&\abs{\p_{x_1} \int_{\R^3} \xi_i \xi_j L_M^{-1}\Pi d\xi}=R\abs{\p_{x_1} \int_{\R^3} \theta B_{ij}\left( \frac{\xi-u}{\sqrt{R\theta}} \right)\frac{1}{M} \left(\p_t G+P_1 \xi \cdot \nabla_x G -Q(G,G)\right) d\xi} \notag\\
				\le&C\sum_{\abs{\gamma}=1}\left[\abs{g}_{\sigma}\abs{\p^{\gamma}g}_{\sigma}+\left(\deltab^{\frac12} \Dt_{-\frac12}+\abs{\nabla_x \vm}\right)\left( \deltab \Dt_{-1}+\abs{\p^{\gamma}g}_{\sigma}+\abs{g}_2\abs{g}_{\sigma}+\deltab \Dt_{-\frac12}\abs{g}_\sigma  \right)\right]\notag\\
				&+C\deltab \tilde D_{-\frac12}(\abs{\nabla_x^2\vm^{\ast}}+\abs{\p_t\nabla_x\vm^{\ast}})+C\sum_{\abs{\alpha}=2}\abs{\p^{\alpha}g}_{\sigma}.\label{non-fluid-s-s-1}
			\end{align}
			Since the treatment of $\p_{x_i} \int_{\R^3} \xi_i \abs{\xi}^2 L_M^{-1}\Pi d\xi$ is similar, we obtain
			\begin{align*}
				&\abs{\p_{x_1}(\mathcal{S}^{m}_{\Psi})^i}+\abs{\p_{x_1}(\mathcal{S}^{m}_{H})}\le C\sum_{\abs{\alpha}=2}\abs{\p^{\alpha}g}_{\sigma}+C\sum_{\abs{\gamma}=1}\bigg[\abs{g}_{\sigma}\abs{\p^{\gamma}g}_{\sigma}+\deltab \tilde D_{-\frac12}(\abs{\nabla_x^2\vm^{\ast}}+\abs{\p_t\nabla_x\vm^{\ast}})\notag\\
				&\qquad\qquad\qquad\qquad\quad+\left(\deltab^{\frac12} \Dt_{-\frac12}+\abs{\nabla_x \vm}\right)\left( \deltab \Dt_{-1}+\abs{\p^{\gamma}g}_{\sigma}+\abs{g}_2\abs{g}_{\sigma}+\deltab \Dt_{-\frac12}\abs{g}_\sigma  \right)\bigg].
			\end{align*}
			Thus \eqref{lem-2-1-n-3} and \eqref{lem-2-1-n-4} have been proved. 
			
			For the calculation of $(\mathcal{S}^{m}_{H})_{4}$ and $\p_{x_1}(\mathcal{S}^{m}_{H})_{4}$, we have
			\begin{align*}
				(\mathcal{S}^{m}_{H})_{4}=\sum_{i=1}^3 u_i (\mathcal{S}^{m}_{\Psi})^i + \Do\int_{\R^3} (u-\bar u) \cdot \xi\xi_1L_{\bar M}^{-1}\bar \Pi_1 d\xi.
			\end{align*}
			Then, by applying the {\it a prior} assumption   \eqref{apa}, we obtain \eqref{lem-2-1-n-5-1-1} and \eqref{lem-2-1-n-5}.
			
			At this point, we then provide an estimate of the macroscopic part, namely \eqref{2026-6-8} and \eqref{lem-2-1-n-2}. By H{\"o}lder's inequality for  the definition of $\mathcal{S}_{\check \Phi}^{f}$ \eqref{T-A-D-1}, we have
			\begin{align}\notag
				\abs{\p_{x_1}^k\mathcal{S}_{\check{\Phi}}^{f}}\le C\deltab \tilde{D}_{-\frac{3+k}{2}}+\D^{(k)}.
			\end{align}
			Thus, \eqref{2026-6-8} has been proved.
			
			Below, we will mainly focus on the estimate of $\mathcal{S}_{\check{\Psi}i}^{f}$ and $\mathcal{S}_{\check{H}i}^{f}$, which consists of two parts via viscosity and flux. Taking
			\begin{equation}
				\label{add.s4p1}
				\kappa(\thetat)\left( \p_{x_1} \zetam-\frac{\p_{x_1}^2 \Hc}{\rhot} \right)
			\end{equation} 
			as an example, we present the calculation of the viscosity part. We remark that from the calculation of the term \eqref{add.s4p1} above, it can  be found that the transformation \eqref{anti-trans-qwe} plays an important role in the viscosity part. We present the following identities needed for calculating this term:
			\begin{align}
				&\zetam = \Do\left[\frac{\E}{\rho}-\frac{\Et}{\rhot}-\left( \frac{\abs{m}^2}{2\rho^2} - \frac{\abs{\mt}^2}{2\rhot^2}  \right)\right],\qquad \frac{\Em}{\rhom}-\frac{\Et}{\rhot}=\frac{\mathring{h}}{\rhom}-\frac{\phim\thetat}{\rhom}-\frac{\abs{\ut}^2\phim}{2\rhom},\label{con-noncon-1}\\
				&\frac{\mm_i}{\rhom}-\frac{\mt_i}{\rhot}=\frac{\varphim_i}{\rhom}-\frac{\mt_i\phim}{\rhot\rhom},\qquad \frac{\mm_i^2}{\rhom^2}-\frac{\mt_i^2}{\rhot^2}=\frac{\rhot\mm_i+\rhom\mt_i}{\rhom^2\rhot} \varphim_i-  \frac{\mt_i(\rhot\mm_i+\rhom\mt_i) }{\rhom^2 \rhot^2}\phim.\label{con-noncon-2}
			\end{align}
			Combining the definition of $\Phic$ \eqref{anti-trans-qwe} and \eqref{con-noncon-1}, the term \eqref{add.s4p1} can be written as
			\begin{align}\label{lem-2-1-3}
				&\eqref{add.s4p1}
				=\kappa(\thetat)\p_{x_1}\Do\left[ \frac{\E}{\rho}-\frac{\Em}{\rhom}+\frac{\abs{\mm}^2}{2\rhom^2}-\frac{\abs{m}^2}{2\rho^2} \right]-\kappa(\thetat)\p_{x_1}\left[\frac{\abs{\mm}^2}{2\rhom^2}-\frac{\abs{\mt}^2}{2\rhot^2}+\frac{\abs{\ut}^2\phim}{2\rhom} \right]\notag
				\\
				&\qquad\qquad\qquad\qquad\qquad+\kappa(\thetat)\left[ \p_{x_1}\left( \frac{\mathring{h}}{\rhom}  \right) - \frac{\p_{x_1}^2H}{\rhot}-\p_{x_1}\left( \frac{\phim \thetat}{\rhom} \right)+\frac{\p_{x_1}^2(\thetat \Phi)}{\rhot}\right]=\mathcal{L}_1+\mathcal{L}_2+\mathcal{L}_3.
			\end{align}
			To estimate \eqref{lem-2-1-3}, by \eqref{con-noncon-2}, it holds that
			\begin{align}
				&\abs{\mathcal{L}_1}\lesssim \abs{\vm_{\neq}}^2+\abs{\vm_{\neq}}\abs{\nabla_x \vm_{\neq}},\notag
				\\
				&\abs{\mathcal{L}_2}+\abs{\mathcal{L}_3}\le C\deltab^{\frac12}\Dt_{-1}\abs{\Vmc}+C\deltab^{\frac12}\Dt_{-\frac{1}{2}}\left(\abs{\p_{x_1}\Vmc}+ \abs{\p_{x_1}^2\Vmc}\right)+C\abs{\p_{x_1}^2\Vmc}\abs{\p_{x_1}\Vmc}+C\abs{\p_{x_1}\Vmc}^2.
				\notag
			\end{align}
			For the flux terms in $\mathcal{S}_{\check{\Psi}i}^{f}$ and $\mathcal{S}_{\check{H}}^{f}$, we take the most complex term 
			\begin{equation}
				\label{add.sec4.sphp}
				\Do\left[ \frac{5}{3}\thetat \p_{x_1} \Psi_{1} -\left( \frac{m_1 \E}{\rho}+\frac{m_1 p}{\rho}-\frac{\mt_1 \Et}{\rhot} - \frac{\mt_1 \tilde{p}}{\rhot} \right)\right]
			\end{equation}			
			as an example to present the calculation. We compute 
			\begin{align}
				&\mathring{p}-\tilde{p}=\frac{2}{3}\mathring{h} - \frac{2}{3} \Do \left[\left( \frac{\varphi_1^2}{2\rho} + \frac{\mt_1 \varphi_1}{\rho} - \frac{\mt_1^2\phi}{\rho \rhot} \right) +\frac{\varphi_2^2}{2 \rho} + \frac{\varphi_3^2}{2 \rho} \right],\label{lem-2-1-6}\\
				&\frac{\mm_1 \Em}{\rhom} - \frac{\mt_1 \tilde{\E}}{\rhot} = \frac{\Et \varphim_1}{\rhom}+\frac{\mt_1 \mathring{h}}{\rhom}   -\frac{\mt_1 \Et \phim}{\rhom \rhot}  + \frac{\varphim_1 \mathring{h}}{\rhom}   , \label{lem-2-1-7}\\
				&\frac{\mm_1 \mathring{p}}{\rhom} - \frac{\mt_1 \tilde{p}}{\rhot} = \frac{\mt_1 (\mathring{p}-\tilde{p})}{\rhom} + \frac{\tilde{p}\varphim_1}{\rhom} - \frac{\mt_1 \tilde{p} \phim}{\rhom \rhot} + \frac{\varphim_1 (\tilde{p}-\tilde{p})}{\rhom}.   \label{lem-2-1-8}
			\end{align}
			Using \eqref{lem-2-1-7} and \eqref{lem-2-1-8}, it follows that 		
			\begin{align}\notag
				\eqref{add.sec4.sphp}
				=\mathcal{L}_4+\mathcal{L}_5+\mathcal{L}_6, 
			\end{align}
			where
			\begin{align*}
				&\mathcal{L}_4=\frac{5}{3}\thetat\p_{x_1}\Psi_1-\frac{\Et \varphim_1}{\rhom}-\frac{\tilde{p}\varphim_1}{\rhom},\quad\mathcal{L}_5=\Do\left( \frac{\mm_1 \Em}{\rhom} -\frac{m_1\E}{\rho}\right)+\Do\left( \frac{\mm_1 \mathring{p}}{\rhom} - \frac{m_1p}{\rho}\right),\\
				&\mathcal{L}_6=\frac{\mt_1 \Et \phim}{\rhom \rhot}-\frac{\mt_1 \mathring{h}}{\rhom}-\frac{\varphim_1\mathring{h}}{\rhom}-\frac{\mt_1(\mathring{p}-\tilde{p})}{\rhom}-\frac{\varphim_1(\mathring{p}-\tilde{p})}{\rhom}+\frac{\mt_1 \tilde{p}\phim}{\rhom\rhot}.
			\end{align*}
			By \eqref{lem-2-1-6}, one has
			\begin{align}\notag
				\abs{\mathcal{L}_5}+\abs{\mathcal{L}_6}  \le C\deltab^{\frac12}\Dt_{-1}\abs{\Vmc}+C\deltab^{\frac12}\Dt_{-\frac{1}{2}}\abs{\p_{x_1}\Vmc}+C\abs{\p_{x_1}\Vmc}^2+C\abs{\vm_{\neq}}^2.
			\end{align}
			From the relationships among $\thetat$, $\Et$ and $\tilde{p}$ \eqref{non-conserved quantities}, we have
			\begin{align}
				\abs{\mathcal{L}_4}&=\abs{\frac{5}{3}\Et \p_{x_1}\Psi_1\left( \frac{1}{\rhot}-\frac{1}{\rhom} \right)-\frac{5}{3}\frac{\abs{\mt}^2}{2\rhot^2}\p_{x_1}\Psi_1+\frac{\abs{\mt}^2}{3 \rhot} \frac{\varphim_1}{\rhom}} \notag\\
				&\le C\deltab^{\frac12} \tilde{D}_{-1}\abs{\p_{x_1}\Vmc}+\abs{\p_{x_1}\Vmc}^2.\notag
			\end{align}
			For $i=1,2,3,4$, the remaining terms in $\mathcal{S}_{\check{\Psi}i}^{f}$ can be calculated using a method similarly for obtaining those estimates as above.
			We thus have
			\begin{align}\label{lem-2-1-12}
				\sum_{i=1}^3\abs{\mathcal{S}_{\check{\Psi}i}^{f}}+\abs{\mathcal{S}_{\check H}^f}\leq C\deltab^{\frac12} \tilde{D}_{-\frac{3}{2}} +\D^{(0)}+\T^{(0)}+\T^{(1)}+\Z^{(0)}+\Z^{(1)}.
			\end{align}
			Using the same argument as for obtaining \eqref{lem-2-1-12}, 
			one has
			\begin{align}\notag
				\sum_{i=1}^3\abs{\p_{x_1}\mathcal{S}_{\check{\Psi}i}^{f}}+\abs{\p_{x_1}\mathcal{S}_{\check H}^f}\le C\deltab^{\frac12} \Dt_{-2}+\D^{(1)}+\T^{(1)}+\T^{(2)}+\Z^{(1)}+\Z^{(2)}.
			\end{align}
			We have finished the proof of \eqref{lem-2-1-n-2}. Thus, we have completed the proof of Lemma \ref{nonlinear-error-estimate}. 
		\end{proof}
		
		Now we can give the $H^2$ estimates for system \eqref{T-A-D-1}.
		
		\begin{Lem}\label{s-t-s-tt-1}
			Recall the definition of $\Vmc$ \eqref{2025-11-5-1}. Under the same assumptions of  Proposition \ref{Thm-ape}, it holds that
			\begin{align}
				&{\frac{d}{dt}}\left(\norm{\Vmc}_{L^2}^2+\sum_{l=1}^4\mathcal{X}^l\right) +\left( \norm{\p_{x_1}\Psic}_{L^2}^2 + \norm{\p_{x_1}\Hc}_{L^2}^2 \right)
				\le C\deltab(1+t)^{-\frac32}+C\deltac\norm{\p_{x_1}\Phic}_{L^2}^2+\eta\norm{\p_t \Vmc}_{L^2}^2 \notag\\
				&\quad\qquad+C\check{\delta}\left[ (1+t)^{-1}\norm{\Vmc}_{L^2}^2+\norm{\p_{x_1}^2\Vmc}_{L^2}^2+\norm{\p_t \vm}_{L^2}^2+\norm{\nabla_x \vm}_{L^2}^2\right]+ C_\eta\sum_{\abs{\gamma}\leq 1} \norm{\p^{\gamma}g}_{\sigma}^2,\label{estimate-of-V-1}\\
				&{\frac{d}{dt}}\norm{\p_{x_1}\Vmc}_{L^2}^2 +\left( \norm{\p_{x_1}^2\Psic}_{L^2}^2 + \norm{\p_{x_1}^2\Hc}_{L^2}^2 \right)
				\le C\deltab(1+t)^{-\frac52}+C\check{\delta}\norm{\p_{x_1}^2\Phi}_{L^2}^2 + C_\eta\sum_{\abs{\gamma}=1} \norm{\p^{\gamma}g}_{\sigma}^2\notag\\
				&\quad\qquad+C\check{\delta}\left[\sum_{i=0}^1(1+t)^{-i-1}\norm{\p_{x_1}^{1-i}\Vmc}_{L^2}^2+(1+t)^{-1}\norm{g}_{\sigma}^2+\norm{\nabla_x \vm}_{L^2}^2+\norm{\p_t\vm^{\ast}}_{L^2}^2\right],\label{estimate-of-V-2}\\
				&{\frac{d}{dt}}\left(\norm{\p_{x_1}^2\left(\Phic,\sum_{i=2,3}\Psic_i,\Hc\right)}_{L^2}^2  + \norm{\p_{x_1}\left( \thetat \p_{x_1} \Psic_1\right)}_{L^2}^2 \right) + \norm{\p_{x_1}^3\left(\Psic,\Hc\right)}_{L^2}^2 \notag\\
				\le& C_\eta\sum_{\abs{\alpha}=2} \norm{\p^{\alpha}g}_{\sigma}^2+C\check{\delta}\left[\norm{\p_{x_1}^3\Phic}_{L^2}^2+\sum_{i=0}^2(1+t)^{-i-1}\norm{\p_{x_1}^{2-i}\Vmc}_{L^2}^2+(1+t)^{-2}\norm{g}_{\sigma}^2+\sum_{\abs{\gamma}=1} \norm{\p^{\gamma}g}_{\sigma}^2\right] \notag\\
				&+C\deltab\left((1+t)^{-2}\norm{\nabla_x \vm}_{L^2}^2+\norm{\partial_t \vm}_{L^2}^2\right)+C\deltab(\norm{\nabla_x^2\vm}_{L^2}^2+\norm{\p_t\nabla_x \vm^{\ast}}_{L^2}^2)+C\deltab (1+t)^{-\frac52}.\label{estimate-of-V-3}
			\end{align}    
		\end{Lem}
		\begin{proof}
			We first prove \eqref{estimate-of-V-1}. Multiplying \eqref{T-A-D-1} by $\left(\frac{2\Phic}{3\thetat},\Psic_1,\Psic_2,\Psic_3,\frac{\Hc}{\thetat}\right)$, one has
			\begin{align*}
				&\frac{d}{dt}\left(\frac{1}{3}\norm{\frac{\Phic}{\thetat^{\frac{1}{2}}}}_{L^2}^2+\frac{1}{2}\norm{\Psic}_{L^2}^2+\frac{1}{2}\norm{\frac{\Hc}{\thetat^{\frac{1}{2}}}}_{L^2}^2 \right)+\norm{\sqrt{\frac{4\mu(\thetat)}{3\rhot}} \p_{x_1}\Psic_1}_{L^2}^2\\
				&\quad+\sum_{i=2}^3\norm{\sqrt{\frac{\mu(\thetat)}{\rhot}} \p_{x_1}\Psic_i}_{L^2}^2+\norm{\sqrt{\frac{\kappa(\thetat)}{\thetat\rhot}}\p_{x_1}\Hc}_{L^2}^2=J_{a1}+J_{a2}+J_{a3},
			\end{align*}
			where
			\begin{align*}
				J_{a1}=&-\sum_{i=1}^3\sum_{j=1}^3\int_{\R}(\mathcal{S}^{m}_{\Psi})_{ja}^i \Psic_i dx - \sum_{i=1}^3\int_{\R} \frac{\left((\mathcal{S}^{m}_{H})_{ia}+(\mathcal{S}^{m}_{H})_{4}\right)\Hc}{\thetat} dx_1, \\
				J_{a2}=&\sum_{i=1}^3\int_{\R}\left(\mathcal{S}_{\check{\Psi}i}^{f} +\mathcal{S}^{m}_{\check{\Psi}i}\right) \Psic_i dx_1+\frac{2}{3}\int_{\R} \frac{\mathcal{S}_{\check{\Phi}}^{f}}{\thetat} \Phic dx_1+\int_{\R} \frac{\left(\mathcal{S}_{\check{H}}^{f}+\mathcal{S}^{m}_{\check{H}}\right) \Hc}{\thetat} dx_1 ,\\
				J_{a3}=&\frac{1}{2}\int_{\R} \p_{x_1}^2\left( \frac{\kappa(\thetat)}{\thetat\rhot}\right)\Hc^2 dx_1+\int_{\R}\p_{x_1}^2\left( \frac{2\mu(\thetat)}{3\rhot}\right)\Psic_{1}^2 dx_1 + \frac{1}{2}\sum_{i=2}^3 \int_{\R}\p_{x_1}^2\left( \frac{\mu(\thetat)}{\rhot}\right)\Psic_{i}^2 dx_1\\
				&-\frac{1}{3}\int_{\R} \frac{\p_t \thetat}{\thetat^2} \Phic^2 dx_1  - \frac{1}{2} \int_{\R} \frac{\p_t \thetat}{\thetat^2} \Hc^2dx_1.
			\end{align*}
			Directly, one has
			\begin{align}
				\abs{J_{a3}}&\leq C\deltab \norm{\tilde{D}_{-1}}_{L^\infty} \norm{\Vmc}_{L^2}^2\leq C \deltab(1+t)^{-1}\norm{\Vmc}_{L^2}^2.  \label{ja13-a}
			\end{align}
			We now give a more refined estimate of $J_{a1}$. By the definition of $(\mathcal{S}^{m}_{\Psi})_{1a}^i$ \eqref{M-1} and applying integration by parts, one has
			\begin{align}\label{math-M-1a}
				\int_{\R} (\mathcal{S}^{m}_{\Psi})_{1a}^i \Psic_i d{x_1}=&R\underbrace{\p_t\int_{\R} \Do \int_{\R^3} \theta B_{1i}\left( \frac{\xi-u}{\sqrt{R\theta}}\right) \frac{\sqrt{\mu}}{M} g d\xi \Psic_i dx_1 }_{\mathcal{X}^i}-R\underbrace{\int_{\R} \Do \int_{\R^3} \p_t\left(\theta B_{1i}\left( \frac{\xi-u}{\sqrt{R\theta}}\right) \frac{\sqrt{\mu}}{M}\right) g d\xi \Psic_i dx_1}_{\mathcal{X}_1} \notag \\
				&-R\underbrace{\int_{\R} \Do \int_{\R^3} \theta B_{1i}\left( \frac{\xi-u}{\sqrt{R\theta}}\right) \frac{\sqrt{\mu}}{M} g d\xi \p_t\Psic_i dx_1 }_{\mathcal{X}_2}.
			\end{align}
			Applying H{\"o}lder's inequality and the {\it a priori} assumptions \eqref{apa}, we obtain
			\begin{align}
				\mathcal{X}_1 \leq &\int_{\Omega} \abs{\p_t (\rhot,\ut,\thetat)} \abs{g}_{\sigma} \abs{\Psic_i} dx + \int_{\Omega}  \abs{\p_t(\phi,\psi,\zeta)} \abs{g}_{\sigma}  \abs{\Psic_i} dx \notag \\
				\leq&C (\deltab^{\frac12}+\chi) \left[  (1+t)^{-1} \norm{\Psic_i}_{L^2}^2 +\norm{\p_t(\phi,\psi,\zeta)}_{L^2}^2 + \norm{g}_{\sigma}^2   \right],\label{A-T-g-1} \\
				\mathcal{X}_2 \leq & C_\eta \norm{g}_{\sigma}^2+\eta\norm{\p_t \Psic_i}_{L^2}^2. \label{A-T-g-2}
			\end{align}
			Combining \eqref{math-M-1a}, \eqref{A-T-g-1} and \eqref{A-T-g-2}, one has
			\begin{align}\label{math-M-1a-1}
				\int_{\R} (\mathcal{S}^{m}_{\Psi})_{1a}^i \Psic_i dx_1-\mathcal{X}^i 
				\leq C\deltac \left[  (1+t)^{-1} \norm{\Psic}_{L^2}^2 +\norm{\p_t(\phi,\psi,\zeta)}_{L^2}^2 \right] + \eta \norm{\p_t \Psic}_{L^2}^2+C_\eta \norm{g}_{\sigma}^2 . 
			\end{align}
			Using the same method, one has
			\begin{align}\label{math-M-4a-1}
				\int_{\R} (\mathcal{S}^{m}_{\Psi})_{2a}^i\Psic_i dx_1 
				\leq C \deltac \left[  (1+t)^{-1} \norm{\Psic}_{L^2}^2 +\norm{\p_{x_1}(\phi,\psi,\zeta)}_{L^2}^2\right] + \eta \norm{\p_{x_1}  \Psic}_{L^2}^2+C_\eta \norm{g}_{\sigma}^2. 
			\end{align}
			And we also have
			\begin{align}
				\int_{\R}(\mathcal{S}^{m}_{\Psi})_{3a}^i\Psic_idx_1&=R\int_{\R} \Do\int_{\R^3} \theta B_{1i}\left( \frac{\xi-u}{\sqrt{R\theta}}\right) \frac{\sqrt{\mu}}{M} \left[\Gamma\left(g,g \right)+\Gamma\left(\frac{\bar{G}_0}{\sqrt{\mu}},g \right)+\Gamma\left(g,\frac{\bar{G}_0}{\sqrt{\mu}}\right)  \right]d\xi \Psic_1 dx_i\notag\\
				&\leq C\deltab(1+t)^{-1}\norm{\Psic}_{L^2}^2+ C(\deltab+\chi)\norm{g}_{\sigma}^2,\label{math-M-5a-1}
			\end{align}
			where we have used the {\it a priori} assumptions \eqref{apa}. Combining \eqref{math-M-1a-1}, \eqref{math-M-4a-1} and \eqref{math-M-5a-1}, one has
			\begin{align}\label{math-M-1-1-1}
				& \sum_{i=1}^3  \sum_{j=1}^3 \int_{\R}(\mathcal{S}^{m}_{\Psi})_{ja}^i  \Psic_i dx_1-\sum_{i=1}^3\mathcal{X}^i \notag\\
				\leq &C\deltac \left[  (1+t)^{-1} \norm{\Vmc}_{L^2}^2 +\norm{\p_t \vm}_{L^2}^2+\norm{\nabla_x \vm}_{L^2}^2\right] +\eta\left( \norm{\p_t \Vmc}_{L^2}^2+\norm{\p_{x_1} \Psic}_{L^2}^2\right)+ C_\eta \norm{g}_{\sigma}^2 .
			\end{align}
			For the terms containing $(\mathcal{S}^{m}_{H})_{4}$, by \eqref{lem-2-1-n-3}, \eqref{lem-2-1-n-5-1-1} and the {\it a priori} assumptions \eqref{apa}, we have
			\begin{align}\label{math-M-1-1-1-1}
				\int_{\R} \abs{\frac{(\mathcal{S}^{m}_{H})_{4}\Hc}{\thetat}}dx_1 \leq C\deltab(1+t)^{-\frac32}+ C\check{\delta}\left[(1+t)^{-1}\norm{\Hc}_{L^2}^2+\sum_{\abs{\alpha}\le1}\norm{\p^{\alpha}g}_{\sigma}^2 +\norm{\vm}_{H^1}^2+\norm{\p_t\vm}_{L^2}^2\right].
			\end{align}
			The remaining terms in $J_{a1}$ can be treated  as for deriving \eqref{math-M-1a}-\eqref{math-M-1-1-1-1}. Then we obtain
			\begin{align}\label{math-M-final}
				J_{a1} + \sum_{l=1}^4 \mathcal{X}^l\leq& C\deltab(1+t)^{-\frac32}+C\check{\delta} \left[  (1+t)^{-1} \norm{\Vmc}_{L^2}^2 +\norm{\p_t \vm}_{L^2}^2+\norm{\nabla_x \vm}_{L^2}^2\right]\notag \\
				&+\eta\left( \norm{\p_t \Vmc}_{L^2}^2+\norm{\p_{x_1} (\Psic,\Hc)}_{L^2}^2\right)+ C_\eta \sum_{\abs{\alpha}\le 1}\norm{\p^{\alpha}g}_{\sigma}^2 .
			\end{align}
			By \eqref{useful-notation}, \eqref{lem-2-1-n-3}, \eqref{lem-2-1-n-2} and the {\it a priori} assumptions \eqref{apa}, one has
			\begin{align}\label{ja12}
				J_{a2}&\le C\int_{\R} \left[\deltab\Dt_{-\frac{3}{2}}+\deltab\Dt_{-1}\abs{\vm}+\deltab \tilde D_{-\frac12}(\abs{\nabla_x\vm^{\ast}}+\abs{\p_t\vm^{\ast}})+\D^{(0)}+\T^{(0)}+\T^{(1)}+\Z^{(0)}+\Z^{(1)}\right] \abs{\Vmc} dx_1\notag\\
				&\leq C\deltab(1+t)^{-\frac32}+C\check{\delta}\left[ (1+t)^{-1}\norm{\Vmc}_{L^2}^2+\norm{\p_{x_1}\Vmc}_{H^1}^2+\norm{\nabla_x \vm}_{L^2}^2 +\norm{\p_t \vm^{\ast}}_{L^2}^2\right] .
			\end{align}
			Combining \eqref{ja13-a}, \eqref{math-M-final} and \eqref{ja12}, we have completed the proof of \eqref{estimate-of-V-1}.
			
			To prove \eqref{estimate-of-V-2}, multiplying $\p_{x_1}$\eqref{T-A-D-1} by $\left(\frac{2}{3}\p_{x_1}\Phic,\thetat\p_{x_1}\Psic_{1},\thetat\p_{x_1}\Psic_{2},\thetat\p_{x_1}\Psic_{3},\p_{x_1}\Hc \right)$, one has
			\begin{align*}
				&\frac{d}{dt}\left(\frac{1}{3}\norm{\p_{x_1}\Phic}_{L^2}^2+\frac{1}{2}\norm{\p_{x_1}\Psic}_{L^2}^2+\frac{1}{2}\norm{\p_{x_1}\Hc}_{L^2}^2 \right)+\norm{\sqrt{\frac{4\mu(\thetat)\thetat}{3\rhot}} \p_{x_1}^2\Psic_1}_{L^2}^2\\
				&\quad+\sum_{i=2}^3\norm{\sqrt{\frac{\mu(\thetat)\thetat}{\rhot}} \p_{x_1}^2\Psic_i}_{L^2}^2+\norm{\sqrt{\frac{\kappa(\thetat)}{\rhot}}\p_{x_1}^2\Hc}_{L^2}^2=J_{b1}+J_{b2}+J_{b3},
			\end{align*}
			where
			\begin{align*}
				J_{b1}=&-\sum_{i=1}^3\sum_{j=1}^3\int_{\R}\thetat\p_{x_1}(\mathcal{S}^{m}_{\Psi})_{ja}^i \p_{x_1}\Psic_i dx_1-\sum_{i=1}^3\int_{\R}\left( \p_{x_1}(\mathcal{S}^{m}_{H})_{ia}+ \p_{x_1}(\mathcal{S}^{m}_{H})_{4}\right)\p_{x_1}\Hc dx_1 , \\
				J_{b2}=&\sum_{i=1}^3\int_{\R}\thetat \left(\p_{x_1}\mathcal{S}_{\check{\Psi}i}^{f}+\p_{x_1} \mathcal{S}^{m}_{\check{\Psi}i} \right) \p_{x_1}\Psic_i dx_1+\frac23\int_{\R} \p_{x_1}\mathcal{S}_{\check{\Phi}}^{f} \p_{x_1}\Phic dx_1+\int_{\R} \left(\p_{x_1}\mathcal{S}_{\check{H}}^{f} + \p_{x_1}\mathcal{S}^{m}_{\check{H}} \right)\p_{x_1}\Hc dx_1,\\
				J_{b3}=&\frac{1}{2}\sum_{i=1}^3\int_{\R} \p_t \thetat \abs{\p_{x_1}\Psic_{i}}^2 dx_1+\int_{\R}\frac{2}{3}\p_{x_1}\left( \frac{\p_{x_1}\thetat\mu(\thetat)}{\rhot}\right)\abs{\p_{x_1}\Psic_1}^2 + \sum_{i=2}^1\frac{1}{2}\p_{x_1}\left( \frac{\p_{x_1}\thetat\mu(\thetat)}{\rhot}\right)\abs{\p_{x_1}\Psi_i}^2 dx_1.
			\end{align*}
			Using \eqref{M-1}, \eqref{M-2t}, \eqref{M-4t-1-1}-\eqref{Q-4t}, \eqref{lem-2-1-n-3}, \eqref{lem-2-1-n-5}, the {\it a priori} assumptions \eqref{apa} and integration by parts, one has
			\begin{align}
				&\abs{J_{b1}} \le \sum_{i=1}^3\sum_{j=1}^3\int_{\R}\left(\deltab\tilde{D}_{-\frac12}\abs{\p_{x_1}\Vmc}+\abs{\p_{x_1}^2 \Vmc^{\ast}}\right)\abs{(\mathcal{S}^{m}_{\Psi})_{ja}^i}  dx_1+\sum_{i=1}^3\int_{\R}\left( \abs{(\mathcal{S}^{m}_{H})_{ia}}+ \abs{(\mathcal{S}^{m}_{H})_{4}}\right)\abs{\p_{x_1}^2\Vmc^{\ast}} dx_1 \notag\\
				&\qquad\le C\deltab(1+t)^{-\frac52}+C\check{\delta}(1+t)^{-1}\left(\norm{\p_{x_1}\Vmc}_{L^2}^2 +\norm{g}_{\sigma}^2\right) +\eta\norm{\p_{x_1}^2 (\Psic,\Hc)}_{L^2}^2\notag\\
				&\qquad\quad+C\deltab\norm{\nabla_x \vm}_{L^2}^2 + C\deltab\norm{\p_t\vm^{\ast}}_{L^2}^2+C_\eta \sum_{\abs{\alpha}=1}\norm{\p^{\alpha}g}_{\sigma}^2,\label{jb11}\\
				&\abs{J_{b3}}\leq C\deltab(1+t)^{-1} \norm{\p_{x_1}\Vmc}_{L^2}^2. \label{jb11-a}
			\end{align}
			Then employing \eqref{lem-2-1-n-3}, \eqref{lem-2-1-n-2}, the {\it a priori} assumptions \eqref{apa} and integration by parts, we obtain
			\begin{align}\label{jb12}
				J_{b2}&\le C\int_{\R} \left[\deltab\Dt_{-\frac{3}{2}}+\deltab\Dt_{-1}\abs{\vm}+\deltab \tilde D_{-\frac12}(\abs{\nabla_x\vm^{\ast}}+\abs{\p_t\vm^{\ast}})+\D^{(0)}+\T^{(0)}+\T^{(1)}+\Z^{(0)}+\Z^{(1)}\right] \notag\\
				&\qquad\qquad\qquad \times\left(\deltab^{\frac12}\Dt_{-\frac12}\abs{\p_{x_1}\Vmc}+\abs{\p_{x_1}^2 \Vmc}\right) dx_1\notag\\
				&\leq C\deltab(1+t)^{-\frac52}+C\check{\delta}\left[ \sum_{i=0}^2(1+t)^{-2+i}\norm{\p_{x_1}^i\Vmc}_{L^2}^2+\norm{\nabla_x \vm}_{L^2}^2+\norm{\p_t \vm^{\ast}}_{L^2}^2 \right] .
			\end{align}
			Combining \eqref{jb11}, \eqref{jb11-a} and \eqref{jb12}, we have completed the proof of \eqref{estimate-of-V-2}.
			
			To prove \eqref{estimate-of-V-3}, we will rewrite the equation $\p_{x_1}^2$\eqref{T-A-D-1} in the following form to avoid the limitations of the structural conditions
			\begin{align}\label{T-A-D-2}
				\left\{ \begin{aligned}
					&\p_t \p_{x_1}^2\Phic +\p_{x_1}^2\left( \thetat\p_{x_1}\Psic_{1}\right)=\p_{x_1}^2\mathcal{S}_{\check{\Phi}}^{f},\\
					&\p_t\p_{x_1}\left( \thetat \p_{x_1}\Psic_{1} \right) + \frac{2}{3} \p_{x_1}\left(\thetat\p_{x_1}^2\Hc\right)+\frac{2}{3}\p_{x_1}\left(\thetat\p_{x_1}^2\Phic\right) -\p_{x_1}\left[\thetat\p_{x_1}\left(\frac{4\mu(\thetat)}{3\rhot} \p_{x_1}^2\Psic_{1} \right) \right]\\
					&\;+\p_{x_1}\left(\thetat\p_{x_1}(\mathcal{S}^{m}_{\Psi})^1\right)= \p_{x_1}\left( \thetat \p_{x_1}\mathcal{S}_{\check{\Psi}1}^{f}\right)     + \p_{x_1}\left( \p_t \thetat \p_{x_1} \Psic_1 \right),\\
					&\p_t \p_{x_1}^2\Psic_i - \p_{x_1}^2\left(\frac{\mu(\thetat)}{\rhot}\p_{x_1}^2 \Psic_i\right) + \p_{x_1}^2 (\mathcal{S}^{m}_{\Psi})^i  =  \p_{x_1}^2 \mathcal{S}_{\check{\Psi}i}^{f} ,\quad \text{for} \quad i=2,3,\\
					&\p_t \p_{x_1}^2\Hc + \frac{2}{3}\p_{x_1}^2\left(\thetat \p_{x_1}\Psic_{1}\right) - \p_{x_1}^2\left(\frac{\kappa(\thetat)}{\rhot}\p_{x_1}^2 \Hc\right) + \p_{x_1}^2 (\mathcal{S}^{m}_{H})= \p_{x_1}^2\mathcal{S}_{\check{H}}^{f},
				\end{aligned}\right.
			\end{align}
			where we have used $(\mathcal{S}^{m}_{\Psi})^{i}=\sum_{j=1}^3\sum_{l=a,b}(\mathcal{S}^{m}_{\Psi})_{jl}^i+(\mathcal{S}^{m}_{\Psi})_{4b}^i$ \eqref{non-fluid-1-a} and $(\mathcal{S}^{m}_{H})=\sum_{j=1}^3\Big((\mathcal{S}^{m}_{H})_{ja}+(\mathcal{S}^{m}_{H})_{jb}\Big)+(\mathcal{S}^{m}_{H})_{4b}+(\mathcal{S}^{m}_{H})_{4}$ \eqref{non-fluid-1-b}. Multiplying \eqref{T-A-D-2} by $\left(\frac{2}{3}\thetat\p_{x_1}^2\Phic,\p_{x_1}\left(\thetat\p_{x_1}\Psic_{1}\right), \p_{x_1}^2\Psic_{2},\p_{x_1}^2\Psic_{3}, \; \thetat \p_{x_1}^2\Hc \right)$, one has
			\begin{align}\label{2025-11-12-2}
				&\frac{d}{dt}\left[\frac{1}{3}\norm{\thetat^{\frac{1}{2}}\p_{x_1}^2\Phic}_{L^2}^2+\frac{1}{2}\left(\norm{\p_{x_1}\left(\thetat\p_{x_1}\Psic_{1}\right)}_{L^2}^2+
				\sum_{i=2}^{3}\norm{\p_{x_1}^2\Psic_i}_{L^2}^2 +\norm{\thetat^{\frac{1}{2}}\p_{x_1}^2\Hc}_{L^2}^2\right)\right]\notag\\
				&+\frac{4}{3}\norm{\thetat\sqrt{\frac{\mu(\thetat)}{\rhot}} \p_{x_1}^3\Psic_1}_{L^2}^2
				+\sum_{i=2}^3\norm{\sqrt{\frac{\mu(\thetat)}{\rhot}} \p_{x_1}^3\Psic_i}_{L^2}^2+\norm{\sqrt{\frac{\thetat\kappa(\thetat)}{\rhot}}\p_{x_1}^3\Hc}_{L^2}^2=J_{c1}+J_{c2}+J_{c3},
			\end{align}
			where
			\begin{align*}
				J_{c1}=&-\int_{\R}\p_{x_1}\left(\thetat\p_{x_1} (\mathcal{S}^{m}_{\Psi})^1\right) \p_{x_1}\left(\thetat \p_{x_1}\Psic_1\right) dx_1-\sum_{i=2}^3\int_{\R}\p_{x_1}^2(\mathcal{S}^{m}_{\Psi})^i \p_{x_1}^2\Psic_i dx_1-\int_{\R} \thetat\p_{x_1}^2(\mathcal{S}^{m}_{H}) \p_{x_1}^2\Hc dx_1 , \\
				J_{c2}=&\frac23\int_{\R} \thetat\p_{x_1}^2\mathcal{S}_{\check{\Phi}}^{f} \p_{x_1}^2\Phic dx_1+\int_{\R} \p_{x_1}\left(\thetat \p_{x_1}\Psic_1\right)\p_{x_1}\left( \thetat\p_{x_1}\mathcal{S}_{\check{\Psi}1}^{f} \right) +\sum_{i=2}^3\int_{\R} \p_{x_1}^2\mathcal{S}_{\check{\Psi}i}^{f} \p_{x_1}^2\Psic_i dx_1+\int_{\R}\thetat\p_{x_1}^2\mathcal{S}_{\check{H}}^{f}\p_{x_1}^2\Hc dx_1,\\
				J_{c3}=&\frac{1}{3}\int_{\R}\p_t\thetat\abs{\p_{x_1}^2 \Phic}^2dx_1+\frac{1}{2}\int_{\R}\p_t\thetat\abs{\p_{x_1}^2 \Hc}^2dx_1+\int_{\R} \left(\thetat \p_{x_1}\Psic_1\right)_{x_1}\left(\p_t\thetat \p_{x_1}\Psic_1 \right)_{x_1}dx_1\\
				&+\frac12\sum_{i=2}^3\int_{\R} \p_{x_1}^2 \left( \frac{\mu(\thetat)}{\rhot}\right)\abs{\p_{x_1}^2\Psic_i}^2dx_1+    \frac{1}{2}\int_{\R} \left[ \p_{x_1}^2\left( \frac{\kappa(\thetat)}{\rhot}\right)\thetat+\frac{\kappa(\thetat)}{\rhot}\p_{x_1}^2\thetat   \right]\abs{\p_{x_1}^2\Hc}^2dx_1\\
				&-\frac{4}{3} \int_{\R}\thetat \p_{x_1} \left(\frac{\mu(\thetat)}{\rhot} \right)\p_{x_1}^{2}\Psic_1\p_{x_1}^2\left(\thetat\p_{x_1}\Psic_1\right)dx_1-\frac{4}{3} \int_{\R} \frac{\thetat\mu(\thetat)}{\rhot} \p_{x_1}^3 \Psic_1 \left( 2\p_{x_1}\thetat \p_{x_1}^2\Psic_1+\p_{x_1}^2\thetat\p_{x_1}\Psic_1\right)dx_1.
			\end{align*}
			Employing \eqref{lem-2-1-n-4} and the {\it a priori} assumptions \eqref{apa}, one has
			\begin{align}
				\abs{J_{c1}}\leq&C \sum_{i=1}^3\int_{\R} \left( \abs{\p_{x_1}(\mathcal{S}^{m}_{\Psi})^i}+\abs{\p_{x_1}(\mathcal{S}^{m}_{H})} \right)\left( \bar{\delta}^{\frac12}\Dt_{-1}\abs{\p_{x_1}\Vmc} +\bar{\delta}^{\frac12}\Dt_{-\frac12}\abs{\p_{x_1}^2\Vmc}+\abs{\p_{x_1}^3(\Psic,\Hc)}\right) dx_1 \notag\\
				\leq& C\check{\delta}\left[\sum_{i=0}^1(1+t)^{-2+i}\norm{\p_{x_1}^{1+i}\Vmc}_{L^2}^2+(1+t)^{-2}(\norm{\nabla_x\vm}_{L^2}^2+\norm{\p_t\vm}_{L^2}^2)+(1+t)^{-2}\norm{g}_{\sigma}^2+\sum_{\abs{\gamma}=1}\norm{\p^{\gamma}g}_{\sigma}^2\right]\notag\\
				&+\eta\norm{\p_{x_1}^3(\Psic,\Hc)}_{L^2}^2+C_\eta \sum_{\abs{\alpha}=2}\norm{\p^{\alpha}g}_{\sigma}^2+C\deltab\Big(\norm{\nabla_x^2 \vm}_{L^2}^2+\norm{\p_t\nabla\vm^{\ast}}_{L^2}^2+(1+t)^{-\frac52}\Big),\label{jc11}\\
				\abs{J_{c3}}\leq&C\check{\delta}\left[(1+t)^{-3}\norm{\Vmc}_{L^2}^2+ (1+t)^{-2}\norm{\p_{x_1}\Vmc}_{L^2}^2+ (1+t)^{-1}\norm{\p_{x_1}^2\Vmc}_{L^2}^2+\norm{\p_{x_1}^3(\Psic,\Hc)}_{L^2}^2\right].\label{jc11-1}
			\end{align}
			Then employing \eqref{useful-notation}, \eqref{lem-2-1-n-2} and integration by parts, we obtain
			\begin{align}\label{jc12}
				J_{c2}&\le C  \int_{\R} \left[\deltab\Dt_{-2}+\D^{(1)}+\T^{(1)}+\T^{(2)}+\Z^{(1)}+\Z^{(2)}\right] \left(\abs{\p_{x_1}^3\Vmc}+\deltab^{\frac12}\Dt_{-\frac12}\abs{\p_{x_1}^2\Vmc}+\deltab^{\frac12} \Dt_{-1}\abs{\p_{x_1}\Vmc}\right) dx_1\notag \\[1.5mm]
				&\leq C\deltab(1+t)^{-\frac72}+C\check{\delta}\left[ \sum_{i=0}^3 (1+t)^{-i}\norm{\p_{x_1}^{3-i}\Vmc}_{L^2}^2+\norm{\nabla_x^2\vm}_{L^2}^2 \right],
			\end{align}
			where we have used the a $prioti$ assumptions \eqref{apa} and 
			\begin{align*}
				&\int_{\R} \abs{\p_{x_1}^2\Vmc}^2 \abs{\p_{x_1}^3\Vmc} dx_1 \le C 
				\norm{\p_{x_1}^2 \Vmc}_{L^4}^2 \norm{\p_{x_1}^3\Vmc}_{L^2} \le C \norm{\p_{x_1}^2 \Vmc}_{L^2}^{\frac{3}{2}} \norm{\p_{x_1}^3 \Vmc}_{L^2}^\frac{3}{2}\\
				&\qquad\le C \chi^{-3}\norm{\p_{x_1}^2\Vmc}_{L^2}^6+C\chi \norm{\p_{x_1}^3 \Vmc}_{L^2}^2 \le  C \chi(1+t)^{-1}\norm{\p_{x_1}^2\Vmc}_{L^2}^2+C \chi\norm{\p_{x_1}^3 \Vmc}_{L^2}^2,  \\
				&\int_{\R} \abs{\p_{x_1}\Vmc}\abs{\p_{x_1}^2\Vmc} \abs{\p_{x_1}^3\Vmc} dx_1 \le C
				\norm{\p_{x_1} \Vmc}_{L^{\infty}}\norm{\p_{x_1}^2 \Vmc}_{L^2} \norm{\p_{x_1}^3\Vmc}_{L^2} \\
				&\qquad\le C\chi\norm{\p_{x_1}^3 \Vmc}_{L^2}^2+ C\chi(1+t)^{-1}\norm{\p_{x_1}^2\Vmc}_{L^2}^2.
			\end{align*}
			Combining \fullrangefour{2025-11-12-2}{jc11}{jc11-1}{jc12}, 
			we have completed the proof of \eqref{estimate-of-V-3}. Hence we have finished the proof of Lemma \ref{s-t-s-tt-1}.
		\end{proof}
		
		\begin{Rem}\label{2026-5-14-1}
			In deriving the energy estimate \eqref{2025-11-12-2}, we have utilized the fact that 
			\begin{align*}
				&\int_{\R} \p_{x_1}^2\left( \thetat\p_{x_1}\Psic_{1}\right) \thetat \p_{x_1}^2 \Phic dx_1 + \int_{\R} \p_{x_1} \left(\thetat \p_{x_1}\Psic_1 \right) \p_{x_1}\left(\thetat \p_{x_1}^2\Phic \right)dx_1=0,\\ 
				& \int_{\R} \p_{x_1}^2\left( \thetat\p_{x_1}\Psic_{1}\right) \thetat \p_{x_1}^2 \Hc dx_1 + \int_{\R} \p_{x_1} \left(\thetat \p_{x_1}\Psic_1 \right) \p_{x_1}\left(\thetat \p_{x_1}^2\Hc \right)dx_1=0.
			\end{align*}
			As a result, the term 
			$
			\deltab \int_{\R} \tilde{D}_{-\frac12} \left[(\p_{x_1}^2\Phic)^2+\abs{\p_{x_1}^2 \Psic}^2 + (\p_{x_1}^2\Hc)^2 \right] dx_1
			$
			does not appear in the energy estimate for system \eqref{T-A-D-2}, which is the main reason why system \eqref{T-A-D-2} is used for the second-order derivative estimate.
		\end{Rem}
		
		Recall the definition of $(\Vmc,\vm,\vm^{\ast})$ \eqref{2025-11-5-1}. Next, we estimate $\norm{\p_t \Vmc}_{L^2}^2$ and $\norm{\p_t \vm}_{L^2}^2$.
		\begin{Lem}\label{s-t-s-tt-2}
			Under the same assumptions of  Proposition \ref{Thm-ape}, it holds that
			\begin{align*}
				&\norm{\p_t\vm}_{L^2}\leq C\bigg[\norm{\nabla_x \vm}_{L^2}+\norm{\nabla_x^2\vm^{\ast}}_{L^2}+\deltab\norm{\p_t\nabla_x\vm^{\ast}}_{L^2}^2+\sum_{1\leq\abs{\alpha}\leq2}\norm{\p^{\alpha}g}_{\sigma} \\
				&\qquad\qquad\qquad\quad+\check{\delta}(1+t)^{-\frac{1}{2}}\left(\norm{g}_{\sigma}+\norm{\vm}_{L^2} \right) + \deltab (1+t)^{-\frac{5}{4}}\bigg],\\
				&\norm{\p_t\Vmc}_{L^2}\leq C\left(\norm{\p_{x_1}\Vmc}_{H^1}^2+\sum_{\abs{\gamma}\leq1}\norm{\p^{\gamma}g}_{\sigma}\right)+C\check{\delta}\left[(1+t)^{-1}\norm{\Vmc}_{L^2}+\norm{\nabla_x \vm}_{L^2}+\norm{\p_t\vm^{\ast}}_{L^2}^2\right]+ C\deltab (1+t)^{-\frac{5}{4}}. 
			\end{align*}
		\end{Lem}
		\begin{proof}
			The desired estimates above can be directly obtained from equations \eqref{perturbation-1} and \eqref{T-A-D-1} and using the {\it a priori} assumptions \eqref{apa}. For brevity of presentation, details of the proof are omitted.
		\end{proof}

		Finally, we estimate $\norm{\p_{x_1}^{k+1}\Phi}_{L^2}^2$.
		\begin{Lem}\label{H-D-Phi-lem}
			Under the same assumptions of  Proposition \ref{Thm-ape}, for $k=0,1$, it holds that
			\begin{align}\label{H-D-Phi-lem.pp}
				&\frac{d}{dt}\left( \norm{\p_{x_1}^{k+1}\Phic}_{L^2}^2+\int_{\R} \p_{x_1}^{k+1}\Phic\p_{x_1}^k\Psic_1 dx_1   \right) + c\norm{\p_{x_1}^{k+1}\Phic}_{L^2}^2\notag \\
				\le& C\norm{\p_{x_1}^{k+1}(\Psic,\Hc)}_{L^2}^2+ C\deltab (1+t)^{-\frac{3+2k}{2}}+C\sum_{k\le\abs{\alpha}\le k+1}\norm{\p^{\alpha}g}_{\sigma}^2+C\deltab\abs{k}(1+t)^{-1}(\norm{\nabla_x\vm}_{L^2}^2+\norm{\p_t\vm^{\ast}}_{L^2}^2)\notag\\
				&+C\check{\delta}\left[ \sum_{i=0}^k (1+t)^{i-k-1}\norm{\p_{x_1}^i \Vmc}_{L^2}^2+(1+t)^{-k-1}\norm{g}_{\sigma}^2+\norm{\nabla_x^{k+1}\vm}_{L^2}^2+\norm{\p_t\nabla_x^k\vm^{\ast}}_{L^2}^2+\norm{\p_{x_1}^{k+2}\Vmc}_{L^2}^2    \right].
			\end{align}
		\end{Lem}
		\begin{proof}
			Multiplying $\p_{x_1}^k$\eqref{T-A-D-1}$_2$ by $\p_{x_1}^{k+1}\Phic$, one has
			\begin{align}\label{H-D-Phi-1}
				&\frac{d}{dt}\int_{\R} \p_{x_1}^k\Psic_1\p_{x_1}^{k+1}\Phic dx_1 + \frac{2}{3}\norm{\p_{x_1}^{k+1}\Phic}_{L^2}^2-\frac{4}{3}\int_{\R}\frac{\mu(\thetat)}{\rhot} \p_{x_1}^{k+2}\Psic_1\p_{x_1}^{k+1}\Phic dx_1 + \int_{\R} \p_t \p_{x_1}^k \Phic \p_{x_1}^{k+1}\Psic_1 dx_1\notag\\
				=&\frac{4}{3}\sum_{i=0}^{k-1}\int_{\R}\p_{x_1}^{k-i}\left( \frac{\mu(\thetat)}{\rhot} \right)\p_{x_1}^{i+2}\Psic_1 \p_{x_1}^{k+1}\Phic dx_1 -\frac{2}{3}\int_{\R}\p_{x_1}^{k+1}\Hc \p_{x_1}^{k+1}\Phic dx_1 \notag\\
				&-\int_{\R}\p_{x_1}^k (\mathcal{S}^{m}_{\Psi})^1\p_{x_1}^{k+1}\Phic dx_1 +\int_{\R} \p_{x_1}^k \mathcal{S}_{\check{\Psi}1}^{f} \p_{x_1}^{k+1}\Phic dx_1.
			\end{align}
			Using  \eqref{T-A-D-1}$_1$, we then obtain
			\begin{align}\label{H-D-Phi-2}
				&-\frac{4}{3}\int_{\R}\frac{\mu(\thetat)}{\rhot} \p_{x_1}^{k+2}\Psic_1\p_{x_1}^{k+1}\Phic dx_1 \notag\\
				=&\frac{4}{3}\int_{\R}\frac{\mu(\thetat)}{\rhot} \p_{x_1}^{k+1} \left( \frac{\p_t\Phic}{\thetat} \right) \p_{x_1}^{k+1}\Phic dx_1 -\frac{4}{3}\int_{\R}\frac{\mu(\thetat)}{\rhot} \p_{x_1}^{k+1} \left( \frac{\mathcal{S}_{\check{\Phi}}^{f}}{\thetat} \right) \p_{x_1}^{k+1}\Phic dx_1 \notag\\
				=&\frac{2}{3}\frac{d}{dt}\left(\int_{\R} \frac{\mu(\thetat)}{\thetat\rhot}  \abs{\p_{x_1}^{k+1} \Phic} ^2 dx_1\right)-\frac{2}{3}\int_{\R} \frac{d}{dt}\left(\frac{\mu(\thetat)}{\thetat\rhot}  \right)\abs{\p_{x_1}^{k+1} \Phic} ^2 dx_1 -\frac{4}{3}\int_{\R}\frac{\mu(\thetat)}{\rhot} \p_{x_1}^{k+1} \left( \frac{\mathcal{S}_{\check{\Phi}}^{f}}{\thetat} \right) \p_{x_1}^{k+1}\Phic dx_1 \notag\\
				&+ \frac{4}{3}\sum_{i=0}^k \int_{\R}\frac{\mu(\thetat)}{\rhot} \p_{x_1}^{k+1-i}\left( \frac{1}{\thetat} \right) \left[\p_{x_1}^i \mathcal{S}_{\check{\Phi}}^{f} -\p_{x_1}^i \left( \thetat \p_{x_1}\Psic_{1}\right) \right]\p_{x_1}^{k+1} \Phic dx_1.
			\end{align}
			And using \eqref{T-A-D-1}$_1$ again, one has
			\begin{align}\label{H-D-Phi-3}
				\int_{\R}\p_t \p_{x_1}^k \Phic \p_{x_1}^{k+1}\Psic_1 dx_1 = -\int_{\R} \p_{x_1}^k\left(\thetat\p_{x_1}\Psic_{1} \right) \p_{x_1}^{k+1}\Psic_1 dx_1 + \int_{\R}\p_{x_1}^k\mathcal{S}_{\check{\Phi}}^{f} \p_{x_1}^{k+1}\Psic_1 dx_1.
			\end{align}
			Combining \eqref{H-D-Phi-1}, \eqref{H-D-Phi-2} and \eqref{H-D-Phi-3}, we obtain
			\begin{align}\label{H-D-Phi-4}
				&\frac{2}{3}\frac{d}{dt}\norm{ \sqrt{\frac{\mu(\thetat)}{\thetat\rhot}}  \p_{x_1}^{k+1} \Phic}_{L^2}^2 + \frac{d}{dt}\int_{\R} \p_{x_1}^k\Psic_1\p_{x_1}^{k+1}\Phic dx_1 + \frac{2}{3}\norm{\p_{x_1}^{k+1}\Phic}_{L^2}^2=J^{(k)}_{d1}+J^{(k)}_{d2}+J^{(k)}_{d3},
			\end{align}
			where
			\begin{align*}
				J_{d1}^{(0)}=& -\int_{\R} (\mathcal{S}^{m}_{\Psi})^1\p_{x_1}\Phic dx_1,        \qquad \qquad  \qquad       J_{d1}^{(1)}= -\int_{\R} \p_{x_1} (\mathcal{S}^{m}_{\Psi})^1\p_{x_1}^2\Phic dx_1           ,\\
				J_{d2}^{(0)}=& \int_{\R}  \mathcal{S}_{\check{\Psi}1}^{f}  \p_{x_1}\Phic dx_1 +\frac{4}{3}\int_{\R}\frac{\mu(\thetat)}{\rhot} \p_{x_1} \left( \frac{\mathcal{S}_{\check{\Phi}}^{f}}{\thetat} \right) \p_{x_1}\Phic dx_1 - \frac{4}{3} \int_{\R}\frac{\mu(\thetat)}{\rhot} \p_{x_1}\left( \frac{1}{\thetat} \right)\mathcal{S}_{\check{\Phi}}^{f}\p_{x_1} \Phic dx_1 - \int_{\R}\mathcal{S}_{\check{\Phi}}^{f} \p_{x_1}\Psic_1  dx_1                              ,\\
				J_{d2}^{(1)}=& \int_{\R} \p_{x_1}\mathcal{S}_{\check{\Psi}1}^{f}  \p_{x_1}^{2}\Phic dx_1  - \int_{\R}\p_{x_1}\mathcal{S}_{\check{\Phi}}^{f} \p_{x_1}^{2}\Psic_1  dx_1                              \\
				&+\frac{4}{3}\int_{\R}\frac{\mu(\thetat)}{\rhot} \p_{x_1}^{2} \left( \frac{\mathcal{S}_{\check{\Phi}}^{f}}{\thetat} \right) \p_{x_1}^{2}\Phic dx_1 - \frac{4}{3}\sum_{i=0}^1 \int_{\R}\frac{\mu(\thetat)}{\rhot} \p_{x_1}^{2-i}\left( \frac{1}{\thetat} \right)\p_{x_1}^i \mathcal{S}_{\check{\Phi}}^{f}\p_{x_1}^{2} \Phic dx_1,\\
				J_{d3}^{(0)}=& -\frac{2}{3}\int_{\R}\p_{x_1}\Hc \p_{x_1}\Phic dx_1  +\frac{2}{3} \int_{\R}\p_t \left( \frac{\mu(\thetat)}{\thetat \rhot} \right) \abs{\p_{x_1}\Phic}^2 dx_1   +\int_{\R} \thetat\abs{\p_{x_1}\Psic_{1}}^2  dx_1 \\
				&+ \frac{4}{3}\int_{\R}\frac{\mu(\thetat)}{\rhot} \p_{x_1}\left( \frac{1}{\thetat} \right)   \thetat\p_{x_1} \Psic_{1}  \p_{x_1} \Phic dx_1,\\
				J_{d3}^{(1)}=&\frac{4}{3}\int_{\R}\p_{x_1}\left( \frac{\mu(\thetat)}{\rhot} \right)\p_{x_1}^{2}\Psic_1 \p_{x_1}^{2}\Phic dx_1 -\frac{2}{3}\int_{\R}\p_{x_1}^{2}\Hc \p_{x_1}^{2}\Phic dx_1  +\frac{2}{3} \int_{\R}\p_t \left( \frac{\mu(\thetat)}{\thetat \rhot} \right) \abs{\p_{x_1}^{2}\Phic}^2 dx_1   \\   
				&+\int_{\R} \p_{x_1}\left(\thetat\p_{x_1}\Psic_{1} \right) \p_{x_1}^{2}\Psic_1  dx_1 + \frac{4}{3}\sum_{i=0}^1 \int_{\R}\frac{\mu(\thetat)}{\rhot} \p_{x_1}^{2-i}\left( \frac{1}{\thetat} \right) \p_{x_1}^i \left( \thetat \p_{x_1}\Psic_{1}\right)  \p_{x_1}^{2} \Phic dx_1.
			\end{align*}
			By \eqref{M-1}, \eqref{M-2t}, \eqref{M-3t}, \eqref{lem-2-1-n-3}, \eqref{lem-2-1-n-4} and the {\it a priori} assumptions \eqref{apa}, we have
			\begin{align}
				J_{d1}^{(0)}+J_{d3}^{(0)} \leq& \frac{1}{1600}\norm{\p_{x_1}\Phic}_{L^2}^2+ C \deltab(1+t)^{-\frac32}+C\norm{\p_{x_1}\left( \Psic,\Hc \right)}_{L^2}^2 + C \sum_{\abs{\alpha}\le 1}\norm{\p^{\alpha}g}_{\sigma}^2\notag\\
				&+C\deltab\norm{\p_t\vm^{\ast}}_{L^2}^2+C\deltab\norm{\nabla_x\vm}_{L^2}^2,\notag
				\\
				J_{d1}^{(1)}+J_{d3}^{(1)} \leq&
				\frac{1}{1600}\norm{\p_{x_1}^2\Phic}_{L^2}^2 + C \deltab(1+t)^{-\frac52}+C\norm{\p_{x_1}^2\left( \Psic,\Hc \right)}_{L^2}^2+C\deltab(\norm{\p_t\nabla_x\vm^{\ast}}_{L^2}^2+\norm{\nabla_x^2\vm}_{L^2}^2)\notag
				\\
				&+C\check{\delta}(1+t)^{-1}\left[  \norm{\p_{x_1}\Vmc}_{L^2}^2 +\norm{\nabla_x \vm}_{L^2}^2+\norm{\p_t\vm^{\ast}}_{L^2}^2+(1+t)^{-1} \norm{g}_{\sigma}^2\right]+C\sum_{1\le \abs{\alpha}\le 2}\norm{\p^{\alpha}g}_{\sigma}^2.\notag
			\end{align}
			Employing \eqref{useful-notation}, \eqref{lem-2-1-n-2} and the {\it a priori} assumptions \eqref{apa}, one has
			\begin{align}
				J_{d2}^{(0)}& \leq C \int_{\R} \abs{\p_{x_1}\Vmc}\left( \deltab\Dt_{-\frac32}+\deltab^{\frac12}\Dt_{-1}\abs{\vm}+\D^{(0)}+\D^{(1)}+\T^{(0)}+\T^{(1)}+\Z^{(0)}+\Z^{(1)}\right) dx_1 \notag\\
				& \le C\deltab(1+t)^{-\frac52}+C\check{\delta}\left[\norm{\p_{x_1}\Vmc}_{H^1}^2 + \norm{\nabla_x \vm}_{L^2}^2+(1+t)^{-2}\norm{\Vmc}_{L^2}^2 \right], \notag
				\\
				J_{d2}^{(1)}& \leq C \int_{\R} \abs{\p_{x_1}^2\Vmc}\left( \deltab\Dt_{-2}+\D^{(1)}+\D^{(2)}+\T^{(1)}+\T^{(2)}+\Z^{(1)}+\Z^{(2)}\right) dx_1 \notag\\
				& \le C\deltab(1+t)^{-\frac72}+C\check{\delta}\left[\norm{\p_{x_1}^2\Vmc}_{H^1}^2 + \norm{\nabla_x^2 \vm}_{L^2}^2+ \sum_{i=0}^1(1+t)^{-2-i}\norm{\p_{x_1}^{1-i}\Vmc}_{L^2}^2 \right]. \notag
			\end{align}
			Plugging all the above estimates to \eqref{H-D-Phi-4} gives \eqref{H-D-Phi-lem.pp}. Then we have completed the proof of Lemma \ref{H-D-Phi-lem}.
		\end{proof}
		
		Combining Lemma \ref{s-t-s-tt-1}, Lemma \ref{s-t-s-tt-2}, and Lemma \ref{H-D-Phi-lem}, we have completed the proof of Theorem \ref{H2-enenrgy-anti}.
		
		\subsection{Estimates of non-zero modes for macroscopic equations}\label{sec4.2}
		We first notice 
		$$
		\int_{\mathbb T^2}(\phi_{\neq},\psi_{\neq},\zeta_{\neq})dx_2dx_3=0.
		$$
		Thus, it can be seen that the Poincar\'e's inequality holds for the non-zero modes:
		\begin{align}\label{2025.6.08-1}
			\|(\phi_{\neq},\psi_{\neq},\zeta_{\neq})\|_{L^2}
			\leq C \| (\nabla^k\phi_{\neq},\nabla^k\psi_{\neq},\nabla^k\zeta_{\neq})\|_{L^2}.
		\end{align}
		Taking $\Dn$ \eqref{2025-11-10-4} for the perturbation system \eqref{perturbation-1}, one has
		\begin{align}\label{eqs201}
			\begin{cases}
				\p_t\phi_{\neq}+\um \cdot \nabla_x \phi_{\neq}+\rhom \operatorname{div}_x \psi_{\neq}=\mathcal {S}_{\phi_{\neq}}, \\
				\p_t\psi_{\neq}+ \um \cdot \nabla_x \psi_{\neq}+\frac{2}{3\rhot}\left(\thetam \nabla_x \phi_{\neq} + \rhom \nabla_x \zeta_{\neq}  \right)=
				\frac{\mu(\tilde\theta)}{\rhot}\left(\Delta_x\psi_{\neq}+\frac{1}{3}\nabla_x\dv_x\psi_{\neq}\right)
				+\mathcal {S}_{\psi_{\neq}},\\
				\p_t \zeta_{\neq}+ \um \cdot \nabla_x \zeta_{\neq}+\frac{2}{3} \thetam \operatorname{div} \psi_{\neq}=\frac{\kappa(\tilde\theta)}{\rhot}\Delta_x\zeta_{\neq}
				+\mathcal {S}_{\zeta_{\neq}},
			\end{cases}
		\end{align}
		where 
		\begin{align*}
			&\mathcal {S}_{\phi_{\neq}}= \um \cdot\nabla_x \phi_{\neq} - \Dn\left( u \cdot \nabla_x\phi\right) + \rhom \dv_x\psi_{\neq}- \Dn \left(\rho\dv_x \psi \right) -\psi_{\neq} \cdot \nabla_x \rhot - \phi_{\neq} \dv_x \ut,\\
			&\mathcal {S}_{\psi_{\neq}}= \um \cdot \nabla_x \psi_{\neq}-\Dn \left( u \cdot \nabla_x\psi \right) + \frac{2}{3\rhot}\left(\thetam \nabla_x \phi_{\neq} + \rhom \nabla_x \zeta_{\neq}  \right) \\
			&\qquad\quad- \frac{\nabla_x p_{\neq}}{\rhot} - \psi_{\neq} \cdot \nabla \ut + \frac{1}{\rhot}\Dn\left(\frac{\phi \nabla_x p}{\rho} \right) + (\mathcal {S}_{\psi}^{f2})_{\neq}+(\mathcal {S}_{\psi}^{m})_{\neq},\\
			&\mathcal {S}_{\zeta_{\neq}}=  \um \cdot \nabla_x \zeta_{\neq}-\Dn \left( u \cdot \nabla_x\zeta \right) + \frac{2}{3} \thetam \dv_x\psi_{\neq} - \frac{2}{3} \Dn \left( \theta \dv_x \psi \right) - \psi_{\neq} \cdot \nabla_x\thetat - \frac{2}{3}\zeta_{\neq}\dv_x \ut + (\mathcal {S}_{\zeta}^{f2})_{\neq}+(\mathcal {S}_{\zeta}^{m})_{\neq}.
		\end{align*}
		Recall the definition of $(\vm,{\vm}^{\ast})$ \eqref{2025-11-5-1}, instant energy functionals $\mathcal{E}_i$ \eqref{2.10}-\eqref{2.10-5} and dissipation energy functionals $\mathcal{D}_i$ \eqref{2.11}-\eqref{2.11-5}. Then, we present the $H^1$ energy estimate for the non-zero mode for macroscopic part $(\phi_{\neq},\psi_{\neq},\zeta_{\neq})$.
		\begin{Thm}\label{non-zero-111-111}
			Under the same assumptions of  Proposition \ref{Thm-ape}, it holds that
			\begin{align*}
				\frac{d}{dt}\|\vm_{\neq}\|_{H^1}^2+c\|(\nabla_x \vm_{\neq},\nabla_x^2{\vm}^{\ast}_{\neq})\|_{L^2}^2 \leq C \check{\delta}\left[(1+t)^{-2}\mathcal{D}_2+\mathcal{D}_3 +(1+t)^{-1}\norm{\p_t\nabla_x\vm^{\ast}}_{L^2}^2 \right]+C_\eta\sum_{\abs{\alpha}=2}\norm{\p^{\alpha}g}_{\sigma}^2.
			\end{align*}
		\end{Thm}
		
		The proof of  Theorem \ref{non-zero-111-111} will be decomposed into two parts including  Lemma \ref{lem99} and  Lemma \ref{lem9} below.
		
		\begin{Lem}\label{lem99} 
			Under the same assumptions of   Proposition \ref{Thm-ape}, it holds that
			\begin{align}\notag
				&\frac{d}{dt}\|\vm_{\neq}\|_{H^1}^2+c\|(\nabla_x \vm_{\neq},\nabla_x^2{\vm}^{\ast}_{\neq})\|_{L^2_x}^2 \leq  C_\eta (\norm{\nabla_x \mathcal{S}_{\phi_{\neq}}}_{L^2}^2+\norm{\mathcal{S}_{\psi_{\neq}}}_{L^2}^2+\norm{\mathcal{S}_{\zeta_{\neq}}}_{L^2}^2).
			\end{align}
		\end{Lem}
		\begin{proof}
			
			{\bf Step 1.} Multiplying \eqref{eqs201}$_1$ by $\frac{2}{3}\frac{\thetam}{\rhom \rhot}\phi_{\neq}$, \eqref{eqs201}$_2$ by $\psi_{\neq}$, \eqref{eqs201}$_3$ by $\frac{\rhom}{\thetam \rhot}\zeta_{\neq}$, and then summing all of them, one has
			\begin{align}\label{eqs210}
				\p_t\left(\frac{\thetam}{3\rhom \rhot}\phi_{\neq}^2+\frac{|\psi_{\neq}|^2}{2}+\frac{\rhom}{2\thetam \rhot}\zeta_{\neq}^2\right)+\frac{\mu(\tilde\theta)}{\rhot}|\nabla_x \psi_{\neq}|^2+\frac{\mu(\tilde\theta)}{3\rhot}|\dv_x\psi_{\neq}|^2+\frac{\rhom \kappa(\tilde\theta)}{\thetam \rhot^2}|\nabla_x \zeta_{\neq}|^2=\mathcal J_0+\dv_x(\cdots),
			\end{align}
			where
			\begin{align*}
				\mathcal J_0:=&\left[\left(\frac{\thetam}{\rhom \rhot}\right)_t+\nabla_x\cdot\left(\frac{\thetam\um}{\rhom \rhot}\right)\right]\frac{\phi_{\neq}^2}{3}+\left[\left(\frac{\rhom}{\thetam \rhot}\right)_t+\nabla_x\cdot\left(\frac{\rhom\um}{\thetam \rhot}\right)\right]\frac{\zeta_{\neq}^2}{2}+\frac{\p_{x_1} \um_1}{2} \abs{\psi_{\neq}}^2  \notag\\
				&-\p_{x_1}\left(\frac{\mu(\tilde\theta)}{\rhot} \right)\left(\sum_{j=1}^3\p_{x_1}\psi_{j\neq}\psi_{j\neq}+\frac{1}{3}\dv_x\psi_{\neq}\psi_{1\neq}\right)+\p_{x_1}^2 \left(\frac{\kappa(\tilde\theta)\rhom}{2\rhot^2\thetam}\right)\abs{\zeta_{\neq}}^2 \notag \\
				&+\frac{2\thetam}{3\rhom \rhot}\phi_{\neq} \mathcal {S}_{\phi_{\neq}}+ \frac{\rhom}{\thetam \rhot} \zeta_{\neq} \mathcal {S}_{\zeta_{\neq}} + \psi_{\neq} \cdot \mathcal {S}_{\psi_{\neq}}+\frac{2\psi_{1\neq}}{3}\left[ \p_{x_1}\left(\frac{\thetam}{\rhot}\right)\phi_{\neq}+\p_{x_1}\left(\frac{\rhom}{\rhot}\right)\zeta_{\neq}    \right].
			\end{align*}
			According to the {\it a priori} assumptions \eqref{apa}, Poincar\'e's inequality \eqref{2025.6.08-1} and  H{\"o}lder's inequality, one has
			\begin{align}\label{eqs212}
				\int_{\R\times \Torus^2} \mathcal{J}_0 dx \leq \eta \norm{(\nabla_x \phi_{\neq},\nabla_x \psi_{\neq} ,\nabla_x \zeta_{\neq})}_{L^2}^2 + C_\eta (\norm{\nabla_x \mathcal{S}_{\phi_{\neq}}}_{L^2}^2+\norm{\mathcal{S}_{\psi_{\neq}}}_{L^2}^2+\norm{\mathcal{S}_{\zeta_{\neq}}}_{L^2}^2).
			\end{align}
			Integrating \eqref{eqs210} on $\Omega$ and combining \eqref{eqs212}, it yields
			\begin{align}\label{eqs213}
				\frac{d}{dt} \norm{\vm_{\neq}}_{L^2}^2+\norm{\nabla_x {\vm}^{\ast}_{\neq}}_{L^2}^2 \leq  \eta\norm{\nabla_x\phi_{\neq}}^2_{L^2}+C_\eta(\norm{\nabla_x \mathcal{S}_{\phi_{\neq}}}_{L^2}^2+\norm{\mathcal{S}_{\psi_{\neq}}}_{L^2}^2+\norm{\mathcal{S}_{\zeta_{\neq}}}_{L^2}^2) .
			\end{align}
			
			{\bf Step 2.} We still need to estimate $\|\nabla_x\phi_{\neq}\|_{L^2}^2$. Taking \eqref{eqs201}$_2\times\frac{\rhot}{\mu(\thetat)}\nabla_x\phi_{\neq}+\nabla_x{\eqref{eqs201}_1}\times\frac{4}{3\rhom}\nabla_x\phi_{\neq},$ one has
			\begin{align}\label{eqs220}
				&\p_t\left(\frac{2}{3\rhom}|\nabla_x\phi_{\neq}|^2+\frac{\rhot}{\mu(\thetat)}\psi_{\neq}\cdot\nabla_x\phi_{\neq}\right)+\frac{2\thetam}{3\mu(\thetat)}|\nabla_x\phi_{\neq}|^2=\mathcal J_1+\dv_x(\cdots),
			\end{align}
			where 
			\begin{align*}
				\mathcal J_1=& \p_t \left( \frac{\rhot}{\mu(\thetat)} \right) \psi_\neq \cdot \nabla_x\phi_{\neq} + \p_{x_1} \left( \frac{\rhot}{\mu(\thetat)}  \right)\psi_{1\neq}\um\cdot\nabla_x \phi_{\neq}+ \p_{x_1} \left( \frac{\rhot}{\mu(\thetat)}  \right)\psi_{1\neq} \rhom \dv_x \psi_{\neq} \notag\\
				&+\frac{\rhot}{\mu(\thetat)} \dv_x \psi_{\neq } \um \cdot\nabla_x\phi_{\neq} +\frac{\rhot \rhom}{\mu(\thetat)} \abs{\dv_x\psi_{\neq}}^2-\frac{2}{3\rhom} \p_{x_1} \phi_{\neq} \p_{x_1} \um \cdot \nabla_x \phi_{\neq}-\frac{4 \p_{x_1} \rhom}{3 \rhom} \p_{x_1} \phi_{\neq}\dv_x \psi_{\neq} \notag\\
				&-\frac{\rhot}{\mu(\thetat)} \um \cdot \nabla_x \psi_{\neq} \cdot\nabla_x\phi_{\neq} - \frac{2\rhom}{3\mu(\thetat)}\nabla_x\zeta_{\neq} \cdot \nabla_x\phi_{\neq} + \frac{\rhot}{\mu(\thetat)} \mathcal {S}_{\psi_{\neq}} \cdot \nabla_x \phi_{\neq} + \frac{4}{3\rhom} \nabla_x \phi_{\neq} \cdot \nabla_x \mathcal {S}_{\phi_{\neq}} \notag\\
				&-\frac{\rhot}{\mu(\thetat)} \dv_x \psi_{\neq} \mathcal {S}_{\phi_{\neq}}-\p_{x_1} \left( \frac{\rhot}{\mu(\thetat)}  \right) \psi_{1\neq} \mathcal {S}_{\phi_{\neq}}.
			\end{align*}
			In order to obtain \eqref{eqs220}, we need to use \eqref{eqs201}$_1$ to calculate $ \frac{\rhot}{\mu(\thetat)} \p_t \psi_{\neq} \nabla_x \phi_{\neq}$ by
			\begin{align*}
				&\frac{\rhot}{\mu(\thetat)} \p_t \psi_{\neq} \cdot \nabla_x \phi_{\neq}=\p_t\left(\frac{\rhot}{\mu(\thetat)}  \psi_{\neq} \cdot \nabla_x \phi_{\neq} \right) - \p_t\left(\frac{\rhot}{\mu(\thetat)}\right) \psi_{\neq} \cdot \nabla_x \phi_{\neq} - \frac{\rhot}{\mu(\thetat)}  \psi_{\neq} \cdot \nabla_x \p_t \phi_{\neq} \\
				=&\p_t\left(\frac{\rhot}{\mu(\thetat)}  \psi_{\neq} \cdot \nabla_x \phi_{\neq} \right) - \p_t\left(\frac{\rhot}{\mu(\thetat)}\right) \psi_{\neq} \cdot \nabla_x \phi_{\neq} +  \p_{x_1}\left(\frac{\rhot}{\mu(\thetat)}\right) \psi_{1\neq}  \p_t \phi_{\neq}+ \frac{\rhot}{\mu(\thetat)} \dv_x\psi_{\neq}  \p_t \phi_{\neq} + \dv_x(\cdots) \\
				=&\p_t\left(\frac{\rhot}{\mu(\thetat)}  \psi_{\neq} \cdot \nabla_x \phi_{\neq} \right) - \p_t\left(\frac{\rhot}{\mu(\thetat)}\right) \psi_{\neq} \cdot \nabla_x \phi_{\neq} - \p_{x_1}\left(\frac{\rhot}{\mu(\thetat)}\right) \psi_{1\neq} \um\cdot \nabla_x \phi_{\neq} - \p_{x_1}\left(\frac{\rhot}{\mu(\thetat)}\right) \psi_{1\neq} \rhom \dv_x \phi_{\neq}\\
				&- \frac{\rhot}{\mu(\thetat)} \dv_x\psi_{\neq} \um\cdot \nabla_x \phi_{\neq} -\frac{\rhot \rhom}{\mu(\thetat)}\abs{\dv_x\psi_{\neq}}^2 + \p_{x_1}\left(\frac{\rhot}{\mu(\thetat)}\right) \psi_{1\neq}\mathcal {S}_{\phi_{\neq}} + \frac{\rhot}{\mu(\thetat)} \dv_x\psi_{\neq}\mathcal {S}_{\phi_{\neq}}+ \dv_x(\cdots).
			\end{align*}
			According to the {\it a priori} assumptions \eqref{apa}, Poincar\'e's inequality \eqref{2025.6.08-1} and applying H{\"o}lder's inequality, one has
			\begin{align}\label{eqs222}
				\int_{\Omega}\mathcal J_1dx\le \eta\|\nabla_x\phi_{\neq}\|_{L^2}^2 + C_\eta \left(\norm{\nabla_x \psi_{\neq}}_{L^2}^2+\norm{ \nabla_x\zeta_{\neq}}_{L^2}^2 + \norm{\nabla_x \mathcal {S}_{\phi_{\neq}}}_{L^2}^2 + \norm{\mathcal {S}_{\psi_{\neq}}}_{L^2}^2\right).
			\end{align}  
			Integrating \eqref{eqs220} on $\Omega$ and combining \eqref{eqs222}, it yields
			\begin{align}\label{eqs223}
				\frac{d}{dt}\int_{\Omega}&\left(\frac{2}{3\rhom}|\nabla_x\phi_{\neq}|^2+\frac{\rhot}{\mu(\thetat)}\psi_{\neq}\cdot\nabla_x\phi_{\neq}\right)dx+\int_{\Omega}\frac{\thetam}{3\mu(\thetat)}|\nabla_x\phi_{\neq}|^2dx \notag\\
				&\le   C_\eta \left(\norm{\nabla_x \psi_{\neq}}_{L^2}^2+\norm{ \nabla_x\zeta_{\neq}}_{L^2}^2 + \norm{\nabla_x \mathcal {S}_{\phi_{\neq}}}_{L^2}^2 + \norm{\mathcal {S}_{\psi_{\neq}}}_{L^2}^2\right).
			\end{align}
			
			{\bf Step 3.} We estimate $\norm{\nabla_x^2\psi_{\neq}}_{L^2}^2$ and $\norm{\nabla_x^2\zeta_{\neq}}_{L^2}^2$. Taking $\nabla_x\eqref{eqs201}_1\times\frac{2\thetam}{3\rhom \rhot}\nabla_x\phi_{\neq}+\nabla_x\eqref{eqs201}_2\times\nabla_x\psi_{\neq}+\nabla_x\eqref{eqs201}_3\times\frac{\rhom}{\thetam\rhot}\nabla_x\zeta_{\neq},$ one has
			\begin{align}\label{eqs230-1}
				\p_t\Big(\frac{\thetam}{3\rhom \rhot}\abs{\nabla_x\phi_{\neq}}&^2+\frac{|\nabla_x\psi_{\neq}|^2}{2}+\frac{\rhom}{2\thetam \rhot}\abs{\nabla_x\zeta_{\neq}}^2\Big)+\frac{\mu(\tilde\theta)}{\rhot}|\Delta_x\psi_{\neq}|^2\notag\\
				&+\frac{\mu(\tilde\theta)}{3\rhot}|\nabla_x\dv_x\psi_{\neq}|^2
				+\frac{\rhom \kappa(\tilde\theta)}{\thetam \rhot^2}|\Delta_x\zeta_{\neq}|^2=\mathcal J_2+\dv_x(\cdots),
			\end{align}
			where 
			\begin{align*}
				\mathcal{J}_{2}=&\left[\left(\frac{\thetam}{\rhom \rhot}\right)_t+\nabla_x\cdot\left(\frac{\thetam\um}{\rhom \rhot}\right)\right]\frac{\abs{\nabla_x\phi_{\neq}}^2}{3}+\left[\left(\frac{\rhom}{\thetam \rhot}\right)_t+\nabla_x\cdot\left(\frac{\rhom\um}{\thetam \rhot}\right)\right]\frac{\abs{\nabla_x\zeta_{\neq}}^2}{2}+\frac{\p_{x_1} \um_1}{2} \abs{\nabla_x\psi_{\neq}}^2 \notag\\
				&-\frac{2\thetam\p_{x_1}\rhom}{3\rhom \rhot} \dv_x \psi_{\neq} \p_{x_1} \phi_{\neq}-\frac{2\rhom \p_{x_1} \thetam}{3\thetam \rhot} \dv_x \psi_{\neq} \p_{x_1} \zeta_{\neq} - \frac{2\thetam}{3\rhom \rhot} \p_{x_1} \um \cdot \nabla_x \phi_{\neq} \p_{x_1} \phi_{\neq} - \p_{x_1} \um \cdot \nabla_x \psi_{\neq} \cdot\p_{x_1} \psi_{\neq}  \notag\\
				&-\frac{\rhom}{\thetam \rhot} \p_{x_1} \um \cdot \nabla_x \zeta_{\neq} \p_{x_1} \zeta_{\neq} - \frac23 \p_{x_1}\left( \frac{\thetam}{\rhot} \right)\p_{x_1}\psi_{\neq} \cdot \nabla_x \phi_{\neq} + \frac23 \p_{x_1}\left( \frac{\thetam}{\rhot} \right) \nabla_x \psi_{1\neq} \cdot \nabla_x \phi_{\neq}+\frac{2\thetam}{3\rhom \rhot}\nabla_x\phi_{\neq}\cdot \nabla_x \mathcal {S}_{\phi_{\neq}} \notag\\
				&-\frac23 \p_{x_1}\left( \frac{\rhom}{\rhot} \right) \p_{x_1} \psi_{\neq} \cdot \nabla_x \zeta_{\neq} + \frac23 \p_{x_1}\left( \frac{\rhom}{\rhot} \right) \nabla_x \psi_{1\neq} \cdot \nabla_x \zeta_{\neq} - \frac{1}{3}\p_{x_1}\left(\frac{\mu(\tilde\theta)}{\rhot} \right) \dv_x \psi_{\neq} \Delta \psi_{1\neq} + \nabla_x\psi_{\neq} \cdot \nabla_x\mathcal {S}_{\psi_{\neq}} \notag\\
				&- \frac{1}{3}\p_{x_1}\left(\frac{\mu(\tilde\theta)}{\rhot} \right) \dv_x \psi_{\neq} \p_{x_1} \dv_x \psi_{\neq} - \frac{\kappa(\thetat)}{\rhot} \p_{x_1} \left(\frac{\rhom}{\thetam \rhot} \right) \p_{x_1} \zeta_{\neq} \Delta_x \zeta_{\neq} + \frac{\rhom}{\thetam \rhot} \nabla_x\zeta_{\neq} \cdot \nabla_x\mathcal {S}_{\zeta_{\neq}}. 
			\end{align*}
			According to the {\it a priori} assumptions \eqref{apa}, Poincar\'e's inequality \eqref{2025.6.08-1} and applying H{\"o}lder's inequality, one has
			\begin{align}\label{eqs232}
				\int_{\Omega}\mathcal J_2dx\le \eta \left(\|\nabla_x\phi_{\neq}\|_{L^2}^2 + \norm{\nabla_x^2 \psi_{\neq}}_{L^2}^2+\norm{ \nabla_x^2\zeta_{\neq}}_{L^2}^2 \right)+ C_\eta (\norm{\nabla_x \mathcal{S}_{\phi_{\neq}}}_{L^2}^2+\norm{\mathcal{S}_{\psi_{\neq}}}_{L^2}^2+\norm{\mathcal{S}_{\zeta_{\neq}}}_{L^2}^2).
			\end{align}  
			Integrating \eqref{eqs230-1} on $\Omega$ and combining \eqref{eqs232}, it yields
			\begin{align}\label{eqs233}
				\frac{d}{dt}& \norm{\nabla_x \vm_{\neq}}_{L^2}^2 + \norm{\nabla_x^2 {\vm}^{\ast}_{\neq}}_{L^2}^2
				\leq\eta \|\nabla_x\phi_{\neq}\|_{L^2}^2 + C_\eta (\norm{\nabla_x \mathcal{S}_{\phi_{\neq}}}_{L^2}^2+\norm{\mathcal{S}_{\psi_{\neq}}}_{L^2}^2+\norm{\mathcal{S}_{\zeta_{\neq}}}_{L^2}^2).
			\end{align}
			Combining \eqref{eqs213}, \eqref{eqs223} and  \eqref{eqs233}, we have completed the proof of  Lemma \ref{lem99}.
		\end{proof}
		\begin{Lem}\label{lem9}
			Under the same assumptions of   Proposition \ref{Thm-ape}, it holds that
			\begin{align*}
				&\norm{\nabla_x \mathcal{S}_{\phi_{\neq}}}_{L^2}^2\le C \check{\delta}\| (\nabla_x\phi_{\neq},\nabla_x^2\psi_{\neq},\nabla_x^2\zeta_{\neq}
				)\|_{L^2}^2,
				\\
				&\norm{\mathcal{S}_{\psi_{\neq}}}_{L^2}^2+\norm{\mathcal{S}_{\zeta_{\neq}}}_{L^2}^2\leq C \sum_{\abs{\alpha}= 2}
				\|\p^{\alpha} g
				\|_{\sigma}^2+ C \bar{\delta} \left( \|\nabla_x\phi_{\neq}
				\|_{L^2}^2+(1+t)^{-2}\left( \norm{\nabla_x \vm}_{L^2}^2+\norm{g}_{\sigma}^2 \right)+\sum_{\abs{\gamma}=1}\norm{\p^{\gamma}g}_{\sigma}^2  \right)\notag\\
				&\qquad\qquad\qquad\qquad\qquad\;\;+C\deltab(1+t)^{-1}(\norm{\nabla_x^2\vm^{\ast}}_{L^2}^2+\norm{\p_t\nabla_x \vm^{\ast}}_{L^2}^2).
			\end{align*}
		\end{Lem}
		
		\begin{proof}
			We first make an estimation for $\mathcal {S}_{\psi_{\neq}}$ \eqref{eqs201}. With respect to $\mathcal {S}_{\psi_{\neq}}$ \eqref{eqs201}, we first calculate the part excluding $(\mathcal {S}_{\psi}^{f2})_{\neq}$ and $(\mathcal {S}_{\psi}^{m})_{\neq}$. Among the parts other than $(\mathcal {S}_{\psi}^{f2})_{\neq}$ and $(\mathcal {S}_{\psi}^{m})_{\neq}$, the term 
			\begin{equation}
				\label{add.lem9.p1}
				\frac{2}{3\rhot}\left(\thetam \nabla_x \phi_{\neq} + \rhom \nabla_x \zeta_{\neq}  \right) - \frac{\nabla_x p_{\neq}}{\rhot}
			\end{equation}			
			is the most difficult one to calculate. Therefore, for this part, we only provide an estimate of \eqref{add.lem9.p1}, and the calculations for the remaining terms are similar. In fact, it holds
			\begin{align}\label{fluid-part-1}
				\norm{\eqref{add.lem9.p1}}_{L^2} 
				&\leq\norm{ \frac{2}{3\rhot}  \left(   \p_{x_1} \phim \zeta_{\neq} + \p_{x_1} \thetam \phi_{\neq}  \right) }_{L^2}+\sum_{i=1}^3\norm{\frac{2}{3\rhot} \left( \p_{xi}\phi_{\neq} \zeta_{\neq} + \p_{xi}\zeta_{\neq}\phi_{\neq} \right)}_{L^2} \notag\\
				&\leq C\chi \norm{(\nabla_x\phi_{\neq},\nabla_x^2 \psi_{\neq},\nabla_x^2 \zeta_{\neq}
					)}_{L^2},
			\end{align}
			where we have used the {\it a priori} assumptions \eqref{apa}. For the fluid part in $(\mathcal{S}_\psi^{f2})_\neq$ \eqref{Q-m}, we only carry out the calculation for 
			\begin{equation}
				\label{add.lem9.p2}
				\Dn \left[\mu(\thetat) \left( \frac{1}{\rho}-\frac{1}{\rhot}\right) \left(\Delta_x u + \frac13 \nabla_x \dv_x u \right)\right],
			\end{equation}
			and the estimations for the remaining parts can be obtained in the same way. By the {\it a priori} assumptions \eqref{apa}, one has
			\begin{align}\label{fluid-part-2}
				\norm{\eqref{add.lem9.p2}}_{L^2} 
				\leq&\norm{\mu(\thetat) \Dn\left( \frac{1}{\rho}-\frac{1}{\rhot}\right) \left(\Delta_x \um + \frac13 \nabla_x \dv_x \um \right)}_{L^2} + \norm{\mu(\thetat) \Do \left( \frac{\phi}{\rho\rhot}\right) \left(\Delta_x \psi_{\neq} + \frac13 \nabla_x \dv_x \psi_{\neq} \right)}_{L^2} \notag\\
				&+\norm{[\mu(\thetat)\Dn \left[\left( \frac{1}{\rho}-\frac{1}{\rhot}\right)_{\neq} \left(\Delta_x \psi_{\neq} + \frac13 \nabla_x \dv_x \psi_{\neq} \right)\right]}_{L^2} \notag\\
				\leq& C\chi \norm{(\nabla_x\phi_{\neq},\nabla_x^2 \psi_{\neq},\nabla_x^2 \zeta_{\neq})}_{L^2}.
			\end{align}
			For the non-fluid part $(\mathcal {S}_{\psi}^{m})_{\neq}$ \eqref{K-m}, we only deal with 
			\begin{equation}
				\label{add.lem9.p3}
				\Dn \left[ \frac{1}{\rho} \int \xi \otimes \xi \cdot \nabla_{x} \left(L_{M}^{-1}\Pi-L_{\Mb}^{-1}\bar{\Pi}_1\right) d \xi  \right].
			\end{equation}
			The calculation of the other term is much simpler. It holds 
			\begin{align}\label{non-fluid-1}
				\eqref{add.lem9.p3} 
				=& \Do\left(\frac{1}{\rho}\right) \Dn  \int \xi \otimes \xi \cdot \nabla_{x} L_{M}^{-1}\Pi d \xi + \Dn\left(\frac{1}{\rho}\right) \Do \int \xi \otimes \xi \cdot \nabla_{x} \left(L_{M}^{-1}\Pi-L_{\Mb}^{-1}\bar{\Pi}_1\right) d \xi \notag \\
				&+\Dn\left[\left(\frac{1}{\rho}\right)_{\neq} \Dn \int \xi \otimes \xi \cdot \nabla_{x} L_{M}^{-1}\Pi d \xi\right].
			\end{align}
			From \eqref{non-fluid-s-s-1} and the a {\it priori} assumptions \eqref{apa}, one has
			\begin{align}
				&\norm{\Dn\left(\frac{1}{\rho}\right) \Do \int \xi \otimes \xi \cdot \nabla_{x} \left(L_{M}^{-1}\Pi-L_{\Mb}^{-1}\bar{\Pi}_1\right) d \xi}_{L^2}\leq C \check{\delta}  \left(\norm{\nabla_x \phi_{\neq}}_{L^2}+\sum_{1\leq\abs{\alpha}\leq2}\norm{\p^\alpha g}_{\sigma}\right)\label{non-fluid-2},\\
				&\norm{\Dn \int_{\R^3} \xi \otimes \xi \cdot \nabla_x L_M^{-1} \Pi d\xi}_{L^2}^2 
				\leq C \sum_{\abs{\alpha}=2}\norm{\p^{\alpha}g}_{\sigma}^2+C\check{\delta}\left[ (1+t)^{-2}\left( \norm{\nabla_x \vm}_{L^2}^2+\norm{g}_{\sigma}^2 \right)+\sum_{\abs{\gamma}=1}\norm{\p^{\gamma}g}_{\sigma}^2  \right] \notag\\
				&\qquad\qquad\qquad\qquad\qquad\qquad\qquad +C\deltab(1+t)^{-1}(\norm{\nabla_x^2\vm^{\ast}}_{L^2}^2+\norm{\p_t\nabla_x \vm^{\ast}}_{L^2}^2)\label{non-fluid-s-s-2}.
			\end{align}
			Combining \eqref{fluid-part-1}-\eqref{non-fluid-s-s-2}, one has
			\begin{align}\label{F-m-111}
				\norm{\mathcal {S}_{\psi_{\neq}}}_{L^2}^2\leq& C \sum_{\abs{\alpha}= 2}
				\|\p^{\alpha} g
				\|_{\sigma}^2+ C \bar{\delta} \left( \|\nabla_x\phi_{\neq}
				\|_{L^2}^2+(1+t)^{-2}\left( \norm{\nabla_x \vm}_{L^2}^2+\norm{g}_{\sigma}^2 \right)+\sum_{\abs{\gamma}=1}\norm{\p^{\gamma}g}_{\sigma}^2  \right)\notag\\
				&+C\deltab(1+t)^{-1}(\norm{\nabla_x^2\vm^{\ast}}_{L^2}^2+\norm{\p_t\nabla_x \vm^{\ast}}_{L^2}^2).
			\end{align}
			The estimates of $\mathcal {S}_{\phi_{\neq}}$  and $\mathcal {S}_{\zeta_{\neq}}$ \eqref{eqs201} are the same as those of $\mathcal {S}_{\psi_{\neq}}$. Thus, using the same argument as   \eqref{F-m-111}, we obtain
			\begin{align}
				&\norm{\nabla_x \mathcal{S}_{\phi_{\neq}}}_{L^2}^2\le C \check{\delta}\| (\nabla_x\phi_{\neq},\nabla_x^2\psi_{\neq},\nabla_x^2\zeta_{\neq})
				\|_{L^2}^2,\notag
				\\
				&\norm{\mathcal {S}_{\psi_{\neq}}}_{L^2}^2+\|\mathcal {S}_{\zeta_{\neq}}\|_{L^2_x}^2\leq C \sum_{\abs{\alpha}= 2}
				\|\p^{\alpha} g
				\|_{\sigma}^2+ C \bar{\delta} \left( \|\nabla_x\phi_{\neq}
				\|_{L^2}^2+(1+t)^{-2}\left( \norm{\nabla_x \vm}_{L^2}^2+\norm{g}_{\sigma}^2 \right)+\sum_{\abs{\gamma}=1}\norm{\p^{\gamma}g}_{\sigma}^2  \right)\notag\\
				&\qquad\qquad\quad\qquad\qquad\quad+C\deltab(1+t)^{-1}(\norm{\nabla_x^2\vm^{\ast}}_{L^2}^2+\norm{\p_t\nabla_x \vm^{\ast}}_{L^2}^2).\label{fluid-and-non-fluid-2}
			\end{align}
			Combining \eqref{F-m-111} and \eqref{fluid-and-non-fluid-2}, we have completed the proof of  Lemma \ref{lem9}.
		\end{proof}
		Combining Lemma \ref{s-t-s-tt-2},  Lemma \ref{lem99} and  Lemma \ref{lem9}, we have completed the proof of Theorem \ref{non-zero-111-111}.
		
		\subsection{Estimates of higher-order derivatives for the macroscopic equations}\label{sec4.3}
		To close the energy estimate, we need to provide the estimates of higher-order derivatives for the fluid part. Recall the definition of $(\vm,{\vm}^{\ast})$ \eqref{2025-11-5-1} and dissipation energy functionals $\mathcal{D}_i$ \eqref{2.11}-\eqref{2.11-5}. Then we have the following result.
		
		\begin{Thm}\label{high-order-111-111}
			Under the same assumptions of  Proposition \ref{Thm-ape}, it holds that
			\begin{align*}
				& \frac{d}{dt} \int_{\Omega} \nabla_x \psi \cdot\nabla_x^2 \phi dx+ \tilde{c}\norm{  \nabla_x^2 \phi}_{L^2}^2\leq C\bar{\delta}(1+t)^{-\frac{5}{2}}+C\check{\delta}\sum_{i=1}^3(1+t)^{-3+i}\mathcal{D}_i+ C_\eta \left( \norm{\nabla_x^2{\vm}^{\ast}}_{L^2}^2 + \sum_{\abs{\alpha}=2}\norm{\p^{\alpha}g}_{\sigma}^2 \right) ,\\
				&  \frac{d}{dt} \left(\norm{\nabla_x^2\vm}_{L^2}^2+\tilde{c}\int_{\Omega} \nabla_x^2 \psi \cdot\nabla_x^3 \phi dx\right)+ \tilde{c}\norm{  \nabla_x^3 \vm}_{L^2}^2\leq C\bar{\delta}(1+t)^{-\frac{5}{2}}+C\check{\delta}\sum_{i=1}^3(1+t)^{-3+i}\mathcal{D}_i+C_\eta \sum_{\abs{\alpha}=3}\norm{\p^{\alpha}g}_{\sigma}^2.
			\end{align*}
		\end{Thm}
		
		The proof of  Theorem \ref{high-order-111-111} will be decomposed into two parts via Lemma \ref{high-der-ori} and  Lemma \ref{high-der-ori-star} below.
		To derive estimates for the high-order derivatives (involving time derivatives) of macroscopic quantities and the high-order derivatives of density, we rewrite the equation \eqref{landau-1} as
		\begin{align}\label{perturbation-ori}
			\begin{cases}
				\p_t \phi = \mathcal{S}_{0},\\
				\p_t \psi + \frac23 \frac{\thetat}{\rhot} \nabla_x\phi=\mathcal{S},\\
				\p_t \zeta=\mathcal{S}_4,
			\end{cases}
		\end{align}
		where
		\begin{align*}
			&\mathcal{S}_{0}:=-\p_t\rhot+\operatorname{div}_x(\rho u), 
			\\
			&\mathcal{S}:=(\mathcal{S}_1,\mathcal{S}_2,\mathcal{S}_3)^{t}=-	 \p_t\ut- u \cdot \nabla_x u -\frac{\nabla_x p}{\rho}+ \frac23 \frac{\thetat}{\rhot} \nabla_x\phi -\frac{1}{\rho}\int \xi \otimes \xi \cdot \nabla_x G d \xi, 
			\\
			&\mathcal{S}_4:=-\p_t\thetat-u \cdot \nabla_x \theta - \theta \dv_x u	-\frac{1}{2\rho}\int_{\R^3} |\xi|^{2} \xi \cdot \nabla_x G d \xi +\frac{u}{\rho}\cdot\int \xi \otimes \xi \cdot \nabla_x G d \xi. 
		\end{align*}
		
		\begin{Lem}\label{high-der-ori}
			Under the same assumptions of  Proposition \ref{Thm-ape},  it holds that
			\begin{align}
				&\norm{ \p_t \nabla_x \vm}_{L^2}^2 \lesssim\left[\deltab(1+t)^{-\frac52}+ \norm{\nabla_x^2 \vm}_{L^2}^2 + \sum_{\abs{\alpha}=2}\norm{\p^{\alpha}g}_{\sigma}^2+\check{\delta} \left(  \sum_{\abs{\gamma}=1} \norm{\p^{\gamma }g}_{\sigma}^2 +\sum_{i=1}^2 (1+t)^{-i} \norm{\nabla_x^{2-i}\vm}_{L^2}^2 \right)\right],\label{high-der-ori-1} \\
				&\norm{ \p_t^2 \vm}_{L^2}^2 \lesssim\left[\deltab(1+t)^{-\frac52}+ \norm{\p_t\nabla_x \vm}_{L^2}^2 + \sum_{\abs{\alpha}=2}\norm{\p^{\alpha}g}_{\sigma}^2+\check{\delta} \left(  \sum_{\abs{\gamma}=1} \norm{\p^{\gamma }g}_{\sigma}^2 +\sum_{i=1}^2 (1+t)^{-i} \norm{\nabla_x^{2-i}\vm}_{L^2}^2 \right)\right],\label{high-der-ori-1-s} \\
				&\norm{ \p_t \nabla_x^2 \vm}_{L^2}^2 \lesssim\left[\deltab(1+t)^{-\frac52}+ \norm{\nabla_x^3 \vm}_{L^2}^2 + \sum_{\abs{\alpha}=3}\norm{\p^{\alpha}g}_{\sigma}^2  + \check{\delta} \left(  \sum_{1\leq\abs{\gamma}\leq 2} \norm{\p^{\gamma }g}_{\sigma}^2 +\sum_{i=1}^3 (1+t)^{-i} \norm{\nabla_x^{3-i}\vm}_{L^2}^2 \right)\right],\label{high-der-ori-2} \\
				&\norm{ \p_t^2 \nabla_x \vm}_{L^2}^2 \lesssim\left[\deltab(1+t)^{-\frac52}+ \norm{\p_t\nabla_x^2 \vm}_{L^2}^2 + \sum_{\abs{\alpha}=3}\norm{\p^{\alpha}g}_{\sigma}^2  + \check{\delta} \left(  \sum_{1\leq\abs{\gamma}\leq 2} \norm{\p^{\gamma }g}_{\sigma}^2 +\sum_{i=1}^3 (1+t)^{-i} \norm{\nabla_x^{3-i}\vm}_{L^2}^2 \right)\right],\label{high-der-ori-2-s} \\
				&\norm{ \p_t^3 \vm}_{L^2}^2 \lesssim\left[\deltab(1+t)^{-\frac52}+ \norm{\p_t^2\nabla_x \vm}_{L^2}^2 + \sum_{\abs{\alpha}=3}\norm{\p^{\alpha}g}_{\sigma}^2  + \check{\delta} \left(  \sum_{1\leq\abs{\gamma}\leq 2} \norm{\p^{\gamma }g}_{\sigma}^2 +\sum_{i=1}^3 (1+t)^{-i} \norm{\nabla_x^{3-i}\vm}_{L^2}^2 \right)\right],\label{high-der-ori-2-ss} \\
				&\p_t \int_{\Omega} \nabla_x \psi \cdot\nabla_x^2 \phi dx+ \tilde c\norm{  \nabla_x^2 \phi}_{L^2}^2 \notag\\
				&\leq C \left( \norm{\nabla_x^2{\vm}^{\ast}}_{L^2}^2 + \sum_{\abs{\alpha}=2}\norm{\p^{\alpha}g}_{\sigma}^2 \right) +C\deltab(1+t)^{-\frac{5}{2}}+ C\check{\delta} \left(  \sum_{\abs{\gamma}=1} \norm{\p^{\gamma }g}_{\sigma}^2 +\sum_{i=1}^2 (1+t)^{-i} \norm{\nabla_x^{2-i}\vm}_{L^2}^2 \right),\label{high-der-ori-3} \\
				&\p_t \int_{\Omega} \nabla_x^2 \psi \cdot\nabla_x^3 \phi dx+ \tilde c \norm{  \nabla_x^3 \phi}_{L^2}^2\notag  \\
				&\leq C \left( \norm{\nabla_x^3{\vm}^{\ast}}_{L^2}^2 + \sum_{\abs{\alpha}=3}\norm{\p^{\alpha}g}_{\sigma}^2 \right) +C\deltab(1+t)^{-\frac{5}{2}}+ C\check{\delta} \left(  \sum_{1\leq\abs{\gamma}\leq2} \norm{\p^{\gamma }g}_{\sigma}^2 +\sum_{i=1}^3 (1+t)^{-i} \norm{\nabla_x^{3-i}\vm}_{L^2}^2 \right).\label{high-der-ori-4}
			\end{align}
		\end{Lem}
		\begin{proof}
			Recall $(\rho,u,\theta)=(\rhot,\ut,\thetat)+(\phi,\psi,\zeta)$ and $G=\bar{G}_0+\sqrt{\mu}g$, then we have
			\begin{align}\notag
				&\sum_{i=0}^4 \abs{\mathcal{S}_i}\le C\deltab^{\frac12}\left( \Dt_{-1}+\Dt_{-\frac12}\abs{\vm}\right)+C\abs{\vm}\abs{\nabla_x \vm}+C\abs{\nabla_x {\vm}^{\ast}}+C\sum_{\abs{\alpha}=1}\left(1+\abs{\vm}\right)\abs{\p^{\alpha}g}_{\sigma}. 
			\end{align}
			Then employing the a {\it priori} assumptions \eqref{apa}, one has
			\begin{align}
				&\sum_{i=0}^4 \norm{\nabla_x \mathcal{S}_i}_{L^2}^2 \leq C\deltab(1+t)^{-\frac52}+C \left( \norm{\nabla_x^2{\vm}^{\ast}}_{L^2}^2 + \sum_{\abs{\alpha}=2}\norm{\p^{\alpha}g}_{\sigma}^2 \right)\notag \\
				&\qquad\qquad+C\check{\delta} \left(\norm{\nabla_x^2 \phi}_{L^2}^2+  \sum_{\abs{\gamma}=1} \norm{\p^{\gamma }g}_{\sigma}^2 +\sum_{i=1}^2 (1+t)^{-i} \norm{\nabla_x^{2-i}\vm}_{L^2}^2 \right) ,\label{high-der-ori-9}
				\\
				&\sum_{i=0}^4 \norm{\nabla_x^2 \mathcal{S}_i}_{L^2}^2  \leq C\deltab(1+t)^{-\frac52} +C \left( \norm{\nabla_x^3 {\vm}^{\ast}}_{L^2}^2 + \sum_{\abs{\alpha}=3}\norm{\p^{\alpha}g}_{\sigma}^2 \right) \notag\\
				&\qquad\qquad+C\check{\delta} \left(\norm{\nabla_x^3 \phi}_{L^2}^2+  \sum_{1\leq\abs{\gamma}\leq2} \norm{\p^{\gamma }g}_{\sigma}^2 +\sum_{i=1}^3 (1+t)^{-i} \norm{\nabla_x^{3-i}\vm}_{L^2}^2 \right) . 
				\label{high-der-ori-10}
			\end{align}
			By system \eqref{perturbation-ori},  \eqref{high-der-ori-9} and \eqref{high-der-ori-10}, one has
			\begin{align*}
				&\norm{\p_t \nabla_x \vm}_{L^2}^2 \leq C\left( \sum_{i=0}^4\norm{\nabla_x \mathcal{S}_i}_{L^2}^2 + \norm{\nabla_x\left( \frac{\thetat}{\rhot} \nabla_x \phi \right)}_{L^2}^2 \right) \notag\\
				\leq& C\deltab(1+t)^{-\frac52}+C \left( \norm{\nabla_x^2\vm}_{L^2}^2 + \sum_{\abs{\alpha}=2}\norm{\p^{\alpha}g}_{\sigma}^2 \right)+C\check{\delta} \left(  \sum_{\abs{\gamma}=1} \norm{\p^{\gamma }g}_{\sigma}^2 +\sum_{i=1}^2 (1+t)^{-i} \norm{\nabla_x^{2-i}\vm}_{L^2}^2 \right),
				\\
				&\norm{\p_t \nabla_x^2 \vm}_{L^2}^2 \leq C\left( \sum_{i=1}^4\norm{\nabla_x^2 \mathcal{S}_i}_{L^2}^2 + \norm{\nabla_x^2\left( \frac{\thetat}{\rhot} \nabla_x \phi \right)}_{L^2}^2 \right)\notag\\
				\leq& C\deltab(1+t)^{-\frac52} +C \left( \norm{\nabla_x^3 \vm}_{L^2}^2 + \sum_{\abs{\alpha}=3}\norm{\p^{\alpha}g}_{\sigma}^2 \right) +C\check{\delta} \left( \sum_{1\leq\abs{\gamma}\leq2} \norm{\p^{\gamma }g}_{\sigma}^2 +\sum_{i=1}^3 (1+t)^{-i} \norm{\nabla_x^{3-i}\vm}_{L^2}^2 \right) .
			\end{align*}
			Thus, we have proven \eqref{high-der-ori-1} and \eqref{high-der-ori-2}. Using the same method, we can prove \eqref{high-der-ori-1-s}, \eqref{high-der-ori-2-s} and \eqref{high-der-ori-2-ss}.  
			
			Next, we will provide the proofs of \eqref{high-der-ori-3} and \eqref{high-der-ori-4}. Applying $\p_{i}$ to \eqref{perturbation-ori}$_2$ and $\p_{ij}$ to \eqref{perturbation-ori}$_2$, $i,j=1,2,3$, respectively, we can obtain
			\begin{align}
				&\p_t \p_{i} \psi_l+\frac23 \frac{\thetat}{\rhot} \p_{il} \phi=\p_{x_i}\mathcal{S}_{l}- \frac23\p_{i}\left( \frac{\thetat}{\rhot} \right)  \p_{j} \phi, \label{high-der-ori-13} \\
				&\p_t \p_{ij} \psi_l+\frac23 \frac{\thetat}{\rhot} \p_{ijl}\phi=\p_{ij}\mathcal{S}_{l}- \frac23\p_{ij}\left( \frac{\thetat}{\rhot}   \p_{j} \phi\right)+ \frac23 \frac{\thetat}{\rhot} \p_{ijl}\phi \label{high-der-ori-14},
			\end{align}
			where $l=1,2,3$.
			Multiplying \eqref{high-der-ori-13} and \eqref{high-der-ori-14} by $\p_{il}\phi$ and $\p_{ijl}\phi$ respectively, then integrating with respect to $\Omega$, and by taking note of 
			\begin{align*}
				&\int_{\Omega} \p_t  \p_{i} \psi_l \p_{il} \phi dx = \p_t\int_{\Omega}  \p_{i} \psi_l \p_{il} \phi dx +  \int_{\Omega} \p_{i l}\psi_l \p_t\p_{x_i}\phi dx =\p_t\int_{\Omega}  \p_{i} \psi_l \p_{il} \phi dx +  \int_{\Omega} \p_{i l}\psi_l \p_{i}\mathcal{S}_{0} dx ,\\
				&\int_{\Omega} \p_t  \p_{ij} \psi_l \p_{ijl}\phi dx = \p_t\int_{\Omega}  \p_{ij} \psi_l \p_{ijl}\phi dx +  \int_{\Omega} \p_{i j l}\psi_l \p_t\p_{ij}\phi dx =\p_t\int_{\Omega}  \p_{ij} \psi_l \p_{ijl}\phi dx +  \int_{\Omega} \p_{i j l}\psi_l \p_{i j}\mathcal{S}_{0} dx,
			\end{align*}
			we obtain 
			\begin{align}
				& \p_t\int_{\Omega} \nabla_x \psi \cdot \nabla_x^2 \phi dx + \tilde c\norm{\nabla_x^2 \phi}_{L^2}^2 \le C \sum_{i=0}^3 \norm{\nabla_x \mathcal{S}_i}_{L^2}^2 + C\norm{\p_{x_1} \left(\frac{\thetat}{\rhot}\right) \nabla_x \phi}_{L^2}^2, \label{high-der-ori-15} \\
				&  \p_t\int_{\Omega} \nabla_x^2 \psi \cdot \nabla_x^3 \phi dx + \tilde c\norm{\nabla_x^3 \phi}_{L^2}^2 \le C\sum_{i=0}^3 \norm{\nabla_x^2 \mathcal{S}_i}_{L^2}^2 + C\deltab\left(\norm{\Dt_{-1}\nabla_x\vm}_{L^2}^2+\norm{\Dt_{-\frac12}\nabla_x^2\vm}_{L^2}^2\right). \label{high-der-ori-16}
			\end{align}
			Combining \eqref{high-der-ori-15} and \eqref{high-der-ori-9}, we can obtain \eqref{high-der-ori-3}. And combining with \eqref{high-der-ori-16} and \eqref{high-der-ori-10}, we have \eqref{high-der-ori-4}. 
			Then we have completed the proof of  Lemma \ref{high-der-ori}.
		\end{proof}
		
		Applying $\p_{ij}$ to system \eqref{perturbation-1}, one has
		\begin{align}\label{perturbation-2}
			\begin{cases}
				\p_t\p_{ij}\phi+\rhot \operatorname{div}_x \p_{ij} \psi=\p_{ij}\mathcal{S}_{\phi}^{f}+\tilde{\mathcal{S}}_0, \\
				\p_t\p_{ij}\psi+\frac{2\thetat}{3\rhot} \nabla_x\p_{ij} \phi + \frac23  \nabla_x \p_{ij}\zeta-\frac{\mu(\thetat)}{\rhot}\p_{ij}\left(\Delta_x \psi+\frac{1}{3} \nabla_x \operatorname{div}_x \psi\right)=\p_{ij}\left(\mathcal{S}_{\psi}^{f1}+\mathcal{S}_{\psi}^{f2}+\mathcal{S}_{\psi}^{m}\right)+\tilde{\mathcal{S}},\\
				\p_t\p_{ij}\zeta+\frac{2}{3} \thetat \operatorname{div}_x \p_{ij}\psi-\frac{\kappa(\thetat)}{\rhot}\Delta_x \p_{ij} \zeta=\p_{ij}\left(\mathcal{S}_{\zeta}^{f1}+\mathcal{S}_{\zeta}^{f2}+\mathcal{S}_{\zeta}^{m}\right)+\tilde{\mathcal{S}}_4,
			\end{cases}
		\end{align}
		where
		\begin{align}
			&\tilde{\mathcal{S}}_0=\rhot \operatorname{div}_x \p_{ij} \psi-\p_{ij} \left( \rhot \operatorname{div}_x  \psi \right), \label{S-m-t-0} \\
			& \tilde{\mathcal{S}}=\frac{2\thetat}{3\rhot} \nabla_x\p_{ij} \phi -\frac23\p_{ij}\left(\frac{\thetat}{\rhot} \nabla_x\phi\right) +\p_{ij}\left[\frac{\mu(\thetat)}{\rhot}\left(\Delta_x \psi+\frac{1}{3} \nabla_x \operatorname{div}_x \psi\right)\right]-\frac{\mu(\thetat)}{\rhot}\p_{i j}\left(\Delta_x \psi+\frac{1}{3} \nabla_x \operatorname{div}_x \psi\right),  \label{S-m-t} \\
			&\tilde{\mathcal{S}}_4=\frac{2}{3} \thetat \operatorname{div}_x \p_{ij}\psi-\frac{2}{3}\p_{ij}\left( \thetat \operatorname{div}_x \psi \right)-\frac{\kappa(\thetat)}{\rhot}\Delta_x \p_{ij} \zeta+\p_{ij}\left(\frac{\kappa(\thetat)}{\rhot}\Delta_x  \zeta\right).\label{S-m-t-4} 
		\end{align}
		\begin{Lem}\label{high-der-ori-star}
			Under the same assumptions of  Proposition \ref{Thm-ape}, it holds that
			\begin{align*}
				\frac{d}{dt}\norm{\nabla_x^2 \vm}_{L^2}^2 + \norm{\nabla_x^3 {\vm}^{\ast}}_{L^2}^2 \leq &C_\eta \sum_{\abs{\alpha}=3} \norm{\p^{\alpha}g}_{\sigma}^2+C\check{\delta} \left[(1+t)^{-3}\norm{\vm}_{L^2}^2+(1+t)^{-1}\norm{\nabla_x \vm}_{L^2}^2+\norm{\nabla_x^2 \vm}_{L^2}^2  \right] \\
				&+C\check{\delta}\left(  (1+t)^{-2} \norm{g}_{\sigma}^2+\sum_{1\leq \abs{\gamma}\leq2} \norm{\p^{\gamma}g}_{\sigma}^2 \right)+C\deltab(1+t)^{-\frac{5}{2}}+\eta\norm{\nabla_x^3\phi}_{L^2}.
			\end{align*}
		\end{Lem}
		\begin{proof}
			By the definition of $\tilde{\mathcal{S}}_0$, $\tilde{\mathcal{S}}=(\tilde{\mathcal{S}}_1,\tilde{\mathcal{S}}_2,\tilde{\mathcal{S}}_3)^{t}$ and  $\tilde{\mathcal{S}}_4$ as in \eqref{S-m-t-0}, \eqref{S-m-t} and \eqref{S-m-t-4}, $\mathcal{S}_{\phi}^f,\mathcal{S}_{\psi}^{f1},\mathcal{S}_{\psi}^{f2},\mathcal{S}_{\zeta}^{f1}$, $\mathcal{S}_{\zeta}^{f2}$ \eqref{S-m-0}-\eqref{Q-m}, \eqref{S-m-4} and \eqref{Q-m_4}, one has
			\begin{align}
				&\sum_{i=0}^4 \abs{\tilde{\mathcal{S}}_i}\le C \deltab^{\frac12} \left[\Dt_{-1}\abs{\nabla_x \vm} + \Dt_{-\frac12}\abs{\nabla_x^2 \vm}+\Dt_{-\frac{1}{2}}\abs{\nabla_x^3 \vm}\right], \label{high-der-s-1}\\
				&\abs{\nabla_x^2\mathcal{S}_{\phi}^f}+\abs{\nabla_x^2\mathcal{S}_{\psi}^{f1}}+\abs{\nabla_x^2 \mathcal{S}_{\zeta}^{f1}}\le C\bigg[ \abs{\vm}\abs{\nabla_x^3 \vm}+\abs{\nabla_x \vm}\abs{\nabla_x^2 \vm}+\abs{\nabla_x \vm}^3+\abs{\vm}\abs{\nabla_x \vm}\abs{\nabla_x^2 \vm} \notag\\
				&\qquad\qquad\quad\qquad\qquad\qquad\qquad+ \deltab^{\frac12}\left(\Dt_{-1}\abs{\nabla_x \vm} + \Dt_{-\frac12}\abs{\nabla_x^2 \vm}+\Dt_{-\frac12}\abs{\nabla_x^3 \vm}+\Dt_{-\frac{3}{2}}\abs{ \vm}\right) +\deltab\Dt_{-\frac52}\bigg],\label{high-der-s-2}\\
				&\abs{\nabla_x\mathcal{S}_{\psi}^{f2}}+\abs{\nabla_x \mathcal{S}_{\zeta}^{f2}}\le C\bigg[ \abs{\vm}\abs{\nabla_x^3 \vm}+\abs{\nabla_x \vm}\abs{\nabla_x^2 \vm}+\abs{\nabla_x \vm}^3+\abs{\vm}\abs{\nabla_x \vm}\abs{\nabla_x^2 \vm} \notag\\
				&\qquad\qquad\qquad\qquad\quad\;+ \deltab^{\frac12}\left(\Dt_{-1}\abs{\nabla_x \vm} + \Dt_{-\frac12}\abs{\nabla_x^2 \vm}+\Dt_{-\frac{3}{2}}\abs{ \vm}\right) +\deltab\Dt_{-2}\bigg].\label{high-der-s-3}
			\end{align}
			For the non-fluid part, we take  $\p_l\left[\frac{1}{\rho}\p_i \int_{\R^3} \xi_i \xi_j L_M^{-1} \Pi d\xi\right]$ in $\p_l\mathcal{S}_{\psi}^{m}$ \eqref{K-m} as an example and perform the calculation. It holds
			\begin{align}
				& \p_{l}\left[ \frac{1}{\rho}\p_i \int_{\R^3} \xi_i \xi_j L_M^{-1} \Pi d\xi \right]=\frac{R}{\rho} \p_{il} \int_{\R^3} \theta B_{ij} \left( \frac{\xi-u}{\sqrt{R \theta}} \right) \frac{1}{M} \left[\p_t G+P_1 \left( \xi \cdot \nabla_x G \right) -Q(G,G)\right]d\xi\notag\\
				&\qquad\qquad\qquad+\p_l\left(\frac{R}{\rho}\right) \p_{i} \int_{\R^3} \theta B_{ij} \left( \frac{\xi-u}{\sqrt{R \theta}} \right) \frac{1}{M} \left[\p_t G+P_1 \left( \xi \cdot \nabla_x G \right) -Q(G,G)\right]d\xi\label{high-der-s-4}.
			\end{align}
			Substituting $(\rho,u,\theta)=(\rhot,\ut,\thetat)+(\phi,\psi,\zeta)$ and $G=\bar{G}_0+\sqrt{\mu}g$ into \eqref{high-der-s-4}, one has
			\begin{align*}
				&\abs{ \p_{l}\left[ \frac{1}{\rho}\p_i \int_{\R^3} \xi_i \xi_j L_M^{-1} \Pi d\xi \right] }\le C\left(\sum_{\abs{\alpha}=3} \abs{\p^{\alpha}g}_{\sigma} + \sum_{\abs{\gamma}=1}\abs{\p^{\gamma} g}_{\sigma}\abs{\p^{\gamma} g}_{2}+\sum_{\abs{\beta}=2}\abs{g}_{\sigma} \abs{\p^{\beta}g}_{\sigma}\right)\\
				&\qquad+C\left(\deltab^{\frac12} \Dt_{-1}+\abs{\nabla_x \vm}^2 + \abs{\nabla_x^2 \vm} + \deltab^{\frac12} \Dt_{-\frac{1}{2}} \abs{\nabla_x \vm} \right) \sum_{\abs{\gamma}=1} \left(\abs{\p^\gamma g}_{\sigma}+\deltab^{\frac12} \Dt_{-\frac12}\abs{g}_{\sigma}+ \abs{ g}_{\sigma}\abs{ g}_{2} + \deltab^{\frac12} \Dt_{-1}\right) \\
				&\qquad+C\left(\deltab^{\frac12}\Dt_{-\frac12}+\abs{\nabla_x \vm}\right) \sum_{\abs{\gamma}=1}\sum_{\abs{\beta}=2} \left(\abs{\p^\beta g}_{\sigma}+\deltab^{\frac12} \Dt_{-\frac12}\abs{\p^\gamma g}_{\sigma} + \deltab^{\frac12}\Dt_{-1}\abs{g}_{\sigma}+ \abs{\p^{\gamma} g}_{\sigma}\abs{ g}_{\sigma} \right)\\
				&\qquad + C \deltab D_{-\frac12}\left(\abs{\nabla_x^3 \vm^{\ast}}+\abs{\p_t \nabla_x^2 \vm^{\ast}} \right)+C\deltab D_{-1}\abs{\p_t\nabla_x\vm^{\ast}}+C\deltab D_{-\frac32}\abs{\p_t \vm^{\ast}}.
			\end{align*}
			The calculation of the other terms in $\nabla_x\mathcal{S}_{\psi}^{m},\nabla_x\mathcal{S}_{\zeta}^{m}$ \eqref{K-m} and \eqref{K-m-4} is similar to the above calculation, and thus we can obtain
			\begin{align}
				&\abs{\nabla_x\mathcal{S}_{\psi}^{m}}+\abs{\nabla_x\mathcal{S}_{\zeta}^{m}}\le C\left(\sum_{\abs{\alpha}=3} \abs{\p^{\alpha}g}_{\sigma} + \sum_{\abs{\gamma}=1}\abs{\p^{\gamma} g}_{\sigma}\abs{\p^{\gamma} g}_{2}+\sum_{\abs{\beta}=2}\abs{g}_{\sigma} \abs{\p^{\beta}g}_{\sigma}\right)\notag\\
				&\qquad+C\left(\deltab^{\frac12} \Dt_{-1}+\abs{\nabla_x \vm}^2 + \abs{\nabla_x^2 \vm} + \deltab^{\frac12} \Dt_{-\frac{1}{2}} \abs{\nabla_x \vm} \right) \sum_{\abs{\gamma}=1} \left(\abs{\p^\gamma g}_{\sigma}+\deltab^{\frac12} \Dt_{-\frac12}\abs{g}_{\sigma}+ \abs{ g}_{\sigma}\abs{ g}_{2} + \deltab^{\frac12} \Dt_{-1}\right)\notag \\
				&\qquad+C\left(\deltab^{\frac12}\Dt_{-\frac12}+\abs{\nabla_x \vm}\right) \sum_{\abs{\gamma}=1}\sum_{\abs{\beta}=2} \left(\abs{\p^\beta g}_{\sigma}+\deltab^{\frac12} \Dt_{-\frac12}\abs{\p^\gamma g}_{\sigma} + \deltab^{\frac12}\Dt_{-1}\abs{g}_{\sigma}+ \abs{\p^{\gamma} g}_{\sigma}\abs{ g}_{\sigma} \right) \notag\\
				&\qquad + C \deltab D_{-\frac12}\left(\abs{\nabla_x^3 \vm^{\ast}}+\abs{\p_t \nabla_x^2 \vm^{\ast}} \right)+C\deltab D_{-1}\abs{\p_t\nabla_x\vm^{\ast}}+C\deltab D_{-\frac32}\abs{\p_t \vm^{\ast}}.\label{high-der-s-5}
			\end{align}
			Multiply  \eqref{perturbation-2}$_1$, \eqref{perturbation-2}$_2$, and \eqref{perturbation-2}$_3$ by $\frac{5\thetat}{3\rhot^2}\p_{ij}\phi$, $\p_{ij} \psi$, and $\frac{5}{2\thetat}\p_{ij} \zeta$, respectively, and then integrate them in region $\Omega$. By the a {\it priori} assumption \eqref{apa}, applying integration by parts and using Lemma \ref{s-t-s-tt-2}, Lemma \ref{high-der-ori}, \eqref{high-der-s-1}-\eqref{high-der-s-3}, \eqref{high-der-s-5}, we can obtain the following
			\begin{align*}
				&\frac{d}{dt}\norm{\nabla_x^2 \vm}_{L^2}^2 + \tilde c\norm{\nabla_x^3 {\vm}^{\ast}}_{L^2}^2\\
				\leq&  C_\eta\sum_{i=0}^4\norm{\tilde{\mathcal{S}}_i}_{L^2}^2+\deltab^{-\frac12}\left(\norm{\nabla_x^2 \mathcal{S}_{\phi}^f}_{L^2}^2+\norm{\nabla_x^2 \mathcal{S}_{\psi}^{f1}}_{L^2}^2+\norm{\nabla_x^2 \mathcal{S}_{\zeta}^{f1}}_{L^2}^2 \right)  +\deltab^{\frac12}\norm{\nabla_x^2\vm}_{L^2}^2+\eta\norm{\nabla_x^3 \vm}_{L^2}^2\\
				&+C_\eta\left( \norm{\nabla_x \mathcal{S}_{\psi}^{f2}}_{L^2}^2+\norm{\nabla_x \mathcal{S}_{\zeta}^{f2}}_{L^2}^2+\norm{\nabla_x \mathcal{S}_{\psi}^{m}}_{L^2}^2+\norm{\nabla_x \mathcal{S}_{\zeta}^{m}}_{L^2}^2  \right) \\
				\leq &C_\eta \sum_{\abs{\alpha}=3} \norm{\p^{\alpha}g}_{\sigma}^2+C\check{\delta} \left[(1+t)^{-3}\norm{\vm}_{L^2}^2+(1+t)^{-1}\norm{\nabla_x \vm}_{L^2}^2  \right] +\deltab^{\frac12}\norm{\nabla_x^2\vm}_{L^2}^2+\eta\norm{\nabla_x^3 \vm}_{L^2}^2\\
				&+C\check{\delta}\left(  (1+t)^{-2} \norm{g}_{\sigma}^2+\sum_{1\leq \abs{\gamma}\leq2} \norm{\p^{\gamma}g}_{\sigma}^2 \right)+C\deltab(1+t)^{-\frac{5}{2}}.
			\end{align*}
			Then we have completed the proof of  Lemma \ref{high-der-ori-star}.
		\end{proof}
		Combining  Lemma \ref{high-der-ori} and  Lemma \ref{high-der-ori-star}, we have completed the proof of  Theorem \ref{high-order-111-111}.
		
		\subsection{Estimate for the microscopic equations}\label{sec4.4} 
		Recall the definition of $\vm$ \eqref{2025-11-5-1} and  dissipation energy functionals $\mathcal{D}_{i,\omega}$ and $\mathcal{D}_i$ \eqref{2.11}-\eqref{2.11-5}. Then we have the following result.
		
		\begin{Lem}\label{mic-energy-t123}
			Under the same assumptions of  Proposition \ref{Thm-ape}, let $|\alpha|+|\beta|\leq 3$,  then for $|\beta|\geq 1$ and $\abs{\beta'}=1$, one has
			\begin{align}
				\frac{d}{dt}\norm{\omega(\beta)\p^{\alpha}_{\beta}g}^2_2+&\norm{\omega(\beta)\p^{\alpha}_{\beta}g}_{\sigma}^2\leq C\bar{\delta}(1+t)^{-\frac32-\abs{\alpha}}+ C\check{\delta}\left[\mathcal{D}_{2,\omega}(t)+(1+t)^{-1}\mathcal{D}_{1,\omega}(t)\right]+\eta\sum_{\abs{\beta_1}=\abs{\beta}}\norm{\p_{\beta_1}^{\alpha}g}_{\sigma,\omega}^2\notag\\
				&\qquad+C_\eta \left(\sum_{|\beta_1|<|\beta|}\norm{\partial^\alpha_{\beta_1} g}_{\sigma,\omega}^2+\norm{\p^\alpha_{\beta-\beta'}\nabla_xg}^2_{\sigma,\omega}+\norm{\p^{\alpha}\nabla_xg}_{\sigma,\omega}^2+\norm{\nabla_x^{\abs{\alpha}+1}\vm}_{L^2}^2\right),\label{5.30}
			\end{align}
			and for $\beta=0$, $|\alpha|\leq 2$, one has
			\begin{align}
				&\frac{d}{dt}\norm{\p^{\alpha}g}^2_{2,\omega}+\norm{\p^{\alpha}g}_{\sigma,\omega}^2\le C \bar{\delta}(1+t)^{-\frac32-\abs{\alpha}}+ C\check{\delta}\left[\mathcal{D}_{2,\omega}(t)+(1+t)^{-1}\mathcal{D}_{1,\omega}(t)\right]\notag\\
				&\qquad\qquad\qquad\qquad\qquad\qquad+C_\eta\sum_{ \abs{\gamma_1}=\abs{\alpha}}\norm{\partial^{\gamma_1} g}_{\sigma}^2+C_\eta\norm{\partial^{\alpha}\nabla_x g}_{\sigma,\omega}^2+C_\eta\norm{\nabla_x^{\abs{\alpha}+1}\vm}_{L^2}^2,\label{5.30-1}\\
				&\frac{d}{dt}\norm{\p^{\alpha}g}^2_2+\norm{\p^{\alpha}g}_{\sigma}^2
				\leq C\bar{\delta}(1+t)^{-\frac32-\abs{\alpha}}+ C\check{\delta}\left[\mathcal{D}_{2}(t)+(1+t)^{-1}\mathcal{D}_{1}(t)\right]+C_\eta\left(\norm{\p^{\alpha}\nabla_xg}_{\sigma}^2+\norm{\nabla_x^{\abs{\alpha}+1} \vm}_{L^2}^2\right).\label{5.30-2}
			\end{align}
			Moreover, for $\abs{\alpha}+\abs{\beta}\le 3$, $\abs{\alpha}\geq1,\ \abs{\beta}\geq1$, $\abs{\beta'}=1$, one has 
			\begin{align}\label{5.30-4}
				\frac{d}{dt}\norm{\p^{\alpha}_{\beta}g}^2_{2,\omega}+&\norm{\p^{\alpha}_{\beta}g}_{\sigma,\omega}^2
				\leq C\bar{\delta}(1+t)^{-\frac32-\abs{\alpha}}+C\check{\delta} \sum_{i=1}^{3}(1+t)^{-3+i}\mathcal{D}_{i,\omega}(t)+\eta\sum_{\abs{\beta_1}=\abs{\beta}}\norm{\p_{\beta_1}^{\alpha}g}_{\sigma,\omega}^2\nonumber\\
				&\qquad+C_\eta\left(\sum_{|\beta_1|<|\beta|}\norm{\partial^\alpha_{\beta_1} g}_{\sigma,\omega}^2+\norm{\p^\alpha_{\beta-\beta'}\nabla_xg}^2_{\sigma,\omega}+\norm{\p^{\alpha}\nabla_xg}_{\sigma,\omega}^2+\norm{\nabla_x^{\abs{\alpha}+1}\vm}_{L^2}^2\right),
			\end{align}
			and for $1\leq \abs{\alpha}\leq 2,\ \abs{\beta}=0$, one has
			\begin{align}
				&\frac{d}{dt}\norm{\p^{\alpha}g}^{2}_{2,\omega}+\norm{\p^{\alpha}g}_{\sigma,\omega}^2
				\leq C\left[\bar{\delta}(1+t)^{-\frac32-\abs{\alpha}}+\check{\delta} \sum_{i=1}^{3}(1+t)^{-3+i}\mathcal{D}_{i,\omega}(t) \right]\notag\\
				&\qquad\qquad\qquad\qquad\qquad\qquad+C_\eta\left[\sum_{ \abs{\gamma_1}=\abs{\alpha}}\norm{\partial^{\gamma_1} g}_{\sigma}^2+\norm{\partial^{\alpha}\nabla_x g}_{\sigma,\omega}^2+\norm{\nabla_x^{\abs{\alpha}+1}\vm}_{L_x^2}^2\right],\label{5.30-5}\\  
				&\frac{d}{dt}\norm{\p^{\alpha}g}^2_2+\norm{\p^{\alpha}g}_{\sigma}^2
				\leq C\bar{\delta}(1+t)^{-\frac32-\abs{\alpha}}+C\check{\delta} \sum_{i=1}^{3}(1+t)^{-3+i}\mathcal{D}_{i}(t)+C_\eta\norm{\p^{\alpha}\nabla_xg}_{\sigma}^2+C_\eta\norm{\nabla_x^{\abs{\alpha}+1}\vm}_{L^2}^2,\label{5.30-6}
			\end{align}
			where $\deltab=\delta+\varepsilon_0$ and $\deltac=\chi+\deltab^{\frac12}$.
		\end{Lem}
		\begin{proof}
			We focus on proving \eqref{5.30} and \eqref{5.30-4}, and the estimates for the remaining parts in the  Lemma \ref{mic-energy-t123} can be obtained by similar calculations.
			Applying $\p^{\alpha}_{\beta}$, $|\alpha|+|\beta|\leq 3,\;\;\abs{\beta}\geq 1$ to \eqref{mic-perturbation}, one has
			\begin{align}\label{mic-perturbation-ab}
				&\p^{\alpha}_{\beta}\p_t g + \xi \cdot \nabla_x\p^{\alpha}_{\beta} g +\sum_{\abs{\beta'}=1}\p_{\beta'}\xi \cdot\nabla_x\p_{\beta-\beta'}^{\alpha}g-\p^{\alpha}_{\beta}\mathcal{L} g =\p_{\beta}^{\alpha}\Bigg(\frac{P_0 \left( \xi \cdot \sqrt{\mu}\nabla_x g \right)}{\sqrt{\mu}}\Bigg)+\p_{\beta}^{\alpha}(\mathcal{S}_{g1}+\mathcal{S}_{g2}+\mathcal{S}_{g3}),
			\end{align}
			where
			\begin{align*}
				\p^{\alpha}_{\beta}\mathcal{S}_{g1}:=&-\p^{\alpha}_{\beta}\Bigg(\frac{P_1 \left( \xi_1 \p_{x_1} \bar{G}_0 \right)}{\sqrt{\mu}}+\frac{\p_t \bar{G}_0}{\sqrt{\mu}}\Bigg),\\
				\p^{\alpha}_{\beta}\mathcal{S}_{g2}:=&-\p_{\beta}^{\alpha}\left\{\frac{1}{\sqrt{\mu}}P_1 \left[\xi_1 \left(\frac{\abs{\xi-u}^2\p_{x_1} (\thetat-\thetab)}{2 R \theta^2} +\frac{\left( \xi -u\right)\cdot \p_{x_1} (\ut-\ub)}{R \theta} \right)M \right]\right\}\notag\\
				& -\p_{\beta}^{\alpha}\left\{\frac{1}{\sqrt{\mu}}P_1 \left[\xi \cdot \left( \frac{\abs{\xi-u}^2\nabla_x \zeta}{2 R \theta^2} +\frac{\left( \xi -u\right)\cdot \nabla_x \psi}{R \theta} \right) M\right] \right\} := \p^{\alpha}_{\beta}\mathcal{S}_{g21}+\p^{\alpha}_{\beta}\mathcal{S}_{g22}, \\
				\p^{\alpha}_{\beta}\mathcal{S}_{g3}:=&\p^{\alpha}_{\beta}\left\{\Gamma\left(g,\frac{M-\mu}{\sqrt{\mu}} \right)+\Gamma\left(\frac{M-\mu}{\sqrt{\mu}},g \right)+\Gamma\left(G,G\right)\right\}.
			\end{align*}
			Since $|\thetat-\thetab|+\abs{\ut-\ub}\lesssim\abs{\Xi}$, by Corollary  \ref{R-m-1-1-1} and Lemma \ref{2026-4-24-1}, we obtain
			\begin{align}\label{2025.6.05-1}
				&\abs{\p_{\beta}\mathcal{S}_{g1}}_{2,\omega}+\abs{\p_{\beta}\mathcal{S}_{g21}}_{2,\omega}\leq C \bar{\delta}\left(\tilde{D}_{-1}+\tilde D_{-\frac12}(\abs{\nabla_x\vm}+\abs{\p_t\vm^{\ast}}) \right),\notag\\ &\abs{\p_{\beta}\mathcal{S}_{g22}}_{2,\omega}\leq C \abs{\vm}\abs{\nabla_x \vm}+C\abs{\nabla_x \vm}+ C\bar{\delta}^{\frac12}\tilde{D}_{-\frac12}\abs{\vm}.
			\end{align}
			Multiplying \eqref{mic-perturbation-ab} by $\omega \p^{\alpha}_{\beta}g$ and then integrating the resulting equation over $\Omega\times\R^3$, one has
			\begin{align}
				&\frac{d}{dt}\norm{\omega \p^{\alpha}_{\beta}g}_2^{2}-\iint_{\Omega\times\R^3}\omega^2\p^{\alpha}_{\beta}g\p^{\alpha}_{\beta}\mathcal{L}gdxd\xi+\sum_{\abs{\beta'}=1}\iint_{\Omega\times\R^3}\omega^2\p^{\alpha}_{\beta}g\p_{\beta'}\xi \cdot\nabla_x\p_{\beta-\beta'}^{\alpha}gdxd\xi \nonumber\\
				=&\iint_{\Omega\times\R^3}\omega^2\Big(\p_{\beta}^{\alpha}\mathcal{S}_{g1}+\p_{\beta}^{\alpha}\mathcal{S}_{g2}+\p_{\beta}^{\alpha}\mathcal{S}_{g3}\Big)\p^{\alpha}_{\beta}g dxd\xi+\iint_{\Omega\times\R^3}\omega^2\p^{\alpha}_{\beta}\Big\{\frac{P_0 \left( \xi \cdot \sqrt{\mu}\nabla_x g \right)}{\sqrt{\mu}}\Big\}\p^{\alpha}_{\beta}g dxd\xi
				.\label{gab}
			\end{align}
			For the second term in the left-hand side of \eqref{gab}, by  Lemma \ref{lem5.3}, we have
			\begin{equation*}
				-\langle\partial^\alpha_\beta\mathcal{L}g,\omega^2(\beta)\partial^\alpha_\beta g\rangle\geq |\omega(\beta)\partial^\alpha_\beta g|_\sigma^2-\eta\sum_{|\beta_1|=|\beta|}|\omega(\beta_1)\partial^\alpha_{\beta_1} g|_\sigma^2
				-C_\eta\sum_{|\beta_1|<|\beta|}|\omega(\beta_1)\partial^\alpha_{\beta_1} g|_\sigma^2.
			\end{equation*}
			Next, we estimate the other terms in \eqref{gab}. Note that terms involving $\mathcal{S}_{g3}$ have been estimated in  Lemma \ref{lem5.7} and  Lemma \ref{lem5.8}. For $\alpha=0$, we have
			\begin{align*}
				\iint_{\Omega\times\R^3}\omega^2\p_{\beta}^{\alpha}\mathcal{S}_{g3}\p^{\alpha}_{\beta}g dxd\xi\leq C\deltac\left[\|\omega(\beta)\p_\beta^\alpha g\|_{\sigma}^2+\mathcal{D}_{2,\omega}(t)+  (1+t)^{-1} \mathcal{D}_{1,\omega} \right]+C\deltab(1+t)^{-\frac32}.
			\end{align*}
			For $\abs{\alpha}\geq1$, by \eqref{high-der-ori-1}, \eqref{high-der-ori-2}, Lemma \ref{lem5.7}, and Lemma \ref{lem5.8}, we obtain
			\begin{align}
				\iint_{\Omega\times\R^3}\omega^2\p_{\beta}^{\alpha}\mathcal{S}_{g3}\p^{\alpha}_{\beta}g dxd\xi\leq &C\deltac\Big[\|\omega(\beta)\p_\beta^\alpha g\|_{\sigma}^2+\sum_{i=1}^3(1+t)^{-3+i}\mathcal{D}_{i,\omega}(t)\Big]+C\deltab(1+t)^{-\frac32-\abs{\alpha}}.\label{2025.6.05-4}
			\end{align}
			Next, we focus on the $\mathcal{S}_{g22}$. For $|\alpha|=0$, by \eqref{2025.6.05-1}, Lemma \ref{s-t-s-tt-2}, Lemma \ref{high-der-ori} and the {\it a priori} assumptions \eqref{apa}, one has
			\begin{align*}
				&\iint_{\Omega\times\R^3}\omega^2\p_{\beta}^{\alpha}\mathcal{S}_{g22}\p^{\alpha}_{\beta}g dxd\xi\leq C_\eta\left[\norm{\nabla_x \vm}_{L^2}^2+\check{\delta}(1+t)^{-1}\norm{\vm}_{L^2}^2 \right]+\eta\norm{\omega(\beta)\p_{\beta}g}_{\sigma}^2.
			\end{align*}
			Then, for $|\alpha|=1$, by Gagliardo-Nirenberg inequality, \eqref{high-der-ori-1}, \eqref{2025.6.05-1}, Lemma \ref{high-der-ori} and the {\it a priori} assumptions \eqref{apa}, one has
			\begin{align*}
				&\iint_{\Omega\times\R^3}\omega^2\p_{\beta}^{\alpha}\mathcal{S}_{g22}\p^{\alpha}_{\beta}g dxd\xi \notag\\
				\leq&C_\eta\norm{\p^{|\alpha|}\nabla_x\vm}_{L_x^2}^2+C_\eta\norm{\p^{\abs{\alpha}}\vm}_{L^4}^4+C\bar{\delta}\norm{\tilde{D}_{-\frac12}}_{L^{\infty}}^2\norm{\p^{\abs{\alpha}}\vm}_{L^2}^2+C_\eta\bar{\delta}\norm{\tilde{D}_{-1}}_{L^{\infty}}^2\norm{\vm}_{L^2}^2+\eta\norm{\omega(\beta)\p_{\beta}^{\alpha}g}_{\sigma}^2\nonumber\\
				\leq &\eta\norm{\omega(\beta)\p^{\alpha}_{\beta}g}_{\sigma}^2+C\check{\delta}\left[(1+t)^{-1}\norm{\nabla_x\vm}_{L^2}^2+(1+t)^{-2}\norm{\vm}_{L^2}^2\right]+C_\eta\norm{\nabla_x^{2}\vm}_{L^2}^2.
			\end{align*}
			Similarly, for $|\alpha|=2$,  by Gagliardo-Nirenberg inequality, \eqref{high-der-ori-2}, \eqref{2025.6.05-1},  Lemma \ref{high-der-ori} and the {\it a priori} assumptions \eqref{apa}, it holds that
			\begin{align*}
				&\iint_{\Omega\times\R^3}\omega^2\p_{\beta}^{\alpha}\mathcal{S}_{g22}\p^{\alpha}_{\beta}g dxd\xi 
				\\
				\leq &\eta\norm{\omega(\beta)\p^{\alpha}_{\beta}g}_{\sigma}^2+C\check{\delta}\left[\norm{\nabla_x^2\vm}_{L^2}^2+(1+t)^{-1}\norm{\nabla_x\vm}_{L^2}^2+(1+t)^{-2}\norm{\vm}_{L^2}^2\right]+C_\eta\norm{\nabla_x^{3}\vm}_{L^2}^2.\notag
			\end{align*}
			For $\mathcal{S}_{g1}$ and $\mathcal{S}_{g21}$, by \eqref{2025.6.05-1}, we only need to treat $\deltab\p^{\alpha}\tilde D_{-1}$ and other terms can be estimated as $\mathcal{S}_{g22}$. It then holds
			\begin{align*}
				\deltab \int_{\R^3}\p^{\alpha} \tilde D_{-1} \abs{\omega \p_\beta^\alpha g}_{\sigma} dx \lesssim C\bar{\delta}\norm{\omega\p^\alpha_{\beta}g}_{\sigma}^2+C\bar{\delta}(1+t)^{-\frac{3}{2}-|\alpha|}.
			\end{align*}
			Then we have
			\begin{align*}
				\iint_{\Omega\times\R^3}\omega^2\p_{\beta}^{\alpha}(\mathcal{S}_{g1}+\mathcal{S}_{g21})\p^{\alpha}_{\beta}g d\xi dx\leq C\bar{\delta}\norm{\omega\p^\alpha_{\beta}g}_{\sigma}^2+C\bar{\delta}(1+t)^{-\frac{3}{2}-|\alpha|}+\iint_{\Omega\times\R^3}\abs{\omega^2\p_\beta^\alpha \mathcal{S}_{g22}\p_{\beta}^\alpha g}d\xi dx.
			\end{align*}
			And for the last term in the first  line in \eqref{gab}, direct calculation yields
			\begin{align*}
				&\sum_{\abs{\beta'}=1}\iint_{\Omega\times\R^3}\omega^2\p^{\alpha}_{\beta}g\p_{\beta'}\xi \cdot\nabla_x\p_{\beta-\beta'}^{\alpha}gdxd\xi\le \eta\norm{\omega\p^\alpha_{\beta}g}^2_{\sigma}+C_\eta\norm{\omega(\beta-\beta')\p^\alpha_{\beta-\beta'}\nabla_xg}^2_{\sigma}.
			\end{align*}
			Finally, for the last term in the second line in \eqref{gab}, for $\alpha=0$, one has
			\begin{align*}
				&\iint_{\Omega\times\R^3}\omega^2\p_{\beta}\Big\{\frac{P_0 \left( \xi \cdot \sqrt{\mu}\nabla_x g \right)}{\sqrt{\mu}}\Big\}\p_{\beta}g dxd\xi\leq\eta\norm{\omega(\beta)\p_{\beta}g}^2_{\sigma}+C_\eta\norm{\omega(0)\nabla_xg}_{\sigma}^2.
			\end{align*}
			For $\abs{\alpha}\ge 1$, one has
			\begin{align*}
				&\iint_{\Omega\times\R^3}\omega^2\p_{\beta}^{\alpha}\Big\{\frac{P_0 \left( \xi \cdot \sqrt{\mu}\nabla_x g \right)}{\sqrt{\mu}}\Big\}\p^{\alpha}_{\beta}g dxd\xi=\iint_{\Omega\times\R^3}\omega^2\p^{\alpha}_{\beta}\Big\{\frac{1}{\sqrt{\mu}}\mathop{\sum }\limits_{{j = 0}}^{4}\left\langle  {\xi \cdot \sqrt{\mu}\nabla_x g,{\chi}_{j}}\right\rangle  {\chi}_{j}\Big\}\p^{\alpha}_{\beta}g dxd\xi\nonumber\\
				\leq&\eta\norm{\omega\p^\alpha_{\beta}g}^2_{\sigma}+C\sum_{0\leq|\bar{\alpha}|\leq| \alpha|-1}\norm{\p^{|\alpha|-|\bar{\alpha}|}(\rho,u,\theta)}_{L^4}^2\norm{\abs{\omega(0)\p^{\bar{\alpha}}\nabla_xg}_{\sigma}}^2_{L^4}+C_\eta\norm{\omega(0)\p^{{\alpha}}\nabla_xg}_{\sigma}^2\nonumber\\
				\leq&\eta\norm{\omega\p^\alpha_{\beta}g}^2_{\sigma}+\check{\delta}\mathcal{D}_{3,\omega}+C_\eta\norm{\omega(0)\p^{{\alpha}}\nabla_xg}_{\sigma}^2. 
			\end{align*}
			Plugging all the above estimates back to \eqref{gab}, we then have finished the proof of \eqref{5.30} and \eqref{5.30-4}. 
			
			The proof for \eqref{5.30-1}, \eqref{5.30-2}, \eqref{5.30-5} and \eqref{5.30-6} are basically the same as those for \eqref{5.30} and \eqref{5.30-4}. The only difference lies in the dissipative estimate derived from the linear operator $\mathcal{L}$:
			\begin{equation}\label{2025.6.06-1}
				-\langle\partial^\alpha\mathcal{L}g,\omega^2(0)\partial^\alpha g\rangle\geq c_{4}|\omega(0)\partial^\alpha g|_\sigma^2-C_\eta|\chi_{\eta}(\xi)\partial^\alpha g|_2^2,\qquad\quad -\langle\partial^\alpha\mathcal{L}g,\partial^\alpha g\rangle\geq c_{4}|\partial^\alpha g|_\sigma^2.
			\end{equation}
			Then we have completed the proof of  Lemma \ref{mic-energy-t123}.
		\end{proof}
		
		\subsection{Estimates of the highest-order derivatives}\label{sec4.5}
		For estimates of the highest-order spatial derivatives of solutions, one can not rely on the microscopic equation for $g$ \eqref{mic-perturbation}. Instead, it is necessary to turn to the original equation for $f$ \eqref{equ-landau}.  
		Recall the definition of $\vm$ \eqref{2025-11-5-1} and dissipation energy functionals $\mathcal{D}_{i,\omega}$ and $\mathcal{D}_i$ \eqref{2.11}-\eqref{2.11-5}. Then we have the following result.
		
		\begin{Lem}\label{mic-energy-t123-123}
			Under the same assumptions of  Proposition \ref{Thm-ape}, it holds that
			\begin{align}
				&\frac{d}{dt}\sum_{|\alpha|=3}\norm{\frac{\partial^\alpha f}{\sqrt{\mu}}}_{2,\omega}^2+c\sum_{|\alpha|=3}\|\partial^\alpha{g}\|_{\sigma,\omega}^2\leq C\check{\delta} \sum_{i=1}^{3}(1+t)^{-3+i}\mathcal{D}_{i,\omega}(t)+C\bar{\delta}(1+t)^{-\frac52}\notag\\
				&\qquad\quad\qquad\qquad\qquad\qquad\qquad\qquad\qquad+C_\eta\sum_{\abs{\alpha}=3}\left(\|\partial^\alpha{g}\|_\sigma^2+\norm{\p^{\alpha}\vm}_{L^2}^2 \right),\label{2025.08.24-1}\\
				& \frac{d}{dt}\sum_{|\alpha|=3}\norm{\frac{\partial^\alpha f}{\sqrt{\mu}}}_2^2+c\sum_{|\alpha|=3}\|\partial^\alpha{g}\|_\sigma^2\leq C\check{\delta} \sum_{i=1}^{3}(1+t)^{-3+i}\mathcal{D}_{i}(t)+C\bar{\delta}(1+t)^{-\frac52}.\label{2025.08.24-2}
			\end{align}
		\end{Lem}
		
		\begin{proof}
			From \eqref{equ-landau}, one has
			\begin{align}\label{equ-f}
				\p_t \left(\frac{f}{\sqrt{\mu}}\right) + \xi \cdot \nabla_x \left( \frac{f}{\sqrt{\mu}} \right)-\mathcal{L}g=\mathcal{S}_{f}+ \frac{1}{\sqrt{\mu}} {P}_1 \left[\xi_1 \left( \frac{\abs{\xi-u}^2}{2R\theta^2}\p_{x_1}\thetab + \frac{(\xi-u)\cdot\p_{x_1}\ub}{R\theta} \right)M\right],
			\end{align}
			where
			\begin{align*}
				\mathcal{S}_{f}=& \Gamma \left( \frac{M-\mu}{\sqrt{\mu}},g \right) + \Gamma\left( g,\frac{M-\mu}{\sqrt{\mu}} \right)+\Gamma\left(\frac{G}{\sqrt{\mu}},\frac{G}{\sqrt{\mu}}\right).
			\end{align*}
			Applying $\p^{\alpha}$ with $|\alpha|=3$ to \eqref{equ-f}, multiplying the results by $\omega^2(0)\frac{\p^{\alpha}f}{\sqrt{\mu}}$ and then integrating the resulting equation over $\Omega\times\R^3$, one has
			\begin{align*}
				&\frac{d}{dt}\norm{\omega(0)\frac{\p^{\alpha}f}{\sqrt{\mu}}}^2_{L^2_{\xi,x}}-\iint_{\Omega\times\R^3}\omega^2(0)\p^{\alpha}g\mathcal{L}\p^{\alpha}gd\xi dx=\iint_{\Omega\times\R^3}\omega^2(0)\p^{\alpha}\Big(\frac{M}{\sqrt{\mu}}+\frac{\bar{G}_0}{\sqrt{\mu}}\Big)\mathcal{L}\p^{\alpha}gd\xi dx\nonumber\\
				&\qquad+\iint_{\Omega\times\R^3}\omega^2(0)\frac{\p^{\alpha}f}{\sqrt{\mu}}\p^{\alpha}\mathcal{S}_{f}d\xi+\iint_{\Omega\times\R^3}\omega^2(0)\frac{\p^{\alpha}f}{\sqrt{\mu}}\p^{\alpha}\left\{\frac{1}{\sqrt{\mu}} {P}_1 \left[\xi_1 \left( \frac{\abs{\xi-u}^2}{2R\theta^2}\p_{x_1}\thetab + \frac{(\xi-u)\cdot\p_{x_1}\ub}{R\theta} \right)M\right]\right\}d\xi dx\nonumber\\
				&\qquad:=I_{f}^1+I_{f}^2+I_{f}^3.
			\end{align*}
			We first  calculate the following 
			\begin{align}
				\label{M1-2}
				\partial^{\alpha}M
				=&\mu\left\{\frac{\partial^{\alpha}\rho}{\rho}
				+\frac{(\xi-u)\cdot\partial^{\alpha}u}{R\theta}
				+\left(\frac{|\xi-u|^2}{2R\theta}-\frac32\right)\frac{\partial^{\alpha}\theta}{\theta}\right\}\nonumber\\
				&+(M-\mu)\left\{\frac{\partial^{\alpha}\rho}{\rho}
				+\frac{(\xi-u)\cdot\partial^{\alpha}u}{R\theta}
				+\left(\frac{|\xi-u|^2}{2R\theta}-\frac32\right)\frac{\partial^{\alpha}\theta}{\theta}\right\}\nonumber\\
				&+\sum_{|\alpha_1|=1,2}C_{\alpha}^{\alpha_1}\bigg\{\partial^{\alpha_1}\left(\frac M\rho\right)\partial^{\alpha-\alpha_1}\rho
				+\partial^{\alpha_1}\left(M\frac{\xi-u}{R\theta}\right)\cdot\partial^{\alpha-\alpha_1}u\nonumber\\
				&+\partial^{\alpha_1}\left(M\frac{|\xi-u|^2}{2R\theta^2}-M\frac{3}{2\theta}\right)\partial^{\alpha-\alpha_1}\theta\bigg\}
				=:J^\alpha_1+J^\alpha_2+J^\alpha_3.
			\end{align}
			For $\iint_{\Omega\times\R^3}\omega^2(0)\frac{\p^{\alpha}M}{\sqrt{\mu}}\mathcal{L}\p^{\alpha}gd\xi dx$, we have
			\begin{align}\label{2026-4-25-1}
				\iint_{\Omega\times\R^3}\omega^2(0)\frac{\p^{\alpha}M}{\sqrt{\mu}}\mathcal{L}\p^{\alpha}gd\xi dx\leq C\check{\delta}\left[\mathcal{D}_{3,\omega}(t)+(1+t)^{-1}\mathcal{D}_{2,\omega}\right]+C\bar{\delta}(1+t)^{-\frac52}+\eta\norm{\p^{\alpha}g}_{\sigma,\omega}^2+C_{\eta}\norm{\p^{\alpha} \vm}_{L^2}^2.
			\end{align}
			To provide estimates for $I_{f}^2$, by Lemma \ref{high-der-ori} and Lemma \ref{2026-4-24-1}, we first note that
			\begin{align*}
				\sum_{\abs{\alpha}=3} \norm{\frac{\p^{\alpha}f}{\sqrt{\mu}}}^2_{\sigma,\omega}&\leq C\sum_{\abs{\alpha}=3}\left(  \norm{\p^{\alpha}g}_{\sigma,\omega}^2 + \norm{\frac{\p^{\alpha}\bar{G}_0}{\sqrt{\mu}}}_{\sigma,\omega}^2 + \norm{\frac{\p^{\alpha}(M-\tilde{M})}{\sqrt{\mu}}}_{\sigma,\omega}^2+ \norm{\frac{\p^{\alpha}\tilde{M}}{\sqrt{\mu}}}_{\sigma,\omega}^2 \right) \notag\\
				&\leq C \sum_{\abs{\alpha}=3} \left( \norm{\p^{\alpha}g}_{\sigma,\omega}^2 + \norm{\nabla_x^3 \vm}_{L^2}^2 \right)+C \deltac \left[  (1+t)^{-\frac52}+(1+t)^{-1} \mathcal{D}_{2,\omega} \right].
			\end{align*}
			Thus, the estimate of $I_{f}^2$ can be found in  Lemma \ref{lem5.7} and Lemma \ref{lem5.8}, which is similar to \eqref{2025.6.05-4}.
			Then, we estimate the rest of $I_{f}^1$ and $I_{f}^3$. 
			By Lemma \ref{2026-4-24-1}, for $\alpha\ge1$, $\beta\ge0$ and any $b\ge0$, we observe that
			\begin{align*}
				\abs{\la \xi \ra^b \p^\alpha_\beta(\frac{\bar G_0}{\sqrt \mu})}_{2,\sigma} \approx \abs{\p_\beta^\alpha\left\{\frac{1}{\sqrt{\mu}} {P}_1 \left[\xi_1 \left( \frac{\abs{\xi-u}^2}{2R\theta^2}\p_{x_1}\thetab + \frac{(\xi-u)\cdot\p_{x_1}\ub}{R\theta} \right)M\right]\right\}}_{2,\sigma}\approx\deltab\tilde D_{-\frac12}\abs{\p_{\beta}^{\alpha}(\frac{M}{\sqrt \mu})}_{2,\sigma}.
			\end{align*}
			Using the same method as \eqref{2026-4-25-1}, it can directly obtain
			\begin{align}\notag
				\iint_{\Omega\times\R^3}\omega^2(0)\frac{\p^{\alpha}\bar{G}_0}{\sqrt{\mu}}\mathcal{L}\p^{\alpha}gd\xi dx+I_{f}^3\leq C\check{\delta}\left[\mathcal{D}_{3,\omega}(t)+(1+t)^{-1}\mathcal{D}_{2,\omega}\right]+C\bar{\delta}\Big((1+t)^{-\frac52}+\norm{\p^{\alpha}g}_{\sigma,\omega}^2+\norm{\p^{\alpha} \vm}_{L^2}^2\Big).
			\end{align}
			And the estimate of the linear operator $\mathcal{L}$ is consistent with \eqref{2025.6.06-1}, thereby enabling us to obtain the proof of \eqref{2025.08.24-1}.
			
			Next, we turn to prove \eqref{2025.08.24-2}. Due to \eqref{M1-2} and $(\rho,u,\theta)=\vm+(\rhot,\ut,\thetat)$, we have
			\begin{align*}
				&\iint_{\Omega\times\R^3}\Bigg(J^{\alpha}_1\mathcal{L}\p^{\alpha}g+J^{\alpha}_2\mathcal{L}\p^{\alpha}g+J^{\alpha}_3\mathcal{L}\p^{\alpha}g\Bigg)d\xi dx\leq C\check{\delta}\left[\mathcal{D}_{3}(t)+(1+t)^{-1}\mathcal{D}_{2}\right]+C\bar{\delta}(1+t)^{-\frac52}+\eta\norm{\p^{\alpha}g}_{\sigma}^2,
			\end{align*}
			where  we have used the fact that $J^{\alpha}_1\in \text{ker}\mathcal{L}$. It then follows that $\iint_{\Omega\times\R^3} J^{\alpha}_1\mathcal{L}\p^{\alpha}g d\xi dx=0 $. The estimate of the remaining terms is similar to the previous calculation, and the proof is omitted for the convenience of reading.
			Thus, we have completed the proof of  Lemma \ref{mic-energy-t123-123}.
		\end{proof}
		
		\section{Decay rate}\label{sec5}
		In this section, we present the differential inequalities for the instant energy functionals $\mathcal{E}_{i,\omega},\mathcal{E}_i$ (see \eqref{2.10}-\eqref{2.10-5}) and  dissipation energy functionals $\mathcal{D}_{i,\omega},\mathcal{D}_i$  (see \eqref{2.11}-\eqref{2.11-5}), and thereby prove Proposition \ref{Thm-ape}.
		At first, we present a useful Lemma. For the Landau collision operator, the dissipative norm is not equivalent to the $L^2$ norm. Motivated by \cite{Wang,Wang-1}, we need the following lemma to perform the time-velocity interpolation technique.
		
		\begin{Lem}\label{key-trans}
			For any integer $\alpha$ and positive numbers $\epsilon>0$ and $\upsilon>0$, one has
			\begin{align*}
				\abs{\p^{\alpha} f}_2^2 \leq \upsilon(1+t)^{\epsilon} \abs{\p^{\alpha} f}_{\sigma}^2+e^{-\frac{q}{4}\upsilon^2( 1+t )^{2\epsilon}} \abs{e^{\frac{q}{8}\la \xi\ra^2} \p^{\alpha} f}_2^2.
			\end{align*}
		\end{Lem}
		\begin{proof}
			By the definition of dissipative norm, one has
			\begin{align*}
				\abs{\p^{\alpha} f}_{\sigma}^2& \geq \int_{\R^3} \la \xi\ra^{-1} \abs{\p^{\alpha} f}^2 d\xi=\int_{\abs{\xi}\leq \upsilon(1+t)^{\epsilon}}+\int_{\abs{\xi}\geq \upsilon(1+t)^{\epsilon}}\nonumber\\
				&\geq\frac{(1+t)^{-\epsilon}}{\upsilon}\int_{\abs{\xi}\leq \upsilon(1+t)^{\epsilon}}\abs{\p^{\alpha} f}^2 d\xi \nonumber\\
				&=\frac{(1+t)^{-\epsilon}}{\upsilon}\abs{\p^{\alpha}f}_{2}^2-\frac{(1+t)^{-\epsilon}}{\upsilon}\int_{\abs{\xi}\geq \upsilon (1+t)^{\epsilon}}\abs{\p^{\alpha} f}^2,
			\end{align*}
			and
			\begin{align*}
				\int_{\abs{\xi}\geq \upsilon(1+t)^{\epsilon}}\abs{\p^{\alpha} f}^2 d\xi\leq e^{-\frac{q}{4}\upsilon^2 (1+t)^{2\epsilon} } \int_{\R^3} e^{\frac{q}{4}\la \xi \ra^2 } \abs{\p^{\alpha} f}^2 d\xi.
			\end{align*}
			Then we complete the proof of  Lemma \ref{key-trans}. 
		\end{proof}
		
		Now we are going to conclude the proof of Proposition \ref{Thm-ape} for obtaining the stability and optimal decay rate of the planar entropy  wave for the Landau equation. 
		
		\begin{proof}[Proof of  Proposition \ref{Thm-ape}]
			Recall the definition of $(\Vmc,\vm)$ \eqref{2025-11-5-1}, instant energy functionals $\mathcal{E}_{i,\omega},\mathcal{E}_i$ \eqref{2.10}-\eqref{2.10-5} and  dissipation energy functionals $\mathcal{D}_{i,\omega},\mathcal{D}_i$  \eqref{2.11}-\eqref{2.11-5}.
			Combining Theorem \ref{H2-enenrgy-anti}, Theorem \ref{non-zero-111-111},  Theorem \ref{high-order-111-111}, Lemma \ref{high-der-ori}, Lemma \ref{mic-energy-t123} and  Lemma \ref{mic-energy-t123-123}, we have
			\begin{align}
				&\frac{d}{dt} \mathcal{E}_{1,\omega} + \mathcal{D}_{1,\omega}
				\leq C\deltac(1+t)^{-1}\mathcal{E}_{1,\omega} + C\deltab (1+t)^{-\frac{3}{2}}, \label{decay-1}\\
				&\frac{d}{dt}\mathcal{E}_{2,\omega}+\mathcal{D}_{2,\omega}\leq C\deltac \left[(1+t)^{-1}\norm{\vm}_{L^2}^2+(1+t)^{-2}\norm{\Vmc}_{L^2}^2 \right]+ C\deltab (1+t)^{-\frac{3}{2}}.\label{decay-1-tt}
			\end{align}
			Integrating \eqref{decay-1} and \eqref{decay-1-tt} with respect to time $t$
			and combining with the {\it a priori} assumptions \eqref{apa}, we have
			\begin{align}\label{2025.6.11-1}
				\mathcal{E}_{2,\omega}\leq C(\delta+\varepsilon_0).
			\end{align}
			Moreover, we derive the following  differential inequalities for the unweighted microscopic quantities
			\begin{align}
				&\frac{d}{dt}\mathcal{E}_1 + \mathcal{D}_1\leq C\deltac(1+t)^{-1}\mathcal{E}_1 + C\deltab (1+t)^{-\frac{3}{2}}, \label{decay-2-1} \\
				&\frac{d}{dt}\mathcal{E}_2 + \mathcal{D}_2\leq C\deltac \left[(1+t)^{-1}\mathcal{D}_1+(1+t)^{-2}\mathcal{E}_1 \right]+ C\deltab (1+t)^{-\frac{3}{2}}, \label{decay-2-2}\\
				&\frac{d}{dt}\tilde{\mathcal{E}}_2 +\tilde{\mathcal{D}}_2 \leq C\deltac\left[(1+t)^{-1}\mathcal{D}_1+(1+t)^{-2}\mathcal{E}_1\right]+C\deltab(1+t)^{-\frac52},\label{I-1-1-1}\\
				&\frac{d}{dt}\mathcal{E}_3 + \mathcal{D}_3\leq C\deltac \left[(1+t)^{-1}\mathcal{D}_2+(1+t)^{-2}\mathcal{D}_1+(1+t)^{-3}\mathcal{E}_1 \right]+ C\deltab (1+t)^{-\frac{5}{2}}.\label{decay-2-3}
			\end{align}
			Recalling $f=M-\Mt+\sqrt{\mu}g + \bar{G}_0+\Mt$, one has
			\begin{align}\label{decay-3-0-2}
				\sum_{\abs{\alpha}=3}\norm{ \frac{\p^{\alpha}f}{\sqrt{\mu}}}_2^2 &\lesssim\sum_{\abs{\alpha}=3}\norm{ \frac{\p^{\alpha}(M-\Mt)}{\sqrt{\mu}}}_2^2+\sum_{\abs{\alpha}=3}\norm{ \p^{\alpha}g}_2^2+\deltab(1+t)^{-\frac52} \notag\\
				&\lesssim \norm{\nabla_x \vm}_{H^2}^2  +\sum_{\abs{\alpha}=3}\norm{ \p^{\alpha}g}_2^2+\deltab(1+t)^{-\frac52}.
			\end{align}
			Combining  Lemma \ref{key-trans} together with \eqref{2025.6.11-1} and \eqref{decay-3-0-2}, and then taking $\epsilon=\frac{1}{10}$ and $\upsilon=\frac{1}{2}$, one has
			\begin{align}
				&\int_0^t \left(\mathcal{E}_2+\tilde{\mathcal{E}}_2\right) d \tau \le C(1+t)^{\frac{1}{10}} \int_0^t \mathcal{D}_1 d\tau +C\deltab,\label{decay-3-0}\\
				&\int_0^t (1+\tau)\mathcal{E}_3 d \tau \le C(1+t)^{\frac{1}{10}} \int_0^t(1+\tau) \tilde{\mathcal{D}}_2 d\tau +C\deltab.\label{decay-3-0-1}
			\end{align}
			By taking $\int_0^t$\eqref{decay-2-1}$d\tau$, $\int_0^t(1+\tau)$\eqref{decay-2-2}$d\tau$, $\int_0^t(1+\tau)$\eqref{I-1-1-1}$d\tau$, and $\int_0^t(1+\tau)^2$\eqref{decay-2-3}$d\tau$, and using \eqref{decay-3-0}-\eqref{decay-3-0-1}, one has
			\begin{align}
				&\mathcal{E}_1 +\int_0^t \mathcal{D}_1d\tau \leq C\deltab(1+t)^{\deltac}\leq C\deltab(1+t)^{\frac{1}{10}}, \label{decay-3-1}\\
				&(1+t)\mathcal{E}_2+\int_0^t (1+\tau)\mathcal{D}_2 d\tau \le \int_0^t \left(\mathcal{E}_2 +\mathcal{D}_1\right)d\tau + \int_0^t(1+\tau)^{-1} \mathcal{E}_1 d\tau+C\deltab(1+t)^{\frac12} \notag\\
				&\qquad\qquad\le C(1+t)^{\frac{1}{10}}\int_0^t \mathcal{D}_1 d\tau + \int_0^t(1+t)^{-1} \mathcal{E}_1 d\tau+C\deltab(1+t)^{\frac12}\leq C\deltab(1+t)^{\frac12},\label{decay-3-2}\\
				&(1+t)\tilde{\mathcal{E}}_2+\int_0^t(1+\tau)\tilde{\mathcal{D}}_2d\tau\leq\int_0^t \left(\tilde{\mathcal{E}}_2+\mathcal{D}_1\right)d\tau + \int_0^t (1+\tau)^{-1} \mathcal{E}_1 d\tau+C\deltab\notag\\
				&\qquad\qquad\leq C(1+t)^{\frac{1}{10}} \int_0^t \mathcal{D}_1 d\tau+ \int_0^t (1+\tau)^{-1} \mathcal{E}_1 d\tau+C\deltab\leq C\deltab(1+t)^{\frac15},\label{I-1-1-2}\\
				&(1+t)^2\mathcal{E}_3+\int_0^t (1+\tau)^2\mathcal{D}_3 d\tau \notag \\
				&\le C\int_0^t (1+\tau)\left(\mathcal{E}_3+\mathcal{D}_2\right) d\tau + C\int_0^t \mathcal{D}_1d\tau+ C\int_0^t(1+\tau)^{-1} \mathcal{E}_1 d\tau+C\deltab(1+t)^{\frac12} \notag\\
				&\le C\int_0^t \left[(1+\tau){\mathcal{D}}_2+ \mathcal{D}_1\right] d\tau+ C(1+t)^{\frac{1}{10}}\int_0^t (1+\tau)\tilde{\mathcal{D}}_2 d\tau + C\int_0^t(1+\tau)^{-1} \mathcal{E}_1 d\tau+C\deltab(1+t)^{\frac12} \notag\\
				&\leq C\deltab(1+t)^{\frac12}\label{decay-3-3}.
			\end{align}
			From \eqref{decay-3-1}, \eqref{decay-3-2}, \eqref{I-1-1-2} and \eqref{decay-3-3}, it is direct to see
			\begin{align}\label{decay-3}
				&  \mathcal{E}_2 \leq C\deltab(1+t)^{-\frac12},\qquad\quad\tilde{\mathcal{E}}_2\leq C\deltab (1+t)^{-\frac45},\qquad\quad\mathcal{E}_3\leq C\deltab(1+t)^{-\frac32}.
			\end{align}
			From \eqref{decay-3}, we then obtain the long-time behavior of the perturbation quantities of the Landau equation with respect to planar entropy waves as 
			\begin{align}\label{2025.6.11-2}
				\|\frac{f(t,x,\xi)-M_{[\rhot,\ut,\thetat](t,x)}(\xi)}{\sqrt{\mu}}\|_{L_{x}^{\infty}L_{\xi}^{2}}\leq C(\delta^{\frac12}+\varepsilon_0^{\frac12})(1+t)^{-\frac{1}{2}}.
			\end{align}
			Combining \eqref{2025.6.11-1}, \eqref{decay-3} and \eqref{2025.6.11-2}, we have completed the proof of  Proposition \ref{Thm-ape}.
		\end{proof}

		\section{Stretched exponential decay for non-zero modes}\label{sec6}
		
		In this section, we present the stretched exponential decay estimates for non-zero modes of solutions to the Landau equation \eqref{equ-landau}. This section is divided into two parts. The first part provides the dissipation estimates for non-zero mode of macroscopic quantities in subsection \ref{sec6.1}, and the second part gives the dissipation estimates for non-zero mode of microscopic quantities in subsection \ref{sec6.2}. By combining these two parts of estimates and applying  Lemma \ref{key-trans}, we can obtain the stretched exponential decay estimates for non-zero mode. We state the result as follows; the proof will be given in subsection \ref{sec6.3}.
		
		\begin{Thm}\label{lem999}
			Under the same assumptions of Theorem \ref{mt}, it holds that
			\begin{align}\label{add.sec6.thm}
				\norm{\frac{(f-M_{[\rhoc,\uc,\check \theta]})_{\neq}}{\sqrt{\mu}}}_{L_x^\infty L_{\xi}^2}^2\leq C (\delta+\varepsilon_0) e^{-ct^{\frac23}}.
			\end{align}
		\end{Thm}
		
		In order to present the proof of the above result more conveniently, we introduce the new perturbation $\f$ as a perturbation near the global Maxwellian $\mu$ in the form that $f=\mu+\sqrt{\mu}\f$. With this notation, it holds that
		\begin{equation}
			\label{add.ff.p1}
			\f_{\neq}=\frac{f_{\neq}}{\sqrt \mu}=\frac{(f-M_{[\rhoc,\uc,\check\theta]})_{\neq}}{\sqrt \mu},
		\end{equation}
		so we only need to prove $\norm{\f_{\neq}}_{L_x^\infty L_\xi^2}\leq C (\delta+\varepsilon_0) e^{-ct^{\frac23}}$ corresponding to the desired result \eqref{add.sec6.thm}.
		Recall $f=M+G$ \eqref{assumption-1}, then we have
		\begin{align}\label{2026-4-26-2}
			\f=\frac{f-\mu}{\sqrt \mu}=\frac{M-M_{[\rhoc,\uc,\check\theta]}+M_{[\rhoc,\uc,\check\theta]}-\mu+\bar G_0+\sqrt \mu g}{\sqrt \mu}.
		\end{align}
		With this observation, by Proposition \ref{Thm-ape}, Corollary \ref{R-m-1-1-1} and \eqref{2026-4-26-2}, we obtain
		\begin{align}\label{2026-4-26-1}
			\norm{\abs{\f}_{2,\omega}}_{L^\infty}+\sum_{1\le|\alpha_1|\le3} \norm{\p^{\alpha_1} \f}_{2,\omega}+\sum_{\abs{\alpha_0}\le3}\norm{\p^{\alpha_0}\f_{\neq}}_{2,\omega}+\norm{\abs{\f_{\neq}}_{\sigma,\omega}}_{L_x^\infty}+\sum_{|\alpha_2|=1}\norm{\abs{\p^{\alpha_2}\Do\f}_{\sigma,\omega}}_{L_x^\infty}  \le C \deltab.
		\end{align}
		
		\subsection{Dissipation estimates for non-zero modes of macroscopic parts}\label{sec6.1}
		For the linearized operator $\mathcal L$ \eqref{2.5} near the global Maxwellian $\mu$, recall that $\ker(\mathcal L)=\text{span} \{\sqrt \mu ,\xi\sqrt \mu,\abs{\xi}^2\sqrt \mu  \}$. We will denote the projection onto the kernel as $\Pp$ and for the new perturbation $\f_{\neq}:=\f-\int_{\T^2}\f dx_2dx_3$ we denote 
		\begin{align*}
			\Pp \f_{\neq}=\Big(\aaa_{\neq}+\bb_{\neq}\cdot\xi +\cc_{\neq}(\abs{\xi}^2-3)\Big)\sqrt{\mu},
		\end{align*}
		where
		\begin{align*}
			\aaa_{\neq}=\int_{\R^3} \f_{\neq} \sqrt \mu d\xi,\qquad \bb_{i\neq}=\int_{\R^3} \f_{\neq} \xi_i\sqrt{\mu} d\xi,\qquad \cc_{\neq}=\frac16\int_{\R^3} \f_{\neq}(\abs{\xi}^2-3)\sqrt{\mu}d\xi.
		\end{align*}
		The  high-order moment functions $\Ee[h]:=[\Ee_{ij}(h)]_{3\times3}$ and $\Ff[h]:=(\Ff_1[h],\Ff_2[h],\Ff_3[h])$ are defined as
		\begin{align}\label{2026-4-26-7}
			\Ee_{ij}[h]=\int_{\R^3} (\xi_i\xi_j-1)\sqrt{\mu} h d\xi,\qquad \Ff_{i}[h]=\int_{\R^3}(\abs{\xi}^2-5) \xi_i \sqrt{\mu} h d\xi.
		\end{align}
		Then we have the hydrodynamic system (see $e.g.$ \cite{BCD,Duan-1})
		\begin{align}
			&\p_t \aaa_{\neq} + \nabla_x \cdot \bb_{\neq}=0, \label{2026-4-26-3}\\
			&\p_t \bb_{\neq}+\nabla_x \aaa_{\neq}=-2\nabla_x \cc_{\neq}-\nabla_x \cdot \Ee[(I-\Pp)\f_{\neq}],\label{2026-4-26-4}\\
			&\p_t \cc_{\neq}=-\frac13 \nabla_x \cdot \bb_{\neq}-\frac16 \nabla_x \cdot \Ff[(I-\Pp)\f_{\neq}] \label{2026-4-23-5},\\
			&\p_t \left( \Ee[(I-\Pp)\f_{\neq}]+2\cc_{\neq} I \right)+\nabla_x\bb_{\neq}+(\nabla_x\bb_{\neq})^{\perp}=-\Ee\left[ l+n\right],\label{2026-4-23-6}\\
			&\p_t\Ff[(I-\Pp)\f_{\neq}] + \nabla_x \cc_{\neq} =-\Ff[l+n],\label{2026-4-23-7}
		\end{align}
		where
		\begin{align*}
			l=\xi \cdot \nabla_x (I-\Pp)\f_{\neq}-\mathcal{L}(I-\Pp)\f_{\neq},\qquad n=-\Dn \Gamma(\f,\f).
		\end{align*}
		For the $\Ee[l],\Ee[n],\Ff[l],\Ff[n]$, we have the following estimate.
		
		\begin{Lem}\label{2026-4-27-2}
			Let $\Ee[\cdot],\Ff[\cdot]$ be defined as in \eqref{2026-4-26-7}, then there exists a universal constant $0<\epsilon\le\frac14$ such that
			\begin{align*}
				&\abs{\Ee_{ij}[(I-\Pp)\f_{\neq}]}+\abs{\Ff_i[(I-\Pp)\f_{\neq}]}\le C \abs{\mu^{\epsilon}(I-\Pp)\f_{\neq}}_2 ,\qquad \abs{\Ee_{ij}[l]}+\abs{\Ff_i[l]}\le C \abs{\mu^{\epsilon}\nabla_x(I-\Pp)\f_{\neq}}_2,\\
				&\abs{\Ee_{ij}[n]}+\abs{F_i[n]}\le C \deltab\left( \abs{\mu^{\epsilon} \f_{\neq}}_2+\abs{\f_{\neq}}_{\sigma}\right).
			\end{align*}
		\end{Lem}
		\begin{proof}
			It should be noted that
			\begin{align}\label{2026-4-27-1}
				\Dn\Gamma(\f,\f)=\Dn\Gamma(\f_{\neq},\f_{\neq})+\Gamma(\Do \f,\f_{\neq})+\Gamma(\f_{\neq},\Do\f).
			\end{align}
			Combining \eqref{2026-4-26-1}, \eqref{2026-4-27-1} and \eqref{5.8}, and using the same method as Lemma 4.5 in \cite{BCD} and Lemma 3.5 in \cite{Duan-1}, we can prove Lemma \ref{2026-4-27-2}. For simplicity of presentation, we omit the details of proof here.
		\end{proof}
		
		Based on the hydrodynamic system \eqref{2026-4-26-3}-\eqref{2026-4-23-6}, the following estimates hold for the non-zero modes of macroscopic quantities.
		
		\begin{Lem}\label{2025-7-30-a}
			Under the same assumptions of Theorem \ref{mt}, for suitable large constants $C_{1},C_{2}$ and a universal constant $C$, it holds that
			\begin{align*}
				&\frac{d}{dt}\mathcal{E}_{mac} + C\Big(\norm{\nabla_x \aaa_{\neq}}_{L^2}^2+\norm{\nabla_x \bb_{\neq}}_{L^2}^2+\norm{\nabla_x \cc_{\neq}}_{L^2}^2\Big) \\
				\le& (C_{\eta_1}+C_{\eta_2})\norm{\mu^{\epsilon}\nabla_x(I-\Pp)\f_{\neq}}_2^2+C\deltab\left( \norm{\mu^{\epsilon} (I-\Pp)\f_{\neq}}_2^2+\norm{(I-\Pp)\f_{\neq}}_\sigma^2\right),
			\end{align*}
			where $\mathcal{E}_{mac}:=\int_{\Omega}\bb_{\neq}\cdot\nabla_x \aaa_{\neq}+C_{1} \left(\Ee[(I-\Pp)\f_{\neq}]+2\cc_{\neq}I \right):[\nabla_x \bb_{\neq}+(\nabla_x \bb_{\neq})^{\perp}]+C_{2}\Ff[(I-\Pp)\f_{\neq}]\cdot\nabla_x \cc_{\neq} dx$.
		\end{Lem}
		\begin{proof}
			Multiplying \eqref{2026-4-26-4} by $\nabla_x \aaa_{\neq}$ and using \eqref{2026-4-26-3}, we have
			\begin{align}\label{2026-4-27-3}
				&\frac{d}{dt}\int_{\Omega} \bb_{\neq} \cdot \nabla_x \aaa_{\neq} dx + \norm{\nabla_x \aaa_{\neq}}_{L^2}^2 \notag\\
				=&-2\int_{\Omega} \nabla_x \cc_{\neq} \cdot \nabla_x \aaa_{\neq} dx-\int_{\Omega} \p_t \aaa_{\neq}  \nabla_x \cdot \bb_{\neq} dx -2\int_{\Omega} \nabla_x \cdot\Ee[(I-\Pp)\f_{\neq}] \cdot \nabla_x \aaa_{\neq} dx \notag\\
				\le& \frac{1}{100} \norm{\nabla_x \aaa_{\neq}}_{L^2}^2+ C\left( \norm{\nabla_x \cc_{\neq}}_{L^2}^2+\norm{\nabla_x \bb_{\neq}}_{L^2}^2 \right)+C\norm{\mu^{\epsilon} \nabla_x(I-\Pp)\f_{\neq}}_2^2.
			\end{align}
			Multiplying \eqref{2026-4-23-6} by $\nabla_x \bb_{\neq}+(\nabla_x \bb_{\neq})^{\perp}$, and using \eqref{2026-4-26-4} and Lemma \ref{2026-4-27-2}, we obtain
			\begin{align}\label{2026-4-27-4}
				& \frac{d}{dt}\int_{\Omega} (\Ee(I-\Pp)\f_{\neq}+2\cc_{\neq}I):(\nabla_x \bb_{\neq}+(\nabla_x \bb_{\neq})^{\perp}) dx +\norm{\nabla_x b+(\nabla_x \bb_{\neq})^{\perp}}_{L^2}^2 \notag\\
				=&-2\int_{\Omega} \p_t \bb_{\neq} \cdot \left( \nabla_x \cdot \Ee[(I-\Pp)\f_{\neq}]+2\nabla_x\cc_{\neq}I \right) dx-\int_{\Omega} \Ee[l+n] :(\nabla_x \bb_{\neq}+(\nabla_x \bb_{\neq})^{\perp}) dx\notag\\
				\le & \eta_1 \norm{\nabla_x \bb_{\neq}+(\nabla_x \bb_{\neq})^{\perp}}_{L^2}^2+ \eta_1\norm{\nabla_x \aaa_{\neq}}_{L^2}^2 + C\deltab\left(\norm{\mu^{\epsilon} \f_{\neq}}_2^2 +\norm{\f_{\neq}}_{\sigma}^2 \right) \notag\\
				&+C_{\eta_1}\left(\norm{\nabla_x \cc_{\neq}}_{L^2}^2 +\norm{\mu^{\epsilon}\nabla_x(I-\Pp)\f_{\neq}}_2^2 \right).
			\end{align}
			Multiplying \eqref{2026-4-23-7} by $\nabla_x \bb_{\neq}+(\nabla_x \bb_{\neq})^{\perp}$, and using \eqref{2026-4-23-5} and Lemma \ref{2026-4-27-2}, one has
			\begin{align}\label{2026-4-27-5}
				&\frac{d}{dt}\int_{\Omega}\Ff[(I-\Pp)\f_{\neq}] \cdot \nabla_x \cc_{\neq} dx + \norm{\nabla_x \cc_{\neq}}_{L^2}^2 \notag\\
				=&-\int_{\Omega} \p_t\cc_{\neq}\nabla_x \cdot\Ff[(I-\Pp)\f_{\neq}] dx-\int_{\Omega} \Ff[l+n] \cdot \nabla_x \cc_{\neq} dx \notag\\
				\le& \eta_2 \left( \norm{\nabla_x \bb_{\neq}}_{L^2}^2 + \norm{\nabla_x \cc_{\neq}}_{L^2}^2 \right) + C_{\eta_2}\norm{\mu^{\epsilon} \nabla_x (I-\Pp)\f_{\neq}}_2^2 + C \deltab\left( \norm{\mu^{\epsilon} \f_{\neq}}_2^2+\norm{\f_{\neq}}_\sigma^2\right).
			\end{align}
			Combining \eqref{2026-4-27-3}, \eqref{2026-4-27-4}, \eqref{2026-4-27-5} and using the small strength $\deltab$ of entropy wave, we have finished the proof of Lemma \ref{2025-7-30-a}.
		\end{proof}
		
		\subsection{Dissipation estimates for non-zero modes of microscopic parts}\label{sec6.2}
		Recall $f=\mu+\sqrt{\mu}\f$ and Landau equation \eqref{equ-landau}, then we can write the equation of $\f_{\neq}$ as
		\begin{align}\label{mic-non-f}
			\p_t \f_{\neq}+\xi \cdot\nabla_x \f_{\neq}-\mathcal{L} \f_{\neq}=\Dn\Gamma(\f,\f).
		\end{align}
		Based on the equation of $\f_{\neq}$ \eqref{mic-non-f}, we have the following result.
		
		\begin{Lem}\label{2025-7-30-b} 
			Under the same assumptions of Theorem \ref{mt},  it holds that
			\begin{align*}
				\frac{d}{dt} \norm{\nabla_x \f_{\neq}}_2^2 + c\norm{(I-\Pp)\nabla_x\f_{\neq}}_\sigma^2 \le C\deltab \left( \norm{\nabla_x \aaa_{\neq}}_{L^2}^2+\norm{\nabla_x \bb_{\neq}}_{L^2}^2+\norm{\nabla_x \cc_{\neq}}_{L^2}^2 \right).
			\end{align*}
		\end{Lem}
		\begin{proof}
			Applying $\p_{x_i}$ to \eqref{mic-non-f}, we have
			\begin{align}\label{mic-non-f-1}
				\p_t \p_{x_i}\f_{\neq} + \xi\cdot \nabla_x \p_{x_i}\f_{\neq} - \mathcal L \p_{x_i} \f_{\neq} = \p_{x_i} \Dn \Gamma(\f,\f).
			\end{align}
			Multiplying \eqref{mic-non-f-1} by $\p_{x_i} \f_{\neq}$, and applying $-\iint_{\Omega\times \R^3} h\mathcal{L}h d\xi dx \ge c \norm{(I-\Pp)h}_{\sigma}^2$, \eqref{2026-4-26-1}, \eqref{2026-4-27-1} and  \eqref{5.7}, one has
			\begin{align*}
				\frac{d}{dt}\norm{\p_{x_i} \f_{\neq}}_2^2 + c\norm{(I-\Pp)\p_{x_i}\f_{\neq}}_\sigma^2 \le C\deltab \left( \norm{\p_{x_i} \aaa_{\neq}}_{L^2}^2+\norm{\p_{x_i} \bb_{\neq}}_{L^2}^2+\norm{\p_{x_i} \cc_{\neq}}_{L^2}^2 \right),
			\end{align*}
			where we have used
			\begin{align*}
				&\iint_{\Omega\times\R^3} \Dn \p_{x_i} \Gamma (\f,\f) \p_{x_i} \f_{\neq} d\xi dx=\iint_{\Omega\times\R^3} \left[\p_{x_i} \Gamma(\f_{\neq},\f_{\neq})+\p_{x_i}\Gamma(\Do \f,\f_{\neq})+\p_{x_i}\Gamma(\f_{\neq},\Do\f) \right] \p_{x_i}\f_{\neq} d\xi dx \\
				\le&C\norm{\p_{x_i} \f_{\neq}}_\sigma \left[(\norm{\abs{\f}_{\sigma}}_{L^\infty}+\norm{\abs{\f}_{2}}_{L^\infty})\norm{\p_{x_i} \f_{\neq}}_{\sigma}+(\norm{\abs{\p_{x_1}\Do \f}_\sigma}_{L^\infty}+\norm{\abs{\p_{x_1}\Do \f}_2}_{L^\infty}) \norm{\f_{\neq}}_{\sigma} \right]\\
				\le& C\deltab \left( \norm{\p_{x_i} \aaa_{\neq}}_{L^2}^2+\norm{\p_{x_i} \bb_{\neq}}_{L^2}^2+\norm{\p_{x_i} \cc_{\neq}}_{L^2}^2 \right)+ C \deltab \norm{(I-\Pp)\p_{x_i}\f_{\neq}}_{\sigma}^2.
			\end{align*}
			Then we have completed the proof of Lemma \ref{2025-7-30-b}.
		\end{proof}
		
		\subsection{Proof of stretched exponential decay}\label{sec6.3}
		Based on Lemma \ref{2025-7-30-a} and  Lemma \ref{2025-7-30-b}, we give the proof of  Theorem \ref{lem999} regarding the stretched exponential decay estimates for non-zero mode.
		
		\begin{proof}[Proof of Theorem \ref{lem999}.]
			Combining  Lemma \ref{2025-7-30-a}, Lemma \ref{2025-7-30-b} and $\norm{\f_{\neq}}_{2}\le \norm{\nabla_x \f_{\neq}}_{2}$,  we have
			\begin{align}\notag
				&\frac{d}{dt} \norm{\nabla_x\f_{\neq}}_{2}^2 + c \norm{ \nabla_x\f_{\neq}}_{\sigma}^2  \leq 0.
			\end{align}
			Further from  Lemma \ref{key-trans}, by taking $\upsilon\ge 1$, we deduce that
			\begin{align}\label{2025-11-13-2}
				\frac{d}{dt} \norm{\nabla_x\f_{\neq}}_2^2 + \frac{c(1+t)^{-\epsilon}}{\upsilon} \norm{\nabla_x\f_{\neq}}_2^2 \le \frac{c(1+t)^{-\epsilon}}{\upsilon} e^{-\frac{q}{4}\upsilon^2(1+t)^{2\epsilon}}\norm{e^{\frac{q}{8}\la \xi \ra^2}\nabla_x\f_{\neq}}_2^2.
			\end{align}
			By \eqref{2026-4-26-1} and \eqref{2025-11-13-2}, we  have
			\begin{align}\notag
				\frac{d}{dt}\left\{ \exp\left[\frac{c(1+t)^{1-\epsilon}}{\upsilon(1-\epsilon)} \right]\norm{\nabla_x\f_{\neq}}_2^2 \right\}\le C (\delta+\varepsilon_0) (1+t)^{-\epsilon} \exp\left[ \frac{c(1+t)^{1-\epsilon}}{\upsilon(1-\epsilon)} - \frac{q^2}{4} \upsilon^2 (1+t)^{2\epsilon}\right].
			\end{align}
			Integrating the above inequality from $0$ to $t$ and taking $\epsilon=\frac13$ and $\upsilon$ sufficiently large,  we obtain the stretched exponential decay for $\norm{\nabla_x\f_{\neq}}_2^2$ by
			\begin{align}\notag
				\norm{\nabla_x\f_{\neq}}_2^2 \leq C (\delta+\varepsilon_0) e^{-ct^{\frac23}}.
			\end{align}
			Using the Gagliardo-Nirenberg inequality, $\norm{\f_{\neq}}_{2}\le \norm{\nabla_x \f_{\neq}}_{2}$ and employing \eqref{2026-4-26-1}, one has
			\begin{align}\notag
				\norm{\f_{\neq}}_{L_x^\infty L_\xi^2}^2\le C \norm{\f_{\neq}}_{2}^\frac{1}{2} \norm{\nabla_x^2\f_{\neq}}_{2}^\frac{3}{2}\le C \norm{\nabla_x\f_{\neq}}_{2}^\frac{1}{2} \norm{\nabla_x^2\f_{\neq}}_{2}^\frac{3}{2} \leq C (\delta+\varepsilon_0) e^{-c t^{\frac23}},
			\end{align}
			which gives the desired estimate \eqref{add.sec6.thm} in terms of \eqref{add.ff.p1}.
			Then we have completed the proof of  Theorem \ref{lem999}.
		\end{proof}
		
		\section{Appendix}\label{sec7}
		In this appendix, we will give some basic estimates that have been used in the previous sections. In section \ref{sec7.1}, we list properties of Burnett functions for the Landau collision operator. In section \ref{sec7.2}, we deduce estimates involving the Landau collision terms.   
		
		\subsection{Burnett functions}\label{sec7.1}
		
		To overcome some difficulties due to the terms involving $L^{-1}_{M}$ and $\bar{G}$, we need to consider the integrality about the velocity. In this subsection, we first list some properties of the Burnett functions and then give the fast decay about the velocity $\xi$ of the Burnett functions.
		Recall the Burnett functions, cf. \cite{Bardos,Chapman-1990,Guo-2006,Ukai-Yang}:
		\begin{equation}
			\label{5.1}
			\hat{A}_{j}(\xi)=\frac{|\xi|^{2}-5}{2}\xi_{j}\quad \mbox{and} \quad \hat{B}_{ij}(\xi)=\xi_{i}\xi_{j}-\frac{1}{3}\delta_{ij}|\xi|^{2} \quad \mbox{for} \quad i,j=1,2,3.
		\end{equation}
		Noting that $\hat{A}_{j}M$ and $\hat{B}_{ij}M$ are orthogonal to the null space $\mathcal{N}$ of $L_{M}$, we can define
		functions $A_{j}(\frac{\xi-u}{\sqrt{R\theta}})$ and $ B_{ij}(\frac{\xi-u}{\sqrt{R\theta}})$ such that $P_{0}A_{j}=0$,
		$P_{0}B_{ij}=0$ and
		\begin{equation}
			\label{5.2}
			A_{j}(\frac{\xi-u}{\sqrt{R\theta}})=L^{-1}_{M}[\hat{A}_{j}(\frac{\xi-u}{\sqrt{R\theta}})M]\quad
			\mbox{and} \quad B_{ij}(\frac{\xi-u}{\sqrt{R\theta}})=L^{-1}_{M}[\hat{B}_{ij}(\frac{\xi-u}{\sqrt{R\theta}})M].
		\end{equation}
		We shall list some elementary but important properties of the Burnett functions summarized in the following lemma, cf. \cite{Guo-2006,Bardos,Ukai-Yang}.
		
		\begin{Lem}
			The Burnett functions have the following properties:
			\begin{itemize}
				\item{$-\langle \hat{A}_{i}, A_{i}\rangle$ ~~is positive and independent of i;}
				\item{$\langle \hat{A}_{i}, A_{j}\rangle=0$ ~~for ~any ~$i\neq j$;\quad $\langle
					\hat{A}_{i}, B_{jk}\rangle=0$~~for ~any ~i,~j,~k;}
				\item{$\langle\hat{B}_{ij},B_{kj}\rangle=\langle\hat{B}_{kl},B_{ij}\rangle=\langle\hat{B}_{ji},B_{kj}\rangle$,~~
					which is independent of ~i,~j, for fixed~~k,~l;}
				\item{$-\langle \hat{B}_{ij}, B_{ij}\rangle$ ~~is positive and independent of i,~j when $i\neq j$;}
				\item{$-\langle \hat{B}_{ii}, B_{jj}\rangle$ ~~is positive and independent of i,~j when $i\neq j$;}
				\item{$-\langle \hat{B}_{ii}, B_{ii}\rangle$ ~~is positive and independent of i;}
				\item{$\langle \hat{B}_{ij}, B_{kl}\rangle=0$ ~~unless~either~$(i,j)=(k,l)$~or~$(l,k)$,~or~i=j~and~k=l;}
				\item{$\langle \hat{B}_{ii}, B_{ii}\rangle-\langle \hat{B}_{ii}, B_{jj}\rangle=2\langle \hat{B}_{ij},
					B_{ij}\rangle$ ~~holds for any~ $i\neq j$.}
			\end{itemize}
		\end{Lem}
		In terms of Burnett functions, the viscosity coefficient $\mu(\theta)$ and heat conductivity
		coefficient $\kappa(\theta)$ in \eqref{diff} can be represented by
		\begin{align*}
			\mu(\theta)=&- R\theta\int_{\mathbb{R}^{3}}\hat{B}_{ij}(\frac{\xi-u}{\sqrt{R\theta}})
			B_{ij}(\frac{\xi-u}{\sqrt{R\theta}})\,d\xi>0,\quad i\neq j,
			\nonumber\\
			\kappa(\theta)=&-R^{2}\theta\int_{\mathbb{R}^{3}}\hat{A}_{j}(\frac{\xi-u}{\sqrt{R\theta}})
			A_{j}(\frac{\xi-u}{\sqrt{R\theta}})\,d\xi>0.
		\end{align*}
		Notice that these coefficients are positive smooth functions depending only on $\theta$.
		\par
		The following lemma is borrowed from \cite[Lemma 6.1]{Duan-Yu1}, which is about
		the fast velocity decay of the Burnett functions.
		\begin{Lem}\label{lem5.2}
			Suppose that $U(\xi)$ is any polynomial of $\frac{\xi-\hat{u}}{\sqrt{R}\hat{\theta}}$ such that
			$U(\xi)\widehat{M}\in(\ker{L_{\widehat{M}}})^{\perp}$ for any Maxwellian $\widehat{M}=M_{[\widehat{\rho},\widehat{u},\widehat{\theta}]}(\xi)$
			as \eqref{2025.6.14-1} where $L_{\widehat{M}}$ is as in \eqref{2025.6.14.-2} .
			For any $\epsilon\in(0,1)$ and any multi-index $\beta$, there exists a constant $C_{\beta}>0$ such that
			$$
			|\partial_{\beta}L^{-1}_{\widehat{M}}(U(\xi)\widehat{M})|\leq C_{\beta}(\widehat{v},\widehat{u},\widehat{\theta})\widehat{M}^{1-\epsilon}.
			$$
			In particular, under the assumptions of \eqref{apa}, there exists a constant $C_{\beta}>0$ such that
			\begin{equation}\notag
				|\partial_{\beta}A_{j}(\frac{\xi-u}{\sqrt{R\theta}})|+|\partial_{\beta}B_{ij}(\frac{\xi-u}{\sqrt{R\theta}})|
				\leq C_{\beta}M^{1-\epsilon}.
			\end{equation}
		\end{Lem}
		
		\subsection{Estimates on collision terms}\label{sec7.2}
		Now, we shall turn to recall the refined estimates for the linearized operator $\mathcal{L}$
		and  the nonlinear collision terms $\Gamma(g_1,g_2)$ defined in \eqref{2.5}. They can be proved by a straightforward modification of the arguments used
		in \cite[Lemmas 9]{Strain-Guo-2008} and \cite[Lemmas 2.2-2.3]{Wang} and we thus omit their proofs for brevity.
		
		\begin{Lem}\label{lem5.3}
			Let  $\omega=\omega(\beta)$ be defined by \eqref{1.32}.
			For any $\eta>0$ small enough, there exists $C>0$ such that
			\begin{align}\notag
				-\langle\partial^\alpha_\beta\mathcal{L}g,\omega^2(\beta)\partial^\alpha_\beta g\rangle\geq |\omega(\beta)\partial^\alpha_\beta g|_\sigma^2-\eta\sum_{|\beta_1|=|\beta|}|\omega(\beta_1)\partial^\alpha_{\beta_1} g|_\sigma^2
				-C\sum_{|\beta_1|<|\beta|}|\omega(\beta_1)\partial^\alpha_{\beta_1} g|_\sigma^2.
			\end{align}
			If $|\beta| = 0$, there exists $c_{4}>0$ such that
			\begin{equation}\notag
				-\langle\partial^\alpha\mathcal{L}g,w^2(0)\partial^\alpha g\rangle\geq c_{4}|w(0)\partial^\alpha g|_\sigma^2-C|\chi_{\eta}(\xi)\partial^\alpha g|_2^2,
			\end{equation}
			where $\chi_\eta(\xi)$ is a general cutoff function depending on $\eta$.
		\end{Lem}
		
		\begin{Lem}
			Under the same assumptions of  Lemma \ref{lem5.3},  for any $\epsilon>0$ small enough, one has
			\begin{equation}
				\label{5.7}
				\langle\partial^\alpha \Gamma(g_1,g_2),    g_3\rangle\leq C\sum_{|\alpha_1|\leq|\alpha|}|\mu^{\epsilon}\partial^{\alpha_1}g_1|_2| \partial^{\alpha-\alpha_1}g_2|_\sigma|  g_3|_\sigma,
			\end{equation}
			and
			\begin{equation}
				\label{5.8}
				\langle\partial^\alpha_\beta \Gamma(g_1,g_2), \omega^2(\beta)   g_3\rangle\leq
				C\sum_{|\alpha_1|\leq|\alpha|}\sum_{|\bar{\beta}|\leq|\beta_1|\leq|\beta|}|\mu^{\epsilon}\partial^{\alpha_1}_{\bar{\beta}}g_1|_2|\omega(\beta)  \partial^{\alpha-\alpha_1}_{\beta-\beta_1}g_2|_{\sigma}|\omega(\beta)    g_3|_{\sigma}.
			\end{equation}
		\end{Lem}
		
		\begin{Lem}\label{Lem-M}
			Under the same assumptions of Theorem \ref{mt-1}, 
			for any $\beta'\geq0$ and $b>0$, one has
			\begin{align}
				| \langle \xi\rangle^{b}\partial _{\beta'}(\frac{M-\mu}{\sqrt{\mu}})|_{\sigma,\omega}^2+| \langle \xi\rangle^{b}\partial _{\beta'}(\frac{M-\mu}{\sqrt{\mu}})|_{2,\omega}^2
				\leq C \deltab,\notag
			\end{align}
			where $\deltab=\delta+\varepsilon_0$.
		\end{Lem}
		\begin{proof}
			For any $\beta'\geq0$  and any $b>0$, from \eqref{1.32}, \eqref{1.35} and \eqref{apa}, there exists a small $\varepsilon_{1}>0$ such that
			\begin{align*}
				| \langle \xi\rangle^{b}\partial _{\beta'}(\frac{M-\mu}{\sqrt{\mu}})|_{\sigma,\omega}^2+| \langle \xi\rangle^{b}\partial _{\beta'}(\frac{M-\mu}{\sqrt{\mu}})|_{2,\omega}^2\leq
				C_b\sum_{|\beta'|\leq|\beta''|\leq|\beta'|+1}\int_{{\mathbb R}^3}\mu^{-\varepsilon_1}
				|\partial _{\beta''}(\frac{M-\mu}{\sqrt{\mu}})|^2\,d\xi.
			\end{align*}
			From the definition of $M$ \eqref{2025.6.14-1} and the a {\it priori} assumptions \eqref{apa}, one has
			\begin{align*}
				\int_{\R^3}&\mu^{-\varepsilon_1}|\partial _{\beta''}(\frac{M-\mu}{\sqrt{\mu}})|^2 \,d\xi\leq C \deltab.
			\end{align*}
			Then we have completed the proof of  Lemma \ref{Lem-M}.
		\end{proof}
		
		\begin{Lem}\label{2026-4-24-1}
			For any $|\bar{\alpha}|\geq 1$ and $|\bar{\beta}|\geq 0$, we use the similar expansion as before to get
			\begin{equation*}
				| \langle \xi\rangle^{b}\partial _{\bar{\beta}}(\frac{\bar{G}_0}{\sqrt{\mu}})|_{2,\omega}+|\langle \xi\rangle^{b} \partial _{\bar{\beta}}(\frac{\bar{G}_0}{\sqrt{\mu}})|_{\sigma,\omega}
				\leq C|[\p_{x_1}{\ub},\p_{x_1}\thetab]|\leq C \deltab D_{-\frac{1}{2}},
			\end{equation*}
			and
			\begin{equation*}
				|\langle \xi\rangle^{b} \partial^{\bar{\alpha}}_{\bar{\beta}}(\frac{\bar{G}_0}{\sqrt{\mu}})|_{2,\omega}+| \langle \xi\rangle^{b} \partial^{\bar{\alpha}}_{\bar{\beta}}(\frac{\bar{G}_0}{\sqrt{\mu}})|_{\sigma,\omega}
				\lesssim\deltab \left(\tilde D_{-\frac{1+\bar\alpha}{2}}+ \sum_{i=1}^{\bar\alpha}\tilde D_{-\frac{1+\bar\alpha-i}{2}}\abs{\partial^{i}\vm^{\ast}}+\sum_{i=2}^{\bar\alpha}\sum_{j=1}^{i-1}\tilde D_{-\frac{1+\bar\alpha-i}{2}}\abs{\partial^j \vm^{\ast}}\abs{\partial^{i-j} \vm^{\ast}}\right).
			\end{equation*}
			Moreover, let $\p^{\bar \alpha}=(\p_{t}^{\bar \alpha_1}\p_{x_2}^{\bar \alpha_2}\p_{x_3}^{\bar \alpha_3})$, then we have
			\begin{align}\label{2026-4-27-a}
				|\langle \xi\rangle^{b} \partial^{\bar{\alpha}}_{\bar{\beta}}(\frac{\bar{G}_0}{\sqrt{\mu}})|_{2,\omega}+| \langle \xi\rangle^{b} \partial^{\bar{\alpha}}_{\bar{\beta}}(\frac{\bar{G}_0}{\sqrt{\mu}})|_{\sigma,\omega}
				\lesssim \deltab \left(\tilde D_{-\frac{2+\bar\alpha}{2}}+ \sum_{i=1}^{\bar\alpha}\tilde D_{-\frac{1+\bar\alpha-i}{2}}\abs{\partial^{i}\vm^{\ast}}+\sum_{i=2}^{\bar\alpha}\sum_{j=1}^{i-1}\tilde D_{-\frac{1+\bar\alpha-i}{2}}\abs{\partial^j \vm^{\ast}}\abs{\partial^{i-j} \vm^{\ast}}\right).
			\end{align}
			Here $\deltab=\delta+\varepsilon_0$ and we also recall the definition of $\vm^{\ast}$ \eqref{2025-11-5-1}.
		\end{Lem}
		
		\begin{proof}
			By \eqref{5.1} and \eqref{5.2}, it holds 
			\begin{equation}
				\bar{G}_0(t,x,\xi)=\frac{\sqrt{R}\;\p_{x_1}\thetab}{\sqrt{\theta}}{A}_1(\frac{\xi-u}{\sqrt{R\theta}})
				+\bar{u}_{jx_1}{B}_{1j}(\frac{\xi-u}{\sqrt{R\theta}}),\notag
			\end{equation}
			which implies that for  $\beta_1=(1,0,0)$,
			\begin{equation*}
				\partial_{\beta_1}\bar{G}_0=\frac{\sqrt{R}\;\p_{x_1}\thetab}{\sqrt{\theta}}\partial_{\xi_1}{A}_1(\frac{\xi-u}{\sqrt{R\theta}})(\frac{1}{\sqrt{R\theta}})
				+\bar{u}_{1x_1}\partial_{\xi_1}{B}_{11}(\frac{\xi-u}{\sqrt{R\theta}})(\frac{1}{\sqrt{R\theta}}).
			\end{equation*}
			Similarly, we also have
			\begin{align}
				\label{5.20}
				\p_{x_1}\bar{G}_0=
				& \frac{\sqrt{R}\;\p_{x_1}^2\thetab}{\sqrt{\theta}}\bar{A}_1(\frac{\xi-u}{\sqrt{R\theta}})
				- \frac{\sqrt{R}\;\p_{x_1}\thetab\p_{x_1}\theta}{2\sqrt{\theta^3}}\bar{A}_1(\frac{\xi-u}{\sqrt{R\theta}})
				\notag\\
				&- \frac{\sqrt{R}\;\p_{x_1}\thetab}{\sqrt{\theta}}\nabla_\xi \bar{A}_1(\frac{\xi-u}{\sqrt{R\theta}})\cdot
				\frac{u_{x_1}}{\sqrt{R\theta}} - \frac{\sqrt{R}\;
					\p_{x_1}\thetab\p_{x_1}\theta}{\sqrt{\theta}}\nabla_\xi \bar{A}_1(\frac{\xi-u}{\sqrt{R\theta}})\cdot\frac{\xi-u}{2\sqrt{R\theta^3}}
				\notag\\
				&+ \p_{x_1}^2{\ub}_{1}\bar{B}_{11}(\frac{\xi-u}{\sqrt{R\theta}})
				- \frac{{\ub}_{1x_1}u_{x_1}}{\sqrt{R\theta}}\cdot \nabla_\xi \bar{B}_{11}(\frac{\xi-u}{\sqrt{R\theta}})
				- \frac{{\ub}_{1x_1}\p_{x_1}\theta(\xi-u)}{2\sqrt{R\theta^3}}\cdot \nabla_\xi \bar{B}_{11}(\frac{\xi-u}{\sqrt{R\theta}}).
			\end{align}
			And $\p_t \bar{G}$ has the similar expression as \eqref{5.20}.
			For $\abs{\bar{\alpha}}>1$ and $\abs{\bar{\beta}}\ge0$, we use the
			similar expansion as before to obtain 
			\begin{align*}
				&| \langle \xi\rangle^{b}\partial _{\bar{\beta}}(\frac{\bar{G}_0}{\sqrt{\mu}})|_{2,\omega}+|\langle \xi\rangle^{b} \partial _{\bar{\beta}}(\frac{\bar{G}_0}{\sqrt{\mu}})|_{\sigma,\omega}
				\leq C|[\p_{x_1}{\ub},\p_{x_1}\thetab]|\leq C \deltab D_{-\frac{1}{2}},\\
				&|\langle \xi\rangle^{b} \partial^{\bar{\alpha}}_{\bar{\beta}}(\frac{\bar{G}_0}{\sqrt{\mu}})|_{2,\omega}+| \langle \xi\rangle^{b} \partial^{\bar{\alpha}}_{\bar{\beta}}(\frac{\bar{G}_0}{\sqrt{\mu}})|_{\sigma,\omega}\le \abs{\p^{\bar\alpha}(\bar u_x,\bar{\theta}_x)}+\cdots+\abs{\p^{\bar\alpha}(u,\theta)}\abs{(\bar u_x,\bar{\theta}_x)}
				\\
				&\qquad\qquad\qquad\qquad\qquad\qquad\qquad\quad\leq C \deltab \left(\tilde D_{-\frac{1+\bar\alpha}{2}}+ \sum_{i=1}^{\bar\alpha}\tilde D_{-\frac{1+\bar\alpha-i}{2}}\abs{\partial^{i}\vm^{\ast}}+\sum_{i=2}^{\bar\alpha}\sum_{j=1}^{i-1}\tilde D_{-\frac{1+\bar\alpha-i}{2}}\abs{\partial^j \vm^{\ast}}\abs{\partial^{i-j} \vm^{\ast}}\right).
			\end{align*}
			Here we have used  Lemma \ref{lem5.2} and the fact that $|\langle \xi\rangle^b w(\bar{\beta})\mu^{-\frac{1}{2}}M^{1-\epsilon}|_2\leq C$
			for any $b\ge0$ and any small $\epsilon>0$. Using the fact that $\p_{x_2}\thetab=0$, $\p_{x_3}\thetab=0$ and $\p_{t}\thetab\lesssim \deltab D_{-1}$, we obtain \eqref{2026-4-27-a}.
		\end{proof}
		
		In what follows, we prove some linear and nonlinear estimates which have been used in  Lemma \ref{mic-energy-t123} and Lemma \ref{mic-energy-t123-123}. Recall the definition of  dissipation energy functionals $\mathcal{D}_{i,\omega}$ and $\mathcal{D}_i$  \eqref{2.11}-\eqref{2.11-5}.
		We first consider the estimates of the terms $\Gamma(g,\frac{M-\mu}{\sqrt{\mu}})$ and $\Gamma(\frac{M-\mu}{\sqrt{\mu}},g)$.
		
		\begin{Lem}\label{lem5.7}
			Under the same assumptions of Theorem \ref{mt}, let $|\alpha|+|\beta|\leq 3$ and  $\omega=\omega(\beta)$ be defined by \eqref{1.32}, then one has
			\begin{eqnarray}
				\label{5.29-t}
				&&\big|(\partial^\alpha_\beta \Gamma(\frac{M-\mu}{\sqrt{\mu}},g), \omega^2(\beta)  h)\big|
				+\big|(\partial^\alpha_\beta \Gamma(g,\frac{M-\mu}{\sqrt{\mu}}),\omega^2(\beta)  h)\big|
				\leq C\deltac\Big(\big\|\omega(\beta)h\big\|_{\sigma}^2+\mathcal{D}_{2,\omega}(t)\Big),
			\end{eqnarray}
			and
			\begin{eqnarray}
				\label{5.30-t}
				&&\big|(\partial^\alpha\Gamma(\frac{M-\mu}{\sqrt{\mu}},g),h)\big|
				+\big|(\partial^\alpha \Gamma(g,\frac{M-\mu}{\sqrt{\mu}}),h)\big|
				\leq C\deltac\Big(\|h\|_{\sigma}^2+\mathcal{D}_{2}(t)\Big),
			\end{eqnarray}
			where $\deltab:=\delta+\varepsilon_0$ and $\deltac:=\chi+\deltab^{\frac12}$.
			Moreover, for $\abs{\alpha}\geq1$, one has 
			\begin{align}\label{5.30-1t}
				&\Big|\big(\partial^\alpha_\beta \Gamma(\frac{M-\mu}{\sqrt{\mu}},g)+\partial^\alpha_\beta \Gamma(g,\frac{M-\mu}{\sqrt{\mu}}),\omega^2(\beta)  h\big)\Big|
				\leq C\deltac \left(\|\omega(\beta)h\|_{\sigma}^2+\mathcal{D}_{3,\omega}(t) + (1+t)^{-1} \mathcal{D}_{2,\omega}(t) \right),
			\end{align}
			and
			\begin{align}\label{5.30-2t}
				&\Big|\big(\partial^\alpha \Gamma(\frac{M-\mu}{\sqrt{\mu}},g)+\partial^\alpha \Gamma(g,\frac{M-\mu}{\sqrt{\mu}}),  h\big)\Big|
				\leq C\deltac \left(\|h\|_{\sigma}^2+\mathcal{D}_{3}(t) + (1+t)^{-1} \mathcal{D}_{2}(t) \right).
			\end{align}
		\end{Lem}
		\begin{proof}
			We only consider the first term on the left hand side of \eqref{5.29-t} while the second term can be handled in the same way.
			It follows from \eqref{5.8} that
			\begin{align}
				\label{5.31}
				|(\partial^\alpha_\beta \Gamma(\frac{M-\mu}{\sqrt{\mu}},g), \omega^2(\beta)  h)|
				\leq C\sum_{|\alpha_1|\leq|\alpha|}\sum_{ |\bar{\beta}|\leq|\beta_1|\leq|\beta|}
				\underbrace{\int_{\Omega}|\mu^{\epsilon}\partial^{\alpha_1}_{\bar{\beta}}(\frac{M-\mu}{\sqrt{\mu}})|_2| \omega(\beta) \partial^{\alpha-\alpha_1}_{\beta-\beta_1}g|_{\sigma}|  \omega(\beta)h|_{\sigma}\,dx}_{I_{1}}.
			\end{align}
			Thus, for any $\abs{\beta'}\geq0$ and $b>0$, we deduce from the  estimates in Lemma \ref{Lem-M} that
			\begin{align}\label{5.32-t}
				| \langle \xi\rangle^{b}\partial _{\beta'}(\frac{M-\mu}{\sqrt{\mu}})|_{\sigma,\omega}^2+| \langle \xi\rangle^{b}\partial _{\beta'}(\frac{M-\mu}{\sqrt{\mu}})|_{2,\omega}^2
				\leq C \deltac.
			\end{align}
			Note that $|\alpha_1|\leq|\alpha|\leq 3$ in \eqref{5.31} since we consider $|\alpha|+|\beta|\leq 3$.
			If $|\alpha_1|=0$ ,
			we have from  \eqref{5.32-t} and \eqref{2.11-1} that
			\begin{align}\label{5-1}
				I_{1}&\le\int_{\Omega}|\mu^{\epsilon}\partial_{\bar{\beta}}(\frac{M-\mu}{\sqrt{\mu}})|_2| \omega(\beta) \partial^{\alpha}_{\beta-\beta_1}g|_{\sigma}|  \omega(\beta)h|_{\sigma}\,dx\leq C\deltac\|\partial^{\alpha}_{\beta-\beta_1}g\|_{\sigma,\omega}\| \omega(\beta)h\|_{\sigma} \leq C\deltac\mathcal{D}_{2,\omega}(t)+C\deltac\| \omega(\beta)h\|_{\sigma}^2.
			\end{align}
			If $|\alpha_1|=1$, we have from the {\it a priori} assumptions \eqref{apa} and  Lemma \ref{high-der-ori} that
			\begin{align}
				I_{1}&\leq C\|\partial^{\alpha_1}[\rho,u,\theta]\|_{L_{x}^\infty}
				\| \omega(\beta)\partial^{\alpha-\alpha_1}_{\beta-\beta_1}g\|_{\sigma}\|  \omega(\beta)  h\|_{\sigma}
				\notag\\
				&\leq C\deltac\left( \norm{\p_{\beta-\beta_1}^{\alpha-\alpha_1}g}_{\sigma,\omega}+\sum_{2\le\abs{\alpha}\le3}\norm{\p^{\alpha}g}_{\sigma} \right) \norm{\omega(\beta)h}_{\sigma} \leq C\deltac\mathcal{D}_{2,\omega}(t)+C\deltac\| \omega(\beta)h\|_{\sigma}^2.\notag
			\end{align}
			If $|\alpha_1|=2$, we can obtain
			\begin{align}\label{higer-order-N2}
				I_{1}&\leq C(\|\partial^{\alpha_1}[\rho,u,\theta]\|_{L_x^2}+\sum_{|\alpha'|=1}\|[\partial^{\alpha'}(\rho,u,\theta)]^2\|_{L_x^2})
				\Big\||\omega(\beta)\partial^{\alpha-\alpha_1}_{\beta-\beta_1}g|_{\sigma}\Big\|_{L_{x}^{\infty}}\|  \omega(\beta)  h\|_{\sigma}
				\leq C\deltac\|  \omega(\beta)  h\|^{2}_{\sigma}+C\deltac\mathcal{D}_{2,\omega}(t).
			\end{align}
			For $|\alpha_1|\geq2$, using a method similar to that of \eqref{higer-order-N2}, we can obtain the same conclusion as that of \eqref{higer-order-N2}.
			Hence, for $\eta_0>0$, $\delta>0$ and $\varepsilon_{0}>0$ small enough, we deduce from the above   estimates that
			\begin{align*}
				|(\partial^\alpha_\beta[ \Gamma(\frac{M-\mu}{\sqrt{\mu}},g)],\omega^2(\beta)   h)|
				\leq C\deltac\big(\|\omega(\beta)h\|_{\sigma}^2+\mathcal{D}_{2,\omega}(t)\big).
			\end{align*}
			Similar arguments as the above give
			\begin{equation*}
				|(\partial^\alpha_\beta [\Gamma(g,\frac{M-\mu}{\sqrt{\mu}})], \omega^2(\beta)  h)|
				\leq C\deltac\big(\|\omega(\beta)h\|_{\sigma}^2+\mathcal{D}_{2,\omega}(t)\big).
			\end{equation*}
			The estimate \eqref{5.29-t} thus follows from the above two estimates.  By \eqref{5.7} and  the similar
			calculations as \eqref{5.29-t}, we can prove that \eqref{5.30-t}  holds and we omit the details for brevity.
			
			We now proceed to prove \eqref{5.30-1t}. When $\abs{\alpha_1}<\abs{\alpha}$, using the method similar to that of \eqref{5-1}-\eqref{higer-order-N2}, we have
			\begin{align}\notag
				I_1\leq C\deltac \mathcal{D}_{3,\omega}(t)+ C \deltac \norm{\omega(\beta)h}_{\sigma}^2.
			\end{align}
			When $\abs{\alpha_1}=\abs{\alpha}$, more meticulous calculations are required. 
			When $|\alpha|=|\alpha_1|=1$, one has
			\begin{align}\label{5-3.1}
				I_{1}&\leq C\underbrace{\|\partial^{\alpha}[\rhot,\ut,\thetat]\|_{L_{x}^\infty}
					\| \omega(\beta)\partial_{\beta-\beta_1}g\|_{\sigma}\|  \omega(\beta)  h\|_{\sigma}}_{I_{11}} +C\underbrace{\|\partial^{\alpha}[\phi,\psi,\zeta]\|_{L_{x}^\infty}
					\norm{ \omega(\beta)\partial_{\beta-\beta_1}g}_{\sigma} \|  \omega(\beta)  h\|_{\sigma}}_{I_{12}}.
			\end{align}
			For $I_{11}$, one has
			\begin{align}\label{5-3.2}
				I_{11}\leq C \deltab^{\frac12}(1+t)^{-1} \| \omega(\beta)\partial_{\beta-\beta_1}g\|_{\sigma}^2 +C\deltab^{\frac12}\| \omega(\beta)  h\|_{\sigma}^2.
			\end{align}
			Applying Lemma \ref{s-t-s-tt-2},  Lemma \ref{high-der-ori} and the {\it a priori} assumptions \eqref{apa}, one has
			\begin{align}\label{5-3.4}
				I_{12}&\leq\norm{\p^{\alpha}(\phi,\psi,\zeta)}_{L_x^2}^{\frac{1}{2}}\norm{\p^{\alpha}\nabla_x(\phi,\psi,\zeta)}_{L_x^2}^{\frac{1}{2}}\norm{\p_{\beta-\beta_1} g}_{\sigma,\omega} \norm{h}_{\sigma,\omega}+\norm{\p^{\alpha}\nabla_x^2(\phi,\psi,\zeta)}_{L_x^2} \norm{\p_{\beta-\beta_1} g}_{\sigma,\omega}\norm{h}_{\sigma,\omega} \notag\\
				&\leq C\chi(1+t)^{-1}\norm{\p_{\beta-\beta_1}g}_{\sigma,\omega}^2+C\chi\norm{\p_x^2\vm}_{H^1}^2+ C\chi \sum_{1\leq \abs{\gamma}\leq 3} \norm{\p^{\gamma}g}_{\sigma}^2 + C\chi \norm{h}_{\sigma,\omega}^2.
			\end{align}
			Combining \eqref{5-3.1}, \eqref{5-3.2} and \eqref{5-3.4}, one has
			\begin{align}\notag
				I_1 \leq  C\deltac(1+t)^{-1}\mathcal{D}_{2,\omega}+C\deltac\mathcal{D}_{3,\omega} + C\deltac \norm{h}_{\sigma,\omega}^2.
			\end{align}
			When $|\alpha_1|=\abs{\alpha}=2$, using Lemma \ref{s-t-s-tt-2},  Lemma \ref{high-der-ori} and the {\it a priori} assumptions \eqref{apa}, we can obtain
			\begin{align}\label{higer-order-N2-1}
				I_{1}\leq& C\left(\|\partial^{\alpha}(\phi,\psi,\zeta)\|_{L_x^2}+\sum_{|\alpha'|=1}\|[\partial^{\alpha'}(\phi,\psi,\zeta)]^2\|_{L_x^{2}}\right)
				\Big\||\omega(\beta)\partial_{\beta-\beta_1}g|_{\sigma}\Big\|_{L_{x}^{\infty}}\|  \omega(\beta)  h\|_{\sigma}
				\notag\\
				&+ C  \left(\norm{\p^{\alpha}(\rhot,\ut,\thetat)}_{L_x^\infty}+\sum_{|\alpha'|=1}\norm{[\partial^{\alpha'}(\rhot,\ut,\thetat)]^2}_{L_x^\infty} \right)\Big\|\omega(\beta)\partial_{\beta-\beta_1}g\Big\|_{\sigma}\|  \omega(\beta)  h\|_{\sigma} \notag\\
				\leq& C\deltac\|  \omega(\beta)  h\|^{2}_{\sigma}+ C\deltac\mathcal{D}_{3,\omega}(t) +  C\deltac(1+t)^{-1} \mathcal{D}_{2,\omega}.
			\end{align}
			For $|\alpha_1|\geq2$, using the method similar to that of \eqref{higer-order-N2-1}, we can obtain the same conclusion as that of \eqref{higer-order-N2-1}.
			Hence, for $\abs{\alpha}\geq1$, we deduce from the above   estimates that
			$$
			|(\partial^\alpha_\beta[ \Gamma(\frac{M-\mu}{\sqrt{\mu}},g)],\omega^2(\beta)   h)|
			\leq C\deltac\big(\|\omega(\beta)h\|_{\sigma}^2+(1+t)^{-1}\mathcal{D}_{2,\omega}(t)+\mathcal{D}_{3,\omega}(t)\big).
			$$
			Similar arguments as the above give
			\begin{equation*}
				|(\partial^\alpha_\beta [\Gamma(g,\frac{M-\mu}{\sqrt{\mu}})], \omega^2(\beta)  h)|
				\leq C\deltac\big(\|\omega(\beta)h\|_{\sigma}^2+(1+t)^{-1}\mathcal{D}_{2,\omega}(t)+\mathcal{D}_{3,\omega}(t)\big).
			\end{equation*}
			Then we have proved \eqref{5.30-1t}. By \eqref{5.7} and  the similar
			calculations as \eqref{5.30-1t}, we can prove that \eqref{5.30-2t}  holds and we omit the details for brevity.
			This completes the proof of Lemma \ref{lem5.7}.
		\end{proof}
		
		The following estimates are concerned with the nonlinear term $\Gamma(\frac{G}{\sqrt{\mu}},\frac{G}{\sqrt{\mu}})$.
		
		\begin{Lem}\label{lem5.8}
			Under the same assumptions of Theorem \ref{mt}, let $|\alpha|+|\beta|\leq 3$ and $\omega=\omega(\beta)$ be defined by \eqref{1.32}.
			For $\alpha=0$ one has
			\begin{equation}
				\label{5.33}
				|(\partial_\beta[\Gamma(\frac{G}{\sqrt{\mu}},\frac{G}{\sqrt{\mu}})], \omega^2(\beta) h)|
				\leq C\deltac\big(\|\omega(\beta)h\|_{\sigma}^2+\mathcal{D}_{2,\omega}(t)\big)+C\deltab(1+t)^{-\frac{3}{2}},
			\end{equation}
			and
			\begin{equation}
				\label{5.34}
				|( \Gamma(\frac{G}{\sqrt{\mu}},\frac{G}{\sqrt{\mu}}),h)|
				\leq C\deltac\big(\|h\|_{\sigma}^2+\mathcal{D}_{2}(t)\big)+C\deltab(1+t)^{-\frac{3}{2}}.
			\end{equation}
			For $\abs{\alpha}\geq1$, one has
			\begin{align}
				\label{5.33-1}
				&|(\partial^\alpha_\beta[\Gamma(\frac{G}{\sqrt{\mu}},\frac{G}{\sqrt{\mu}})], \omega^2(\beta) h)| \leq C\deltac\Big(\|\omega(\beta)h\|_{\sigma}^2+(1+t)^{-1}\mathcal{D}_{2,\omega}(t)+\mathcal{D}_{3,\omega}(t)\Big)+C\deltab(1+t)^{-\frac{3}{2}-\abs{\alpha}},
			\end{align}
			and
			\begin{align}
				\label{5.34-1}
				|(\partial^\alpha [\Gamma(\frac{G}{\sqrt{\mu}},\frac{G}{\sqrt{\mu}})],h)|
				\leq C\deltac\Big(\|h\|_{\sigma}^2+(1+t)^{-1}\mathcal{D}_{2}(t)+\mathcal{D}_{3}(t)\Big)+C\deltab(1+t)^{-\frac{3}{2}-\abs{\alpha}}.
			\end{align}
		\end{Lem}
		\begin{proof}
			Recalling that  $G=\bar{G}_0+\sqrt{\mu}g$, we   see
			\begin{equation}
				\label{5.35}
				\Gamma(\frac{G}{\sqrt{\mu}},\frac{G}{\sqrt{\mu}})=\Gamma(\frac{\bar{G}_0}{\sqrt{\mu}},\frac{\bar{G}_0}{\sqrt{\mu}})
				+\Gamma(\frac{\bar{G}_0}{\sqrt{\mu}},g)+\Gamma(g,\frac{\bar{G}_0}{\sqrt{\mu}})+\Gamma(g,g).
			\end{equation}
			For the first term in \eqref{5.35}, we have from the similar arguments as \eqref{5.31} that
			\begin{align}
				\label{5.36}
				&|(\partial^\alpha_\beta [\Gamma(\frac{\bar{G}_0}{\sqrt{\mu}},\frac{\bar{G}_0}{\sqrt{\mu}})], \omega^2(\beta)h)|\leq C\sum_{|\alpha_1|\leq|\alpha|}\sum_{ |\bar{\beta}|\leq|\beta_1|\leq|\beta|}
				\underbrace{\int_{\Omega}|\mu^{\epsilon}\partial^{\alpha_1}_{\bar{\beta}}(\frac{\bar{G}_0}{\sqrt{\mu}})|_2| \omega(\beta) \partial^{\alpha-\alpha_1}_{\beta-\beta_1}(\frac{\bar{G}_0}{\sqrt{\mu}})|_{\sigma}|  \omega(\beta)h|_{\sigma}\,dx}_{I_{2}}.
			\end{align}
			Note that $ |\alpha_1|\leq|\alpha|\leq 3$ in \eqref{5.36} due to the fact that $|\alpha|+|\beta|\leq 3$.
			For $\abs{\alpha-\alpha_1}\le\abs{\alpha_1}$,  one has from \eqref{5.36} that
			\begin{align}\label{5-4-1}
				I_{2}&\leq C \deltab^{-1} \norm{\mu^{\epsilon}\p_{\bar\beta}^{\alpha_1}(\frac{\bar G_0}{\sqrt{\mu}})}_{L^2}^2\norm{\omega(\beta)\p_{\beta-\beta_1}^{\alpha-\alpha_1}(\frac{\bar G_0}{\sqrt{\mu}})}_{L^\infty}^2 +C \deltab \norm{\omega(\beta) h}_{\sigma}^2.
			\end{align}
			By  Lemma \ref{2026-4-24-1} and \eqref{5-4-1}, it holds
			\begin{align*}
				&I_2 \leq C\deltab (1+t)^{-\frac32}+C\deltac(\norm{\omega(\beta)h}_{\sigma}^2+\mathcal{D}_{2,\omega}(t)), \quad \text{for}\quad \alpha=0,\\
				&I_2 \leq C\deltab (1+t)^{-\frac32-\abs{\alpha}}+C\deltac\Big(\norm{\omega(\beta)h}_{\sigma}^2+(1+t)^{-1}\mathcal{D}_{2,\omega}(t)+\mathcal{D}_{3,\omega}(t)\Big),\quad \text{for} \quad \alpha\ge1.
			\end{align*}
			For $\abs{\alpha-\alpha_1}\ge\abs{\alpha_1}$, it can be treated in the same way. By the above estimate and \eqref{5.36}, we have
			\begin{align*}
				&|(\partial^\alpha_\beta [\Gamma(\frac{\bar{G}_0}{\sqrt{\mu}},\frac{\bar{G}_0}{\sqrt{\mu}})], \omega^2(\beta)h)|
				\leq  C\deltab (1+t)^{-\frac32}+C\deltac(\norm{\omega(\beta)h}_{\sigma}^2+\mathcal{D}_{2,\omega}(t)), \quad \text{for}\quad \alpha=0,\\
				&|(\partial^\alpha_\beta [\Gamma(\frac{\bar{G}_0}{\sqrt{\mu}},\frac{\bar{G}_0}{\sqrt{\mu}})], \omega^2(\beta)h)|\leq C\deltab (1+t)^{-\frac32-\abs{\alpha}}+C\deltac\Big(\norm{\omega(\beta)h}_{\sigma}^2+(1+t)^{-1}\mathcal{D}_{2,\omega}(t)+\mathcal{D}_{3,\omega}(t)\Big),\quad\text{for} \;\; \alpha\ge1.\notag
			\end{align*}
			
			For the second term in \eqref{5.35}, by \eqref{5.8}, we can obtain
			\begin{align}
				\label{5.40c}
				&|(\partial^\alpha_\beta [\Gamma(\frac{\bar{G}_0}{\sqrt{\mu}},g)], \omega^2(\beta)h)|
				\leq C\sum_{ |\alpha_1|\leq|\alpha|}\sum_{ |\bar{\beta}|\leq|\beta_1|\leq|\beta|}
				\underbrace{\int_{\Omega}|\mu^{\epsilon}\partial^{\alpha_1}_{\bar{\beta}}(\frac{\bar{G}_0}{\sqrt{\mu}})|_2| \omega(\beta) \partial^{\alpha-\alpha_1}_{\beta-\beta_1}g|_{\sigma}|  \omega(\beta)h|_{\sigma}\,dx}_{I_{3}}.
			\end{align}
			From the estimate of $\bar G_0$ in Lemma \ref{2026-4-24-1}, one has
			\begin{align*}
				\abs{\mu^{\epsilon} \p_{\bar\beta}^{\alpha_1}(\frac{\bar G_0}{\sqrt{\mu}})}_2\approx C\deltab\tilde D_{-\frac12}\abs{\mu^{\epsilon} \p_{\bar\beta}^{\alpha_1}(\frac{M-\mu}{\sqrt{\mu}})}_2.
			\end{align*}
			It follows that the estimate of $I_3$ is similar to that of 
			$I_1$ \eqref{5.8}, and using the same method, we obtain from \eqref{5.40c} that
			\begin{align}
				\label{5.41c}
				|(\partial^\alpha_\beta [\Gamma(\frac{\bar{G}_0}{\sqrt{\mu}},g)], \omega^2(\beta)h)|
				\leq C\deltab\left(\|\omega(\beta)h\|_{\sigma}^2+(1+t)^{-1}\mathcal{D}_{2,\omega}(t)\right).
			\end{align}
			Similar arguments as for obtaining \eqref{5.41c} imply
			\begin{align*}
				|(\partial^\alpha_\beta [\Gamma(g,\frac{\bar{G}_0}{\sqrt{\mu}})], \omega^2(\beta)h)|
				\leq C \deltab\big(\|\omega(\beta)h\|_{\sigma}^2+(1+t)^{-1}\mathcal{D}_{2,\omega}(t)\big).
			\end{align*}
			By \eqref{5.8}, for $\alpha=0$, we can arrive at
			\begin{align*}
				|(\partial_\beta [\Gamma(g,g)],\omega^2(\beta)h)|
				&\leq C\sum_{ |\bar{\beta}|\leq|\beta_1|\leq|\beta|}
				\int_{\mathbb R}|\mu^{\epsilon}\partial_{\bar{\beta}}g|_2| \omega(\beta) \partial_{\beta-\beta_1}g|_{\sigma}|  \omega(\beta)h|_{\sigma}\,dx
				\leq C\deltac\|\omega(\beta)h\|_{\sigma}^2+C\deltac\mathcal{D}_{2,\omega}(t).
			\end{align*}
			For $\abs{\alpha}\geq1$, one has
			\begin{align*}
				|(\partial^\alpha_\beta [\Gamma(g,g)],\omega^2(\beta)h)|
				&\leq C\sum_{ |\alpha_1|\leq|\alpha|}\sum_{ |\bar{\beta}|\leq|\beta_1|\leq|\beta|}
				\int_{\mathbb R}|\mu^{\epsilon}\partial^{\alpha_1}_{\bar{\beta}}g|_2| \omega(\beta) \partial^{\alpha-\alpha_1}_{\beta-\beta_1}g|_{\sigma}|  \omega(\beta)h|_{\sigma}\,dx
				\nonumber\\
				&\leq C\deltac\|\omega(\beta)h\|_{\sigma}^2+C\deltac\mathcal{D}_{3,\omega}(t).
			\end{align*}
			Plugging all the above estimates back to \eqref{5.36}, one has \eqref{5.33} and \eqref{5.33-1}. We can use the similar calculations for obtaining \eqref{5.34} and \eqref{5.34-1}. Therefore, the proof of Lemma \ref{lem5.8} is completed.
		\end{proof}
		
		\medskip
		\noindent {\bf Acknowledgment:}\,
		The research of Renjun Duan was partially supported by the grant from the National Natural Science Foundation of China (Project No.~12425109). The  research of Feimin Huang was partially supported by the Strategic Priority Research Program of
		the Chinese Academy of Sciences (No. XDB0510201) and National Natural Sciences Foundation of China (Project  No.
		12288201).

		\medskip
		\noindent{\bf Data availability:} The manuscript contains no associated data.

		\medskip
		\noindent{\bf Conflict of Interest:} The authors declare that they have no conflict of interest.


		\normalsize
	\end{document}